\documentclass[10pt,letterpaper]{amsart}
\usepackage{dslihead}

\title{On close fields and the local Langlands correspondence}

\author{Siyan Daniel Li-Huerta}
\address{Department of Mathematics\\Massachusetts Institute of Technology\\77 Massachusetts Avenue\\Cambridge, MA 02139}
\email{sdlh@mit.edu}

\usepackage{amsmath}

\swapnumbers
\theoremstyle{plain}

\newtheorem*{thm*}{Theorem}
\newtheorem{lem}[subsection]{Lemma}
\newtheorem*{lem*}{Lemma}
\newtheorem{prop}[subsection]{Proposition}
\newtheorem*{prop*}{Proposition}

\newtheorem*{conj*}{Conjecture}
\newtheorem{cor}[subsection]{Corollary}
\newtheorem*{cor*}{Corollary}

\newtheorem*{thmA}{Theorem A}
\newtheorem*{thmB}{Theorem B}

\theoremstyle{definition}

\newtheorem*{defn*}{Definition}

\theoremstyle{remark}

\newtheorem*{rem*}{Remark}
\newtheorem*{rems*}{Remarks}

\newcommand{\sHom}{\sH\!\mathrm{om}}
\DeclareMathOperator{\s}{s}
\newcommand{\cHck}{\cH\mathrm{ck}}
\DeclareMathOperator{\bd}{bd}
\DeclareMathOperator{\ULA}{ULA}
\DeclareMathOperator{\Sat}{Sat}
\DeclareMathOperator{\sw}{sw}
\DeclareMathOperator{\LocSys}{LocSys}

\DeclareMathOperator{\finproj}{fp}
\DeclareMathOperator{\bigboxtimes}{{\mathlarger{\mathlarger{\boxtimes}}}}
\DeclareMathOperator{\bigast}{{\mathlarger{\mathlarger{\ast}}}}

\begin{document}

\begin{abstract}
We prove that Fargues--Scholze's semisimplified local Langlands correspondence (for quasisplit groups) with $\ov\bF_\ell$-coefficients is compatible with Deligne and Kazhdan's philosophy of close fields. From this, we deduce that the same holds with $\ov\bQ_\ell$-coefficients after restricting to wild inertia, addressing questions of Gan--Harris--Sawin and Scholze. The proof involves constructing a moduli space of nonarchimedean local fields and then extending Fargues--Scholze's work to this context.
\end{abstract}

\maketitle
\vspace{-.5cm}
\tableofcontents

\section*{Introduction}
\subsection*{Close fields}
In number theory, there is a heuristic that
\begin{align*}\label{eqn:slogan}
  \mbox{``as absolute ramification tends to }\infty\mbox{, }p\mbox{-adic fields tend to function fields.''}\tag{$\star$}
\end{align*}
For example, let $E$ be a $p$-adic field with residue field $\bF_q$ and absolute ramification index $e$, and write $E'$ for the function field $\bF_q\lp{t}$. Write $\Ga$ and $\Ga'$ for the absolute Galois groups of $E$ and $E'$, respectively. Building on work of Krasner \cite{Kra50}, Deligne \cite{Del84} showed that there exists a canonical\footnote{Everything depends on a choice of uniformizer (up to multiplication by $1+\fp^e$, where $\fp$ denotes the maximal ideal of the ring of integers) of $E$. However, we ignore this in the introduction.} isomorphism (up to conjugation)
\begin{align*}
\Ga/I^e\cong\Ga'/I'^e,
\end{align*}
where $I^e$ and $I'^e$ denote the $e$-th ramification subgroup of $\Ga$ and $\Ga'$, respectively. Consequently, we get a canonical bijection between representations of $\Ga$ that are trivial on $I^e$ and representations of $\Ga'$ that are trivial on $I'^e$.

This bijection also applies to Galois representations with extra structure. More precisely, Deligne's isomorphism induces a correspondence $G/E\leftrightarrow G'/E'$ between isomorphism classes of quasisplit\footnote{One can even extend this to \emph{all} connected reductive groups that split over an $e$-ramified extension \cite[Lemma 5.1]{Gan19}, though we do not work in this generality.} connected reductive groups that split over an $e$-ramified extension. Under this correspondence, Deligne's isomorphism also induces a canonical bijection
\begin{align*}
\left\{{\begin{tabular}{c}
  $L$-parameters $\rho$\\
  for $G$ with $\rho|_{I^e}=1$
  \end{tabular}}\right\}\longleftrightarrow\left\{{\begin{tabular}{c}
  $L$-parameters $\rho'$\\
  for $G'$ with $\rho'|_{I'^e}=1$
  \end{tabular}}\right\}.
\end{align*}

Something similar happens for representations of $p$-adic groups. Fix a positive integer $n$. One can compatibly define $n$-th congruence subgroups $K^n$ and $K'^n$ of $G(E)$ and $G'(E')$, respectively, by using the corresponding Chevalley group over $\bZ$. Let $\ell\neq p$ be a prime, and let $\La$ be either $\ov\bF_\ell$ or $\ov\bQ_\ell$. Extending results of Kazhdan \cite{Kaz86b}, Ganapathy \cite{Gan24} showed that there exists a canonical isomorphism\footnote{While Kazhdan and Ganapathy work over $\bC\approx\ov\bQ_\ell$, one can deduce the result over $\ov\bF_\ell$ using their work. Also, strictly speaking, we use an isomorphism whose construction differs from that of \cite{Gan24}, and we do not know whether these isomorphisms are equal. See Appendix \ref{s:congruence}.} between the bi-invariant Hecke algebras
\begin{align*}
\cH(G(E),K^n)_\La\cong\cH(G'(E'),K'^n)_\La
\end{align*}
over $\La$ \emph{whenever $e$ is large enough}. Consequently, we get a canonical bijection
\begin{align*}
\left\{{\begin{tabular}{c}
  smooth irreps $\pi$ of $G(E)$\\
  over $\La$ with $(\pi)^{K^n}\neq0$
  \end{tabular}}\right\}\longleftrightarrow\left\{{\begin{tabular}{c}
  smooth irreps $\pi'$ of $G'(E')$\\
  over $\La$ with $(\pi')^{K'^n}\neq0$
  \end{tabular}}\right\}.
\end{align*}

\subsection*{Local Langlands}
The Langlands program predicts a relationship between automorphic forms and Galois representations. For example, in the context of connected reductive groups $G$ over nonarchimedean local fields $E$, recent groundbreaking work of Fargues--Scholze \cite{FS21} yields a natural map
  \begin{align*}
\LLC^{\semis}_G:\left\{{\begin{tabular}{c}
    smooth irreps\\
    of $G(E)$ over $\La$
  \end{tabular}}\right\}\lra\left\{{\begin{tabular}{c}
  semisimple $L$-parameters\\
  for $G$ over $\La$
  \end{tabular}}\right\}.
  \end{align*}

  The goal of this paper is to prove that $\LLC^{\semis}_G$ is compatible with Heuristic \eqref{eqn:slogan}. Our first main theorem concerns the case of $\ov\bF_\ell$-coefficients.
\begin{thmA}
  Let $G'$ be a quasisplit connected reductive group over $E'$ that splits over an $n$-ramified extension. There exists an integer $d\geq n$ such that, for all $p$-adic fields $E$ with residue field $\bF_q$ and absolute ramification index $e\geq d$,
  \begin{enumerate}[i)]
  \item the map $\LLC^{\semis}_G$ restricts to a map
  \begin{align*}
\LLC^{\semis}_G:\left\{{\begin{tabular}{c}
    smooth irreps $\pi$ of $G(E)$\\
    over $\ov\bF_\ell$ with $(\pi)^{K^n}\neq0$
  \end{tabular}}\right\}\lra\left\{{\begin{tabular}{c}
  semisimple $L$-parameters $\rho$\\
  for $G$ over $\ov\bF_\ell$ with $\rho|_{I^d}=1$
  \end{tabular}}\right\}.
  \end{align*}
\item the square
    \begin{align*}
    \xymatrix{\left\{{\begin{tabular}{c}
    smooth irreps $\pi$ of $G(E)$\\
    over $\ov\bF_\ell$ with $(\pi)^{K^n}\neq0$
  \end{tabular}}\right\}\ar[r]^-{\LLC_G^{\semis}}\ar@{<->}[d] & \left\{{\begin{tabular}{c}
  semisimple $L$-parameters $\rho$\\
  for $G$ over $\ov\bF_\ell$ with $\rho|_{I^d}=1$
  \end{tabular}}\right\}\ar@{<->}[d] \\
  \left\{{\begin{tabular}{c}
    smooth irreps $\pi'$ of $G'(E')$\\
     over $\ov\bF_\ell$ with $(\pi')^{K'^n}\neq0$
  \end{tabular}}\right\}\ar[r]^-{\LLC_{G'}^{\semis}} & \left\{{\begin{tabular}{c}
  semisimple $L$-parameters $\rho'$\\
  for $G'$ over $\ov\bF_\ell$ with $\rho'|_{I'^d}=1$
  \end{tabular}}\right\}}
    \end{align*}
    commutes.
  \end{enumerate}
\end{thmA}
We actually prove a refinement of Theorem A on the level of \emph{excursion algebras}; see Theorem \ref{ss:spreadoutexcursion}. Using this refinement, we prove our second main theorem in the case of $\ov\bQ_\ell$-coefficients.
\begin{thmB}
  Let $G'$ be a quasisplit connected reductive group over $E'$ that splits over an $n$-ramified extension. There exists an integer $d\geq n$ such that, for all $p$-adic fields $E$ with residue field $\bF_q$ and absolute ramification index $e\geq d$,
  \begin{enumerate}[i)]
  \item the map $\LLC^{\semis}_G$ restricts to a map
  \begin{align*}
\LLC^{\semis}_G:\left\{{\begin{tabular}{c}
    smooth irreps $\pi$ of $G(E)$\\
    over $\ov\bQ_\ell$ with $(\pi)^{K^n}\neq0$
  \end{tabular}}\right\}\lra\left\{{\begin{tabular}{c}
  semisimple $L$-parameters $\rho$\\
  for $G$ over $\ov\bQ_\ell$ with $\rho|_{I^d}=1$
  \end{tabular}}\right\}.
  \end{align*}
\item the square
    \begin{align*}
    \xymatrix{\left\{{\begin{tabular}{c}
    smooth irreps $\pi$ of $G(E)$\\
    over $\ov\bQ_\ell$ with $(\pi)^{K^n}\neq0$
  \end{tabular}}\right\}\ar[r]^-{\LLC_G^{\semis}}\ar@{<->}[d] & \left\{{\begin{tabular}{c}
  semisimple $L$-parameters $\rho$\\
  for $G$ over $\ov\bQ_\ell$ with $\rho|_{I^d}=1$
  \end{tabular}}\right\}\ar@{<->}[d] \\
  \left\{{\begin{tabular}{c}
    smooth irreps $\pi'$ of $G'(E')$\\
     over $\ov\bQ_\ell$ with $(\pi')^{K'^n}\neq0$
  \end{tabular}}\right\}\ar[r]^-{\LLC_{G'}^{\semis}} & \left\{{\begin{tabular}{c}
  semisimple $L$-parameters $\rho'$\\
  for $G'$ over $\ov\bQ_\ell$ with $\rho'|_{I'^d}=1$
  \end{tabular}}\right\}}
    \end{align*}
    commutes \emph{after restricting to the wild inertia subgroup}.
  \end{enumerate}
\end{thmB}
This addresses (a generalization to quasisplit connected reductive groups of) a question of Gan--Harris--Sawin \cite[Conjecture 11.7]{GHSBP21}\footnote{Gan--Harris--Sawin work with the Genestier--Lafforgue correspondence \cite{GL17} for $G'$, but this agrees with the Fargues--Scholze correspondence by \cite[Theorem C]{LH23}.} and Scholze \cite[p.~470]{Sch14}.
\begin{rems*}\hfill
  \begin{enumerate}[1)]
  \item Since $\LLC^{\semis}_G$ equals the expected local Langlands correspondence when $G$ is a torus \cite[p.~331]{FS21}, in this case Theorem A and Theorem B follow from work of Aubert--Varma \cite[Theorem 1.2.1(iv)]{AV24}.

  \item One can ask whether our main theorems hold for $d=n$. The answer for part i) is already negative when $G$ is a wildly ramified induced torus \cite[Sec. 7.1]{MP19}. When $G$ is split but not a torus, the answer for part i) is also negative because our congruence subgroups arise from hyperspecial subgroups and hence are ``too big.'' However, if one instead uses congruence subgroups arising from Iwahori subgroups, then one expects the answer to be positive when $G$ is split \cite[Conjecture 11.7]{GHSBP21}.
  \end{enumerate}
\end{rems*}
\subsection*{Idea of proof}
We begin by introducing a \emph{moduli space of nonarchimedean local fields}, which allows us to conceptualize Heuristic \eqref{eqn:slogan}. This starts with the observation that, by Krasner's lemma, there exists an enumeration $\{E_i\}_{i\in\bN}$ of all $p$-adic fields with residue field $\bF_q$ (up to isomorphism), and as $i$ tends to infinity, the absolute ramification index of $E_i$ also tends to infinity. We use this to construct a topological ring $E$ such that
\begin{enumerate}[$\bullet$]
\item the topological space $\abs{\Spa{E}}$ is naturally homeomorphic to the one-point compactification $\bN\cup\{\infty\}$ of the discrete space $\bN$,
\item for all $i$ in $\bN$, the residue field of $\Spa{E}$ at $i$ is $E_i$, and the residue field of $\Spa{E}$ at $\infty$ is $E'\coloneqq\bF_q\lp{t}$.
\end{enumerate}
Therefore $\Spa{E}$ can be regarded as the moduli space of nonarchimedean local fields with residue field $\bF_q$. We prove that finite \'etale $E$-algebras are spread out from finite \'etale $E'$-algebras in a way that is compatible with Deligne's isomorphism, so the finite \'etale site of $\Spa{E}$ geometrizes the Galois side in our main theorems.

What about the automorphic side? First, using the above results, we spread out quasisplit connected reductive groups $G'$ over $E'$ to reductive group schemes $G$ over $E$ in a way that is compatible with Deligne's isomorphism. Next, we refine our construction of $E$ to produce a \emph{ring topological space} $\bE$ over $\bN\cup\{\infty\}$ that recovers $E$ after taking global sections. We use this to construct a group topological space $G(\bE)$ over $\bN\cup\{\infty\}$ that interpolates between $p$-adic groups for different nonarchimedean local fields, as well as to define a compact open group subspace $\bK^n\subseteq G(\bE)$ over $\bN\cup\{\infty\}$ that interpolates between $n$-th congruence subgroups. Finally, we introduce and study a notion of \emph{smooth representations} of $G(\bE)$ with $\La$-coefficients, which are sheaves of $\La$-modules on $\bN\cup\{\infty\}$ equipped with certain continuous actions of $G(\bE)$. An important example is the compactly supported induction $\cInd^{G(\bE)}_{\bK^n}\La$; crucially, we prove that its endomorphism sheaf
\begin{align*}
\ul\End_{G(\bE)}(\cInd_{\bK^n}^{G(\bE)}\La)
\end{align*}
is isomorphic to the constant sheaf $\ul{\cH(G'(E'),K'^n)_\La}$ on $\bN\cup\{\infty\}$ in a way that is compatible with Ganapathy's isomorphism.

From here, we geometrize the above theory as in work of Fargues--Scholze \cite{FS21}. Write $\cO\subseteq E$ for the subring of powerbounded elements. We begin by constructing a generalization of Witt vectors for $\cO$-algebras that interpolates between $\cO_i$-typical Witt vectors for all $i$ in $\bN$, where $\cO_i$ denotes the ring of integers of $E_i$. Next, we use this to define and study a generalization $X_S\ra\Spa{E}$ of the relative Fargues--Fontaine curve for any perfectoid space $S$ over $\ul{\bN\cup\{\infty\}}_{\bF_q}$. Restricting to perfectoid test objects $S$ puts us in the world of \emph{diamonds} \cite{Sch17},\footnote{Via the tilting equivalence and almost purity, perfectoid spaces provide \emph{another} example of Heuristic \eqref{eqn:slogan}. Hence their appearance here should be unsurprising. See \cite[p.~470]{Sch14}.} and here we define a space 
\begin{align*}
\Div^1_X\ra\ul{\bN\cup\{\infty\}}_{\bF_q}
\end{align*}
of degree $1$ relative effective Cartier divisors on $X_S$. By reducing to the situation considered in Fargues--Scholze \cite{FS21}, we prove a version of Drinfeld's lemma for $\Div^1_X$.

Continuing as in work of Fargues--Scholze \cite{FS21}, we introduce the moduli stack
\begin{align*}
  \Bun_G\ra\ul{\bN\cup\{\infty\}}_{\bF_q}
\end{align*}
of $G$-bundles on $X_S$. To study $\Bun_G$, we reduce many statements to the case of vector bundles by proving a Tannakian description of $G$-torsors, which is subtle because $E$ is non-noetherian. We then handle the case of vector bundles by using our generalization of Witt vectors to prove an \emph{explicit} version of Lubin--Tate theory over $E$, generalizing arguments of Fontaine \cite{Fon77} and Fargues--Fontaine \cite{FF18}.

At this point, we prove a version of the geometric Satake equivalence for $G/E$ with $\bZ_\ell$-coefficients. This requires working over non-noetherian rings like
\begin{align*}
\Cont(\bN\cup\{\infty\},\bZ_\ell),
\end{align*}
which seem inaccessible to Tannakian identification results, so we instead use an elaborate argument to reduce to the situation considered in Fargues--Scholze \cite{FS21}. Our geometric Satake equivalence lets us spread out geometric Hecke operators $T_V'$ from the stack $\Bun_{G'}$ studied in Fargues--Scholze \cite{FS21} to our stack $\Bun_G$.

We can now put everything together to prove Theorem A. Recall that $\LLC^{\semis}_{G'}$ is defined using \emph{excursion algebras}, which are certain $\ov\bZ_\ell$-algebras with canonical generators indexed by tuples $(J,V,x,\xi,\ga'_\bullet)$, where $J$ is a finite set, $V$ is a representation of the $J$-th power $(\prescript{L}{}G')^J$ of the $L$-group, $x$ and $\xi$ are elements of $V$ and $V^\vee$, respectively, that are fixed by the image of the diagonal map $\wh{G}'\hra(\prescript{L}{}G')^J$, and $\ga_\bullet'$ is a $J$-tuple of elements of the absolute Weil group of $E'$. Any such tuple induces an endomorphism of $A'\coloneqq\cInd_{K'^n}^{G'(E')}\ov\bF_\ell$ via the composition
\begin{align*}\label{eqn:excursion}
A'=T_{\mathbf{1}}'(A')\lra^x T_V'(A')\lra^{\ga_\bullet'}T_V'(A')\lra^\xi T_{\mathbf{1}}'(A')=A',\tag{$\times$}
\end{align*}
where we view $A'$ as an \'etale $\ov\bF_\ell$-sheaf on the open locus of trivial $G'$-bundles $\Bun^1_{G'}\subseteq\Bun_{G'}$. Using the geometric Hecke operators $T_V$ on our stack $\Bun_G$, we can spread out \eqref{eqn:excursion} to an endomorphism of $A\coloneqq\cInd_{\bK^n}^{G(\bE)}\ov\bF_\ell$ via the composition
\begin{align*}
A=T_{\mathbf{1}}(A)\lra^x T_V(A)\lra^{\ga_\bullet}T_V(A)\lra^\xi T_{\mathbf{1}}(A)=A,
\end{align*}
where we similarly view $A$ as an \'etale $\ov\bF_\ell$-sheaf on the open locus of trivial $G$-bundles $\Bun_G^1\subseteq\Bun_G$. Then part i) follows from compactness properties of $T_V$, and part ii) follows from our isomorphism
\begin{align*}
\ul{\cH(G'(E'),K'^n)_{\ov\bF_\ell}}\ra^\sim\ul\End_{G(\bE)}(A).
\end{align*}

Finally, let us sketch the proof of Theorem B. Work of Bernstein \cite{Ber84} reduces us to proving part ii) for individual irreducible cuspidal representations $\pi'$ of $G'(E')$ over $\ov\bQ_\ell$. After twisting $\pi'$ by an unramified character, we can assume that $\pi'$ is defined over $\ov\bZ_\ell$, which implies that the corresponding irreducible smooth representation $\pi$ of $G(E)$ is also defined over $\ov\bZ_\ell$. By applying (the excursion algebra version of) Theorem A, the desired result then follows from comparing $\LLC^{\semis}_{G'}(\pi')$ and $\LLC^{\semis}_G(\pi)$ mod $\ell$ and using the fact that the kernel of $\wh{G}(\ov\bZ_\ell)\ra\wh{G}(\ov\bF_\ell)$ is (ind-)pro-$\ell$.

\subsection*{Outline}
In \S\ref{s:smoothrepresentations}, we develop a theory of smooth representations over profinite sets. In \S\ref{s:localfields}, we introduce our moduli space of nonarchimedean local fields and study its Galois theory. In \S\ref{s:FFcurves}, we construct a generalization of Witt vectors over $\cO$, which we use to define a version of relative Fargues--Fontaine curves. In \S\ref{s:LT}, we prove an explicit version of Lubin--Tate theory over $E$, and we apply this in \S\ref{s:vectorbundles} to study vector bundles on Fargues--Fontaine curves in our context.

In \S\ref{s:reductivegroups}, we introduce the connected reductive group $G$ over $E$, and we prove that Hecke algebras for the resulting family of $p$-adic groups over $\bN\cup\{\infty\}$ geometrize Ganapathy's isomorphism. (Actually, we use an isomorphism whose construction differs from that of Ganapathy; see Appendix \ref{s:congruence}.) We also define the moduli stack $\Bun_G$. In \S\ref{s:Grassmannians}, we introduce Beilinson--Drinfeld affine Grassmannians in our context, and in \S\ref{s:Satake} we prove a version of the geometric Satake equivalence for $G/E$. Finally, we put everything together in \S\ref{s:final} to prove Theorem A and Theorem B.

\subsection*{Notation}
We endow all finite free $\bZ$-modules with the discrete topology. For all rings $R$ and affine groups $G$ over $R$, write $\Rep{G}$ for the category of representations of $G$ on finite free $R$-modules. We view all functors between derived categories as derived functors, and starting in \S\ref{s:Satake}, we view derived categories as $\infty$-categories.

Starting in \S\ref{s:FFcurves}, we freely use definitions from perfectoid geometry as in \cite{Sch17} and \cite{FS21}. We work over $\Spd\bF_q$ (unless otherwise specified) until \ref{ss:Drinfeldapplied}, after which we work over $\Spd\ov\bF_q$. For all small v-stacks $Z$ on $\Perf_{\bF_q}$ and rings $\La$ that are $\ell$-power torsion, write $\Shv_{\et}(Z,\La)\subseteq D_{\et}(Z,\La)$ for the heart of the standard $t$-structure. Write $D_{\lc}(Z,\La)$ for the full subcategory of $D_{\et}(Z,\La)$ given by objects that are \'etale-locally constant with perfect fibers, and write $\LocSys(Z,\La)$ for the full subcategory of $D_{\et}(Z,\La)$ given by objects that are \'etale-locally constant with finite projective fibers.

\subsection*{Acknowledgements}
The author is indebted to Peter Scholze for suggesting that \cite{FS21} should extend to families of local fields over $\bN\cup\{\infty\}$. The author would also like to thank Radhika Ganapathy for answering many questions about her work, as well as for sharing a revised version of \cite{Gan24}. Finally, the author would like to thank Michael Harris for his encouragement, and to thank Tasho Kaletha and Zhiwei Yun for helpful discussions.

During the completion of this work, the author was partially supported by NSF Grant \#2303195. The author also benefitted greatly from visiting the Hausdorff Research Institute for Mathematics, supported by the DFG under Germany's Excellence Strategy -- EXC-2047/1 -- 390685813.

\section{Smooth representations over profinite sets}\label{s:smoothrepresentations}
In this section, we develop a theory of \emph{smooth representations} for certain group topological spaces over a profinite set. The group topological spaces that we consider are a relative version of locally compact totally disconnected (or \emph{lctd}) topological groups. Examples will include the analogues of absolute Galois groups in \S\ref{s:localfields}, absolute Weil groups in \S\ref{s:FFcurves}, and $p$-adic groups in \S\ref{s:reductivegroups}.

This section is elementary and self-contained; although nobody would want to, it can be read independently of the other sections.

\subsection{}\label{ss:profinitesetsheaf}
We start with some recollections about topological spaces $S$. Recall that the \'etal\'e space construction yields an equivalence from the category of sheaves on $S$ to the category of \'etale topological spaces over $S$, with a quasi-inverse given by the sheaf of continuous sections \cite[II.6, Corollary 3]{MM94}. For any topological space $X$ over $S$ and open subset $U$ of $S$, write $X(U)$ for the set of continuous sections of $X_U\ra U$.

For the rest of this section, assume that $S$ is profinite. Then compact open subsets $U$ of $S$ form a basis for the topology of $S$, and because $U$ is also profinite, every open cover of $U$ splits. Hence when evaluating on this basis, it suffices to check the sheaf condition for $S$ on pairs of disjoint compact open subsets of $S$.

\subsection{}\label{ss:grouphypotheses}
We study the following generalization of lctd groups in the sense of \cite[2.2]{Kot05}.
\begin{defn*}
  Let $G$ be a group topological space over $S$. We say that $G$ is \emph{lctd over $S$} if $G$ is Hausdorff and
  \begin{enumerate}[a)]
\item there exists a family $\{K^\al\}_\al$ of compact open group subspaces of $G$ over $S$ that is cofinal among neighborhoods of the identity section in $G$,
\item for all $s$ in $S$ and $g_s$ in the fiber $G_s$, there exists a neighborhood $U$ of $s$ and $g$ in $G(U)$ such that $g(s)=g_s$.
\end{enumerate}
\end{defn*}
For the rest of this section, let $G$ be a group topological space over $S$ that is lctd over $S$, and let $\{K^\al\}_\al$ be a family of compact open group subspaces of $G$ over $S$ satisfying Definition \ref{ss:grouphypotheses}.a).

For all $s$ in $S$, evaluation at $s$ yields a homomorphism $\ev_s:\varinjlim_UG(U)\ra G_s$, where $U$ runs over compact neighborhoods of $s$. Note that Definition \ref{ss:grouphypotheses}.b) is equivalent to $\ev_s$ being surjective.

\subsection{}\label{ss:evaluationkernel}
While $\ev_s$ need not be bijective, it does satisfy the following relationship with open group subspaces $H$ of $G$. Note that any such $H$ is Hausdorff and satisfies Definition \ref{ss:grouphypotheses}.b).
\begin{lem*}
  For all $s$ in $S$, the kernel of $\ev_s$ lies in $\varinjlim_UH(U)$, where $U$ runs over compact neighborhoods of $s$. Consequently, $\ev_s$ induces a bijection
  \begin{align*}
   \varinjlim_UG(U)/\varinjlim_UH(U)\ra^\sim G_s/H_s.
  \end{align*}
\end{lem*}
\begin{proof}
Let $g$ be in $\ker\ev_s$, and view $g$ as a continuous map $g:U\ra G$ for some compact neighborhood $U$ of $s$. Then $g^{-1}(H)$ is a neighborhood of $s$, so it contains a compact neighborhood $U'$ of $s$. Hence $g|_{U'}$ lies in $H(U')$. Finally, the last statement follows from Definition \ref{ss:grouphypotheses}.b) applied to $G$ and $H$.
\end{proof}

\subsection{}\label{ss:openquotient}
Since $H$ is open in $G$, we expect the ``quotient'' of $G$ by $H$ to be ``discrete.'' We make this precise as follows. Write $G$ and $H$ for their associated sheaves on $S$, and consider the quotient $G/H$ of sheaves on $S$.
\begin{prop*}
  For all compact open subsets $U$ of $S$, we have
  \begin{align*}
    G(U)/H(U)=(G/H)(U),
  \end{align*}
  and for all $s$ in $S$, we have $G_s/H_s=(G/H)_s$. Moreover, the map $q:G\ra G/H$ over $S$ whose fiber at $s$ equals the quotient map $G_s\ra G_s/H_s$ is continuous.
\end{prop*}
\begin{proof}
  Let $U$ and $U'$ be disjoint compact open subsets of $S$. Then
  \begin{align*}
    G(U\cup U')/H(U\cup U') &= (G(U)\times G(U'))/(H(U)\times H(U')) \\
    &= \big[G(U)/H(U)\big] \times \big[G(U')/H(U')\big],
  \end{align*}
so \ref{ss:profinitesetsheaf} shows that, when evaluating $G/H$ on compact open subsets of $S$, we do not need to sheafify. Lemma \ref{ss:evaluationkernel} immediately implies that $G_s/H_s=(G/H)_s$.

  To see that $q$ is continuous, it suffices to show that, for all compact open subsets $U$ of $S$ and $\sigma$ in $(G/H)(U)$, the preimage under $q$ of $\sigma(U)$ is open. Now $\sigma=gH(U)$ for some $g$ in $G(U)$, so this preimage equals the union of $g(u)H_u$ for all $u$ in $U$. But this is precisely the image of the open subset $H_U\subseteq G_U$ under the homeomorphism
  \begin{gather*}
    \xymatrix{G_U=U\times_SG\ar[r]^-{g\times\id} & G_U\times_SG\ar[r]^-m & G_U.}\qedhere
  \end{gather*}
\end{proof}

\subsection{}\label{ss:smoothreps}
The notion of discrete spaces with continuous actions (and consequently of smooth representations) generalizes to our setting as follows. Let $\La$ be a ring.
\begin{defn*}\hfill
  \begin{enumerate}[a)]
  \item An \emph{\'etale $G$-space} is an \'etale topological space $X$ over $S$ along with a continuous action $a:G\times_SX\ra X$ over $S$.
  \item A \emph{smooth representation of $G$ over $\La$} is a $\La$-module \'etale topological space $V$ over $S$ along with a continuous $\La$-linear action $a:G\times_SV\ra V$ over $S$.
  \end{enumerate}
  Write $\Rep(G,\La)$ for the category of smooth representations of $G$ over $\La$, and write $\ul{\La}$ for the constant topological space over $S$ along with the trivial action of $G$.
\end{defn*}
For all smooth representations $V$ and $V'$ of $G$ over $\La$, write $\ul{\Hom}_G(V,V')$ for the sheaf of $\La$-modules on $S$ that sends compact open subsets $U$ of $S$ to the set of morphisms $V_U\ra V_U'$ of smooth representations of $G_U$ over $\La$.

\subsection{}\label{ss:equivariantsheaves}
We now translate the various ways of interpreting discrete spaces with continuous actions (and consequently of smooth representations) into our more general setting. Let $X$ be an \'etale $G$-space, and consider the continuous map
\begin{align*}
\xymatrix{\rho:G\times_SX\ar[r]^-{(\pr_1,a)} & G\times_SX}
\end{align*}
 of \'etale topological spaces over $G$.
\begin{lem*}
  This association induces an equivalence of categories between
  \begin{enumerate}[a)]
  \item \'etale $G$-spaces $X$,
  \item $G$-equivariant sheaves $(X,\rho:\pr^*X\ra^\sim\pr^*X)$ on $S$.
  \end{enumerate}
\end{lem*}
\begin{proof}
  By checking on fibers, we see that $\rho:G\times_SX\ra G\times_SX$ satisfies the cocycle condition and is bijective. Since $G\times_SX$ is \'etale over $G$, the latter implies that $\rho$ is a homeomorphism, as desired.

  In the other direction, let $X$ be a sheaf on $S$ along with an isomorphism
  \begin{align*}
    \rho:\pr^*X\ra^\sim\pr^*X
  \end{align*}
  of sheaves on $G$ satisfying the cocycle condition. Then the composition
  \begin{align*}
    \xymatrix{a:G\times_SX\ar[r]^-\rho &G\times_SX\ar[r]^-{\pr_2} &X}
  \end{align*}
is continuous, and checking on fibers shows that $a$ is an action over $S$.
\end{proof}

\subsection{}\label{ss:continuitycriterion}
Let $X$ be an \'etale $G$-space, and consider the action $G\times X\ra X$ of sheaves on $S$ induced by the continuous map $a:G\times_SX\ra X$.
\begin{prop*}
  This association induces an equivalence of categories between
  \begin{enumerate}[a)]
  \item \'etale $G$-spaces $X$,
  \item sheaves $X$ on $S$ along with an action $G\times X\ra X$ of sheaves on $S$ such that, for all compact open subsets $U$ of $S$ and $\tau$ in $X(U)$, there exists an $\al$ with $K^\al(U)$ stabilizing $\tau$.
  \end{enumerate}
\end{prop*}
Proposition \ref{ss:continuitycriterion} implies that $\Rep(G,\La)$ is naturally an abelian category, so we can form its derived category $D(G,\La)$.
\begin{proof}
  Since $a$ is continuous and $\tau$ is an open embedding, the cartesian diagram
  \begin{align*}
    \xymatrix{\stab_G\tau\ar[r]\ar[d] & U\ar[dd]^-\tau \\
    G\times_SU\ar[d]^-{\id\times\tau} & \\
    G\times_SX\ar[r]^-a & X}
  \end{align*}
  shows that $\stab_G\tau$ is a neighborhood of the identity section in $G_U$. By Definition \ref{ss:grouphypotheses}.a), there exists an $\al$ with $K^\al_U$ lying in $\stab_G\tau$, so our action $G\times X\ra X$ of sheaves on $S$ satisfies the property in part b).

  In the other direction, let $X$ be a sheaf on $S$ along with an action $G\times X\ra X$ of sheaves on $S$ satisfying the property in part b). For all $s$ in $S$, taking stalks yields an action of $\varinjlim_UG(U)$ on $X_s$, where $U$ runs over compact neighborhoods of $s$. We claim that this action factors through $\ev_s:\varinjlim_UG(U)\ra G_s$. To see this, let $g$ be in $\ker\ev_s$, let $\tau$ be in $X_s$, and view $\tau$ as an element of $X(U)$ for some compact neighborhood $U$ of $s$. Lemma \ref{ss:evaluationkernel} represents $g$ as an element of $K^\al(U')$ for some compact neighborhood $U'\subseteq U$ of $s$, so $g$ stabilizes $\tau|_{U'}$, as desired.

  The claim induces an action $G_s\times X_s\ra X_s$, and taking the disjoint union over $s$ yields an action $a:G\times_SX\ra X$ over $S$. As $U$ and $\tau$ vary, the open subsets $G\times_S\tau(U)$ of $G\times_SX$ form a cover, so it suffices to check the continuity of $a$ by restricting to $G\times_S\tau(U)$. Then Proposition \ref{ss:openquotient} and checking on fibers show that $a:G\times_S\tau(U)\ra X$ equals the composition
  \begin{align*}
    \xymatrix{G\times_S\tau(U)\ar[r]^-{q\times\id} & G/K^\al\times_S\tau(U)\ar[r]& X,}
  \end{align*}
  so it suffices to show that the right map is continuous. Because $G/K^\al$ is \'etale over $S$, as $\sg$ varies through $(G/K^\al)(U)$, the open subsets $\sg(U)\times_S\tau(U)$ of $G/K^\al\times_S\tau(U)$ form a cover. Therefore it suffices to restrict to $\sg(U)\times_S\tau(U)$. Finally, Proposition \ref{ss:openquotient} identifies $\sg=gK^\al(U)$ for some $g$ in $G(U)$, and we see that $a:\sg(U)\times_S\tau(U)\ra X$ equals the map $g\cdot\tau:U\ra X$. This is indeed continuous.
\end{proof}

\subsection{}\label{ss:etaleGspaces}
Endow the category of \'etale $G$-spaces with the Grothendieck topology whose coverings are given by jointly surjective collections of \'etale maps.
\begin{cor*}
The site $\{\mbox{\'etale }G\mbox{-spaces}\}$ is subcanonical, and the Yoneda embedding
  \begin{align*}
    \{\mbox{\'etale }G\mbox{-spaces}\}\hra\big\{\mbox{sheaves on }\{\mbox{\'etale }G\mbox{-spaces}\}\big\}
  \end{align*}
is an equivalence of categories.
\end{cor*}
\begin{proof}
Note that surjective \'etale maps are quotient maps, which implies the first statement. For the second statement, let $\cF$ be a sheaf on $\{\mbox{\'etale }G\mbox{-spaces}\}$, and write $X$ for the presheaf on $S$ given by $U\mapsto\varinjlim_\al\cF(G/K^\al|_U)$ for all compact open subsets $U$ of $S$. For all disjoint compact open subsets $U$ and $U'$ of $S$, we have
  \begin{align*}
&\quad \varinjlim_\al\cF(G/K^\al|_{U\cup U'}) = \varinjlim_\al\cF\big((G/K^\al|_U)\cup(G/K^\al|_{U'})\big)\\
    &= \varinjlim_\al\big(\cF(G/K^\al|_U)\times\cF(G/K^\al|_{U'})\big) = \big[\varinjlim_\al\cF(G/K^\al|_U)\big]\times\big[\varinjlim_\al\cF(G/K^\al|_{U'})\big],
  \end{align*}
  so \ref{ss:profinitesetsheaf} shows that, when evaluating $X$ on compact open subsets of $S$, it is already a sheaf.

  Let $g$ be in $G(U)$, and write $g^{-1}K^\al_Ug$ for the image of the open group subspace $K^\al_U\subseteq G_U$ under the homeomorphism
  \begin{align*}
    \xymatrixcolsep{4pc}\xymatrix{K^\al_U=U\times_SK^\al\times_SU\ar[r]^-{g^{-1}\times\id\times g} & K^\al_U\times_SK^\al\times_S K^\al_U\ar[r]^-m & K^\al_U.}
  \end{align*}
  Because $g^{-1}K^\al_Ug$ is a neighborhood of the identity section in $G_U$, Definition \ref{ss:grouphypotheses}.a) yields an $\al'$ such that $K^{\al'}_U$ lies in $g^{-1}K^\al g$. Therefore right translation by $g$ induces a map $G/K^\al|_U\ra G/K^{\al'}|_U$ of \'etale $G_U$-spaces. Taking $\varinjlim_\al\cF(-)$ yields a map $X(U)\ra X(U)$, and as $U$ varies, we obtain an action $G\times X\ra X$ of sheaves on $S$. By construction, this action satisfies the property in Proposition \ref{ss:continuitycriterion}.b), so Proposition \ref{ss:continuitycriterion} endows $X$ with the structure of an \'etale $G$-space. Finally, the image of $X$ under the Yoneda embedding is naturally isomorphic to $\cF$, as desired.
\end{proof}

\subsection{}\label{ss:cInd}
Next, we introduce our generalization of compactly supported induction. For the rest of this section, assume that $H$ is closed in $G$. By taking intersections, this implies that $H$ also satisfies Definition \ref{ss:grouphypotheses}.a), so $H$ is also lctd over $S$.

Let $W$ be a smooth representation of $H$ over $\La$.
\begin{defn*}
  Write $\cInd_H^GW$ for the sheaf of $\La$-modules on $S$ that sends compact open subsets $U$ of $S$ to the set of continuous functions $f:G_U\ra W$ over $S$ such that
  \begin{enumerate}[a)]
  \item  the square
  \begin{align*}
    \xymatrix{H\times_SG_U\ar[r]^-m\ar[d]^-{\id\times f} & G_U\ar[d]^-f \\
    H\times_SW\ar[r]^-a & W}
  \end{align*}
  commutes,
  \item there exist $g_1,\dotsc,g_r$ in $G(U)$ such that, for all $u$ in $U$, the function $f_u:G_u\ra W_u$ is supported on $\bigcup_{j=1}^rH_ug_j(u)$.
  \end{enumerate}
\end{defn*}

\begin{lem}\label{ss:cIndfibers}
  For all $s$ in $S$, restricting to $G_s$ yields a bijection
  \begin{align*}
    (\cInd_H^GW)_s\ra^\sim\cInd_{H_s}^{G_s}(W_s).
  \end{align*}
\end{lem}
\begin{proof}
  For injectivity, let $U$ be a compact neighborhood of $s$, let $f$ be in
  \begin{align*}
    (\cInd_H^GW)(U),
  \end{align*}
  and let $g_1,\dotsc,g_r$ be the elements of $G(U)$ provided by Definition \ref{ss:cInd}.b). Suppose that $f|_{G_s}=0$. Then, for all $1\leq j\leq r$, the preimage under $f\circ g_j$ of $0(U)\subseteq(\cInd_H^GW)_U$ is a neighborhood of $s$, so there exists a compact neighborhood $U'$ of $s$ contained in all of them. Definition \ref{ss:cInd}.b) indicates that $f|_{G_{U'}}=0$.

For surjectivity, let $f_s:G_s\ra W_s$ be in $\cInd_{H_s}^{G_s}(W_s)$. Since restricting to $G_s$ is additive, it suffices to consider $f_s$ of the form
  \begin{align*}
    x_s\mapsto
    \begin{cases}
      (x_sg_s^{-1})w_s & \mbox{if }x_s\mbox{ lies in }H_sg_s, \\
      0 & \mbox{otherwise,}
    \end{cases}
  \end{align*}
for some $w_s$ in $W_s$ and $g_s$ in $G_s$. By Definition \ref{ss:grouphypotheses}.b), there exists a compact neighborhood $U$ of $s$ and $g$ in $G(U)$ such that $g(s)=g_s$, and after shrinking $U$, there exists $w$ in $W(U)$ such that $w(s)=w_s$.

  Write $H_Ug$ for the image of $H_U\subseteq G_U$ under the homeomorphism
  \begin{align*}
    \xymatrix{G_U=G\times_SU\ar[r]^-{\id\times g} & G\times_SG_U\ar[r]^-m & G_U.}
  \end{align*}
  Write $f:G_U\ra W$ for the continuous function over $S$ whose value on $H_Ug$ equals
  \begin{align*}
    \xymatrix{H_Ug\ar[r]^-\sim & H_U = H\times_SU\ar[r]^-{\id\times w} & H\times_SW\ar[r]^-a & W}
  \end{align*}
  and whose value on $G_U-H_Ug$ equals $0\circ\pr$, which is well-defined because $H_Ug$ is clopen in $G_U$. By checking on fibers, we see that $f$ lies in $(\cInd_H^GW)(U)$, and its restriction to $G_s$ equals $f_s$, as desired.
\end{proof}

\subsection{}\label{ss:cIndaction}
We get an action $G\times\cInd^G_HW\ra\cInd^G_HW$ of sheaves on $S$ as follows. For all compact open subsets $U$ of $S$, $g$ in $G(U)$, and $f$ in $(\cInd_H^GW)(U)$, define $g\cdot f$ to be the composition
\begin{align*}
\xymatrix{G_U=G\times_SU\ar[r]^-{\id\times g} & G\times_SG_U\ar[r]^-m & G_U\ar[r]^-f & W.}
\end{align*}
Note that this action is $\Lambda$-linear.
\begin{lem*}
For all $f$ in $(\cInd^G_HW)(U)$, there exists an $\al$ with $K^\al(U)$ stabilizing $f$.
\end{lem*}
Combined with Proposition \ref{ss:continuitycriterion}, Lemma \ref{ss:cIndaction} yields a continuous $\La$-linear action $G\times_S\cInd^G_HW\ra\cInd^G_HW$ over $S$, so we obtain a functor
\begin{align*}
  \cInd_H^G:\Rep(H,\La)\ra\Rep(G,\La).
\end{align*}
\begin{proof}
  Let $g_1,\dotsc,g_r$ be the elements of $G(U)$ provided by Definition \ref{ss:cInd}.b). For all $1\leq j\leq r$, the subspace $g_j^{-1}\stab_H(f\circ g_j)g_j\subseteq G_U$ is a neighborhood of the identity section in $G_U$, so Definition \ref{ss:grouphypotheses}.a) yields an $\al$ such that $K^\al_U$ lies in
  \begin{align*}
    g_j^{-1}\stab_H(f\circ g_j)g_j
  \end{align*}
  for all $1\leq j\leq r$. By using Lemma \ref{ss:cIndfibers} to check on fibers, Definition \ref{ss:cInd}.a) implies that every $g$ in $K^\al(U)$ stabilizes $f$, as desired.
\end{proof}

\subsection{}\label{ss:separateHcosets}
We need the following separatedness property of $G/H$.
\begin{lem*}
Let $U$ be a compact open subset of $S$, let $\sg_1$ and $\sg_2$ be in $(G/H)(U)$, and let $u$ be in $U$. If $\sg_{1u}\neq\sg_{2u}$, then there exists a compact neighborhood $U'\subseteq U$ of $u$ such that $\sg_{1u'}\neq\sg_{2u'}$ for all $u'$ in $U'$.
\end{lem*}
\begin{proof}
By Proposition \ref{ss:openquotient}, there exist $g_1$ and $g_2$ in $G(U)$ such that $\sg_1=g_1H(U)$ and $\sg_2=g_2H(U)$. Proposition \ref{ss:openquotient} also indicates that $g_1(u)H_u\neq g_2(u)H_u$, so $(g_1^{-1}g_2)(u)$ does not lie in $H_u$. Since $g_1^{-1}g_2:U\ra G$ is continuous, the preimage under $g_1^{-1}g_2$ of $G-H$ is a neighborhood of $u$, so it contains a compact neighborhood $U'$ of $u$. Hence $g_1(u')H_{u'}\neq g_2(u')H_{u'}$ for all $u'$ in $U'$, as desired.
\end{proof}

\subsection{}\label{ss:counit}
We want to prove that $\cInd^G_H$ is left adjoint to restriction, so we start by constructing the counit. Let $V$ be a smooth representation of $G$ over $\La$, and consider the map
\begin{align*}
\ve:\cInd_H^G(V|_H)\ra V
\end{align*}
over $S$ whose fiber at $s$ equals
\begin{align*}
f_s\mapsto\sum_{\sg_s\in H_s\bs G_s}\sg^{-1}_sf_s(\sg_s),
\end{align*}
where the sum is finite by Definition \ref{ss:cInd}.b), and its terms are well-defined by Definition \ref{ss:cInd}.a).
\begin{prop*}
The map $\ve$ is a morphism of smooth representations of $G$ over $\La$.
\end{prop*}
\begin{proof}
  Let $U$ be a compact open subset of $S$, let $f$ be in $(\cInd^G_H(V|_H))(U)$, and let $g_1,\dotsc,g_r$ be the elements of $G(U)$ provided by Definition \ref{ss:cInd}.b). Note that Definition \ref{ss:cInd}.b) depends only on the images $\sg_1,\dotsc,\sg_r$ in $(H\bs G)(U)$ of the $g_1,\dotsc,g_r$.

  Let $u$ be in $U$. We claim that there exists a compact neighborhood $U'\subseteq U$ of $u$ and representatives $g'_1,\dotsc,g'_r$ in $G(U')$ of the $\sg_1|_{U'},\dotsc,\sg_r|_{U'}$ such that, for all $1\leq j_1,j_2\leq r$, if $H_{u'}g'_{j_1}(u')=H_{u'}g'_{j_2}(u')$ for some $u'$ in $U'$, then $g'_{j_1}=g'_{j_2}$. To see this, note that Proposition \ref{ss:openquotient} and Lemma \ref{ss:separateHcosets} yield a compact neighborhood $U'\subseteq U$ of $u$ such that, if $\sg_{j_1u'}=\sg_{j_2u'}$ for some $u'$ in $U'$, then $\sg_{j_1}|_{U'}=\sg_{j_2}|_{U'}$. Then take $g_1',\dotsc,g_r'$ in $G(U')$ to be representatives of the $\sg_1|_{U'},\dotsc,\sg_r|_{U'}$ such that, if $\sg_{j_1}|_{U'}=\sg_{j_2}|_{U'}$, then $g'_{j_1}=g'_{j_2}$.

  By construction, the restriction of $\ve\circ f$ to $U'$ equals the sum of the compositions 
  \begin{align*}
    \xymatrixcolsep{3pc}\xymatrix{U'\ar[r]^-{(g_j'^{-1},g_j')} & G\times_SG_U\ar[r]^-{\id\times f}& G\times_SV\ar[r]^-a & V}
  \end{align*}
  as $g_j'$ runs over distinct elements in $g_1',\dotsc,g_r'$. Therefore the restriction of $\ve\circ f$ to $U'$ is continuous, and as $u$ varies, this implies that $\ve\circ f$ is continuous. As $U$ and $f$ vary, this implies that $\ve$ is continuous. Finally, using Lemma \ref{ss:cIndfibers} to check on fibers shows that $\ve$ is a morphism of smooth representations of $G$ over $\La$.
\end{proof}

\subsection{}\label{ss:rescIndadjunction}
We now turn to the unit. Let $W$ be a smooth representation of $H$ over $\La$. For all compact open subsets $U$ of $S$ and $w$ in $W(U)$, write $\eta(w):G_U\ra W$ for the continuous map over $S$ whose value on $H_U$ equals the composition
\begin{align*}
  \xymatrix{H\times_SU\ar[r]^-{\id\times w}& H\times_SW\ar[r]^-a & W,}
\end{align*}
and whose value on $G_U-H_U$ equals $0\circ\pr$, which is well-defined because $H_U$ is clopen in $G_U$.
\begin{prop*}
The map $\eta$ yields a morphism $W\ra(\cInd^G_HW)|_H$ of smooth representations of $H$ over $\La$. Moreover, $\cInd^G_H$ is left adjoint to the restriction functor $(-)|_H:\Rep(G,\La)\ra\Rep(G,\La)$, with counit given by $\ve$ and unit given by $\eta$, and $\cInd^G_H$ is exact.
\end{prop*}
Now $(-)|_H$ is also exact, so $(-)|_H$ and $\cInd^G_H$ agree with their derived functors.
\begin{proof}
Checking on fibers indicates that $\eta(w)$ lies in $(\cInd_H^GW)(U)$. Then Lemma \ref{ss:cIndfibers} lets us check on stalks to see that the induced map $\eta:W\ra\cInd^G_HW$ is a morphism of smooth representations of $H$ over $\La$, that $\ve$ and $\eta$ satisfy the counit-unit equations for $\cInd^G_H$ and $(-)|_H$, and that $\cInd^G_H$ is exact.
\end{proof}

\subsection{}\label{ss:finitegeneration}
We define finite generation in terms of $\cInd$ as follows.
\begin{defn*}
Let $V$ be a smooth representation of $G$ over $\La$. We say that $V$ is \emph{finitely generated} if there exists a surjective morphism $\bigoplus_{j=1}^r\cInd_{K^{\al_j}}^G\ul\La\ra V$ of smooth representations of $G$ over $\La$ for some $\al_1,\dotsc,\al_r$.
\end{defn*}
If $W$ is a finitely generated smooth representation of $H$ over $\La$, then Proposition \ref{ss:rescIndadjunction} implies that $\cInd_H^GW$ is a finitely generated smooth representation of $G$ over $\La$.

\subsection{}\label{ss:Homstalks}
We will need the following lemma.
\begin{lem*}
Let $V$ and $V'$ be smooth representations of $G$ over $\La$, and assume that $V$ is finitely generated. For all $s$ in $S$, restricting to $V_s$ yields an injection
  \begin{align*}
    \ul{\Hom}_G(V,V')_s\hra\Hom_{G_s}(V_s,V_s').
  \end{align*}
\end{lem*}
\begin{proof}
  Let $\bigoplus_{j=1}^r\cInd^G_{K^{\al_j}}\ul\La\ra V$ be the surjective morphism provided by Definition \ref{ss:finitegeneration}. Because Lemma \ref{ss:cIndfibers} induces a commutative square
  \begin{align*}
    \xymatrix{\ul\Hom_G(V,V')_s\ar[r]\ar[d] & \ul\Hom_G(\bigoplus_{j=1}^r\cInd^G_{K^{\al_j}}\ul\La,V')_s\ar[d]\\
    \Hom_{G_s}(V_s,V'_s)\ar[r] & \Hom_{G_s}(\bigoplus_{j=1}^r\cInd^{G_s}_{K^{\al_j}_s}\La,V'_s) &}
  \end{align*}
  with injective rows, it suffices to consider $V=\bigoplus_{j=1}^r\cInd^G_{K^{\al_j}}\ul\La$. Then we have
  \begin{align*}
    \ul\Hom_G(\textstyle\bigoplus_{j=1}^r\cInd^G_{K^{\al_j}}\ul\La,V')_s  &=\textstyle\bigoplus_{j=1}^r\ul\Hom_G(\cInd^G_{K^{\al_j}}\ul\La,V')_s & \mbox{by finitude}\\
                                                                       &= \textstyle\bigoplus_{j=1}^r\ul\Hom_{K^{\al_j}}(\ul\La,V'|_{K^{\al_j}})_s & \mbox{by Proposition \ref{ss:rescIndadjunction}}\\
                                                                          &=\textstyle\bigoplus_{j=1}^r(V_s')^{K^{\al_j}_s} \\
                                                                          &=\textstyle\bigoplus_{j=1}^r\Hom_{K^{\al_j}_s}(\La,V_s'|_{K^{\al_j}_s})\\
                                                                          &=\textstyle\bigoplus_{j=1}^r\Hom_{G_s}(\cInd^{G_s}_{K^{\al_j}_s}\La,V_s') & \mbox{by Proposition \ref{ss:rescIndadjunction}}\\
    &= \Hom_{G_s}(\textstyle\bigoplus_{j=1}^r\cInd^{G_s}_{K^{\al_j}_s}\La,V'_s), & \mbox{by finitude}
  \end{align*}
as desired.
\end{proof}

\subsection{}\label{ss:Heckealgebra}
To compute the endomorphisms of $\cInd^G_HW$, we will use the following generalization of Hecke algebras. Let $W$ be a smooth representation of $H$ over $\La$, and write $\ul{\End}_\La(W)$ for the sheaf of $\La$-module endomorphisms of $W$.
\begin{defn*}
  Write $\cH(G,W)$ for the sheaf of $\La$-modules on $S$ that sends compact open subsets $U$ of $S$ to the set of continuous functions $\phi:G_U\ra\ul{\End}_\La(W)$ over $S$ such that
  \begin{enumerate}[a)]
  \item the square
    \begin{align*}
      \xymatrix{H\times_SG_U\times_SH\ar[r]^-m\ar[d]^-{\id\times\phi\times\id} & G_U\ar[d]^-\phi \\
      H\times_S\ul{\End}_\La(W)\times_SH\ar[r] & \ul{\End}_\La(W)}
    \end{align*}
    commutes, where the bottom arrow is the map over $S$ whose fiber at $s$ equals
    \begin{align*}
      (h_s,\nu_s,h'_s)\mapsto h_s\circ\nu_s\circ h'_s,
    \end{align*}
  \item there exist $g_1,\dotsc,g_r$ in $G(U)$ such that, for all $u$ in $U$, the function $\phi_u:G_u\ra\End_\La(W_u)$ is supported on $\bigcup_{j=1}^rH_ug_j(u)H_u$.
  \end{enumerate}
When $W=\ul\La$, write $\cH(G,H)_\La$ for $\cH(G,W)$.
\end{defn*}

\subsection{}\label{prop:HeckealgebrainjectsintoEnd}
In the generality of this section, we can only prove the following injectivity result. However, in our case of interest, we actually upgrade it into a bijectivity result: see Theorem \ref{ss:Heckealgebrasinfamilies}.
\begin{prop*}
We have a natural morphism of sheaves on $S$
  \begin{align*}
    \vsg:\cH(G,W)\ra\ul{\End}_G(\cInd_H^GW).
  \end{align*}
When $W$ is finitely generated, $\vsg$ is injective.
\end{prop*}
\begin{proof}
  Proposition \ref{ss:rescIndadjunction} identifies $\ul{\End}_G(\cInd_H^GW)$ with $\ul{\Hom}_H(W,(\cInd_H^GW)|_H)$, so it suffices to construct a morphism from $\cH(G,W)$ to the latter. Let $U$ be a compact open subset of $S$, and let $\phi$ be in $\cH(G,W)(U)$. For all $w$ in $W(U)$, consider the map $\vsg(\phi)(w):G_U\ra W$ over $S$ whose fiber at $s$ equals
  \begin{align*}
    x_s\mapsto\sum_{\sg_s\in G_s/H_s}\phi_s(\sg_s)
    \begin{cases}
      (\sg_s^{-1}x_s)w_s &\mbox{if }x_s\mbox{ lies in }\sg_sH_s,\\
      0 & \mbox{otherwise,}
    \end{cases}
  \end{align*}
  where the sum is finite by Definition \ref{ss:Heckealgebra}.b), and its terms are well-defined by Definition \ref{ss:Heckealgebra}.a). Definition \ref{ss:Heckealgebra}.a) and Definition \ref{ss:Heckealgebra}.b) also imply that $\vsg(\phi)(w)$ satisfies Definition \ref{ss:cInd}.a) and Definition \ref{ss:cInd}.b), respectively.

  We turn to the continuity of $\vsg(\phi)(w)$. Let $g_1,\dotsc,g_r$ be the elements of $G(U)$ provided by Definition \ref{ss:Heckealgebra}.b), and write $\sg_1,\dotsc,\sg_r$ for their images in $(G/H)(U)$. Let $u$ be in $U$. The proof of Proposition \ref{ss:counit} yields a compact neighborhood $U'\subseteq U$ of $u$ and representatives $g'_1,\dotsc,g'_r$ in $G(U')$ of the $\sg_1|_{U'},\dotsc,\sg_r|_{U'}$ such that, for all $1\leq j_1,j_2\leq r$, if $g_{j_1}'(u')H_{u'}=g_{j_2}'(u')H_{u'}$ for some $u'$ in $U'$, then $g_{j_1}'=g_{j_2}'$. Write $f_j:G_{U'}\ra W$ for the continuous function over $S$ whose value on $g'_jH_{U'}$ equals
  \begin{align*}
    \xymatrixcolsep{4pc}
    \xymatrix{g'_jH_{U'}\ar[r]^-{(\phi,\cong,w\circ\pr)} & \ul{\End}_\La(W)\times_SH_{U'}\times_SW\ar[r]^-a & W}
  \end{align*}
  and whose value on $G_{U'}-g'_jH_{U'}$ equals $0\circ\pr$, which is well-defined because $g'_jH_{U'}$ is clopen in $G_{U'}$. By construction, the restriction of $\vsg(\phi)(w)$ to $U'$ equals the sum of the $f_j$ as $g'_j$ runs over distinct elements in $g'_1,\dotsc,g'_r$. Therefore the restriction of $\vsg(\phi)(w)$ to $U'$ is continuous, and as $u$ varies, this implies that $\vsg(\phi)(w)$ is continuous.

  Note that $\vsg(\phi)$ yields a morphism $W_U\ra(\cInd_{H_U}^{G_U}W_U)|_{H_U}$ of smooth representations of $H_U$ over $\La$. To see that the resulting morphism $\vsg$ of sheaves on $S$ is injective when $W$ is finitely generated, it suffices to check on stalks. For all $s$ in $S$, restriction to the fiber at $s$ yields a commutative square
  \begin{align*}
    \xymatrix{\cH(G,W)_s\ar[r]\ar[d] & \ul{\End}_G(\cInd_H^GW)_s\ar[d]\\
    \cH(G_s,W_s)\ar[r] & \End_{G_s}(\cInd_{H_s}^{G_s}W_s).}
  \end{align*}
The bottom arrow is a bijection, and Lemma \ref{ss:Homstalks} and Lemma \ref{ss:cIndfibers} show that the right arrow is injective. By using Definition \ref{ss:Heckealgebra}.b) instead of Definition \ref{ss:cInd}.b), the proof of Lemma \ref{ss:cIndfibers} implies that the left arrow is injective, so the top arrow is injective, as desired.
\end{proof}

\subsection{}\label{ss:doublecosetopen}
Later, we will need the following fact. Let $g$ be in $G(S)$, and write $HgH$ for the image of the composition
\begin{align*}
\xymatrixcolsep{3pc}
\xymatrix{H\times_SS_S\times H\ar[r]^-{\id\times g\times\id} & G\times_SG\times_SG\ar[r]^-m & G}.
\end{align*}
\begin{lem*}
The subset $HgH\subseteq G$ is open.
\end{lem*}
\begin{proof}
Let $s$ be in $S$, and let $h_sg(s)h_s'$ be in the fiber of $HgH$ at $s$ for some $h_s$ and $h_s'$ in $H_s$. By Definition \ref{ss:grouphypotheses}.b), there exists a compact neighborhood $U$ of $s$ and $h$ in $H(U)$ such that $h(s)=h_s$. Then $(hg)H_U$ is a neighborhood of $h_s(g)h'_s$ that lies in $HgH$, as desired.
\end{proof}

\subsection{}\label{ss:relativepoints}
We conclude this section by providing a way to construct group topological spaces over $S$. Let $R$ be a ring topological space over $S$. For all sets $I$, write $R^I$ for the $I$-fold fiber power of $R$ over $S$.
\begin{prop*}
  There exists a unique functor
  \begin{align*}
    (-)(R):\{\mbox{affine schemes over }R(S)\}\ra\{\mbox{topological spaces over }S\}
  \end{align*}
  that preserves fiber products and, for all sets $I$, sends the affine space $\bA^I_{R(S)}$ to $R^I$. Moreover, for all affine schemes $X$ over $R(S)$, we have a natural identification $X(R)(S)=X(R(S))$.
\end{prop*}
In particular, for all affine group schemes $X$ over $R(S)$, we naturally obtain a group topological space $X(R)$ over $S$. Write $R^\times$ for $\bG_m(R)$.
\begin{proof}
  Let $X=\Spec{A}$ be an affine scheme over $R(S)$, and fix a presentation $A\cong R(S)[T_i]_{i\in I}/\fa$. Because $R$ is a ring topological space over $S$, elements of $R(S)[T_i]_{i\in I}$ induce continuous maps $R^I\ra R$ over $S$. Applying this to every element of $\fa$ yields a continuous map $R^I\ra R^\fa$ over $S$; define $X(R)$ to be the fiber product
  \begin{align*}
    \xymatrix{X(R)\ar[r]\ar[d] & R^I\ar[d] \\
    S\ar[r]^-0 & R^\fa.}
  \end{align*}

  Now let $Y=\Spec{B}$ be another affine scheme over $R(S)$. Fix a presentation $B\cong R(S)[U_j]_{j\in J}/\fb$, and let $f:Y\ra X$ be a morphism over $R(S)$. For all $i$ in $I$, choose a lift to $R(S)[U_j]_{j\in J}$ of $f^*T_i$. As $i$ varies, this yields a continuous map $R^J\ra R^I$ over $S$. Checking on fibers indicates that the diagram
  \begin{align*}
    \xymatrix{Y(R)\ar@{.>}[rd]\ar[r]\ar@/_0.75cm/[rdd] & R^J\ar[rd] \\
                   & X(R)\ar[r]\ar[d] & R^I\ar[d] \\
    & S\ar[r]^-0 & R^\fa}
  \end{align*}
  commutes, which induces a continuous map $Y(R)\ra X(R)$ over $S$. This shows that our construction is functorial, as well as independent of the presentation.

  Our construction evidently sends $\bA^I_{R(S)}$ to $R^I$. Let $Z=\Spec{C}$ be yet another affine scheme over $R(S)$ along with a morphism $Z\ra X$ over $R(S)$. By choosing presentations of $B$ and $C$ over $A$, we see that our construction preserves fiber products. The cartesian square
  \begin{align*}
    \xymatrix{X\ar[r]\ar[d] & \bA^I_{R(S)}\ar[d] \\
    \Spec{R(S)}\ar[r]^-0 & \bA^\fa_{R(S)}}
  \end{align*}
indicates that these properties uniquely determine the value of $X(R)$. Finally, the fiber product definition of $X(R)$ immediately yields the desired description of $X(R)(S)$.
\end{proof}

\section{Local fields in families}\label{s:localfields}
In this section, we conceptualize Heuristic \eqref{eqn:slogan} by converting sequences $\{E_i\}_{i\in\bN}$ of $p$-adic local fields into nontrivial families over $\bN\cup\{\infty\}$. We begin by constructing a ring topological space $\bE$ over $\bN\cup\{\infty\}$ whose fibers recover the $E_i$. Next, we define the topological ring $E$ to be the global sections of $\bE$, and we study finite \'etale $E$-algebras. We prove that $E$ enjoys a version of Galois theory over $\bN\cup\{\infty\}$ that is compatible with Deligne's isomorphism. Finally, we conclude by extending various constructions (e.g. maximal unramified extensions, adjoining $q$-th power roots of a uniformizer, separable closures, and Lubin--Tate extensions) to the context of $E$.

\subsection{}\label{ss:closefieldssetup}
When forming families of local fields, it will be convenient to work in the following generality. Let $\ka$ be a perfect field. Let $\{E_i\}_{i\in\bN}$ be a sequence of complete discretely valued fields of characteristic $0$ with residue field $\ka$, and write $E_\infty$ for $\ka\lp{t}$. For all $i$ in $\bN\cup\{\infty\}$, write $v_i:E_i\ra\bZ\cup\{\infty\}$ for the (normalized) valuation, write $\cO_i$ for the ring of integers of $E_i$, and write $\fp_i$ for the maximal ideal of $\cO_i$. For all positive integers $n$, write $\Tr_n(E_i)$ for the triple as in \cite[(1.2)]{Del84}.

When $i$ lies in $\bN$, write $e_i$ for the absolute ramification index $v_i(p)$, assume that $e_i\to\infty$ as $i\to\infty$, and fix an isomorphism $\Tr_{e_i}(E_i)\cong\Tr_{e_i}(E_\infty)$. Note that this includes an isomorphism $\cO_i/\fp_i^{e_i}\cong\cO_\infty/\fp_\infty^{e_i}$. Write $\pi_\infty$ for $t$, and choose a uniformizer $\pi_i$ of $E_i$ whose image in $\cO_i/\fp_i^{e_i}\cong\cO_\infty/\fp_\infty^{e_i}$ equals the image of $\pi_\infty$.

\subsection{}\label{defn:rings}
View $\bN$ as a discrete topological space, and write $\bN\cup\{\infty\}$ for its one-point compactification. Note that $\bN\cup\{\infty\}$ is profinite; in fact, it is homeomorphic to
\begin{align*}
  \varprojlim_d\{i\in\bN\mid i\leq d\}\cup\{\infty\},
\end{align*}
where $d$ runs over natural numbers, and the transition map
\begin{align*}
\{i\in\bN\mid i\leq d+1\}\cup\{\infty\}\ra\{i\in\bN\mid i\leq d\}\cup\{\infty\}
\end{align*}
sends $d+1$ to $\infty$ and equals the identity otherwise.

First, we construct our family as a ring topological space over $\bN\cup\{\infty\}$.
\begin{defn*}\hfill
  \begin{enumerate}[a)]
  \item Let $n$ be a positive integer. Write $\bO_n$ for the ring topological space over $\bN\cup\{\infty\}$ given by $\varprojlim_d$ of the discrete ring topological spaces 
    \begin{align*}
      \Big(\coprod_{i\leq d}\cO_i/\fp_i^n\Big)\textstyle\coprod\cO_\infty/\fp_\infty^n
    \end{align*}
    over $\{i\in\bN\mid i\leq d\}\cup\{\infty\}$, where, for large enough $d$, the transition map
    \begin{align*}
      \Big(\coprod_{i\leq d+1}\cO_i/\fp_i^n\Big)\textstyle\coprod\cO_\infty/\fp_\infty^n\ra\Big(\coprod_{i\leq d}\cO_i/\fp_i^n\Big)\textstyle\coprod\cO_\infty/\fp_\infty^n
    \end{align*}
    sends $\cO_{d+1}/\fp_{d+1}^n$ to $\cO_\infty/\fp_\infty^n$ via the isomorphism $\cO_{d+1}/\fp_{d+1}^{e_{d+1}}\cong\cO_\infty/\fp_\infty^{e_{d+1}}$ (since $d$ is large enough) and equals the identity otherwise.
  \item Write $\bO$ for the ring topological space over $\bN\cup\{\infty\}$ given by $\varprojlim_n\bO_n$, where $n$ runs over positive integers, and the transition map $\bO_{n+1}\ra\bO_n$ is induced from reduction mod $\pi_i^n$.
  \item Write $\pi$ for the continuous section of $\bO\ra\bN\cup\{\infty\}$ whose value on $i$ in $\bN\cup\{\infty\}$ equals $\pi_i$. Write $\bE$ for the ring topological space over $\bN\cup\{\infty\}$ given by $\varinjlim_r\bO$, where $r$ runs over non-negative integers, and the transition maps are given by
    \begin{align*}
      \xymatrix{\bO=(\bN\cup\{\infty\})\times_{\bN\cup\{\infty\}}\bO\ar[r]^-{\pi\times\id} & \bO\times_{\bN\times\{\infty\}}\bO\ar[r]^-m & \bO.}
    \end{align*}
  \end{enumerate}
\end{defn*}
Note that $\bO_n$ is locally constant over $\bN\cup\{\infty\}$, so $\bO$ is Hausdorff. Because $\bE$ is an increasing union of open subspaces homeomorphic to $\bO$, it is also Hausdorff. Finally, we see that $\bE$ is independent of the choice of $\pi_i$.

\subsection{}
We have the following version of the valuation map for $\bE^\times$ over $\bN\cup\{\infty\}$. Write $v:\bE^\times\ra\bZ\times(\bN\cup\{\infty\})$ for the map over $\bN\cup\{\infty\}$ whose fiber at $i$ equals $v_i$. Note that $\bO^\times$ is an open group topological subspace of $\bE^\times$ over $\bN\cup\{\infty\}$. 
\begin{prop*}
The map $v$ is a morphism of group topological spaces over $\bN\cup\{\infty\}$, and $\ker{v}$ equals $\bO^\times$.
\end{prop*}
\begin{proof}
Now $\pi$ induces a morphism $\ul\bZ\ra\bE^\times/\bO^\times$ of abelian sheaves on $\bN\cup\{\infty\}$. By using Proposition \ref{ss:openquotient} to check on stalks, we see that this is an isomorphism. Its inverse yields an isomorphism $\bE^\times/\bO^\times\ra^\sim\bZ\times(\bN\cup\{\infty\})$ of group topological spaces over $\bN\cup\{\infty\}$, and $v$ equals the composition
  \begin{align*}
    \bE^\times\ra \bE^\times/\bO^\times\ra^\sim \bZ\times(\bN\cup\{\infty\}).
  \end{align*}
Proposition \ref{ss:openquotient} shows that $\bE^\times\ra\bE^\times/\bO^\times$ is continuous, and checking on fibers shows that it is a morphism of group topological spaces over $\bN\cup\{\infty\}$. So the same holds for $v$. Finally, the above work also shows that $\ker v$ equals $\bO^\times$.
\end{proof}

\subsection{}\label{ss:E}
Next, we study the topological ring associated with $\bE$. Write $E$ for $\bE(\bN\cup\{\infty\})$ equipped with the compact-open topology, which is a topological ring. Similarly, write $\cO$ for the topological ring $\bO(\bN\cup\{\infty\})$. Note that $\cO$ is an open subring of $E$, equals the subset $E^\circ$ of powerbounded elements, and contains $\pi$.

\begin{lem*}
  As topological rings, $\cO$ is naturally isomorphic to
  \begin{align*}
    \varprojlim_n\varinjlim_d\big(\prod_{i\leq d}\cO_i/\fp_i^n\big)\times\cO_\infty/\fp_\infty^n,
  \end{align*}
and $E$ is isomorphic to $\cO[\frac1\pi]$. Consequently, $\cO$ carries the $\pi$-adic topology and is complete, and we have $\cO/\pi^n=\varinjlim_d\big(\prod_{i\leq d}\cO_i/\fp_i^n\big)\times\cO_\infty/\fp_\infty^n=\bO_n(\bN\cup\{\infty\})$.
\end{lem*}
\begin{proof}
Since $\bO=\varprojlim_n\bO_n$ and $\bO_n$ is locally constant, we have
  \begin{align*}
    \cO=\Cont_{\bN\cup\{\infty\}}(\bN\cup\{\infty\},\bO) &= \varprojlim_n\Cont_{\bN\cup\{\infty\}}(\bN\cup\{\infty\},\bO_n)\\
    &=\varprojlim_n\varinjlim_d\Big(\prod_{i\leq d}\cO_i/\fp_i^n\Big)\times\cO_\infty/\fp_\infty^n
  \end{align*}
  as topological rings. Because $\bN\cup\{\infty\}$ is compact, we also have
  \begin{flalign*}
    E=\Cont_{\bN\cup\{\infty\}}(\bN\cup\{\infty\},\bE)=\varinjlim_r\Cont_{\bN\cup\{\infty\}}(\bN\cup\{\infty\},\bO)=\varinjlim_r\cO=\cO[\textstyle\frac1\pi].\qedhere
  \end{flalign*}
\end{proof}

\subsection{}
For all $i$ in $\bN\cup\{\infty\}$, evaluation at $i$ yields a continuous ring homomorphism $\ev_i:E\ra E_i$. Note that $\ev_i(\cO)=\cO_i$, so $\ev_i$ is surjective and $\ker\ev_i$ is maximal.
\begin{prop*}
Every closed prime ideal of $E$ equals $\ker\ev_i$ for some $i$ in $\bN\cup\{\infty\}$. Thus $\abs{\Spa{E}}$ is naturally homeomorphic to $\bN\cup\{\infty\}$, and $\Spa{E}$ is an adic space.
\end{prop*}
\begin{proof}
  For all $i$ in $\bN$, write $b_i$ for the continuous section of $\bE\ra\bN\cup\{\infty\}$ whose value on $i'$ in $\bN\cup\{\infty\}$ equals $1$ if $i'=i$ and $0$ otherwise. Let $\fa$ be a closed prime ideal of $E$, and first assume that $\fa$ does not contain $b_i$ for some $i$ in $\bN$. Since $b_ix=0$ for all $x$ in $\ker\ev_i$, $\fa$ being prime implies that it equals $\ker\ev_i$.

  Next, assume that $\fa$ contains $b_i$ for all $i$ in $\bN$. Because $\fa$ is closed and $\cO$ is $\pi$-adically complete, this implies that $\fa$ contains $\sum_{i\in\bN}f_ib_i$ for all sequences $\{f_i\}_{i\in\bN}$ with $f_i$ in $E_i$ such that $v_i(f_i)\to\infty$ as $i\to\infty$. The set of such $\sum_{i\in\bN}f_ib_i$ is precisely $\ker\ev_\infty$, so $\fa$ equals $\ker\ev_\infty$.

Since $\abs{\Spa{E_i}}$ is a singleton, the above discussion identifies $\abs{\Spa{E}}$ with $\bN\cup\{\infty\}$ as sets. Note that rational open subspaces of $\Spa{E}$ correspond to compact open subsets of $\bN\cup\{\infty\}$, so this identification is a homeomorphism. In fact, for any compact open subset $U$ of $\bN\cup\{\infty\}$, the corresponding rational localization of $E$ equals $\bE(U)$, and this implies that $\sO_{\Spa{E}}$ is a sheaf.
\end{proof}

\subsection{}\label{ss:Galoisgroups}
Let us recall the setup of Deligne's isomorphism. For all $i$ in $\bN\cup\{\infty\}$, fix a separable closure $\ov{E}_i$ of $E_i$, and write $\Ga_i$ for $\Gal(\ov{E}_i/E_i)$. Write $C_i$ for the completion of $\ov{E}_i$, and recall that $C_i$ is algebraically closed.

For all positive integers $n$, write $I^n_i$ for the $n$-th ramification subgroup of $\Ga_i$ in the upper numbering. Using our isomorphism $\Tr_{e_i}(E_i)\cong\Tr_{e_i}(E_\infty)$, we obtain a canonical isomorphism $\Ga_i/I_i^{e_i}\cong\Ga_\infty/I_\infty^{e_i}$ of topological groups (up to conjugation) \cite[(3.5.1)]{Del84}. Fix an isomorphism in this conjugacy class.

We now introduce the analogue of the absolute Galois group for $E$.
\begin{defn*}\hfill
  \begin{enumerate}[a)]
  \item Write $\Ga^n$ for the group topological space over $\bN\cup\{\infty\}$ given by $\varprojlim_d$ of the group topological spaces
    \begin{align*}
      \Big(\coprod_{i\leq d}\Ga_i/I_i^n\Big)\textstyle\coprod\Ga_\infty/I_\infty^n
    \end{align*}
    over $\{i\in\bN\mid i\leq d\}\cup\{\infty\}$, where, for large enough $d$, the transition map
    \begin{align*}
      \Big(\coprod_{i\leq d+1}\Ga_i/I_i^n\Big)\textstyle\coprod\Ga_\infty/I_\infty^n\ra\Big(\coprod_{i\leq d}\Ga_i/I_i^n\Big)\textstyle\coprod\Ga_\infty/I_\infty^n
    \end{align*}
    sends $\Ga_{d+1}/I^n_{d+1}$ to $\Ga_\infty/I^n_\infty$ via the isomorphism $\Ga_{d+1}/I^{e_{d+1}}_{d+1}\cong\Ga_\infty/I^{e_{d+1}}_\infty$ (since $d$ is large enough) and equals the identity otherwise.
  \item Write $\Ga$ for the group topological space over $\bN\cup\{\infty\}$ given by $\varprojlim_n\Ga^n$, where $n$ runs over positive integers, and the transition map $\Ga^{n+1}\ra\Ga^n$ is induced from quotienting by $I_i^n$.
  \end{enumerate}
\end{defn*}
Because $\Ga^n$ is profinite, so is $\Ga$. Moreover, note that $\Ga$ is lctd over $\bN\cup\{\infty\}$.

\subsection{}\label{lem:etalepointwise}
We will use the following fiberwise criterion to construct \'etale maps over $E$.
\begin{lem*}
  Let $(A,A^+)$ be a sheafy complete Huber pair over $(E,\cO)$, and let $(A',A'^+)$ be a finite free Huber pair over $(A,A^+)$. Then $A'$ is \'etale over $A$ if and only if, for all $i$ in $\bN\cup\{\infty\}$, the base change $A'\otimes_EE_i$ is \'etale over $A\otimes_EE_i$.
\end{lem*}
\begin{proof}
  Write $Z$ for $\Spa(A,A^+)$, and write $Z'$ for $\Spa(A',A'^+)$. By \cite[(1.7.5)]{Hub96}, the morphism $f:Z'\ra Z$ is \'etale if and only if, for every rank-$1$ point $z$ of $Z$, the fiber $f^{-1}(z)$ is \'etale over $z$. Gathering rank-$1$ points $z$ according to their images in $\abs{\Spa{E}}\cong\bN\cup\{\infty\}$ and applying \cite[(1.7.5)]{Hub96} again shows that this holds if and only if $Z'\times_{\Spa{E}}\Spa{E_i}\ra Z\times_{\Spa{E}}\Spa{E_i}$ is \'etale for all $i$ in $\bN\cup\{\infty\}$.

  Because $E$ is Tate, the morphism $(E,\cO)\ra(A,A^+)$ is adic. Therefore \cite[(3.7)]{Hub94} shows that $Z'\times_{\Spa{E}}\Spa{E_i}$ and $Z\times_{\Spa{E}}\Spa{E_i}$ are affinoid, with global sections equal to $A'\otimes_EE_i$ and $A\otimes_EE_i$. Combined with the above criterion, we get the desired result.  
\end{proof}

\subsection{}\label{ss:explicithenselian}
Via Deligne's isomorphism, finite \'etale $E$-algebras spread out from finite \'etale $E_\infty$-algebras as follows.
\begin{prop*}\hfill
  \begin{enumerate}[i)]
  \item   Base change induces an equivalence of categories
  \begin{align*}
    \textstyle\varinjlim_U\{\mbox{finite \'etale }\bE(U)\mbox{-algebras}\}\ra^\sim\{\mbox{finite \'etale }E_\infty\mbox{-algebras}\},
  \end{align*}
  where $U$ runs over compact neighborhoods of $\infty$.
  \item Let $F_\infty$ be a finite \'etale $E_\infty$-algebra such that $I^n_\infty$ acts trivially on its associated $\Ga_\infty$-set. For any compact neighborhood $U$ of $\infty$ and finite \'etale $\bE(U)$-algebra $F$ with $F\otimes_{\bE(U)}E_\infty\cong F_\infty$, there exists a compact neighborhood $U'\subseteq U$ of $\infty$ such that,  for all $i$ in $U'$, we have $e_i\geq n$, and the finite \'etale extension $(F\otimes_{\bE(U)}E_i)/E_i$ corresponds to $F_\infty/E_\infty$ via the isomorphism $\Ga_i/I_i^n\cong\Ga_\infty/I^n_\infty$.
  \end{enumerate}
\end{prop*}
\begin{proof}
  Because $\sO_{\Spa{E},\infty}=\varinjlim_U\bE(U)$, base change induces an equivalence
  \begin{align*}
    \textstyle\varinjlim_U\{\mbox{finite \'etale }\bE(U)\mbox{-algebras}\}\ra^\sim\{\mbox{finite \'etale }\sO_{\Spa{E},\infty}\mbox{-algebras}\}.
  \end{align*}
Note that the residue field of $\sO_{\Spa{E},\infty}$ is $E_\infty$. Since $\sO_{\Spa{E},\infty}$ is henselian \cite[Lemma 2.4.17 (a)]{KL15}, part i) follows from \cite[Tag 09ZL]{stacks-project}.

We turn to part ii). Part i) indicates that it suffices to find one $U$ and $F$ satisfying the desired property for $U=U'$. By taking finite products, we can assume that $F_\infty$ is a field. The maximal unramified subextension of $F_\infty$ is handled by repeating the construction of $E$ (except that in \ref{ss:closefieldssetup} we replace $\ka$ with a finite extension, and we replace $E_i$ with its corresponding unramified extension).

Thus we can assume that $F_\infty/E_\infty$ is totally ramified. Then $F_\infty$ is generated by an element with minimal polynomial
  \begin{align*}
    f_\infty = T^r+\pi_\infty(a_{r-1,\infty}T^{r-1}+\dotsb+a_{0,\infty}),
  \end{align*}
  where the $a_{r-1,\infty},\dotsc,a_{0,\infty}$ lie in $\cO_\infty$. Choose $a_{r-1},\dotsc,a_0$ in $\cO$ such that $\ev_\infty(a_j)=a_{j,\infty}$ for all $0\leq j\leq r-1$. Let $U$ be a compact neighborhood of $\infty$ such that, for all $i$ in $U$, we have $e_i\geq n$, and for all $0\leq j\leq r-1$, the image of $\ev_i(a_j)$ in $\cO_i/\fp_i^n$ equals the image of $a_{j,\infty}$ in $\cO_\infty/\fp_\infty^n$ under the isomorphism $\cO_i/\fp_i^n\cong\cO_\infty/\fp_\infty^n$.

Let $F$ be the finite free $\bE(U)$-algebra obtained via quotienting $\bE(U)[T]$ by
  \begin{align*}
    f\coloneqq T^r+\pi(a_{r-1}T^{r-1}+\dotsb+a_0).
  \end{align*}
  For all $i$ in $U$, \cite[(1.3)]{Del84} shows that $(F\otimes_{\bE(U)}E_i)/E_i$ is the finite separable extension corresponding to $F_\infty/E_\infty$ via the isomorphism $\Ga_i/I_i^n\cong\Ga_\infty/I_\infty^n$. Therefore Lemma \ref{lem:etalepointwise} implies that $F\otimes_E\bE(U)$ is \'etale over $\bE(U)$, as desired.
\end{proof}

\subsection{}\label{cor:Efet}
Proposition \ref{ss:explicithenselian} implies that $E$ enjoys the following version of Galois theory.
\begin{cor*}
The category of finite \'etale $E$-algebras is naturally anti-equivalent to the category of finite locally constant topological spaces $X$ over $\bN\cup\{\infty\}$ along with a continuous action $\Ga\times_{\bN\cup\{\infty\}}X\ra X$ over $\bN\cup\{\infty\}$.
\end{cor*}
\begin{proof}
  Let $F$ be a finite \'etale $E$-algebra. Let $n$ be a positive integer such that $I^n_\infty$ acts trivially on the $\Ga_\infty$-set associated with $F\otimes_EE_\infty$, and let $U'$ be a compact neighborhood of $\infty$ as in Proposition \ref{ss:explicithenselian}.ii). Write $X$ for the finite locally constant topological space over $\bN\cup\{\infty\}$ along with a continuous action $\Ga\times_{\bN\cup\{\infty\}}X\ra X$ over $\bN\cup\{\infty\}$ such that
  \begin{enumerate}[a)]
  \item for all $i$ not in $U'$, its fiber $X_i$ at $i$ is the $\Ga_i$-set associated with $F\otimes_EE_i$,
  \item $X_{U'}$ is the pullback to $U'$ of the $\Ga_\infty/I_\infty^n$-set associated with $F\otimes_EE_\infty$ (since $\Ga^n_{U'}\cong\Ga_\infty/I^n_\infty\times U'$ as group topological spaces over $U'$).
  \end{enumerate}

By Proposition \ref{ss:explicithenselian}.i), the assignment $F\mapsto X$ is functorial. This functor is fully faithful by Proposition \ref{ss:explicithenselian}.i) and essentially surjective by Proposition \ref{ss:continuitycriterion}.
\end{proof}

\subsection{}\label{ss:realizingextensions}
Later, we will need to spread out finite Galois extensions from $E_\infty$ to $E$ as follows. Let $F_\infty/E_\infty$ be a finite Galois extension, and let $n$ be a positive integer such that the image of $I^n_\infty$ in $\Gal(F_\infty/E_\infty)$ is trivial. Write $\psi:[-1,\infty)\ra[-1,\infty)$ for the Hasse--Herbrand function associated with $F_\infty/E_\infty$ as in \cite[Chap. IV, \S3]{Ser79}.

In this subsection, assume that $e_i\geq n$ for all $i$ in $\bN$. Write $F_i/E_i$ for the finite Galois extension corresponding to $F_\infty/E_\infty$ via the isomorphism $\Ga_i/I_i^n\cong\Ga_\infty/I^n_\infty$, and recall that our isomorphism $\Tr_{e_i}(E_i)\cong\Tr_{e_i}(E_\infty)$ induces an isomorphism $\Tr_{\psi(e_i)}(F_i)\cong\Tr_{\psi(e_i)}(F_\infty)$ \cite[(3.4.1)]{Del84}. Because $\psi(x)\to\infty$ as $x\to\infty$, repeating Definition \ref{defn:rings} (except that in \ref{ss:closefieldssetup} we replace $\ka$ with the residue field of $F_\infty$, and we replace $E_i$ with $F_i$) yields ring topological spaces $\bO_{\bF,n}$, $\bO_\bF$, and $\bF$ over $\bN\cup\{\infty\}$.

 Corollary \ref{cor:Efet} yields a natural finite \'etale $E$-algebra $F$ such that, for all $i$ in $\bN\cup\{\infty\}$, the $E_i$-algebra $F\otimes_EE_i$ is isomorphic to $F_i$. Note that $\bO_\bF(\bN\cup\{\infty\})$ is naturally isomorphic to $F^\circ$, and $\bF(\bN\cup\{\infty\})$ is naturally isomorphic to $F$.

 Finally, Corollary \ref{cor:Efet} implies that the above construction is functorial in $F_\infty$. Therefore we obtain actions of $\Gal(F_\infty/E_\infty)$ on $\bO_{\bF,n}$, $\bO_\bF$, and $\bF$ over $\bN\cup\{\infty\}$. By checking on fibers, we see that the natural closed embedding $\bE\hra\bF$ of ring topological spaces over $\bN\cup\{\infty\}$ identifies $\bE$ with the fixed point locus of $\Gal(F_\infty/E_\infty)$ in $\bF$.

 \subsection{}\label{ss:realizingEbreve}
\textbf{For the rest of this paper, assume that $\ka=\bF_q$ is a finite field.} Then $\bO_n$ is finite locally constant over $\bN\cup\{\infty\}$, so $\bO$ is profinite and $\bE$ is locally profinite.

Let us describe the analogue of (the completion of) the maximal unramified extension of $E$. For all $i$ in $\bN\cup\{\infty\}$, write $\breve{E}_i$ for the completion of the maximal unramified extension of $E_i$, and write $\breve\cO_i$ for its ring of integers. Applying $-\otimes_{\bF_q}\ov\bF_q$ to $\Tr_{e_i}(E_i)\cong\Tr_{e_i}(E_\infty)$ induces an isomorphism $\Tr_{e_i}(\breve{E}_i)\cong\Tr_{e_i}(\breve{E}_\infty)$, so repeating Definition \ref{defn:rings} (except that in \ref{ss:closefieldssetup} we replace $\ka$ with $\ov\bF_q$, and we replace $E_i$ with $\breve{E}_i$) yields ring topological spaces $\breve\bO$ and $\breve\bE$ over $\bN\cup\{\infty\}$. Write $\breve{E}$ for $\breve\bE(\bN\cup\{\infty\})$, and write $\breve\cO$ for $\breve\bO(\bN\cup\{\infty\})$.

Finally, write $\vp:\breve\bE\ra\breve\bE$ for the map over $\bN\cup\{\infty\}$ whose fiber at $i$ equals the lift $\vp_i:\breve{E}_i\ra\breve{E}_i$ of absolute $q$-Frobenius. Note that $\vp$ is a morphism of ring topological spaces over $\bN\cup\{\infty\}$ and preserves $\breve\bO$. By checking on fibers, we see that the natural closed embedding $\bE\hra\breve\bE$ of ring topological spaces over $\bN\cup\{\infty\}$ identifies $\bE$ with the fixed point locus of $\vp$ in $\breve\bE$.

\subsection{}\label{ss:fieldperf}
We now consider the analogue of (the completion of) adjoining all $q$-th power roots of $\pi$ to $E$. For all $i$ in $\bN\cup\{\infty\}$, write $E_i^{\perf}$ for the completion of $\bigcup_{r=0}^\infty E_i(\pi_i^{1/q^r})$. Our choice of uniformizer $\pi_i$ induces an identification $(E_i^{\perf})^\flat=E_\infty^{\perf}$ such that $\pi_i=\pi_\infty^\sharp$. Choose a system of $q$-th power roots of $\pi_i$ in $C_i$, which induces a map $E^{\perf}_i\ra C_i$ over $E$. Since $E^{\perf}_i$ is perfectoid and this realizes $C_i$ as the completion of its separable closure, the tilting equivalence induces an identification $C_i^\flat=C_\infty$.

Write $\cO^{\perf}$ for the $\pi$-adic completion of $\bigcup_{r=0}^\infty\cO[\pi^{1/q^r}]$, and write $E^{\perf}$ for $\cO^{\perf}[\textstyle\frac1\pi]$. Note that $\cO$ is a direct summand of $\cO^{\perf}$ as topological $\cO$-modules, and $E^{\perf}\otimes_EE_i$ is naturally isomorphic to $E^{\perf}_i$.

\subsection{}\label{ss:Cperfectoid}
The following is the analogue of (the completion of) the separable closure of $E$. Let $\{K^\al\}_\al$ be a directed family of compact open group subspaces of $\Ga$ over $\bN\cup\{\infty\}$ that is cofinal among neighborhoods of the identity section in $\Ga$. Then $\Ga/K^\al$ is finite locally constant over $\bN\cup\{\infty\}$. Write $F^\al$ for the finite \'etale $E$-algebra corresponding to $\Ga/K^\al$ via Corollary \ref{cor:Efet}, write $\cO_C$ for the $\pi$-adic completion of $\bigcup_\al(F^{\al})^\circ$, and write $C$ for $\cO_C[\frac1\pi]$. By cofinality, $\cO_C$ is independent of the $\{K^\al\}_\al$.
\begin{prop*}
The ring $C$ is perfectoid, and our identifications $C^\flat_i=C_\infty$ induce an identification $C^\flat=\Cont(\bN\cup\{\infty\},C_\infty)$ such that $\pi_\infty^\sharp=\pi$.
\end{prop*}
\begin{proof}
For all $i$ in $\bN\cup\{\infty\}$, write $F^\al_i$ for the finite \'etale $E_i$-algebra $F^\al\otimes_EE_i$, and note that $F^\al_i$ corresponds to the $\Ga_i$-set $\Ga_i/K^\al_i$. Now $C^\circ=\cO_C$, and we have
  \begin{align*}
    \cO_C/\pi = \varinjlim_\al(F^{\al})^\circ/\pi = \varinjlim_\al\varinjlim_d\prod_{i\leq d}\cO_{F^\al_i}/\pi_i\times\cO_{F^\al_\infty}/\pi_\infty,
  \end{align*}
  where $d$ is large enough with respect to $\al$. Since $\{K^\al_i\}_\al$ is a cofinal family of open compact subgroups of $\Ga_i$, we see that $\varinjlim_\al\cO_{F^\al_i}/\pi_i=\cO_{\ov{E}_i}/\pi_i=\cO_{C_i}/\pi_i$. Our identification $C_i^\flat=C_\infty$ induces an identification $\cO_{C_i}/\pi_i=\cO_{C_\infty}/\pi_\infty$, so switching the order of the direct limits yields 
  \begin{align*}
    \varinjlim_d\varinjlim_\al\prod_{i\leq d}\cO_{F^\al_i}/\pi_i\times\cO_{F^\al_\infty}/\pi_\infty &= \varinjlim_d\prod_{i\leq d}\cO_{C_\infty}/\pi_\infty\times\cO_{C_\infty}/\pi_\infty \\
    &= \Cont(\bN\cup\{\infty\},\cO_{C_\infty}/\pi_\infty).
  \end{align*}
This description implies that the absolute $q$-Frobenius map $\cO_C/p\ra\cO_C/p$ is surjective, so $C$ is perfectoid \cite[Remark 3.2]{Sch17}. Moreover, the above work identifies $\cO_C^\flat$ with $\Cont(\bN\cup\{\infty\},\cO_{C_\infty})$, so $C^\flat$ is naturally isomorphic to $\Cont(\bN\cup\{\infty\},C_\infty)$. Finally, the fact that $\pi_\infty^\sharp=\pi_i$ under the identification $C^\flat_i=C_\infty$ implies the desired description of $\pi_\infty^\sharp$ in $C$.
\end{proof}

\subsection{}\label{ss:LTextensions}
Finally, we define the analogues of Lubin--Tate extensions over $E$ as follows. (In \S\ref{s:LT}, we will explain their relationship with Lubin--Tate theory over $E$.) For all $i$ in $\bN\cup\{\infty\}$, write $\wh{E}_i^\times$ for the profinite completion of $E_i^\times$, and write $\Art_i:\wh{E}_i^\times\ra^\sim\Ga_i^{\ab}$ for the local Artin isomorphism normalized by sending uniformizers to geometric $q$-Frobenii. For all positive integers $n$, note that $\Art_i$ and our choice of uniformizer $\pi_i$ induce a surjective continuous homomorphism $\Ga_i/I_i^n\ra(\cO_i/\fp^n_i)^\times$. Write $E_i^{\LT,n}$ for the corresponding finite abelian extension of $E_i$. Write $E^{\LT}_i$ for the completion of $\bigcup_{n=1}^\infty E^{\LT,n}_i$, and recall that $E^{\LT}_i$ is perfectoid.

Note that $\bO_n^\times$ is finite locally constant over $\bN\cup\{\infty\}$.
\begin{lem*}
The map $\Ga^n\ra\bO_n^\times$ over $\bN\cup\{\infty\}$ whose fiber at $i$ equals $\Ga_i/I^n_i\ra(\cO_i/\fp^n_i)^\times$ is continuous. Consequently, there exists a natural finite \'etale $E$-algebra $E^{\LT,n}$ such that, for all $i$ in $\bN\cup\{\infty\}$, the $E_i$-algebra $E^{\LT,n}\otimes_EE_i$ is isomorphic to $E_i^{\LT,n}$.
\end{lem*}
Write $\cO^{\LT}$ for the $\pi$-adic completion of $\bigcup_{n=1}^\infty(E^{\LT,n})^\circ$, and write $E^{\LT}$ for $\cO^{\LT}[\frac1\pi]$.
\begin{proof}
Because Proposition \ref{ss:relativepoints} is compatible with inverse limits in $S$, we have
    \begin{align*}
      \bO^\times_n = \varprojlim_d\Big[\Big(\coprod_{i\leq d}(\cO_i/\fp_i^n)^\times\Big)\textstyle\coprod(\cO_\infty/\fp_\infty^n)^\times\Big].
    \end{align*}
as group topological spaces over $\bN\cup\{\infty\}$. Next, consider the continuous map 
    \begin{align*}
      \Big(\coprod_{i\leq d}\Ga_i/I_i^n\Big)\textstyle\coprod\Ga_\infty/I_\infty^n\ra\Big(\coprod_{i\leq d}(\cO_i/\fp_i^n)^\times\Big)\textstyle\coprod(\cO_\infty/\fp_\infty^n)^\times
    \end{align*}
over $\{i\in\bN\mid i\leq d\}\cup\{\infty\}$ given by the disjoint union of the $\Ga_i/I^n_i\ra(\cO_i/\fp_i^n)^\times$. Now \cite[(3.6.1)]{Del84} shows that this is compatible with the $\varprojlim_d$ in Definition \ref{ss:Galoisgroups}.a), so taking $\varprojlim_d$ yields a continuous map $\Ga^n\ra\bO_n^\times$ with the desired property.

    By checking on fibers, we see that $\Ga^n\ra\bO_n^\times$ is a morphism of group topological spaces over $\bN\cup\{\infty\}$. Therefore precomposing with $\Ga\ra\Ga^n$ yields a continuous action $\Ga\times_{\bN\cup\{\infty\}}\bO_n^\times\ra\bO_n^\times$ over $\bN\cup\{\infty\}$. Finally, the corresponding finite \'etale $E$-algebra $E^{\LT,n}$ from Corollary \ref{cor:Efet} has the desired property.
  \end{proof}

  \section{Fargues--Fontaine curves}\label{s:FFcurves}
  To carry out Fargues--Scholze's program \cite{FS21} in our context, we need to define a version of relative Fargues--Fontaine curves over $E$. This is the goal of this section. First, we construct a generalization of Witt vectors for $\cO$-algebras, which specializes to $\cO_i$-typical Witt vectors when applied to $\cO_i$-algebras. We prove that various properties of $\cO_i$-typical Witt vectors extend to our generalization, and we also construct a version of Witt bivectors over $\cO$ (which will be needed for computations in \S\ref{s:LT}).

  Next, we use this to define a relative Fargues--Fontaine curve $X_S$ over $\Spa{E}$ for any affinoid perfectoid space $S$ over $\ul{\bN\cup\{\infty\}}$. After proving basic facts about $X_S$, we show that $S$-points of $\Div_X^1\coloneqq(\Spd{E})/\phi^\bZ$ induce closed Cartier divisors on $X_S$, where $\phi$ denotes the relative $q$-Frobenius morphism. Finally, we prove a version of Drinfeld's lemma for $(\Div_X^1)_{\ov\bF_q}$ by reducing to the situation considered in Fargues--Scholze \cite{FS21}.
  
\subsection{}\label{ss:Witt}
We start by constructing the analogue of Witt vectors over $\cO$. For all non-negative integers $j$, write $w_j$ for the polynomial in $\cO[T_0,\dotsc,T_j]$ given by
\begin{align*}
T_0^{q^j}+\pi T_1^{q^{j-1}}+\dotsb+\pi^jT_j.
\end{align*}
For all $\cO$-algebras $A$, view $w_j$ as a function
\begin{align*}
w_j:\{\mbox{sequences }(a_0,a_1,\dotsc)\mbox{ in }A\}\ra A
\end{align*}
by sending $(a_0,a_1,\dotsc)\mapsto w_j(a_0,\dotsc,a_j)$.
\begin{prop*}\hfill
  \begin{enumerate}[i)]
  \item Let $A$ be an $\cO$-algebra that is $\pi$-torsionfree and has an $\cO$-algebra endomorphism $\vp$ whose reduction mod $\pi$ equals absolute $q$-Frobenius. Then the map $(w_0,w_1,\dotsc)$ is injective with image equal to
  \begin{align*}
    \{\mbox{sequences }(b_0,b_1,\dotsc)\mbox{ in }A\mid\mbox{for all }j\geq0\mbox{, }b_{j+1}\equiv\vp(b_j)\pmod{\pi^{j+1}}\}.
  \end{align*}
  \item The functor $\{\cO\mbox{-algebras}\}\ra\{\mbox{sets}\}$ given by
  \begin{align*}
    A\mapsto\{\mbox{sequences }(a_0,a_1,\dotsc)\mbox{ in }A\}
  \end{align*}
factors uniquely through a functor $W:\{\cO\mbox{-algebras}\}\ra\{\cO\mbox{-algebras}\}$ such that, for all non-negative integers $j$, the map $w_j:W(A)\ra A$ is an $\cO$-algebra homomorphism.
  \end{enumerate}
\end{prop*}
\begin{proof}
  For part i), $A$ being $\pi$-torsionfree implies that $(w_0,w_1,\dotsc)$ is injective. Next, $a_{j'}^q\equiv\vp(a_{j'})\pmod{\pi}$ implies that $a_{j'}^{q^{j-j'+1}}\equiv\vp(a_{j'})^{q^{j-j'}}\pmod{\pi^{j-j'+1}}$ for all $0\leq j'\leq j$, so the image of $(w_0,w_1,\dotsc)$ lies in the above subset.

  In the other direction, let $(b_0,b_1,\dotsc)$ be a sequence in $A$ satisfying $b_{j+1}\equiv\vp(b_j)\pmod{\pi^{j+1}}$ for all $j\geq0$. Define $a_0\coloneqq b_0$, and inductively apply
  \begin{align*}
    b_{j+1}&\equiv \vp(b_j) = \vp(a_0)^{q^j}+\pi\vp(a_1)^{q^{j-1}}+\dotsb+\pi^j\vp(a_j)\pmod{\pi^{j+1}}\\
           &\equiv a_0^{q^{j+1}}+\pi a_1^{q^j}+\dotsb+\pi^ja_j^q\pmod{\pi^{j+1}},
  \end{align*}
which implies that $\pi^{j+1}$ divides $b_{j+1}-(a_0^{q^{j+1}}+\pi a_1^{q^j}+\dotsb+\pi^ja_j^q)$, to define
  \begin{align*}
    a_{j+1}\coloneqq\frac1{\pi^{j+1}}\big(b_{j+1}-(a_0^{q^{j+1}}+\pi a_1^{q^j}+\dotsb+\pi^ja_j^q)\big).
  \end{align*}
Then the image of $(a_0,a_1,\dotsc)$ under $(w_0,w_1,\dotsc)$ equals $(b_0,b_1,\dotsc)$, which yields the desired result.

For part ii), note that part i) uniquely determines the value of $W$ when $A$ is a polynomial $\cO$-algebra. Finally, reducing to this universal case uniquely determines the value of $W$ on any $\cO$-algebra.
\end{proof}

\begin{prop}\label{prop:Wittprop}
Let $A$ be an $\cO$-algebra.
\begin{enumerate}[i)]
\item The map $[-]:A\ra W(A)$ given by $a\mapsto(a,0,\dotsc)$ is multiplicative,
\item There exists unique maps $\vp:W(A)\ra W(A)$ and $V:W(A)\ra W(A)$ such that, for all non-negative integers $j$, we have $w_j\circ\vp=w_{j+1}$ and
  \begin{align*}
    w_j\circ V=
    \begin{cases}
      0 & \mbox{if }j=0,\\
      \pi\cdot w_{j-1} & \mbox{otherwise.}
    \end{cases}
  \end{align*}
  Consequently, $\vp$ is an $\cO$-algebra homomorphism, $V$ is $\cO$-linear, and $\vp\circ V=\pi$.
\item Let $n$ be a positive integer. Then $\im V^n$ is an ideal of $W(A)$, and the map $W(A)\ra A^n$ given by $(a_0,a_1,\dotsc)\mapsto(a_0,\dotsc,a_{n-1})$ induces an identification from $W^n(A)\coloneqq W(A)/\im V^n$ to $A^n$.
\item Assume that $A$ is an $\cO/\pi$-algebra. Then $V\circ\vp=\pi$, and $\vp$ equals the $\cO$-algebra homomorphism $W(A)\ra W(A)$ induced by absolute $q$-Frobenius $A\ra A$.
\end{enumerate}
\end{prop}
\begin{proof}
  First, assume that $A$ is $\pi$-torsionfree and has an $\cO$-algebra endomorphism whose reduction mod $\pi$ equals absolute $q$-Frobenius. Then Proposition \ref{ss:Witt}.i) identifies $W(A)$ with its image under $(w_0,w_1,\dotsc)$. This uniquely determines $\vp$ and $V$, as well as verifies part i), part ii), and the first statement in part iii). Note that $V$ sends $(a_0,a_1,\dotsc)$ to $(0,a_0,a_1,\dotsc)$, which implies the second statement in part iii).

  The above includes the case where $A$ is a polynomial $\cO$-algebra. Reducing to this universal case uniquely determines $\vp$ and $V$, as well as verifies parts i)--iii), for any $\cO$-algebra. Similarly, reducing to the universal case and taking reduction mod $\pi$ verifies part iv).
\end{proof}

\subsection{}\label{ss:Witttiltadjunction}
When restricted to perfect $\cO/\pi$-algebras $A$, our Witt vectors satisfy the following usual properties. Proposition \ref{prop:Wittprop}.ii) and Proposition \ref{prop:Wittprop}.iv) imply that $\vp^n:W^n(A)\ra W(A)/\pi^n$ is an isomorphism for all positive integers $n$. Proposition \ref{prop:Wittprop}.iii) identifies $W(A)$ with $\varprojlim_n W^n(A)$, so this shows that $W(A)$ is $\pi$-adically complete and $\pi$-torsionfree. Note that $(a_0,a_1,\dotsc)$ in $W(A)$ equals $\sum_{n=0}^\infty[a_n^{1/q^n}]\pi^n$.

For all $\pi$-adically complete $\cO$-algebras $B$, recall that the tilt $B^\flat$ is naturally isomorphic to the inverse limit perfection of $B/\pi$ \cite[Lemma 3.2(i)]{BMS18}.
\begin{prop*}
Consider the functor
  \begin{align*}
    W:\{\mbox{perfect }\cO/\pi\mbox{-algebras}\}\ra\{\pi\mbox{-adically complete }\cO\mbox{-algebras}\}.
  \end{align*}
  \begin{enumerate}[i)]
  \item It is left adjoint to the tilt functor $(-)^\flat$. The unit $A\ra W(A)^\flat$ is an isomorphism, and the counit $\te:W(B^\flat)\ra B$ sends $\sum_{n=0}^\infty[r_n]\pi^n$ to $\sum_{n=0}^\infty r_n^\sharp\pi^n$.
  \item It is fully faithful, and its essential image consists of $\pi$-adically complete, $\pi$-torsionfree $B$ such that $B/\pi$ is perfect.
  \end{enumerate}
\end{prop*}
\begin{proof}
For part i), our candidate for the unit is the isomorphism
  \begin{align*}
    A=\varprojlim_rA=\varprojlim_rW(A)/\pi = W(A)^\flat,
  \end{align*}
  where $r$ runs over non-negative integers, and the transition maps are given by absolute $q$-Frobenius.

  We turn to our candidate for the counit. Let $n$ be a positive integer, and recall that the $\cO$-algebra homomorphism $w_{n-1}:W^n(B)\ra B$ is given by
  \begin{align*}
    (b_0,\dotsc,b_{n-1})\mapsto b_0^{q^{n-1}}+\pi b_1^{q^{n-2}}+\dotsb+\pi^{n-1}b_{n-1}.
  \end{align*}
For all $0\leq j\leq n-1$, we see that $b\equiv b'\pmod{\pi}$ implies that $b^{q^j}\equiv b'^{q^j}\pmod{\pi^{j+1}}$, so the composition $W^n(B)\ra B\ra B/\pi^n$ factors through an $\cO$-algebra homomorphism $\te_n:W^n(B/\pi)\ra B/\pi^n$. Proposition \ref{prop:Wittprop}.ii) implies that the square
  \begin{align*}
    \xymatrix{W^{n+1}(B/\pi)\ar[r]^-\vp\ar[d]^-{\te_{n+1}} & W^{n+1}(B/\pi)\ar[r] & W^n(B/\pi)\ar[d]^-{\te_n} \\
    B/\pi^{n+1}\ar[rr] & & B/\pi^n}
  \end{align*}
  commutes, so taking $\varprojlim_n$ yields an $\cO$-algebra homomorphism
  \begin{align*}
    \te:W(B^\flat)=\varprojlim_nW^n(B/\pi)\ra\varprojlim_n B/\pi^n=B.
  \end{align*}
  Tracing through our identifications shows that our candidates satisfy the counit-unit equations for $W$ and $(-)^\flat$, as well as the desired description of $\te$.

For part ii), full faithfulness follows from part i), so we focus on the essential image. When $B/\pi$ is perfect, note that $(-)^\sharp$ yields a section $B/\pi\cong B^\flat\ra B$ of $B\ra B/\pi$. Therefore, when $B$ is also $\pi$-torsionfree, every element of $B$ equals $\sum_{n=0}^\infty r_n^\sharp\pi^n$ for uniquely determined $r_n$ in $B/\pi$. This shows that $\te:W(B^\flat)\ra B$ is an isomorphism, so $B$ lies in the essential image of $W$.
\end{proof}

\subsection{}\label{ss:Wittfet}
In this subsection, assume that $A$ is a perfect $\cO/\pi$-algebra.
\begin{cor*}
  Reduction mod $\pi$ induces an equivalence of categories
  \begin{align*}
    \{\mbox{finite \'etale }W(A)\mbox{-algebras}\}\ra^\sim\{\mbox{finite \'etale }A\mbox{-algebras}\}.
  \end{align*}
Moreover, it has a quasi-inverse given by $W$.
\end{cor*}
\begin{proof}
  The first statement follows from the $\pi$-adic completeness of $W(A)$, since $W(A)/\pi=A$. For the second statement, let $A'$ be a finite \'etale $A$-algebra, and $B'$ be a finite \'etale $W(A)$-algebra satisfying $B'/\pi\cong A'$. Since $W(A)$ is $\pi$-adically complete and $\pi$-torsionfree, so is the finite projective $W(A)$-module $B'$. Moreover, \cite[Lemma 3.1.5]{KL15} shows that $A'$ is perfect. Hence Proposition \ref{ss:Witttiltadjunction}.ii) implies that $B'\cong W(A')$, as desired.
\end{proof}

\subsection{}
Later, we will need the analogue of Witt bivectors over $\cO$, so first let us introduce the analogue of Witt covectors over $\cO$. Write $CW^u(A)$ for the $\cO$-module $\varinjlim_nW^n(A)$, where the transition maps are given by $V:W^n(A)\ra W^{n+1}(A)$, and write $V:CW^u(A)\ra CW^u(A)$ for the map given by $\varinjlim_n$ of the composition
\begin{align*}
\xymatrix{W^n(A)\ar[r]^-V & W^{n+1}(A)\ar[r] & W^n(A).}
\end{align*}
Identify $CW^u(A)$ with the set of sequences $(\dotsc,a_{-1},a_0)$ in $A$ satisfying the following property: there exists an $M\geq0$ such that $a_{-m}=0$ for all $m\geq M$. Under this identification, $V$ sends $(\dotsc,a_{-1},a_0)$ to $(\dotsc,a_{-2},a_{-1})$.

\subsection{}
Write $CW(A)$ for the set of sequences $(\dotsc,a_{-1},a_0)$ in $A$ satisfying the following weaker property:
\begin{align*}\label{eq:nilpotenttail}
  \mbox{there exists an }M\geq0\mbox{ such that the ideal }(a_{-m})_{m\geq M}\mbox{ of }A\mbox{ is nilpotent.}\tag{$\triangleleft$}
\end{align*}
\begin{lem*}
The $\cO$-module structure on $CW^u(A)$ naturally extends to $CW(A)$, and the natural extension of $V$ to $CW(A)$ is $\cO$-linear.
\end{lem*}
Write $BW(A)$ for the $\cO$-module $\varprojlim_rCW(A)$, where $r$ runs over non-negative integers, and the transition maps are given by $V:CW(A)\ra CW(A)$. We identify $BW(A)$ with the set of double sequences $(\dotsc,a_{-1},a_0,a_1,\dotsc)$ in $A$ satisfying \eqref{eq:nilpotenttail}.
\begin{proof}
Note that \cite[n$^\circ$ II.1.5, lemme 1.2]{Fon77} holds after replacing $p$ with $q$. Therefore, for the polynomials $S_j$ that compute the $j$-th coordinate of addition on $W$, the analogue of \cite[n$^\circ$ II.1.5, lemme 1.3]{Fon77} also holds, and the bound depends only on $q$. Thus the desired result follows from the proof of \cite[n$^\circ$ II.1.5, proposition 1.1]{Fon77}.
\end{proof}

\subsection{}\label{ss:Wittspecialization}
Crucially, our Witt vectors specialize to the usual Witt vectors as follows. For all $i$ in $\bN\cup\{\infty\}$, write $W_i$ for the ring of $\cO_i$-typical Witt vectors as in \cite[Proposition 1.1]{Dri76}, and view $\cO_i$ as an $\cO$-algebra via the surjective ring homomorphism $\ev_i:\cO\ra\cO_i$.

In this subsection, assume that $A$ is an $\cO_i$-algebra.
\begin{prop*}
The $\cO$-algebra $W(A)$ is naturally isomorphic to $W_i(A)$.
\end{prop*}
\begin{proof}
  First, assume that $A$ is $\pi$-torsionfree and has an $\cO$-algebra endomorphism whose reduction mod $\pi$ equals absolute $q$-Frobenius. Then Proposition \ref{ss:Witt}.i) identifies $W(A)$ with its image under $(w_0,w_1,\dotsc)$, and this satisfies the characterizing property in \cite[Proposition 1.1]{Dri76}. Therefore uniqueness in \cite[Proposition 1.1]{Dri76} implies that $W(A)$ is naturally isomorphic to $W_i(A)$.

  The above includes the case where $A$ is a polynomial $\cO_i$-algebra. Finally, reducing to this universal case shows that $W(A)$ is naturally isomorphic to $W_i(A)$ for any $\cO_i$-algebra $A$.
\end{proof}

\subsection{}\label{ss:Ycurve}
Note that $\cO/\pi\cong\Cont(\bN\cup\{\infty\},\bF_q)$ and that $\Spd\Cont(\bN\cup\{\infty\},\bF_q)$ equals the v-sheaf $\ul{\bN\cup\{\infty\}}$ on $\Perf_{\bF_q}$. Hence $\Spa$ yields an anti-equivalence from the category of perfectoid Huber pairs over $\cO/\pi$ to the category of affinoid perfectoid spaces over $\bF_q$ equipped with a morphism to $\ul{\bN\cup\{\infty\}}$.

We can now define the analogue of relative Fargues--Fontaine curves over $\cO$. Let $S=\Spa(R,R^+)$ be an affinoid perfectoid space over $\bF_q$ equipped with a morphism $S\ra\ul{\bN\cup\{\infty\}}$. Choose a pseudouniformizer $\vpi$ of $R$, and write $\norm{-}$ for the associated spectral norm on $R$ normalized by $\norm{\vpi}=\frac1q$. Endow $W(R^+)$ with the $(\pi,[\vpi])$-adic topology, write $\cY_S$ for the complement of the vanishing locus of $(\pi,[\vpi])$ in $\Spa{W(R^+)}$, and note that $\cY_S$ is the analytic locus of $\Spa{W(R^+)}$.

We have a continuous map $\rad:\abs{\cY_S}\ra[0,\infty]$ given by
\begin{align*}
  x\mapsto\frac{\log\abs{[\vpi](\wt{x})}}{\log\abs{\pi(\wt{x})}},
\end{align*}
where $\wt{x}$ denotes the unique rank-$1$ generalization of $x$ in $\cY_S$. For any closed interval $\cI$ in $[0,\infty]$ with rational endpoints, write $\cY_{S,\cI}=\Spa(B_{S,\cI},B_{S,\cI}^+)$ for the associated rational open subspace of $\Spa{W(R^+)}$, which lies in $\cY_S$. More generally, for any subset $\cI$ of $[0,\infty]$, write $\cY_{S,\cI}$ for the open subspace $\bigcup_{\cI'}\cY_{S,\cI'}$ of $\cY_S$, where $\cI'$ runs over closed intervals in $\cI$ with rational endpoints. Write $Y_S$ for $\cY_{S,(0,\infty)}$, and note that $\cY_{S,[0,\infty)}$ is the complement of the vanishing locus of $[\vpi]$ in $\Spa{W(R^+)}$.

Write $\vp:S\ra S$ for the absolute $q$-Frobenius automorphism, as well as the induced automorphisms of $W(R^+)$ and $\cY_S$. Note that $\rad\circ\vp=q\cdot\rad$, so $\vp$ acts freely and totally discontinuously on $Y_S$ over $\Spa{E}$. Write $X_S$ for the quotient $Y_S/\vp^\bZ$.

\subsection{}\label{ss:Yinftysousperfectoid}
Write $W'(R^+)$ for the ring $W(R^+)$ endowed with the $\pi$-adic topology.
\begin{lem*}\hfill
  \begin{enumerate}[i)]
  \item The ring $W'(R^+)[\frac1\pi]$ is sousperfectoid,
  \item The pre-adic space $\cY_{S,[1,\infty]}$ is sousperfectoid and hence adic.
  \end{enumerate}
\end{lem*}
\begin{proof}
For part i), we claim that $W'(R^+)\wh\otimes_\cO\cO^{\perf}$ is integral perfectoid. Note that
  \begin{align*}
\cO^{\perf}/\pi = \bigcup_{r=0}^\infty\cO[\pi^{1/q^r}]/\pi = \bigcup_{r=0}^\infty(\cO/\pi)[t^{1/q^r}]/t,
  \end{align*}
  where $t^{1/q^r}$ equals the image of $\pi^{1/q^r}$, so we get
  \begin{align*}
    (W'(R^+)\wh\otimes_\cO\cO^{\perf})/\pi = R^+\otimes_{\cO/\pi}(\cO^{\perf}/\pi) = \bigcup_{r=0}^\infty R^+[t^{1/q^r}]/t.
  \end{align*}
  Replacing $\pi$ with $\pi^{1/q}$ yields $(W'(R^+)\wh\otimes_\cO\cO^{\perf})/\pi^{1/q}=\bigcup_{r=0}^\infty R^+[t^{1/q^r}]/t^{1/q}$. Since $R^+$ is integral perfectoid, these descriptions imply that the absolute $q$-Frobenius
  \begin{align*}
    (W'(R^+)\wh\otimes_\cO\cO^{\perf})/\pi^{1/q}\ra(W'(R^+)\wh\otimes_\cO\cO^{\perf})/\pi
  \end{align*}
  is an isomorphism, so \cite[Lemma 3.10(ii)]{BMS18} yields the claim. From here, \cite[Lemma 3.21]{BMS18} indicates that $(W'(R^+)\wh\otimes_\cO\cO^{\perf})[\textstyle\frac1\pi]=W'(R^+)[\frac1\pi]\wh\otimes_EE^{\perf}$ is perfectoid, which shows that $W'(R^+)[\frac1\pi]$ is sousperfectoid.

  For part ii), note that the $(\pi,[\vpi])$-adic topology on $W(R^+)[[\vpi]/\pi]$ equals the $\pi$-adic topology. Therefore $B_{S,[1,\infty]}$ equals the global sections of the rational open subspace $\{\abs{[\vpi]}\leq\abs{\pi}\neq0\}$ of $\Spa W'(R^+)[\frac1\pi]$, which is sousperfectoid by part i).
\end{proof}

\begin{prop}\label{ss:Ysousperfectoid}
  For all non-negative integers $r$, the pre-adic space
  \begin{align*}
    \cY_{S,[0,q^r]}\times_{\Spa\cO}\Spa\cO^{\perf}
  \end{align*}
  is affinoid perfectoid, and our choices of $\pi$ and $\vpi$ induce an identification between $(\cY_{S,[0,q^r]}\times_{\Spa\cO}\Spa\cO^{\perf})^\flat$ and the closed perfectoid unit disk over $S$. Consequently, $\cY_S$, $Y_S$, and $X_S$ are sousperfectoid adic spaces.
\end{prop}
\begin{proof}
  We claim that $B_{S,[0,1]}\wh\otimes_\cO\cO^{\perf}$ is perfectoid. To see this, write $A$ for the $[\vpi]$-adic completion of $\bigcup_{m=0}^\infty(W(R^+)\wh\otimes_\cO\cO^{\perf})[(\pi/[\vpi])^{1/q^m}]$. Because $B_{S,[0,1]}$ equals $W(R^+)\ang{\pi/[\vpi]}[\textstyle\frac1{[\vpi]}]$, we see that $B_{S,[0,1]}\wh\otimes_\cO\cO^{\perf}$ equals $A[\textstyle\frac1{[\vpi]}]$. Now
  \begin{align*}
    A/[\vpi] &= \bigcup_{m=0}^\infty\big(W(R^+)\wh\otimes_\cO\cO^{\perf}\big)[(\pi/[\vpi])^{1/q^m}]/[\vpi] \\
    &= \bigcup_{m=0}^\infty\big((R^+/\vpi)\otimes_{\cO/\pi}(\cO^{\perf}/\pi)\big)[X^{1/q^m}]/(\pi^{1/p^m}-\vpi^{1/q^m}X^{1/q^m}),
  \end{align*}
  where $X^{1/q^m}$ equals the image of $(\pi/[\vpi])^{1/q^m}$. The proof of Lemma \ref{ss:Yinftysousperfectoid}.i) identifies $\cO^{\perf}/\pi$ with $\bigcup_{r=0}^\infty(\cO/\pi)[t^{1/q^r}]/t$, so the above expression for $A/[\vpi]$ becomes
  \begin{align*}
\bigcup_{m=0}^\infty\Big(\bigcup_{r=0}^\infty(R^+/\vpi)[t^{1/q^r}]/t\Big)[X^{1/q^m}]/(t^{1/q^m}-\vpi^{1/q^m}X^{1/q^m}) = \bigcup_{m=0}^\infty(R^+/\vpi)[X^{1/q^m}].
  \end{align*}
  Replacing $\vpi$ with $\vpi^{1/q}$ yields $A/[\vpi]^{1/q}=\bigcup_{m=0}^\infty(R^+/\vpi^{1/q})[X^{1/q^m}]$. Since $R^+$ is integral perfectoid, these descriptions imply that the absolute $q$-Frobenius morphism $A/[\vpi]^{1/q}\ra A/[\vpi]$ is an isomorphism, so \cite[Lemma 3.10(ii)]{BMS18} shows that $A$ is integral perfectoid. From here, \cite[Lemma 3.21]{BMS18} indicates that $A[\textstyle\frac1{[\vpi]}]=B_{S,[0,1]}\wh\otimes_\cO\cO^{\perf}$ is perfectoid, as desired.

  The claim shows that $\cY_{S,[0,1]}$ is sousperfectoid. Because $\cY_S$ is covered by $\cY_{S,[0,1]}$ and $\cY_{S,[1,\infty]}$, combining this with Lemma \ref{ss:Yinftysousperfectoid}.ii) indicates that $\cY_S$ is sousperfectoid. Moreover, the above work identifies $A^\flat$ with $R^+\ang{X^{1/p^\infty}}$ and hence identifies
  \begin{align*}
    (\cY_{S,[0,1]}\times_{\Spa\cO}\Spa\cO^{\perf})^\flat
  \end{align*}
with the closed perfectoid unit disk over $S$. Since $\vp^r$ induces an isomorphism $\cY_{S,[0,1]}\ra^\sim\cY_{S,[0,q^r]}$, the same holds for $(\cY_{S,[0,q^r]}\times_{\Spa\cO}\Spa\cO^{\perf})^\flat$.
\end{proof}

\subsection{}\label{ss:diamondformula}
As usual, relative Fargues--Fontaine curves enjoy the following ``diamond formula.'' Because $\cO^\flat\cong\cO/\pi$, we get a natural morphism $\Spd{\cO}\ra\ul{\bN\cup\{\infty\}}$. 
\begin{prop*}
The v-sheaves $\cY_{S,[0,\infty)}^\Diamond$ and $\Spd{\cO}\times_{\ul{\bN\cup\{\infty\}}}S$ are naturally isomorphic over $\Spd\cO$. Consequently, we get a natural continuous map $\abs{\cY_{S,[0,\infty)}}\ra\abs{S}$.
\end{prop*}
\begin{proof}
  For the first statement, let $T=\Spa(A,A^+)$ be an affinoid perfectoid space along with an untilt $T^\sharp=\Spa(A^\sharp,A^{\sharp+})$ of $T$ over $\Spa{\cO}$. A morphism $T^\sharp\ra\cY_{S,[0,\infty)}$ over $\Spa\cO$ is equivalent to an $\cO$-algebra homomorphism $W(R^+)\ra A^{\sharp+}$ such that the image of $[\vpi]$ is invertible in $A^\sharp$. By Proposition \ref{ss:Witttiltadjunction}.i), this is equivalent to an $\cO/\pi$-algebra homomorphism $R^+\ra A^+$ such that the image of $\vpi$ is invertible in $A$, i.e. a morphism of Huber pairs $(R,R^+)\ra(A,A^+)$ over $\cO/\pi$. This yields the desired result. By \cite[Lemma 15.6]{Sch17}, the second statement follows from applying $\abs{-}$ to the composition $\cY^\Diamond_{S,[0,\infty)}\cong\Spd\cO\times_{\ul{\bN\cup\{\infty\}}}S\lra^{\pr_2}S$.
\end{proof}

\subsection{}\label{ss:Yrationalopensubspace}
Relative Fargues--Fontaine curves are open-local in the following sense. Let $U$ be a rational open subspace of $S$.
\begin{prop*}
  The morphism $\cY_{U,[0,\infty)}\ra\cY_{S,[0,\infty)}$ is an open embedding, and its image is the subset $|\cY_{S,[0,\infty)}|\times_{\abs{S}}\abs{U}$ of $|\cY_{S,[0,\infty)}|$.
\end{prop*}
\begin{proof}
  Write $U=\{\abs{f_j}\leq\abs{g}\neq0\mbox{ for all }1\leq j\leq m\}$ for some $g,f_1,\dotsc,f_m$ in $R$ such that the ideal generated by $f_1,\dotsc,f_m$ is open. Write $\cY(U)$ for the open subspace $\{\abs{[f_j]}\leq\abs{[g]}\neq0\mbox{ for all }1\leq j\leq m\}$ of $\cY_{S,[0,\infty)}$, and note that the image of $\cY_{U,[0,\infty)}\ra\cY_{S,[0,\infty)}$ lies in $\cY(U)$. The proof of Proposition \ref{ss:diamondformula} shows that $\cY(U)^\Diamond$ is naturally isomorphic to $\Spd\cO\times_{\ul{\bN\cup\{\infty\}}}U$ over $\Spd\cO$, so \cite[Lemma 15.6]{Sch17} implies that $\abs{\cY(U)}=|\cY_{S,[0,\infty)}|\times_{\abs{S}}\abs{U}$.

  Now $\cY_{U,[0,\infty)}\ra\cY(U)$ equals the increasing union of its restrictions
  \begin{align*}
    \cY_{U,[0,q^r]}\ra \cY(U)\cap\cY_{S,[0,q^r]},
  \end{align*}
where $r$ runs over non-negative integers. Thus it suffices to prove that these restrictions are isomorphisms. Both sides are affinoid, so it suffices to check on global sections, and because $\cO$ is a direct summand of $\cO^{\perf}$ as topological $\cO$-modules, it suffices to show that
  \begin{align*}
    \cY_{U,[0,q^r]}\times_{\Spa\cO}\Spa\cO^{\perf}\ra(\cY(U)\cap\cY_{S,[0,q^r]})\times_{\Spa\cO}\Spa\cO^{\perf}
  \end{align*}
  is an isomorphism. Proposition \ref{ss:Ysousperfectoid} indicates that both sides are perfectoid, so it suffices to check after applying $(-)^\flat$. Finally, this follows from Proposition \ref{ss:Ysousperfectoid}.
\end{proof}

\subsection{}\label{ss:kertheta}
Next, we turn to divisors on relative Fargues--Fontaine curves over $\cO$. Let $S\ra\Spd{\cO}$ be a morphism over $\ul{\bN\cup\{\infty\}}$, and write $S^\sharp=\Spa(R^\sharp,R^{\sharp+})$ for the associated untilt of $S$ over $\Spa{\cO}$.
\begin{prop*}
  The $\cO$-algebra homomorphism $\te:W(R^+)\ra R^{\sharp+}$ is surjective. Its kernel is generated by an element of the form $\pi+[\vpi]\al$ for some pseudouniformizer $\vpi$ of $R$ and $\al$ in $W(R^+)$, and $\ker\te$ induces a closed Cartier divisor $S^\sharp\hra\cY_{S,[0,\infty)}$.
\end{prop*}
\begin{proof}
  Choose a pseudouniformizer $\vpi$ of $R$ such that $\vpi^\sharp$ divides $\pi$ in $R^{\sharp+}$. Since $W(R^+)$ and $R^{\sharp+}$ are $[\vpi]$-adically complete, it suffices to check the surjectivity of $\te$ mod $[\vpi]$. There, $\te$ becomes $R^+\ra R^{\sharp+}/\vpi^\sharp$, which is surjective because $R^{\sharp+}$ is integral perfectoid.

  The fact that $R^{\sharp+}$ is integral perfectoid also shows that there exists $\be$ in $R^+$ such that $\be^\sharp\equiv\pi/\vpi^\sharp\pmod{\pi}$, so $\vpi^\sharp\be^\sharp\equiv\pi\pmod{\vpi^\sharp\pi}$. The surjectivity of $\te$ then yields $\sum_{n=0}^\infty[r_n]\pi^n$ in $W(R^+)$ such that $\vpi^\sharp\be^\sharp=\pi+\vpi^\sharp\pi\sum_{n=0}^\infty r_n^\sharp\pi^n$. By setting $\al\coloneqq[\be]-\sum_{n=0}^\infty[r_n]\pi^{n+1}$, we see that $\xi\coloneqq\pi+[\vpi]\al$ lies in $\ker\te$.

  Note that $\xi$ is not divisible by $[\vpi]$, so $W(R^+)/\xi$ remains $[\vpi]$-adically complete and $[\vpi]$-torsionfree. Therefore, to show that the induced map $W(R^+)/\xi\ra R^{\sharp+}$ is an isomorphism, it suffices to check mod $[\vpi]$. There, it becomes
  \begin{align*}
    W(R^+)/(\xi,[\vpi]) = W(R^+)/(\pi,[\vpi])=R^+/\vpi\ra R^{\sharp+}/\vpi^\sharp,
  \end{align*}
which is indeed an isomorphism.

Finally, suppose that $\xi\sum_{n=0}^\infty[s_n]\pi^n$ for some $\sum_{n=0}^\infty[s_n]\pi^n$ in $W(R^+)$. Then
\begin{align*}
  \sum_{n=0}^\infty[s_n]\pi^{n+1}=\pi\sum_{n=0}^\infty[s_n]\pi^n\equiv\xi\sum_{n=0}^\infty[s_n]\pi^n=0\pmod{[\vpi]},
\end{align*}
so the $s_n$ are divisible by $\vpi$. By repeatedly replacing the $s_n$ with $s_n/\vpi$, this shows that the $s_n=0$. Hence $\xi$ is a non-zerodivisor. Finally, by using Proposition \ref{ss:Ysousperfectoid} instead of \cite[Proposition II.1.1]{FS21} and the fact that the tilt of $S^\sharp\ra\Spa\cO$ equals $S\ra\ul{\bN\cup\{\infty\}}$, the proof of \cite[Proposition II.1.4]{FS21} shows that the induced Cartier divisor $S^\sharp\ra\cY_{S,[0,\infty)}$ is closed.
\end{proof}

\subsection{}\label{ss:Div1}
Write $\Div^1_Y$ for the v-sheaf $\Spd{E}$ over $\ul{\bN\cup\{\infty\}}$, and write $\phi:\Div^1_Y\ra\Div^1_Y$ for the absolute $q$-Frobenius automorphism. By Proposition \ref{ss:kertheta}, the $S$-points of $\Div^1_Y$ naturally induce closed Cartier divisors $S^\sharp\hra Y_S$. Note that $\phi$ sends $S^\sharp\hra Y_S$ to its image under $\vp:Y_S\ra Y_S$, so $\phi$ acts freely on $\Div^1_Y$.

Write $\Div^1_X$ for the quotient $\Div^1_Y\!/\phi^\bZ$. Since $\vp$ acts totally discontinuously on $Y_S$, the $S$-points of $\Div_X^1$ naturally induce closed Cartier divisors $S^\sharp\hra X_S$.
\begin{lem*}
Our $\Div^1_X\ra\ul{\bN\cup\{\infty\}}$ is representable in spatial diamonds and proper.
\end{lem*}
\begin{proof}
  Because absolute $q$-Frobenius is compatible with products and acts trivially on underlying topological spaces, we have
  \begin{align*}
    \abs{\Div^1_X\times_{\ul{\bN\cup\{\infty\}}}S} = \abs{\Div_Y^1\times_{\ul{\bN\cup\{\infty\}}}S}\big/\big(\phi\times\id\big)^\bZ = \abs{\Div_Y^1\times_{\ul{\bN\cup\{\infty\}}}S}\big/\big(\id\times\vp\big)^\bZ.
  \end{align*}
  Proposition \ref{ss:diamondformula} indicates that $\Div^1_Y\times_{\ul{\bN\cup\{\infty\}}}S\cong Y^{\Diamond}_S$. Therefore the above implies that $\Div_X^1\times_{\ul{\bN\cup\{\infty\}}}S$ has an open cover by v-sheaves of the form $\cY_{S,\cI}^\Diamond$ for small enough closed intervals $\cI$ in $[1,q]$, so \cite[Lemma 15.6]{Sch17} shows that $\Div^1_X\times_{\bN\cup\{\infty\}}S$ is a locally spatial diamond. Similarly, we get a surjective continuous map $|\cY_{S,[1,q]}|\twoheadrightarrow|\Div^1_X\times_{\ul{\bN\cup\{\infty\}}}S|$, so the latter is qcqs. Hence \cite[Proposition 11.19 (iii)]{Sch17} shows that $\Div^1_X\times_{\ul{\bN\cup\{\infty\}}}S$ is spatial, and properness follows from \cite[Proposition 18.3]{Sch17}.
\end{proof}

\subsection{}\label{ss:Ebreve}
Write $(-)_{\ov\bF_q}$ for the fiber product $-\times_{\Spd\bF_q}\Spd\ov\bF_q$.
\begin{lem*}
We have a natural isomorphism $\Spd\breve\cO\cong(\Spd\cO)_{\ov\bF_q}$ over $\Spd\cO$.
\end{lem*}
\begin{proof}
  Note that $\breve\cO$ is $\pi$-adically complete, $\pi$-torsionfree, and satisfies
  \begin{align*}
    \breve\cO/\pi\cong\Cont(\bN\cup\{\infty\},\ov\bF_q).
  \end{align*}
  Hence Proposition \ref{ss:Witttiltadjunction}.ii) implies that $\breve\cO\cong W(\Cont(\bN\cup\{\infty\},\ov\bF_q))$, so a morphism $S^\sharp\ra\Spa\breve\cO$ over $\Spa\cO$ is equivalent to an $\cO$-algebra homomorphism
  \begin{align*}
    W(\Cont(\bN\cup\{\infty\},\ov\bF_q))\ra R^{\sharp+}.
  \end{align*}
  By Proposition \ref{ss:Witttiltadjunction}.i), this is equivalent to an $\cO/\pi$-algebra homomorphism
  \begin{align*}
    (\cO/\pi)\otimes_{\bF_q}\ov\bF_q\cong\Cont(\bN\cup\{\infty\},\ov\bF_q)\ra R^+.
  \end{align*}
  Applying Proposition \ref{ss:Witttiltadjunction}.i) again shows that this is equivalent to an $\cO$-algebra homomorphism $W(\cO/\pi)\ra R^{\sharp+}$ along with an $\bF_q$-algebra homomorphism $\ov\bF_q\ra R^+$. Finally, Proposition \ref{ss:Witttiltadjunction}.ii) implies that $W(\cO/\pi)\cong\cO$, and this yields the desired result.
\end{proof}

\subsection{}\label{ss:Waction}
To state the analogue of Drinfeld's lemma in our context, we need the analogue of the absolute Weil group for $E$. For all $i$ in $\bN\cup\{\infty\}$, write $W_i$ for the absolute Weil group of $\ov{E}_i/E_i$, write $I_i$ for its inertia subgroup, and write $v_i:W_i/I_i\ra\bZ$ for the isomorphism that sends geometric $q$-Frobenius to $1$. By repeating Definition \ref{ss:Galoisgroups} (except that we replace $\Ga_i/I_i^n$ with $W_i/I_i^n$ or $I_i/I_i^n$) we obtain natural group topological spaces $\bW$ and $\bI$ over $\bN\cup\{\infty\}$ whose fibers at $i$ are isomorphic to $W_i$ and $I_i$, respectively. Note that $\bW$ and $\bI$ are lctd over $\bN\cup\{\infty\}$ as in Definition \ref{ss:grouphypotheses}.

View $\breve{E}$ as an $E$-subalgebra of $C$. Because $\breve\cO^\flat\cong\Cont(\bN\cup\{\infty\},\ov\bF_q)$, we get natural morphisms $\Spd{C}\ra\Spd\breve{E}\ra\ul{\bN\cup\{\infty\}}_{\ov\bF_q}$.
\begin{prop*}
We have a natural action of $\ul{\bW}$ on $\Spd{C}$ over $\ul{\bN\cup\{\infty\}}_{\ov\bF_q}$ such that $\Spd{C}\ra\Spd\breve{E}$ induces an isomorphism $(\Spd{C})/\ul{\bW}\ra^\sim(\Div^1_X)_{\ov\bF_q}$ over $\ul{\bN\cup\{\infty\}}_{\ov\bF_q}$.
\end{prop*}
\begin{proof}
  Arguing as in the proof of Lemma \ref{ss:LTextensions} shows that the maps
  \begin{align*}
     \bW\ra\bZ\times(\bN\cup\{\infty\})\mbox{ and }\bW\ra\Ga
  \end{align*}
over $\bN\cup\{\infty\}$ whose fibers at $i$ equal $W_i\lra W_i/I_i\lra^{v_i}\bZ$ and $W_i\hra\Ga_i$, respectively, are continuous. By checking on fibers, we see that they are morphisms of group topological spaces over $\bN\cup\{\infty\}$.

  Corollary \ref{cor:Efet} and \cite[Lemma 15.6]{Sch17} imply that $\Spd{C}\ra\Spd{E}$ is a $\ul{\Ga}$-torsor. Therefore precomposing with $\ul{\bW}\ra\ul{\Ga}$ yields an action $a$ of $\ul{\bW}$ on $\Spd{C}$ over $\ul{\bN\cup\{\infty\}}$. Since absolute $q$-Frobenius $\vp$ is an automorphism of $\Spd{C}$ over $\ul{\bN\cup\{\infty\}}$ that commutes with $a$, it induces an action of $\ul{\bZ\times(\bN\cup\{\infty\})}$ on $\Spd{C}$ over $\ul{\bN\cup\{\infty\}}$ that commutes with $a$. By precomposing with $\ul{\bW}\ra\ul{\bZ\times(\bN\cup\{\infty\})}$ and multiplying by $a$, we get our candidate action of $\ul{\bW}$ on $\Spd{C}$ over $\ul{\bN\cup\{\infty\}}_{\ov\bF_q}$.

  Because the image of $\bI$ under $\bW\ra\bZ\times(\bN\cup\{\infty\})$ is trivial, the restriction to $\ul{\bI}$ of our candidate action equals $a$. Hence $(\Spd{C})/\ul{\bI}$ is naturally isomorphic to $\Spd\breve{E}$. Under the identification $\Spd\breve{E}\cong(\Div^1_Y)_{\ov\bF_q}$ from Lemma \ref{ss:Ebreve}, we see that the resulting action of $\ul{\bW}/\ul{\bI}\ra^\sim\ul{\bZ\times(\bN\cup\{\infty\})}$ on $(\Div^1_Y)_{\ov\bF_q}$ is induced by $\phi\times\id_{\ov\bF_q}$. Thus further quotienting by $\ul{\bW}/\ul{\bI}$ yields an natural isomorphism $(\Spd{C})/\ul{\bW}\ra^\sim(\Div^1_X)_{\ov\bF_q}$, as desired.
\end{proof}

\subsection{}\label{ss:Drinfeldfullyfaithful}
We conclude this section by proving some analogues of Drinfeld's lemma in our context. In this subsection, we work over $\Spd\ov\bF_q$. For any group topological space $G$ over $\bN\cup\{\infty\}$, write $BG$ for the v-stack $\ul{\bN\cup\{\infty\}}\big/\ul{G}$. Proposition \ref{ss:Waction} shows that $\Spd{C}\ra\ul{\bN\cup\{\infty\}}$ induces a morphism $\Div_X^1\ra B\bW$.

Let $\La$ be a ring that is $\ell$-power torsion.
\begin{prop*}
For all small v-stacks $Z$ over $\ul{\bN\cup\{\infty\}}$, the pullback functor
  \begin{align*}
    D_{\et}(Z\times_{\ul{\bN\cup\{\infty\}}}B\bW,\La)\ra D_{\et}(Z\times_{\ul{\bN\cup\{\infty\}}}\Div^1_X,\La)
  \end{align*}
  is fully faithful. Moreover, it is essentially surjective when the pullback functor
  \begin{align*}
    D_{\et}(Z,\La)\ra D_{\et}(Z\times_{\ul{\bN\cup\{\infty\}}}\Spd C,\La)
  \end{align*}
is essentially surjective.
\end{prop*}
\begin{proof}
Note that $\ul{\bN\cup\{\infty\}}\ra B\bW$ and $\Spd{C}\ra\Div^1_X$ induce simplicial v-hypercovers $\ul{\bW}^\bullet\ra B\bW$ and $\ul{\bW}^\bullet\times_{\ul{\bN\cup\{\infty\}}}\Spd{C}\ra\Div^1_X$, respectively, and this yields a commutative square
  \begin{align*}
    \xymatrix{D_{\et}(Z\times_{\ul{\bN\cup\{\infty\}}}B\bW,\La)\ar[r]\ar[d] & D_{\et}(Z\times_{\ul{\bN\cup\{\infty\}}}\ul{\bW}^{\bullet},\La)\ar[d] \\
    D_{\et}(Z\times_{\ul{\bN\cup\{\infty\}}}\Div^1_X,\La)\ar[r] & D_{\et}(Z\times_{\ul{\bN\cup\{\infty\}}}\ul{\bW}^{\bullet}\times_{\ul{\bN\cup\{\infty\}}}\Spd{C},\La).}
  \end{align*}
  The top and bottom functors are fully faithful by \cite[Proposition 17.3]{Sch17}. Proposition \ref{ss:Cperfectoid} identifies $\Spd{C}$ with $\ul{\bN\cup\{\infty\}}\times\Spa C_\infty$, so applying \cite[Theorem 19.5 (ii)]{Sch17} to $\ov\bF_q\hra C_\infty$ shows that the right functor is also fully faithful. Therefore the left functor is fully faithful, which yields the first statement.

  For the second statement, consider an object of $D_{\et}(Z\times_{\ul{\bN\cup\{\infty\}}}\Div_X^1,\La)$, and write $A$ for its image in $D_{\et}(Z\times_{\ul{\bN\cup\{\infty\}}}\ul{\bW}^{\bullet}\times_{\ul{\bN\cup\{\infty\}}}\Spd{C},\La)$. Then $A$ is cartesian, and by assumption its $\bullet=0$ part lies in the image of $D_{\et}(Z,\La)$. Since the right functor is fully faithful, this shows that $A$ arises from a cartesian object of $D_{\et}(Z\times_{\ul{\bN\cup\{\infty\}}}\ul{\bW}^\bullet,\La)$, so \cite[Proposition 17.3]{Sch17} indicates that $A$ lies in the essential image of the top functor, as desired.
\end{proof}

\subsection{}\label{ss:Drinfeldslemma}
In this subsection, we work over $\Spd\ov\bF_q$. For any small v-stack $Z$ over $\ul{\bN\cup\{\infty\}}$ and finite set $J$, write $Z^J$ for the $J$-fold fiber power of $Z$ over $\ul{\bN\cup\{\infty\}}$.
\begin{cor*}
  For all small v-stacks $Z$ over $\ul{\bN\cup\{\infty\}}$ and finite sets $J$, pullback
  \begin{align*}
    D_{\lc}(Z\times_{\ul{\bN\cup\{\infty\}}}B\bW^J,\La)\ra D_{\lc}(Z\times_{\ul{\bN\cup\{\infty\}}}(\Div_X^1)^J,\La)
  \end{align*}
  is an equivalence of categories.
\end{cor*}
\begin{proof}
By induction on $\#J$, we can assume that $\#J=1$. Full faithfulness follows from Proposition \ref{ss:Drinfeldfullyfaithful}, so we focus on essential surjectivity. Now full faithfulness and descent \cite[Proposition 17.3]{Sch17} indicate that, after replacing $Z$ with a v-cover, we can assume that $Z$ is strictly totally disconnected. Then Lemma \ref{ss:Div1} shows that $Z\times_{\ul{\bN\cup\{\infty\}}}\Div^1_X$ is a spatial diamond. Hence \cite[Proposition 20.15]{Sch17} implies that, after replacing $Z$ with a connected component, we can assume that $Z$ is a geometric point. Finally, the image of $\abs{Z}$ in $\bN\cup\{\infty\}$ equals $\{i\}$ for some $i$ in $\bN\cup\{\infty\}$, so the result follows from Proposition \ref{ss:Wittspecialization} and \cite[Proposition IV.7.3]{FS21}.
\end{proof}

\section{Lubin--Tate theory}\label{s:LT}
In this section, we develop a version of Lubin--Tate theory over $E$. We begin by using our generalization of Witt vectors to construct an analogue $\LT$ over $\cO$ of Lubin--Tate formal group laws, as in work of Fargues--Fontaine \cite{FF18}. This yields a natural choice of coordinates for $\LT$, which we use to explicitly define an analogue of the logarithm map. We also show that $\LT$ satisfies the expected relationship with the Lubin--Tate extensions defined in \S\ref{s:localfields}.

This version of Lubin--Tate theory has the following consequences. First, after computing the global sections of $\sO_{X_S}(1)$ on $X_S$ using Witt bivectors as in work of Fontaine \cite{Fon77} and Fargues--Fontaine \cite{FF18},\footnote{See \ref{ss:absoluteBanachColmez} for the definition of the line bundle $\sO_{X_S}(1)$ on $X_S$. In particular, there is no circularity.} it lets us reinterpret sections of $\sO_{X_S}(1)$ in terms of $\Div_X^1$. Second, it lets us prove that the morphism $\Div_X^1\ra\ul{\bN\cup\{\infty\}}$ is cohomologically smooth of dimension $1$, which will be needed in \S\ref{s:Satake}.

\subsection{}\label{ss:LT}
Recall that $\Spec$ yields an anti-equivalence from the category of $\cO$-algebras $A$ where $\pi$ is nilpotent to the category of affine schemes over $\Spf\cO$. Write $\wh{W}$ for the infinite-dimensional formal disk over $\Spf\cO$ whose $A$-points equal
\begin{align*}
\{(a_0,a_1,\dotsc)\in W(A)\mid\mbox{the }a_j\mbox{ are nilpotent, and cofinitely many }a_j\mbox{ are zero}\}.
\end{align*}
Note that $V$ preserves $\wh{W}$. Also, for all non-negative integers $j$, we see that the polynomials that compute the $j$-th coordinate of $\cO$-module operations on $W$ are homogeneous of degree $q^j$, where $T_{j'}$ has degree $q^{j-j'}$ for all $0\leq j'\leq j$, so $\wh{W}$ is naturally a formal $\cO$-module.

Write $\LT$ for the $\cO$-module object $\coker(V-1:\wh{W}\ra\wh{W})$ over $\Spf\cO$. By the discussion on \cite[p.~158--159]{FF18}, we have an isomorphism $\Spf\cO\lb{X_0}\ra^\sim\LT$ of sheaves over $\Spf\cO$ given by sending $a_0\mapsto[a_0]$, so $\LT$ is a $1$-dimensional formal $\cO$-module. Moreover, for all $i$ in $\bN\cup\{\infty\}$, Proposition \ref{ss:Wittspecialization} and the discussion on \cite[p.~156--158]{FF18} show that $\LT\times_{\Spf\cO}\Spf\cO_i$ is a Lubin--Tate formal $\cO_i$-module $\LT_i$ over $\Spf\cO_i$; in particular, the element $f_\pi$ of $\cO\lb{X_0}$ that computes multiplication by $\pi$ on $\LT$ satisfies $f_\pi\equiv X_0^q\pmod{\pi}$.

\subsection{}\label{ss:LTtorsion}
View $\LT$ as the pre-adic space $\Spa\cO\lb{X_0}$. Write $(-)_E$ for the fiber product $-\times_{\Spa\cO}\Spa{E}$, and note that $\LT_E$ is isomorphic to the open unit disk over $\Spa{E}$. Recall from \ref{ss:relativepoints} the group topological space $\bO_n^\times$ over $\bN\cup\{\infty\}$.

Our formal group $\LT_E$ is related to the $E$-algebras $E^{\LT,n}$ from \ref{ss:LTextensions} as follows.
\begin{prop*}
For all positive integers $n$, multiplication by $\pi^n$ yields a finite \'etale cover $\LT_E\ra\LT_E$ of degree $q^n$. Its kernel $\LT_E[\pi^n]$ is the finite \'etale cover of $\Spa{E}$ corresponding to $\bO_n$ via Corollary \ref{cor:Efet}, where $\Ga$ acts on $\bO_n$ through $\Ga\ra\Ga^n\ra\bO_n^\times$.
\end{prop*}
\begin{proof}
Note that $\LT_E$ equals the increasing union of open subspaces
  \begin{align*}
    \LT_E=\bigcup_{r=1}^\infty\Spa E\ang{X_0,X_0^r/\pi},
  \end{align*}
  where each $\Spa E\ang{X_0,X_0^r/\pi}$ is preserved by multiplication by $\pi^n$ on $\LT_E$. Because the element $f_{\pi^n}$ of $\cO\lb{X_0}$ that computes multiplication by $\pi^n$ on $\LT$ satisfies $f_{\pi^n}\equiv X_0^{q^n}\pmod{\pi}$, the analogue of Weierstrass division for $E\ang{X_0,X_0^r/\pi}$ implies that the resulting map $E\ang{X_0,X_0^r/\pi}\ra E\ang{X_0,X_0^r/\pi}$ is finite free of degree $q^n$.

  To show that $E\ang{X_0,X_0^r/\pi}\ra E\ang{X_0,X_0^r/\pi}$ is \'etale, Lemma \ref{lem:etalepointwise} indicates that it suffices to check after applying $-\otimes_EE_i$ for all $i$ in $\bN\cup\{\infty\}$. Now $\LT_E\times_{\Spa E}\Spa E_i$ is isomorphic to $\LT_i\times_{\Spa\cO_i}\Spa E_i$, so this follows from Lubin--Tate theory over $E_i$. Similarly, Lubin--Tate theory over $E_i$ yields the desired description of $\LT_E[\pi^n]$.
\end{proof}

\subsection{}\label{ss:infiniteLTcoordinate}
Next, we consider the ``universal cover'' of $\LT$ in the sense of \cite[p.~170]{SW13}. Write $\wt{\LT}$ for $\varprojlim_n\LT$, where $n$ runs over non-negative integers, and the transition morphisms are given by multiplication by $\pi$ on $\LT$. As in \ref{ss:LT}, identify the $n$-th copy of $\LT$ with $\Spf\cO\lb{X_n}$. Since multiplication by $\pi$ on $\wt{\LT}$ is invertible, $\wt{\LT}$ is a formal $E$-module over $\Spf\cO$.

Write $\cO\lb{\wt{X}^{1/q^\infty}}$ for the $\pi$-adic completion of $\bigcup_{r=0}^\infty\cO\lb{\wt{X}^{1/q^r}}$.
\begin{lem*}
We have an isomorphism $\wt{\LT}\ra^\sim\Spf\cO\lb{\wt{X}^{1/q^\infty}}$ of formal schemes over $\Spf\cO$ given by sending $(a_n)_{n\geq0}\mapsto\big(\lim_{n\to\infty}a_n^{q^{n-r}}\big)_{r\geq0}$.
\end{lem*}
\begin{proof}
  Note that $\wt{\LT}=\Spf{B}$, where $B$ denotes the $\pi$-adic completion of $\varinjlim_n\cO\lb{X_n}$, and the transition maps send $X_n\mapsto f_\pi(X_{n+1})$. For any non-negative integer $r$, we claim that the sequence $(X_n^{q^{n-r}})_{n\geq r}$ converges in $B$. To see this, observe that $X_n=f_\pi(X_{n+1})\equiv X_{n+1}^q\pmod{\pi}$ implies that $X_n^{q^{n-r}}\equiv X_{n+1}^{q^{n+1-r}}\pmod{\pi^{n+1-r}}$, which yields the claim.

The claim indicates that sending $\wt{X}\mapsto\lim_{n\to\infty}X_n^{q^n}$ induces a map of $\cO$-algebras $\cO\lb{\wt{X}^{1/q^\infty}}\ra B$. Since $\cO\lb{\wt{X}^{1/q^\infty}}$ and $B$ are $\pi$-adically complete and $\pi$-torsionfree, to show that $\cO\lb{\wt{X}^{1/q^\infty}}\ra B$ is an isomorphism, it suffices to check mod $\pi$. There, it becomes the map of $\cO/\pi$-algebras $\bigcup_{r=0}^\infty(\cO/\pi)\lb{\wt{X}^{1/q^r}}\ra\textstyle\bigcup_{n=0}^\infty(\cO/\pi)\lb{X_0^{1/q^n}}$ given by sending $\wt{X}\mapsto X_0$, which is indeed an isomorphism.
\end{proof}

\subsection{}\label{ss:log}
Using our choice of coordinates for $\LT_E$ from \ref{ss:LT}, we can explicitly define its logarithm as follows. By the proof of \cite[Proposition 4.3.1]{FF18}, we have a morphism of $\cO$-module objects $\log:\LT_E\ra\bG^{\an}_a$ over $\Spa{E}$ given by sending
\begin{align*}
a_0\mapsto\sum_{r=0}^\infty a_0^{q^r}\!/\pi^r
\end{align*}
Note that $\log$ sends $\{\abs{X_0}\leq\abs{\pi}\}\subseteq\LT_E$ to $\{\abs{X}\leq\abs{\pi}\}\subseteq\bG^{\an}_a$.
\begin{prop*}
  The induced morphism $\log:\{\abs{X_0}\leq\abs{\pi}\}\ra\{\abs{X}\leq\abs{\pi}\}$ is an isomorphism. Consequently, we get a short exact sequence
  \begin{align*}
    \xymatrix{0\ar[r] & \displaystyle\bigcup_{n=1}^\infty\LT_E[\pi^n]\ar[r] & \LT_E\ar[r]^-{\log} &\bG^{\an}_a\ar[r] &0,}
  \end{align*}
where $\log$ is \'etale-locally surjective.
\end{prop*}
\begin{proof}
  Since the power series $\sum_{r=0}^\infty T^{q^r}\!/\pi^r$ lies in $T\cdot E\lb{T}$ and has linear term $T$, it has a compositional inverse $\exp$ in $T\cdot E\lb{T}$. Expanding the coefficients of $\exp$ shows that it lies in $E\ang{T/\pi}$, so the morphism given by sending $a\mapsto\exp(a)$ induces the desired inverse to $\log:\{\abs{X_0}\leq\abs{\pi}\}\ra\{\abs{X}\leq\abs{\pi}\}$.

  We turn to the short exact sequence. By checking on fibers over $\Spa{E}$, we see that, for all positive integers $m$, the intersection of $\{\abs{X_0}^{q^{m-1}(q-1)}\leq\abs{\pi}\}\subseteq\LT_E$ with $\bigcup_{n=1}^\infty\LT_E[\pi^n]$ equals $\LT_E[\pi^m]$. This implies that $\bigcup_{n=1}^\infty\LT_E[\pi^n]\ra\LT_E$ is a closed embedding. Similarly, checking on fibers indicates that $\bigcup_{n=1}^\infty\LT_E[\pi^n]$ is the kernel of $\log$.

  To show that $\log$ is \'etale-locally surjective, let $a$ be an $(A,A^+)$-point of $\bG^{\an}_a$. Because $\Spa(A,A^+)$ is compact, the image of $a$ lies in $\{\abs{X}\leq\abs{\pi}^{-n+1}\}\subseteq\bG^{\an}_a$ for some positive integer $n$, so $\pi^na$ lies in $\{\abs{X}\leq\abs{\pi}\}$. The invertibility of $\log$ on $\{\abs{X}\leq\abs{\pi}\}$ lets us lift $\pi^na$ to an $(A,A^+)$-point $p_0$ of $\LT_E$, and Proposition \ref{ss:LTtorsion} indicates, after replacing $(A,A^+)$ with a finite \'etale cover, that there exists an $(A,A^+)$-point $a_0$ of $\LT_E$ satisfying $f_{\pi^n}(a_0)=p_0$. Since $\log$ is $\cO$-linear, we get
  \begin{align*}
    \pi^n\log(a_0) = \log(f_{\pi^n}(a_0)) = \log(p_0) = \pi^na,
  \end{align*}
and the invertibility of $\pi$ on $\bG^{\an}_a$ indicates that $\log(a_0)=a$, as desired.
\end{proof}

\subsection{}\label{ss:tildelog}
Using the identification from Lemma \ref{ss:infiniteLTcoordinate}, view $\wt{\LT}$ as the pre-adic space $\Spa\cO\lb{\wt{X}^{1/q^\infty}}$. Write $\wt{\log}:\wt{\LT}_E\ra\bG^{\an}_a$ for the composition $\wt{\LT}_E\lra^{\pr_0}\LT_E\lra^{\log}\bG^{\an}_a$.
\begin{cor*}
  We have a short exact sequence
  \begin{align*}
    \xymatrix{0\ar[r]& \displaystyle\varprojlim_n\Big(\bigcup_{m=1}^\infty\LT_E[\pi^m]\Big)\ar[r] & \wt{\LT}_E\ar[r]^-{\wt{\log}} & \bG^{\an}_a\ar[r]& 0,}
  \end{align*}
where $\wt{\log}$ is pro-\'etale-locally surjective.
\end{cor*}
\begin{proof}
  Proposition \ref{ss:LTtorsion} identifies $\bigcup_{m=1}^\infty\LT_E[\pi^m]$ with the disjoint union
  \begin{align*}
    \Big(\coprod_{j=1}^\infty\Spa E^{\LT,j}\Big)\textstyle\coprod\Spa E.
  \end{align*}
  Under this identification, multiplication by $\pi$ on $\LT$ restricts to the morphism
  \begin{align*}
    \bigcup_{m=1}^\infty\LT_E[\pi^m]\ra\bigcup_{m=1}^\infty\LT_E[\pi^m]
  \end{align*}
  that sends $\Spa{E}$ to $\Spa{E}$ and $\Spa E^{\LT,j}$ to $\Spa E^{\LT,j-1}$ via the natural morphism $\Spa E^{\LT,j}\ra\Spa E^{\LT,j-1}$, where $E^{\LT,0}$ denotes $E$. In particular, this is surjective.

Since multiplication by $\pi$ on $\bG^{\an}_a$ is invertible, we can identify $\varprojlim_n\bG^{\an}_a$ with $\bG^{\an}_a$ via $\pr_0$. From here, taking $\varprojlim_n$ of the short exact sequence from Proposition \ref{ss:log} and applying the Mittag--Leffler criterion yield the desired result.
\end{proof}

\subsection{}\label{ss:O(1)sections}
We now use Witt bivectors to compute $B^{\vp=\pi}_{S,[1,\infty]}$, as in \cite[Proposition 4.2.1]{FF18}. Write $\wh{BW}(R^+)$ for the $\cO$-module $\varprojlim_mBW(R^+/\vpi^m)$, where $m$ runs over positive integers. Note that we can identify $\wh{BW}(R^+)$ with the set of double sequences $(\dotsc,a_{-1},a_0,a_1,\dotsc)$ in $R^+$ satisfying $\limsup_{n\to-\infty}\norm{a_n}<1$. Therefore we have an $\cO$-linear map $\wh{BW}(R^+)\ra B_{S,[1,\infty]}$ given by sending
\begin{align*}
(\dotsc,a_{-1},a_0,a_1,\dotsc) \mapsto \sum_{n\in\bZ}[a_n^{1/q^n}]\pi^n.
\end{align*}

For all $i$ in $\bN\cup\{\infty\}$, write $S_i$ for the fiber of $S\ra\ul{\bN\cup\{\infty\}}$ at $\ul{\{i\}}$. Note that $S_i$ is affinoid perfectoid; write $S_i=\Spa(R_i,R_i^+)$, and for all $r$ in $R$, write $r_i$ for its image in $R_i$.

\begin{prop*}
  The map $\wh{BW}(R^+)\ra B_{S,[1,\infty]}$ is injective, and it induces a bijection
  \begin{align*}
    \wh{BW}(R^+)^{\vp=\pi}\ra^\sim B_{S,[1,\infty]}^{\vp=\pi}.
  \end{align*}
  Consequently, the map $\wt{a}\mapsto\sum_{n\in\bZ}[\wt{a}^{1/q^n}]\pi^n$ yields a bijection $R^{\circ\circ}\ra^\sim B_{S,[1,\infty]}^{\vp=\pi}$.
\end{prop*}
\begin{proof}
  First, suppose that $\sum_{n\in\bZ}[a_n^{1/q^n}]\pi^n=0$ for some $(\dotsc,a_{-1},a_0,a_1,\dotsc)$ in $\wh{BW}(R^+)$. Then its image in $B_{S_i,[1,\infty]}$ is $0$, and Proposition \ref{ss:Wittspecialization} shows that this image equals $\sum_{n\in\bZ}[a_{n,i}^{1/q^n}]\pi_i^n$. Hence \cite[p.~95]{FF18} indicates that $a_{n,i}=0$ for all $n$ in $\bZ$. As $i$ runs over $\bN\cup\{\infty\}$, \cite[Theorem 5.2.1]{SW20} implies that $a_n=0$ for all $n$ in $\bZ$.

  Next, write $[\vpi^{1/q^\infty}]$ for the ideal of $B_{S,[1,\infty]}$ generated by $[\vpi^{1/q^r}]$ for all non-negative integers $r$. We claim that reduction mod $[\vpi^{1/q^\infty}]$ induces an injection
  \begin{align*}
    B_{S,[1,\infty]}^{\vp=\pi}\hra(B_{S,[1,\infty]}/[\vpi^{1/q^\infty}])^{\vp=\pi},
  \end{align*}
  and precomposing with $\wh{BW}(R^+)^{\vp=\pi}\hra B_{S,[1,\infty]}^{\vp=\pi}$ yields a bijection. For injectivity, suppose that $x$ in $B_{S,[1,\infty]}^{\vp=\pi}$ lies in $[\vpi^{1/q^r}]B_{S,[1,\infty]}$ for some non-negative integer $r$. Then $x=\pi^{-k}\vp^k(x)$ lies in $[\vpi^{q^{k-r}}]B_{S,[1,\infty]}$ for all non-negative integers $k$, so taking $k\to\infty$ shows that $x=0$.

  For bijectivity, note that $W(R^+)[\frac1\pi]$ surjects onto $B_{S,[1,\infty]}/[\vpi^{1/q^\infty}]$. Therefore it suffices to consider $\wt{x}$ in $W(R^+)[\frac1\pi]$ such that $y\coloneqq\pi^{-1}\vp(\wt{x})-\wt{x}$ lies in $[\vpi^{1/q^r}]B_{S,[1,\infty]}$. Then $y$ also lies in $W(R^+)[\frac1\pi]$, so $y$ even lies in $[\vpi^{1/q^r}]W(R^+)[\frac1\pi]$. Hence $\vp^k(y)$ lies in $[\vpi^{q^{k-r}}]W(R^+)[\frac1\pi]$. Because the $\pi$-adic valuation of $\pi^{-k}\vp^k(y)$ decreases linearly in $k$ and we have
  \begin{align*}
    \pi^{-(k+1)}\vp^{k+1}(\wt{x})-\pi^{-k}\vp^k(\wt{x}) = \pi^{-k}\vp^k(\pi^{-1}\vp(\wt{x})-\wt{x}) = \pi^{-k}\vp^k(y),
  \end{align*}
this shows that $x\coloneqq\lim_{k\to\infty}\pi^{-k}\vp^k(\wt{x})=\wt{x}+\sum_{k=0}^\infty\pi^{-k}\vp^k(y)$ converges in $\wh{BW}(R^+)$. By construction, $x$ lies in $\wh{BW}(R^+)^{\vp=\pi}$ and satisfies $x\equiv\wt{x}\pmod{[\vpi^{1/q^\infty}]}$, which yields the claim.

  The claim yields the first bijection. For the second bijection, note that $\pi=\vp\circ V$ implies that $\wh{BW}(R^+)^{\vp=\pi}$ equals the set of constant double sequences $(\dotsc,\wt{a},\wt{a},\wt{a},\dotsc)$ satisfying $\norm{\wt{a}}<1$. Now $\{\wt{a}\in R^+\mid\norm{\wt{a}}<1\}$ equals $R^{\circ\circ}$, as desired.
\end{proof}

\subsection{}\label{ss:logsandtheta}
See \ref{ss:absoluteBanachColmez} for the definition of the line bundle $\sO_{X_S}(1)$ on $X_S$. Let us interpret Proposition \ref{ss:O(1)sections} in terms of $\sO_{X_S}(1)$ and $\LT$ as follows. Write $(-)_{\cO/\pi}$ for the fiber product $-\times_{\Spa\cO}\Spa(\cO/\pi)$. Lemma \ref{ss:infiniteLTcoordinate} identifies $\wt{\LT}_{\cO/\pi}$ and $\wt{\LT}_E$ with the open perfectoid unit disk over $\Spa(\cO/\pi)$ and $\Spa{E}$, respectively.

\begin{cor*}Under the identification from Lemma \ref{ss:infiniteLTcoordinate},
  \begin{enumerate}[i)]
  \item the map $\wt{a}\mapsto\sum_{r\in\bZ}[\wt{a}^{q^r}]/\pi^r$ induces an $E$-linear bijection
    \begin{align*}
      \wt{\LT}(S)\ra^\sim H^0(X_S,\sO_{X_S}(1)),
    \end{align*}
  \item the map $\wt{\log}:\wt{\LT}_E\ra\bG^{\an}_a$ sends $(\wt{a}^{1/q^r})_{r\geq0}\mapsto\sum_{r\in\bZ}\wt{a}^{q^r}\!/\pi^r$ on $(A,A^+)$-points,
  \item for all untilts $S^\sharp=\Spa(R^\sharp,R^{\sharp+})$ of $S$ over $\Spa{E}$, the square
    \begin{align*}
      \xymatrix{\wt{\LT}(S)\ar[r]^-\sim\ar@{=}[d] & H^0(X_S,\sO_{X_S}(1))\ar[d] \\
      \wt{\LT}(S^\sharp)\ar[r]^-{\wt{\log}} & R^\sharp}
    \end{align*}
    commutes, where the left arrow is $R^{\circ\circ}=\varprojlim_r(R^{\sharp})^{\circ\circ}$, and the right arrow is given by pullback to the closed Cartier divisor $S^\sharp\hra X_S$ from \ref{ss:Div1}.
  \end{enumerate}
\end{cor*}
\begin{proof}
  For part i), we identify $H^0(X_S,\sO_{X_S}(1))$ with $B_{S,[1,\infty]}^{\vp=\pi}$ by Proposition \ref{ss:absoluteBanachColmez}. Bijectivity follows from Proposition \ref{ss:O(1)sections}, and $E$-linearity follows from the proof of \cite[Proposition 4.4.5]{FF18}.

  For part ii), let $(a_n)_{n\geq0}$ be an $(A,A^+)$-point of $\wt{\LT}_E$. As $\log$ is $\cO$-linear, we get
  \begin{align*}
    \log(a_0) = \log(f_{\pi^n}(a_n)) = \pi^n\log(a_n) = \sum_{r=-n}^\infty a_n^{q^{r+n}}\!/\pi^r
  \end{align*}
  for all non-negative integers $n$. The proof of Lemma \ref{ss:infiniteLTcoordinate} indicates that
  \begin{align*}
    a_n^{q^{r+n}}\!/\pi^r\equiv a_{n+1}^{q^{r+n+1}}\!/\pi^r\pmod{\pi^{n+1}},
  \end{align*}
  so taking $\lim_{n\to\infty}$ shows that
  \begin{align*}
    \log(a_0) = \lim_{n\to\infty}\sum_{r=-n}^\infty a_n^{q^{r+n}}\!/\pi^r = \sum_{r\in\bZ}\lim_{n\to\infty}a_n^{q^{r+n}}/\pi^r = \sum_{r\in\bZ}\wt{a}^{q^r}\!/\pi^r.
  \end{align*}
This yield part ii). Finally, because the composition $R^{\circ\circ}=\varprojlim_r(R^\sharp)^{\circ\circ}\lra^{\pr_0}R^\sharp$ equals $(-)^\sharp$, part iii) follows immediately from part ii), part i), and the description of $\te$ in Proposition \ref{ss:Witttiltadjunction}.i).
\end{proof}

\subsection{}\label{lem:LTbasis}
Next, we turn to study $E^{\LT}$. For this, we use the following analogue of bases for the Tate module of Lubin--Tate formal groups.
\begin{lem*}
  There exists a sequence $(t_n)_{n\geq1}$ in $(E^{\LT})^{\circ\circ}$ such that, for all positive integers $n$,
  \begin{enumerate}[$\bullet$]
  \item $t_1$ generates $E^{\LT,1}$ over $E$, and $t_{n+1}$ generates $E^{\LT,n+1}$ over $E^{\LT,n}$,
  \item $t_n$ lies in $(E^{\LT,n})^{\circ\circ}$, we have $f_\pi(t_{n+1})=t_n$, and we have $f_\pi(t_1)=0$.
  \end{enumerate}
\end{lem*}
\begin{proof}
  Note that $f_\pi$ is divisible by $X_0$, and $f_\pi/X_0\equiv X_0^{q-1}\pmod{\pi}$. Therefore \cite[Chapter IV, Theorem 9.2]{Lan02}\footnote{Note that \cite[Chapter IV, Theorem 9.1]{Lan02} and \cite[Chapter IV, Theorem 9.2]{Lan02} hold for any ring $\fo$ and ideal $\fm\subseteq\fo$ such that is $\fm$-adically complete (with the same proofs).} yields a monic polynomial $g_1$ in $\cO[X_0]$ such that $f_\pi/X_0=u_1g_1$ for some unit $u_1$ in $\cO\lb{X_0}$. Since $f_\pi=X_0u_1g_1$, one can use \cite[Chapter IV, Theorem 9.1]{Lan02} to identify $\LT_E[\pi]$ with $\Spa E[X_0]/(X_0g_1)$, which identifies $\LT_E[\pi]-\{0\}$ with $\Spa E[X_0]/g_1$. Proposition \ref{ss:LTtorsion} identifies $\LT_E[\pi]-\{0\}$ with $\Spa E^{\LT,1}$, so the image $t_1$ of $X_0$ in $E^{\LT,1}$ is a generator over $E$. By checking the image of $t_1$ in $E^{\LT,1}_i$ for all $i$ in $\bN\cup\{\infty\}$, we see that $t_1$ lies in $(E^{\LT,1})^{\circ\circ}$, which lets us evaluate $f_\pi(t_1)=t_1u_1(t_1)g_1(t_1)=0$.

  To inductively obtain $t_{n+1}$ from $t_n$, note that $f_\pi-t_n\equiv X^q_0\pmod{t_n}$. Therefore \cite[Chapter IV, Theorem 9.2]{Lan02} yields a monic polynomial $g_{n+1}$ in $(E^{\LT,n})^\circ[X_0]$ such that $f_\pi-t_n=u_{n+1}g_{n+1}$ for some unit $u_{n+1}$ in $(E^{\LT,n})^\circ\lb{X_0}$. Arguing as in the proof of Proposition \ref{ss:LTtorsion} identifies $E^{\LT,n+1}$ with $E^{\LT,n}[X_0]/g_{n+1}$, so the image $t_{n+1}$ of $X_0$ in $E^{\LT,n+1}$ is a generator over $E^{\LT,n}$. By checking the image of $t_{n+1}$ in $E^{\LT,n+1}_i$ for all $i$ in $\bN\cup\{\infty\}$, we see that $t_{n+1}$ lies in $(E^{\LT,n+1})^{\circ\circ}$, which lets us evaluate $f_\pi(t_{n+1}) = u_{n+1}(t_{n+1})g_{n+1}(t_{n+1})+t_n=t_n$.
\end{proof}

\subsection{}\label{ss:ELTperfectoid}
Fix a sequence $(t_n)_{n\geq1}$ as in Lemma \ref{lem:LTbasis}. For all $i$ in $\bN\cup\{\infty\}$, write $t_{n,i}$ for the image of $t_n$ in $E^{\LT,n}_i$. Recall from \ref{ss:fieldperf} the topological field $E^{\perf}_\infty$, and note that $E^{\perf}_\infty$ equals the completion $\bF_q\lp{t^{1/q^\infty}}$ of $\bigcup_{r=0}^\infty\bF_q\lp{t^{1/q^r}}$.
\begin{prop*}
The ring $E^{\LT}$ is perfectoid, and our choice of $(t_n)_{n\geq1}$ induces an identification $(E^{\LT})^\flat = \Cont(\bN\cup\{\infty\},E^{\perf}_\infty)$ such that $\pi_\infty^\sharp=\lim_{n\to\infty}t_n^{q^n}$.
\end{prop*}
\begin{proof}
  Since $\pi$ is a multiple of $t_1$ and $t_2^q\equiv f_\pi(t_2)=t_1\pmod{\pi}$, we see that $t_2^q$ is a unit multiple of $t_1$. Now $(E^{\LT})^\circ=\cO^{\LT}$, and we have
  \begin{align*}
    \cO^{\LT}\!/t_1 = \varinjlim_n(E^{\LT,n})^\circ/t_1 = \varinjlim_n\varinjlim_d\prod_{i\leq d}\cO_{E^{\LT,n}_i}/t_{1,i}\times\cO_{E^{\LT,n}_\infty}/t_{1,\infty},
  \end{align*}
  where $d$ is large enough with respect to $n$. Because $\varinjlim_n\cO_{E^{\LT,n}_i}/t_{1,i}=\cO_{E^{\LT}_i}/t_{1,i}$, switching the order of the direct limits yields
  \begin{align*}
    \varinjlim_d\varinjlim_n\prod_{i\leq d}\cO_{E^{\LT,n}_i}/t_{1,i}\times\cO_{E^{\LT,n}_\infty}/t_{1,\infty} = \varinjlim_d\prod_{i\leq d}\cO_{E^{\LT}_i}/t_{1,i}\times\cO_{E^{\LT}_\infty}/t_{1,\infty}.
  \end{align*}
  Replacing $t_1$ with $t_2$ yields $\cO^{\LT}\!/t_2=\varinjlim_d\prod_{i\leq d}\cO_{E^{\LT}_i}/t_{2,i}\times\cO_{E^{\LT}_\infty}/t_{2,\infty}$. Since $\cO_{E^{\LT}_i}$ is integral perfectoid, these descriptions imply that the absolute $q$-Frobenius map $\cO^{\LT}\!/t_2\ra\cO^{\LT}\!/t_2^q=\cO^{\LT}\!/t_1$ is an isomorphism, so $E^{\LT}$ is perfectoid.

We have $t^{q-1}_1\equiv f_\pi(t_1)/t_1= 0\pmod{\pi}$ and $t^q_{n+1}\equiv f_\pi(t_{n+1})=t_n\pmod{\pi}$. Therefore our choice of $(t_n)_{n\geq1}$ induces an identification
  \begin{align*}
    \cO^{\LT}\!/\pi = \bigcup_{n=1}^\infty(\cO/\pi)[t^{1/q^{n-1}(q-1)}]/t = \Cont\Big(\bN\cup\{\infty\},\bigcup_{n=1}^\infty\bF_q[t^{1/q^{n-1}(q-1)}]/t\Big),
  \end{align*}
  where $t^{1/q^{n-1}(q-1)}$ equals the image of $t_n$. Taking the inverse limit perfection yields
  \begin{align*}
    (\cO^{\LT})^\flat = \Cont(\bN\cup\{\infty\},(E^{\perf}_\infty)^\circ),
  \end{align*}
  where $\pi_\infty$ equals $(\dotsc,t^{1/q(q-1)},t^{1/(q-1)},0)$ in the inverse limit perfection of $\cO^{\LT}\!/\pi$. This implies the desired description of $(E^{\LT})^\flat$. Finally, because the sequence $(\dotsc,t_2,t_1,0)$ in $\cO^{\LT}$ is a lift of the sequence $(\dotsc,t^{1/q(q-1)},t^{1/(q-1)},0)$ in $\cO^{\LT}/\pi$, we see that $\pi_\infty^\sharp=\lim_{n\to\infty}t_n^{q^n}$, as desired.
\end{proof}

\subsection{}\label{ss:Div1O(1)}
At this point, we can explain the relationship between $\sO_{X_S}(1)$ and divisors on $X_S$. See \ref{ss:absoluteBanachColmez} for the definition of the $E$-module v-sheaf $\cB\cC(\sO(1))$ over $\ul{\bN\cup\{\infty\}}$.
\begin{prop*}
  Our choice of $(t_n)_{n\geq1}$ induces an isomorphism over $\ul{\bN\cup\{\infty\}}$
  \begin{align*}
    \Spd E^{\LT} \ra^\sim \cB\cC(\sO(1))-\{0\}.
  \end{align*}
\end{prop*}
\begin{proof}
  Since $f_\pi(t_1)=0$ and $f_\pi(t_{n+1})=t_n$ for all positive integers $n$, our choice of $(t_n)_{n\geq1}$ induces a morphism $\Spa E^{\LT}\ra\wt{\LT}$. Therefore precomposing with an untilt $S^\sharp=\Spa(R^\sharp,R^{\sharp+})$ of $S$ over $\Spa E^{\LT}$ yields an element of
  \begin{align*}
    \wt{\LT}(S^\sharp)=\wt{\LT}(S)\ra^\sim H^0(X_S,\sO_{X_S}(1)),
  \end{align*}
and this induces a morphism $\Spd E^{\LT}\ra\cB\cC(\sO(1))$ over $\ul{\bN\cup\{\infty\}}$.

  Proposition \ref{ss:ELTperfectoid} shows that $E^{\LT}$ is perfectoid. Moreover, Proposition \ref{ss:ELTperfectoid} identifies
  \begin{align*}
    \Spd E^{\LT} = \Spa(E^{\LT})^\flat = \Spa\Cont(\bN\cup\{\infty\},E^{\perf}_\infty),
  \end{align*}
and under the identification from Corollary \ref{ss:logsandtheta}.i), the morphism $\Spd E^{\LT}\ra\cB\cC(\sO(1))$ sends $S^\sharp\ra\Spa E^{\LT}$ to the image of $\pi_\infty$ in $R$. Hence the image of this morphism lies in $\cB\cC(\sO(1))-\{0\}$. Corollary \ref{ss:logsandtheta}.i) even identifies $\cB\cC(\sO(1))-\{0\}$ with the punctured open perfectoid unit disk $\Spa\Cont(\bN\cup\{\infty\},\bF_q\lp{\wt{X}^{1/q^\infty}})$ over $\ul{\bN\cup\{\infty\}}$, and $\Spd E^{\LT}\ra\cB\cC(\sO(1))-\{0\}$ corresponds to the map of $\cO/\pi$-algebras
  \begin{align*}
    \Cont(\bN\cup\{\infty\},\bF_q\lp{\wt{X}^{1/q^\infty}})\ra\Cont(\bN\cup\{\infty\},E^{\perf}_\infty)
  \end{align*}
  that sends $\wt{X}\mapsto\pi_\infty$. This is indeed an isomorphism.
\end{proof}

\subsection{}\label{ss:O(1)zeroes}
Let $\sg$ be an $S$-point of $\cB\cC(\sO(1))-\{0\}$. By Proposition \ref{ss:Div1O(1)}, $\sg$ corresponds to an $S$-point of $\Spd E^{\LT}$, whose image under $\Spd E^{\LT}\ra\Div_Y^1\ra\Div^1_X$ naturally induces a closed Cartier divisor $S^\sharp\hra X_S$ by \ref{ss:Div1}.
\begin{prop*}
We get a short exact sequence
  \begin{align*}
    \xymatrix{0\ar[r] & \sO_{X_S}\ar[r]^-\sg & \sO_{X_S}(1)\ar[r] & \sO_{S^\sharp}\ar[r]& 0}
  \end{align*}
of sheaves on $X_S$.
\end{prop*}
\begin{proof}
  Write $\sI_{S^\sharp}$ for the ideal of $\sO_{X_S}$ corresponding to the closed Cartier divisor $S^\sharp\hra X_S$. Note that the morphism $\Spa E^{\LT}\ra\wt{\LT}_E$ induced by our choice of $(t_n)_{n\geq1}$ lies in the kernel of $\wt{\log}$, so Corollary \ref{ss:logsandtheta}.iii) indicates that the image of $\sg:\sO_{X_S}\ra\sO_{X_S}(1)$ lies in $\sI_{S^\sharp}(1)$.

Since $\sI_{S^\sharp}(1)$ is a line bundle, to show that $\sg:\sO_{X_S}\ra\sI_{S^\sharp}(1)$ is an isomorphism, it suffices to check after pulling back to geometric points $\ov{s}$ of $S$. Now the image of $\ov{s}$ in $\bN\cup\{\infty\}$ equals $\{i\}$ for some $i$ in $\bN\cup\{\infty\}$, so the result follows from Proposition \ref{ss:Wittspecialization} and \cite[Proposition II.2.3]{FS21}.
\end{proof}

\subsection{}\label{ss:torsorsareproetale}
We want to use our study of $\Spd E^{\LT}$ to prove that $\Div_X^1\ra\ul{\bN\cup\{\infty\}}$ is cohomologically smooth of dimension $1$. This requires the following generalization of \cite[Lemma 10.13]{Sch17}, which will also be useful later. Let $G$ be a group topological space over $\bN\cup\{\infty\}$ that is lctd over $\bN\cup\{\infty\}$ as in Definition \ref{ss:grouphypotheses}, and let $\{K^\al\}_\al$ be a family of compact open group subspaces of $G$ over $\bN\cup\{\infty\}$ satisfying Definition \ref{ss:grouphypotheses}.a).
\begin{lem*}
  Let $\wt{S}\ra S$ be a $\ul{G}$-torsor in the category of v-sheaves. For all $\al$, the morphism $\ul{K^\al\bs G}\times^{\ul{G}}\wt{S}\ra S$ is separated \'etale, and the natural morphism
  \begin{align*}
    \wt{S}\ra\textstyle\varprojlim_\al\ul{K^\al\bs G}\times^{\ul{G}}\wt{S}
  \end{align*}
  is an isomorphism. Consequently, $\wt{S}\ra S$ is a pro-\'etale cover and universally open.
\end{lem*}
\begin{proof}
  First, assume that $\wt{S}$ is the trivial $\wt{G}$-torsor $\ul{G}\times_{\ul{\bN\cup\{\infty\}}}S$. Lemma \ref{ss:separateHcosets} implies that the \'etale topological space $K^\al\bs G$ over $\bN\cup\{\infty\}$ is Hausdorff, so the morphism
  \begin{align*}
    \ul{K^\al\bs G}\times^{\ul{G}}\wt{S} = \ul{K^\al\bs G}\times_{\ul{\bN\cup\{\infty\}}}S\ra S
  \end{align*}
is separated \'etale. Moreover, Definition \ref{ss:grouphypotheses}.a) indicates that $G\ra\textstyle\varprojlim_\al K^\al\bs G$ is an isomorphism, so the same holds for $\wt{S}\ra\textstyle\varprojlim_\al\ul{K^\al\bs G}\times^{\ul{G}}\wt{S}=\varprojlim_\al\ul{K^\al\bs G}\times_{\ul{\bN\cup\{\infty\}}}S$.

In general, $\wt{S}$ is trivial after replacing $S$ with a v-cover, so the above and \cite[Proposition 10.11 (iv)]{Sch17} show that $\ul{K^\al\bs G}\times^{\ul{G}}\wt{S}\ra S$ is separated \'etale in general. Because isomorphisms of v-sheaves can be checked after a v-cover, the above also shows that $\wt{S}\ra\textstyle\varprojlim_\al\ul{K^\al\bs G}\times^{\ul{G}}\wt{S}$ is an isomorphism in general.

The presentation $\wt{S}=\textstyle\varprojlim_\al\ul{K^\al\bs G}\times^{\ul{G}}\wt{S}$ implies that $\wt{S}\ra S$ is a pro-\'etale cover. Since the morphisms $\ul{K^\al\bs G}\times^{\ul{G}}\wt{S}\ra S$ are universally open and every open subspace of $\wt{S}$ arises via pullback from $\ul{K^\al\bs G}\times^{\ul{G}}\wt{S}$ for some $\al$, this also indicates that $\wt{S}\ra S$ is universally open, as desired.
\end{proof}

\subsection{}\label{ss:profinitequotient}
We also need the following variant of \cite[Proposition 24.2]{Sch17}. Write $(\bO^\times)^n$ for the kernel of $\bO^\times\ra\bO^\times_n$. Note that $\bO^\times$ is lctd over $\bN\cup\{\infty\}$, and $\{(\bO^\times)^n\}_{n\geq1}$ forms a family of compact open group subspaces of $\bO^\times$ over $\bN\cup\{\infty\}$ satisfying Definition \ref{ss:grouphypotheses}.a).

Let $f:Z'\ra Z$ be a morphism of small v-stacks over $\ul{\bN\cup\{\infty\}}$, and let
\begin{align*}
\ul{(\bO^\times)^1}\times_ZZ'\ra Z'
\end{align*}
be a free action of $\ul{(\bO^\times)^1}$ on $Z'$ over $Z$. For all positive integers $n$, write
\begin{align*}
\xymatrixcolsep{3pc}
\xymatrix{Z'\ar[r]^-{q_n} & Z'/\ul{(\bO^\times)^n}\ar[r]^-{f_{(\bO^\times)^n}} & Z} 
\end{align*}
for the natural morphisms. Let $\La$ be a ring that is $\ell$-power torsion.
\begin{prop*}
If $f$ is representable in locally spatial diamonds, separated, and cohomologically smooth, then the same holds for $f_{(\bO^\times)^1}$. Moreover, there is a natural transformation $q^*_1\ra q^!_1$ of functors
  \begin{align*}
    D_{\et}(Z'/\ul{(\bO^\times)^1},\La)\ra D_{\et}(Z',\La)
  \end{align*}
  such that $q^*_1f_{(\bO^\times)^1}^!\ra q^!_1f_{(\bO^\times)^1}^!=f^!$ is an isomorphism.
\end{prop*}
\begin{proof}
  Let $S\ra Z$ be a morphism over $\ul{\bN\cup\{\infty\}}$. Then $S\times_ZZ'$ is a locally spatial diamond, $S\times_ZZ/\ul{(\bO^\times)^1}$ is a diamond, and $S\times_ZZ'\ra S\times_ZZ/\ul{(\bO^\times)^1}$ is an $\ul{(\bO^\times)^1}$-torsor. By replacing $S\times_ZZ/\ul{(\bO^\times)^1}$ with a v-cover by an affinoid perfectoid space, we can use Lemma \ref{ss:torsorsareproetale} to see that $S\times_ZZ'\ra S\times_ZZ/\ul{(\bO^\times)^1}$ is open, so \cite[Lemma 2.10]{Sch17} and \cite[Proposition 11.15]{Sch17} imply that $S\times_ZZ/\ul{(\bO^\times)^1}$ is locally spatial.

  The separatedness of $f_{(\bO^\times)^1}$ follows from \cite[Proposition 10.9]{Sch17}. Note that the relative compactification $\ov{Z'/\ul{(\bO^\times)^1}}^{/Z}$ is naturally isomorphic to $(\ov{Z'}^{/Z})/\ul{(\bO^{\times})^1}$, and arguing as above shows that
  \begin{align*}
    \abs{\ov{Z'}^{/Z}}\ra\abs{(\ov{Z'}^{/Z})/\ul{(\bO^{\times})^1}}
  \end{align*}
is a quotient map. Because $f$ is compactifiable, this implies that $f_{(\bO^\times)^1}$ is compactifiable.

Next, we turn to our candidate for $q^*_1\ra q^!_1$. Since $q_1$ is proper, it suffices to construct a natural transformation $q_{1*}q^*_1\ra\id$. For all positive integers $n\geq m$, write $q_{n,m}:Z'/\ul{(\bO^\times)^n}\ra Z'/\ul{(\bO^\times)^m}$ for the quotient morphism, which is finite \'etale of degree $q^{n-m}$. Lemma \ref{ss:torsorsareproetale} identifies $q_1$ with $\varprojlim_nq_{n,1}$, so $q_{1*}q^*_1$ is naturally isomorphic to $\varinjlim_nq_{n,1*}q_{n,1}^*$. Because $q_{n+1,n}$ is finite \'etale of degree $q$, the trace natural transformation $\tr_n:q_{n,1*}q^*_{n,1}\ra\id$ satisfies the commutative square
  \begin{align*}
    \xymatrix{q_{n+1,1*}q_{n+1,1}^*\ar[r]\ar[d]^-{\tr_{n+1}} & q_{n,1*}q_{n,1}^*\ar[d]^-{\tr_n}\\
    \id\ar[r]^-{\times q} & \id.}
  \end{align*}
Hence taking $\varinjlim_n$ of $q^{-n}\tr_n$ yields a natural transformation $q_{1*}q^*_1\ra\id$.

  We claim that the natural transformation $q^*_1f^!_{(\bO^\times)^1}\ra f^!$ is an isomorphism. To see this, after replacing $Z$ with a v-cover, we can assume that $Z$ is strictly totally disconnected. Then $Z'$ is a locally spatial diamond, and since the claim is open-local on $Z'$, we can assume that $Z'$ is spatial. Because $Z$ is strictly totally disconnected, \cite[Proposition 20.17]{Sch17} indicates that it suffices to show that
  \begin{align*}\label{eqn:profinitequotient}
    \Hom(B,q^*_1f_{(\bO^\times)^1}^!A)\ra\Hom(B,f^!A)\tag{$\dagger$}
  \end{align*}
  is an isomorphism for all $B$ in $D_{\et,pc}(Z',\La)$ and $A$ in $D_{\et}(Z,\La)$. Lemma \ref{ss:torsorsareproetale} identifies $Z'$ with $\varprojlim_nZ'/\ul{(\bO^\times)^n}$, so \cite[Proposition 20.15]{Sch17} implies that $B$ is isomorphic to $q_m^*B_m$ for some positive integer $m$ and $B_m$ in $D_{\et,pc}(Z'/\ul{(\bO^\times)^m},\La)$. Therefore the left-hand side of (\ref{eqn:profinitequotient}) becomes
  \begin{align*}
    &\quad\Hom(q^*_mB_m,q^*_1f_{(\bO^\times)^1}^!A) \\
    &=\Hom(B_m,q_{m*}q^*_1f_{(\bO^\times)^1}^!A) \\
    &=\Hom(B_m,\textstyle\varinjlim_nq_{n,m*}q^*_{n,1}f_{(\bO^\times)^1}^!A) & \mbox{since }\textstyle\varinjlim_nq_{n,m*}q_{n,1}^*\ra^\sim q_{m*}q_1^*\\
    &=\textstyle\varinjlim_n\Hom(B_m,q_{n,m*}q^*_{n,1}f_{(\bO^\times)^1}^!A) & \mbox{by \cite[Proposition 20.10]{Sch17}} \\
    &=\textstyle\varinjlim_n\Hom(B_m,q_{n,m*}q^!_{n,1}f_{(\bO^\times)^1}^!A) & \mbox{because }q_{n,1}\mbox{ is \'etale} \\
    &=\textstyle\varinjlim_n\Hom(B_m,q_{n,m*}f_{(\bO^\times)^n}^!A) \\
    &=\textstyle\varinjlim_n\Hom(f_{(\bO^\times)^n!}q_{n,m}^*B_m,A),
  \end{align*}
  where $n$ runs over positive integers satisfying $n\geq m$. Next, note that
  \begin{align*}
    f_!B &= f_{(\bO^\times)^m!}q_{m!}q^*_mB_m = f_{(\bO^\times)^m!}q_{m*}q^*_mB_m & \mbox{because }q_m\mbox{ is proper}\\
         &= \textstyle\varinjlim_nf_{(\bO^\times)^m!}q_{n,m*}q^*_{n,m}B_m & \mbox{since }\textstyle\varinjlim_nq_{n,m*}q^*_{n,m}\ra^\sim q_{m*}q_m^* \\
         &= \textstyle\varinjlim_nf_{(\bO^\times)^m!}q_{n,m!}q^*_{n,m}B_m & \mbox{because }q_{n,m}\mbox{ is finite} \\
    &= \textstyle\varinjlim_nf_{(\bO^\times)^n!}q^*_{n,m}B_m,
  \end{align*}
  and the transition morphisms in the directed system $\{f_{(\bO^\times)^n!}q^*_{n,m}B_m\}_n$ have splittings given by normalized traces as above. Since $f$ is quasicompact cohomologically smooth, $f_!B$ lies in $D_{\et}(Z,\La)$ by \cite[Proposition 23.13 (ii)]{Sch17}, so \cite[Proposition 20.17]{Sch17} implies that the directed system $\{f_{(\bO^\times)^n!}q^*_{n,m}B_m\}_n$ is eventually constant with value $f_!B$. Therefore the left-hand side of (\ref{eqn:profinitequotient}) becomes
  \begin{align*}
    \textstyle\varinjlim_n\Hom(f_{(\bO^\times)^n!}q_{n,m}^*B_m,A) = \Hom(f_!B,A) = \Hom(B,f^!A),
  \end{align*}
  which is precisely the right-hand side of (\ref{eqn:profinitequotient}). This yields the claim.

To show that $f_{(\bO^\times)^1}$ is cohomologically smooth, \cite[Proposition 23.15]{Sch17} indicates that, after replacing $Z$ with a v-cover, we can assume that $Z$ is strictly totally disconnected. Then it suffices to check condition (iii) in \cite[Proposition 23.4]{Sch17}, and this follows from the claim and applying \cite[Proposition 23.4 (iii)]{Sch17} again to $f$.
\end{proof}

\subsection{}\label{ss:Div1smooth}
The proof of Lemma \ref{ss:Ebreve} shows that $W(\cO/\pi)\cong\cO$, so the composition
\begin{align*}
\bF_q\hra \Cont(\bN\cup\{\infty\},\bF_q)\cong\cO/\pi\lra^{[-]} W(\cO/\pi)\cong\cO
\end{align*}
induces a morphism $\bF_q^\times\times(\bN\cup\{\infty\})\ra\bO^\times$ of group topological spaces over $\bN\cup\{\infty\}$. By using profinitude and checking on fibers, we see that the induced morphism
\begin{align*}
  \bF_q^\times\times(\bN\cup\{\infty\})\times_{\bN\cup\{\infty\}}(\bO^\times)^1\ra\bO^\times
\end{align*}
of group topological spaces over $\bN\cup\{\infty\}$ is an isomorphism.

We conclude this section by proving the promised smoothness result.
\begin{cor*}
Our $\Div^1_X\ra\ul{\bN\cup\{\infty\}}$ is cohomologically smooth of dimension $1$.
\end{cor*}
\begin{proof}
  Proposition \ref{ss:ELTperfectoid} identifies $\Spd E^{\LT}$ with the punctured open perfectoid unit disk over $\ul{\bN\cup\{\infty\}}$. This identifies $\Spd E^{\LT}\times_{\ul{\bN\cup\{\infty\}}}\Spd C$ with the punctured open perfectoid unit disk over $\Spd{C}$, so $\Spd E^{\LT}\times_{\ul{\bN\cup\{\infty\}}}\Spd C\ra\Spd C$ is cohomologically smooth of dimension $1$.

  Corollary \ref{cor:Efet} and \cite[Lemma 15.6]{Sch17} imply that $\Spd E^{\LT}\ra\Spd{E}$ is an $\ul{\bO^\times}$-torsor. Therefore $\Spd E^{\LT}\ra(\Spd E^{\LT})/\ul{\bF_q^\times}$ is a finite \'etale cover, so the above and \cite[Proposition 23.13]{Sch17} show that $(\Spd E^{\LT})/\ul{\bF_q^\times}\times_{\ul{\bN\cup\{\infty\}}}\Spd{C}\ra\Spd{C}$ is cohomologically smooth of dimension $1$. Since $(\Spd E^{\LT})/\ul{\bF_q^\times}\ra\Spd E^{\LT}$ is a $\ul{(\bO^\times)^1}$-torsor, Proposition \ref{ss:profinitequotient} indicates that $\Spd{E}\times_{\ul{\bN\cup\{\infty\}}}\Spd{C}\ra\Spd{C}$ is cohomologically smooth of dimension $1$. From here, Lemma \ref{ss:Div1} and \cite[Proposition 23.15]{Sch17} imply that $\Spd{E}\ra\ul{\bN\cup\{\infty\}}$ is cohomologically smooth of dimension $1$. Finally, the proof of Lemma \ref{ss:Div1} shows that the same holds for $\Div^1_X\ra\ul{\bN\cup\{\infty\}}$.
\end{proof}

\section{Vector bundles on Fargues--Fontaine curves}\label{s:vectorbundles}
In this section, we study vector bundles on $X_S$. First, we prove that they satisfy v-descent with respect to $S$. Next, we define the analogue of absolute Banach--Colmez spaces over $E$, and we use results from \S\ref{s:LT} to prove various facts about them. We also define general Banach--Colmez spaces over $E$, and we use an ampleness result to extend the aforementioned facts to this setting. Finally, we apply these facts to prove properties about the Harder--Narasimhan filtration in our context, as well as to generalize a theorem of Kedlaya--Liu \cite{KL15}.

\subsection{}\label{ss:BanachColmezsheaf}
Let $\sE$ be a vector bundle on $X_S$. Write $\cB\cC(\sE)$ for the presheaf of $E$-modules on $\Perf_{\bF_q}$ over $S$ whose $S'$-points equal $H^0(X_{S'},\sE|_{X_{S'}})$.
\begin{prop*}
  The presheaf of categories on $\Perf_{\bF_q}$ over $\ul{\bN\cup\{\infty\}}$ given by
  \begin{align*}
    S\mapsto\{\mbox{vector bundles on }X_S\}
  \end{align*}
  satisfies v-descent. In particular, $\cB\cC(\sE)$ is a v-sheaf.
\end{prop*}
\begin{proof}
  Because $X_S = Y_S/\vp^\bZ$, vector bundles on $X_S$ are equivalent to vector bundles on $Y_S$ equipped with a $\vp$-semilinear automorphism. Therefore it suffices to prove that $S\mapsto\{\mbox{vector bundles on }Y_S\}$ satisfies v-descent. Let $S'\ra S$ be an affinoid perfectoid v-cover. Since $Y_S=\bigcup_\cI\cY_{S,\cI}$, where $\cI$ runs over closed intervals in $(0,\infty)$ with rational endpoints, it suffices to prove that vector bundles satisfy descent with respect to $\cY_{S',\cI}\ra\cY_{S,\cI}$.

  Let $\sE'$ be a vector bundle on $\cY_{S',\cI}$ with a descent datum $\al$ with respect to $\cY_{S',\cI}\ra\cY_{S,\cI}$. Pullback yields a vector bundle $\wt\sE'$ on $\cY_{S',\cI}\times_{\Spa\cO}\Spa\cO^{\perf}$ with commuting descent data $\wt\al$ with respect to
  \begin{align*}
\cY_{S',\cI}\times_{\Spa\cO}\Spa\cO^{\perf}\ra\cY_{S,\cI}\times_{\Spa\cO}\Spa\cO^{\perf}
  \end{align*}
  and $\be'$ with respect to $\cY_{S',\cI}\times_{\Spa\cO}\Spa\cO^{\perf}\ra\cY_{S',\cI}$.

  Proposition \ref{ss:Ysousperfectoid} indicates that the $\cY_{-,\cI}\times_{\Spa\cO}\Spa\cO^{\perf}$ are affinoid perfectoid. Hence \cite[Lemma 17.1.8]{SW20} enables us to descend $(\wt\sE',\wt\al)$ and $(\be',\wt\al)$ to a vector bundle $\wt\sE$ on $\cY_{S,\cI}\times_{\Spa\cO}\Spa\cO^{\perf}$ with a descent datum $\be$ with respect to
  \begin{align*}
    \cY_{S,\cI}\times_{\Spa\cO}\Spa\cO^{\perf}\ra\cY_{S,\cI}.
  \end{align*}
Because $\cO$ is a direct summand of $\cO^{\perf}$ as topological $\cO$-modules, \cite[Tag 08XA]{stacks-project} and \cite[Tag 08XD]{stacks-project} enable us to descend $(\wt\sE,\be)$ to a vector bundle $\sE$ on $\cY_{S,\cI}$. This yields the desired result.
\end{proof}

\subsection{}\label{ss:absoluteBanachColmez}
Next, we study the analogue of absolute Banach--Colmez spaces in our context. Let $\la$ be a rational number, and write $\la=\frac{c}d$ for coprime integers $c$ and $d$ such that $d$ is positive. Write $\sO_{X_S}(\la)$ for the rank-$d$ vector bundle on $X_S$ given by descending $\sO_{Y_S}^{\oplus d}$ along $Y_S\ra X_S$ via
\begin{align*}
  \begin{bmatrix}
    & & & \pi^{-c} \\
    1& & & \\
    &\ddots & & \\
    & & 1 &
  \end{bmatrix}\circ\vp^{\oplus d},
\end{align*}
and write $\cB\cC(\sO(\la))$ for the presheaf of $E$-modules on $\Perf_{\bF_q}$ over $\ul{\bN\cup\{\infty\}}$ whose $S$-points equal $H^0(X_S,\sO_{X_S}(\la))$. Since $\sO_{X_S}(\la)$ is compatible with base change in $S$, Proposition \ref{ss:BanachColmezsheaf} implies that $\cB\cC(\sO(\la))$ is a v-sheaf.
\begin{prop*}
If $\la$ is positive, then $H^0(X_S,\sO_{X_S}(\la))$ is naturally isomorphic to $B_{S,[1,\infty]}^{\vp^d=\pi^c}$, and $H^1(X_S,\sO_{X_S}(\la))$ is zero.
\end{prop*}
\begin{proof}
  By repeating the construction of $E$ (except that in \ref{ss:closefieldssetup} we replace $\bF_q$ with $\bF_{q^d}$, and we replace $E_i$ with its corresponding unramified extension), we obtain a topological ring $E^d$. Write $X_S(E^d)$ for the adic space associated with $E^d$ as in \ref{ss:Ycurve}. By construction, we have a degree $d$ finite \'etale morphism $X_S(E^d)\ra X_S$ given by quotienting by $\vp$, and $\sO_{X_S}(\la)$ is the pushforward of $\sO_{X_S(E^d)}(c)$. Therefore Shapiro's lemma lets us assume, after replacing $E$ with $E^d$, that $\la=c$ is an integer.

  Note that $X_S$ is obtained from $\cY_{S,[1,q]}$ by using $\vp$ to glue $\cY_{S,[1,1]}\ra^\sim\cY_{S,[q,q]}$. Because the $\cY_{S,\cI}$ are affinoid, we see that $R\Ga(X_S,\sO_{X_S}(c))$ is quasi-isomorphic to 
  \begin{align*}
     \xymatrix{B_{S,[1,q]}\ar[r]^-{\vp-\pi^c}& B_{S,[1,1]}.}
  \end{align*}
  We claim that restriction yields a quasi-isomorphism to this complex from
  \begin{align*}
    \xymatrix{B_{S,[1,\infty]}\ar[r]^-{\vp-\pi^c}& B_{S,[1,\infty]}.}
  \end{align*}
  To see this, recall from \ref{ss:Yinftysousperfectoid} the topological ring $W'(R^+)$. Now $\Spa W'(R^+)[\frac1\pi]$ is covered by the rational open subspaces
  \begin{align*}
    \{\abs{\pi}\leq\abs{[\vpi]}\neq0\}\mbox{ and }\{\abs{[\vpi]}\leq\abs{\pi}\neq0\},
  \end{align*}
and the proof of Lemma \ref{ss:Yinftysousperfectoid}.ii) shows that the global sections of the latter equals $B_{S,[1,\infty]}$. Arguing similarly shows that the global sections of the former equals $B_{S,[0,1]}[\frac1\pi]$, and the global sections of their intersection equals $B_{S,[1,1]}$. Using the two-sided $\pi$-adic expansions of elements of $B_{S,[1,1]}$, we see that the difference map $B_{S,[0,1]}[\frac1\pi]\times B_{S,[1,\infty]}\ra B_{S,[1,1]}$ is surjective, so Lemma \ref{ss:Yinftysousperfectoid}.i) yields a short exact sequence
  \begin{align*}
    \xymatrix{0\ar[r]& W(R^+)[\frac1\pi]\ar[r]& B_{S,[0,1]}[\frac1\pi]\times B_{S,[1,\infty]}\ar[r]& B_{S,[1,1]}\ra0.}
  \end{align*}
  Replacing $\{\abs{\pi}\leq\abs{[\vpi]}\neq0\}$ with $\{\abs{\pi^q}\leq\abs{[\vpi]}\neq0\}$ instead yields
  \begin{align*}
    \xymatrix{0\ar[r]& W(R^+)[\frac1\pi]\ar[r]& B_{S,[0,q]}[\frac1\pi]\times B_{S,[1,\infty]}\ar[r]& B_{S,[1,q]}\ar[r]& 0,}
  \end{align*}
  and $\vp-\pi^c$ induces a map from the bottom row to the top row. By the snake lemma, it suffices to see that the maps $W(R^+)[\frac1\pi]\ra W(R^+)[\textstyle\frac1\pi]$ and $B_{S,[0,q]}\ra B_{S,[0,1]}$ are isomorphisms. Now $\sum_{m=0}^\infty\pi^{mc}\vp^{-1-m}$ converges on $W(R^+)[\frac1\pi]$ and $B_{S,[0,1]}$, which provides the inverse to $\vp-\pi^c$ on $W(R^+)[\frac1\pi]$ and $B_{S,[0,1]}$. This yields the claim.

  The claim yields the first statement. For the second statement, note that $B_{S,[1,\infty]}$ is spanned by $W(R^+)[\frac1\pi]$ and $[\vpi]B_{S,[1,\infty]}$, and the above work shows that $W(R^+)[\frac1\pi]$ lies in the image of $\vp-\pi^c$. Now $-\sum_{m=0}^\infty\pi^{(-1-m)c}\vp^m$ converges on $[\vpi]B_{S,[1,\infty]}$, which implies that $[\vpi]B_{S,[1,\infty]}$ also lies in the image of $\vp-\pi^c$.
\end{proof}

\begin{prop}\label{prop:absoluteBanachColmez}Let $\la$ be a rational number.
  \begin{enumerate}[i)]
  \item If $\la$ is positive, then the morphism $\cB\cC(\sO(\la))\ra\ul{\bN\cup\{\infty\}}$ is representable in locally spatial diamonds, partially proper, and cohomologically smooth.
  \item If $\la$ is zero, then $\cB\cC(\sO(\la))$ is naturally isomorphic to $\ul{\bE}$, and the pro-\'etale sheafification of the presheaf on $\Perf_{\bF_q}$ over $\ul{\bN\cup\{\infty\}}$ whose $S$-points equal $H^1(X_S,\sO_{X_S}(\la))$ is zero.
  \item If $\la$ is negative, then $H^0(X_S,\sO_{X_S}(\la))$ is zero.
  \end{enumerate}
\end{prop}
\begin{proof}
  The proof of Proposition \ref{ss:absoluteBanachColmez} lets us assume that $\la=c$ is an integer. Under the identification from Corollary \ref{ss:logsandtheta}.i), any pseudouniformizer of $R$ yields an $S$-point $\sg$ of $\cB\cC(\sO(1))-\{0\}$. By Proposition \ref{ss:O(1)zeroes}, we get a short exact sequence
  \begin{align*}
    \xymatrix{0\ar[r] & \sO_{X_S}\ar[r]^-\sg & \sO_{X_S}(1)\ar[r] & \sO_{S^\sharp}\ar[r]& 0}
  \end{align*}
  of sheaves on $X_S$, where $S^\sharp\hra X_S$ is the closed Cartier divisor from \ref{ss:Div1} associated with the image of $\sg$ under $\cB\cC(\sO(1))-\{0\}\cong\Spd E^{\LT}\ra\Div^1_Y\ra\Div^1_X$. Since the pullback to $Y_S$ of $\sO_{X_S}(1)$ is trivial, the same holds for $S^\sharp$, so twisting by $\sO_{X_S}(c)$ yields a short exact sequence
  \begin{align*}\label{eqn:BanachColmezSES}
    \xymatrix{0\ar[r] & \sO_{X_S}(c)\ar[r]^-\sg & \sO_{X_S}(c+1)\ar[r] & \sO_{S^\sharp}\ar[r]& 0.}\tag{$\triangleright$}
  \end{align*}
  
First, suppose that $c$ is positive. If $c=1$, the result follows from the identification in Corollary \ref{ss:logsandtheta}.i). To inductively get the case of $c+1$ from the case of $c$, Proposition \ref{ss:absoluteBanachColmez} shows that the long exact sequence induced by \eqref{eqn:BanachColmezSES} yields a short exact sequence
  \begin{align*}
    \xymatrix{0\ar[r]& \cB\cC(\sO_{X_S}(c))\ar[r]& \cB\cC(\sO_{X_S}(c+1))\ar[r]& (\bG^{\an}_{a,S^\sharp})^\Diamond\ar[r]& 0,}
  \end{align*}
  where the morphism $\cB\cC(\sO_{X_S}(c+1))\ra(\bG^{\an}_{a,S^\sharp})^\Diamond$ is surjective on $S$-points. Because $\cB\cC(\sO_{X_S}(c+1))\ra(\bG^{\an}_{a,S^\sharp})^\Diamond$ is also a $\cB\cC(\sO_{X_S}(c))$-torsor, the result for $\cB\cC(\sO_{X_S}(c))$ implies that $\cB\cC(\sO_{X_S}(c+1))\ra(\bG^{\an}_{a,S^\sharp})^\Diamond$ is representable in locally spatial diamonds, partially proper, and cohomologically smooth. Now $(\bG^{\an}_{a,S^\sharp})^\Diamond\ra S$ is a locally spatial diamond, partially proper, and cohomologically smooth, so \cite[Proposition 23.13]{Sch17} shows that the same holds for the composition $\cB\cC(\sO_{X_S}(c+1))\ra(\bG^{\an}_{a,S})^\Diamond\ra S$.

Next, suppose that $c$ is zero. Then Proposition \ref{ss:absoluteBanachColmez} shows that the long exact sequence induced by \eqref{eqn:BanachColmezSES} yields an exact sequence
\begin{align*}
  \xymatrix{0\ar[r] & H^0(X_S,\sO_{X_S})\ar[r] & H^0(X_S,\sO_{X_S}(1))\ar[r] & R^\sharp\ar[r] & H^1(X_S,\sO_{X_S})\ar[r] & 0,}
\end{align*}
and Corollary \ref{ss:logsandtheta}.iii) and Corollary \ref{ss:tildelog} identify its first four terms with
\begin{align*}
  \xymatrix{0\ar[r] & \displaystyle\varprojlim_n\Big(\bigcup_{m=1}^\infty\LT_E[\pi^m]\Big)(S^\sharp)\ar[r] & \wt{\LT}(S^\sharp)\ar[r]^-{\wt\log} & R^\sharp.}
\end{align*}
Corollary \ref{ss:tildelog} indicates that $\wt\log$ is pro-\'etale-locally surjective, so our exact sequence implies that the pro-\'etale sheafification of $S\mapsto H^1(X_S,\sO_{X_S})$ is zero. As for $H^0(X_S,\sO_{X_S})$, Proposition \ref{ss:LTtorsion} shows that $\LT_E[\pi^m](S^\sharp)$ is a pseudotorsor for $\Cont_{\bN\cup\{\infty\}}(|S^\sharp|,\bO_n)$. Now $\sg$ induces a morphism $S^\sharp\ra\Spa E^{\LT}\ra\Spa E^{\LT,m}$, so Proposition \ref{ss:LTtorsion} also shows that this pseudotorsor is naturally trivial. Therefore the compactness of $\abs{S}=\abs{S^\sharp}$ yields the desired description
\begin{align*}
\varprojlim_n\Big(\bigcup_{m=1}^\infty\LT_E[\pi^m]\Big)(S^\sharp) = \varprojlim_n\bigcup_{m=1}^\infty\Cont_{\bN\cup\{\infty\}}(\abs{S},\bO_n) = \Cont_{\bN\cup\{\infty\}}(\abs{S},\bE).
\end{align*}

Finally, suppose that $c$ is negative. Then the long exact sequence induced by \eqref{eqn:BanachColmezSES} yields an exact sequence
\begin{align*}
\xymatrix{0\ar[r]& H^0(X_S,\sO_{X_S}(c))\ar[r]& H^0(X_S,\sO_{X_S}(c+1))\ar[r]& R^\sharp,}
\end{align*}
so if $c=-1$, the result follows from part ii) and the injectivity of $\ul{\bE}\ra(\bG^{\an}_{a,S^\sharp})^\Diamond$. This exact sequence also indicates that the case of $c+1$ implies the case of $c$.
\end{proof}

\subsection{}\label{ss:amplelinebundle}
To extend our discussion to general Banach--Colmez spaces, we use the following ampleness result for $\sO_{X_S}(1)$.

\begin{prop*}
Let $\sE$ be a vector bundle on $X_S$. Then there are positive integers $c$ and $m$ such that there exists a surjective morphism $\sO_{X_S}(-c)^{\oplus m}\twoheadrightarrow\sE$.
\end{prop*}
\begin{proof}
After replacing the usual Witt vectors with our Witt vectors, the proof proceeds as in \cite[Theorem II.2.6]{FS21} verbatim.
\end{proof}

\subsection{}\label{ss:projectivizedBanachColmez}
Note that $\bO^\times=\varprojlim_n\bO_n^\times$ is profinite, and the morphism $\bZ\times(\bN\cup\{\infty\})\ra\bE^\times$ of group topological spaces over $\bN\cup\{\infty\}$ arising from $\pi$ induces an isomorphism
\begin{align*}
\big[\bZ\times(\bN\cup\{\infty\})\big]\times_{\bN\cup\{\infty\}}\bO^\times\ra^\sim\bE^\times.
\end{align*}
We now consider the analogue of projectivized Banach--Colmez spaces in our context. For all vector bundles $\sE$ on $X_S$, Proposition \ref{prop:absoluteBanachColmez}.ii) shows that $\cB\cC(\sE)$ is naturally an $\ul{\bE}$-module over $S$.
\begin{cor*}
The morphism $\cB\cC(\sE)\ra S$ is representable in locally spatial diamonds and partially proper. Moreover, the morphism $(\cB\cC(\sE)-\{0\})/\ul{\bE^\times}\ra S$ is representable in spatial diamonds and proper.
\end{cor*}
\begin{proof}
  Applying Proposition \ref{ss:amplelinebundle} twice yields an exact sequence
  \begin{align*}
    \xymatrix{\sO_{X_S}(-c_2)^{\oplus m_2}\ar[r] & \sO_{X_S}(-c_1)^{\oplus m_1}\ar[r] & \sE^\vee\ar[r] & 0}
  \end{align*}
  of sheaves on $X_S$, where the $c_1,c_2,m_1,m_2$ are positive integers. From here, taking duals and global sections yields an exact sequence
  \begin{align*}
    \xymatrix{0\ar[r] & \cB\cC(\sE)\ar[r] & \cB\cC(\sO_{X_S}(c_1))^{\oplus m_1}\ar[r] & \cB\cC(\sO_{X_S}(c_2))^{\oplus m_2}.}
  \end{align*}
  Now $\cB\cC(\sO_{X_S}(c_2))^{\oplus m_2}\ra S$ is separated by Proposition \ref{prop:absoluteBanachColmez}.i), so its zero section is a closed embedding. Hence our exact sequence implies that $\cB\cC(\sE)\ra\cB\cC(\sO_{X_S}(c_1))^{\oplus m_1}$ is a closed embedding. Proposition \ref{prop:absoluteBanachColmez}.i) indicates that $\cB\cC(\sO_{X_S}(c_1))^{\oplus m_1}\ra S$ is representable in locally spatial diamonds and partially proper, so this shows that the same holds for the composition $\cB\cC(\sE)\ra\cB\cC(\sO_{X_S}(c_1)^{\oplus m_1}\ra S$.

  We turn to the second statement. We claim that $\abs{\cB\cC(\sE)}$ and the automorphism arising from $\pi$ satisfy the conditions in \cite[Lemma II.2.17]{FS21}. The first paragraph is satisfied by \cite[Remark II.2.18]{FS21}, and since $\cB\cC(\sE)$ is a closed subspace and subsheaf of $E$-modules of $\cB\cC(\sO_{X_S}(c_1))^{\oplus m_1}$, it suffices to check conditions (i)--(ii) for $|\cB\cC(\sO_{X_S}(c_1))^{\oplus m_1}|$. The proof of Proposition \ref{prop:absoluteBanachColmez}.i) identifies $\cB\cC(\sO_{X_S}(c_1))$ with an iterated extension of $(\wt{\LT}\times_{\Spa\cO/\pi}S)^\Diamond$ and $(\bG^{\an}_{a,S^\sharp})^\Diamond$, which themselves satisfy conditions (i)--(ii). This yields the claim.

  Because the fixed point locus of $\pi$ is precisely the zero section, the claim and \cite[Lemma II.2.17]{FS21} imply that
  \begin{align*}
    (\cB\cC(\sE)-\{0\})/\pi^\bZ = (\cB\cC(\sE)-\{0\})/\ul{\bZ\times(\bN\cup\{\infty\})}
  \end{align*}
  is a spatial diamond. Note that $\ul{\bO^\times}$ acts freely on $(\cB\cC(\sE)-\{0\})/\ul{\bZ\times(\bN\cup\{\infty\})}$ over $S$, so \cite[Remark 11.25]{Sch17} shows that
  \begin{align*}
    \Big((\cB\cC(\sE)-\{0\})/\ul{\bZ\times(\bN\cup\{\infty\})}\Big)\big/\ul{\bO^\times} &= (\cB\cC(\sE)-\{0\})/\ul{\big[\bZ\times(\bN\cup\{\infty\})\big]\times_{\bN\cup\{\infty\}}\bO^\times} \\
    &=(\cB\cC(\sE)-\{0\})/\ul{\bE^\times}
  \end{align*}
is a spatial diamond. Finally, we see that $(\cB\cC(\sE)-\{0\})/\ul{\bE^\times}\ra S$ remains partially proper, so \cite[Proposition 18.3]{Sch17} indicates that it is proper, as desired.
\end{proof}

\subsection{}
Let us recall the notion of Harder--Narasimhan polygons. Let $\sE$ be a vector bundle on $X_S$. For any geometric point $\ov{s}$ of $S$, the image of $\ov{s}$ in $\bN\cup\{\infty\}$ equals $\{i\}$ for some $i$ in $\bN\cup\{\infty\}$, so Proposition \ref{ss:Wittspecialization} and \cite[Theorem II.2.14]{FS21} imply that $\sE|_{X_{\ov{s}}}$ is isomorphic to
\begin{align*}
  \sO_{X_{\ov{s}}}(\la_1)\oplus\dotsb\oplus\sO_{X_{\ov{s}}}(\la_r)
\end{align*}
for some rational numbers $\la_1\geq\dotsb\geq\la_r$. For all $1\leq j\leq r$, write $d_j$ for the rank of $\sO_{X_{\ov{s}}}(\la_j)$.
\begin{defn*}
  The \emph{Harder--Narasimhan polygon} of $\sE$ is the function
  \begin{align*}
    p_\sE:\abs{S}\ra\{\mbox{convex polygons in }\bR^2\}
  \end{align*}
sending $\ov{s}$ to the polygon with vertices given by
  \begin{align*}
    (0,0),(d_1,d_1\la_1),\dotsc,(d_1+\dotsb+d_r,d_1\la_1+\dotsb+d_r\la_r).
  \end{align*}
\end{defn*}

\subsection{}\label{prop:HarderNarasimhan}
As usual, the properness of projectivized Banach--Colmez spaces has the following consequence for Harder--Narasimhan polygons.
\begin{prop*}Let $\sE$ be a vector bundle on $X_S$.
  \begin{enumerate}[i)]
  \item The function $p_\sE$ is upper semicontinuous.
  \item If $p_\sE$ is constant, then there exists a decreasing filtration $\{\sE^{\geq\la}\}_{\la\in\bQ}$ of $\sE$ such that, for all geometric points $\ov{s}$ of $S$, the pullback $\{\sE^{\geq\la}|_{X_{\ov{s}}}\}_{\la\in\bQ}$ equals the Harder--Narasimhan filtration of $\sE|_{X_{\ov{s}}}$. Moreover, there exists an affinoid perfectoid v-cover $S'\ra S$ such that, for all rational numbers $\la$, there is a non-negative integer $m$ and an isomorphism
    \begin{align*}
      \Big(\sE^{\geq\la}\big/\bigcup_{\la'>\la}\sE^{\geq\la'}\Big)\Big|_{X_{S'}} \cong \sO_{X_{S'}}(\la)^{\oplus m}.
    \end{align*}
  \end{enumerate}
\end{prop*}
\begin{proof}
  For part i), after replacing $S$ with a clopen cover, we can assume that $\sE$ has constant rank $d$. Note that $p_\sE(\ov{s})$ equals the convex hull of
  \begin{align*}
    \{(j,m)\in\bZ^2\mid 0\leq j\leq d\mbox{ and }H^0(X_{\ov{s}},(\textstyle\bigwedge^j\sE)\otimes\sO_{X_S}(-m)\big|_{X_{\ov{s}}})\mbox{ is nonzero}\},
  \end{align*}
so it suffices to show that, for all vector bundles $\sF$ on $X_S$, 
  \begin{align*}
    \{\ov{s}\in\abs{S}\mid H^0(X_{\ov{s}},\sF|_{X_{\ov{s}}})\mbox{ is nonzero}\}
  \end{align*}
  is a closed subset of $\abs{S}$. Now the latter equals equals the image of the morphism $(\cB\cC(\sF)-\{0\})/\ul{\bE^\times}\ra S$, which is proper by Corollary \ref{ss:projectivizedBanachColmez}. Hence its image is indeed a closed subset of $\abs{S}$.

  For part ii), suppose that $p_\sE$ is the polygon with vertices given by
  \begin{align*}
    (0,0),(d_1,d_1\la_1),\dotsc,(d_1+\dotsb+d_r,d_1\la_1+\dotsb+d_r\la_r)
  \end{align*}
  for some rational numbers $\la_1\geq\dotsb\geq\la_r$. We claim that there exists an affinoid perfectoid v-cover $S'\ra S$ and an isomorphism
  \begin{align*}
    \sE|_{X_{S'}}\cong\sO_{X_{S'}}(\la_1)\oplus\dotsb\oplus\sO_{X_{S'}}(\la_r).
  \end{align*}
  To see this, consider the vector bundle $\sF\coloneqq\sHom(\sO_{X_S}(\la_1),\sE)$ on $X_S$. The morphism $(\cB\cC(\sF)-\{0\})/\ul{\bE^\times}\ra S$ is quasicompact by Corollary \ref{ss:projectivizedBanachColmez}, and by assumption it is surjective on geometric points. Therefore \cite[Lemma 12.11]{Sch17} indicates that $(\cB\cC(\sF)-\{0\})/\ul{\bE^\times}\ra S$ is a v-cover. Since $\cB\cC(\sF)-\{0\}\ra(\cB\cC(\sF)-\{0\})/\ul{\bE^\times}$ is also a v-cover, we see that, after replacing $S$ with a v-cover, there exists a morphism $\sg:\sO_{X_S}(\la_1)\ra\sE$ whose pullbacks to geometric points $\ov{s}$ of $S$ are nonzero. Because $\sO_{X_{\ov{s}}}(-\la_1)$ is irreducible, checking on geometric points $\ov{s}$ of $S$ shows that $\sg^\vee:\sE^\vee\ra\sO_{X_S}(\la_1)^\vee=\sO_{X_S}(-\la_1)$ is surjective, so $\sg$ is injective and has a vector bundle cokernel. By induction on $\rk\sE$, after replacing $S$ with a v-cover, we have an isomorphism $\coker\sg\cong\sO_{X_S}(\la_2)\oplus\dotsb\oplus\sO_{X_S}(\la_r)$. The resulting extension
  \begin{align*}
    \xymatrix{0\ar[r] & \sO_{X_S}(\la_1)\ar[r]^-\sg & \sE\ar[r] & \sO_{X_S}(\la_2)\oplus\dotsb\oplus\sO_{X_S}(\la_r)\ar[r] & 0}
  \end{align*}
  splits after replacing $S$ with a pro-\'etale cover by Proposition \ref{ss:absoluteBanachColmez} and Proposition \ref{prop:absoluteBanachColmez}.ii). This yields the claim.

Let $S'\ra S$ be as in the claim, and for all rational numbers $\la$, write $\sE'^{\geq\la}$ for the sub-bundle $\bigoplus_{\la_j\geq\la}\sO_{X_{S'}}(\la_j)$ of $\sE|_{X_{S'}}\cong\sO_{X_{S'}}(\la_1)\oplus\dotsb\oplus\sO_{X_{S'}}(\la_r)$. Note that $\sE|_{X_{S'}}$ has a descent datum $\al$ with respect to $X_{S'}\ra X_S$. Proposition \ref{prop:absoluteBanachColmez}.iii) implies that $\al$ preserves the sub-bundles $\sE'^{\geq\la}$, so Proposition \ref{ss:BanachColmezsheaf} enables us to descend $(\sE'^{\geq\la},\al)$ to a sub-bundle $\sE^{\geq\la}$ of $\sE$. By construction, $\sE^{\geq\la}$ satisfies the desired properties.
\end{proof}

\subsection{}\label{cor:slopezerobundles}
We conclude this section by using Proposition \ref{prop:HarderNarasimhan} to prove the analogue of \cite[Theorem 8.5.12]{KL15} in our more general setting.
\begin{cor*}
  We have an exact tensor equivalence of categories
  \begin{align*}
    \{\mbox{pro-\'etale }\ul{\bE}\mbox{-local systems on }S\}\ra^\sim\left\{
  {\begin{tabular}{c}
    vector bundles $\sE$ on $X_S$\\
    with slope-zero $p_\sE$
  \end{tabular}}
\right\}
  \end{align*}
given by $\bL\mapsto\bL\otimes_{\ul{\bE}}\sO_{X_S}$.
\end{cor*}
\begin{proof}
Note that $-\otimes_{\ul{\bE}}\sO_{X_S}$ preserves tensor products and duals. Therefore internal homs let us reduce full faithfulness to proving that, for all pro-\'etale $\ul{\bE}$-local systems $\bL$ on $S$, the map $\Hom_{\ul{\bE}}(\ul{\bE},\bL)\ra\Hom_{\sO_{X_S}}(\sO_{X_S},\bL\otimes_{\ul{\bE}}\sO_{X_S})$ is a bijection. Now $\bL$ is trivial after replacing $S$ with a pro-\'etale cover, $\Hom_{\ul{\bE}}(\ul{\bE},\bL)$ satisfies pro-\'etale descent on $S$, and $\Hom_{\sO_{X_S}}(\sO_{X_S},\bL\otimes_{\ul{\bE}}\sO_{X_S})$ satisfies pro-\'etale descent on $S$ by Proposition \ref{ss:BanachColmezsheaf}. Hence we can assume that $\bL=\ul{\bE}$, where the result follows from Proposition \ref{prop:absoluteBanachColmez}.ii).

  For essential surjectivity, let $\sE$ be a vector bundle on $X_S$ with slope-zero $p_\sE$. After replacing $S$ with a clopen cover, we can assume that $\sE$ has constant rank $d$. Then Proposition \ref{ss:BanachColmezsheaf} and Proposition \ref{prop:absoluteBanachColmez}.ii) indicate that the presheaf $\wt{S}$ on $\Perf_{\bF_q}$ over $S$ whose $S'$-points equal $\Iso(\sE|_{X_{S'}},\sO_{X_{S'}}^{\oplus d})$ is a pseudotorsor for $\GL_d(\ul{\bE})=\ul{\GL_d(\bE)}$, and Proposition \ref{prop:HarderNarasimhan}.ii) shows that $\wt{S}$ is a torsor. Proposition \ref{ss:G(E)hypothesisA} and Proposition \ref{ss:G(E)hypothesisB} indicate that $\GL_d(\bE)$ is lctd over $\bN\cup\{\infty\}$, so Lemma \ref{ss:torsorsareproetale} shows that $\wt{S}\ra S$ is a pro-\'etale cover. This yields an affinoid perfectoid pro-\'etale cover $S'\ra S$ and an isomorphism $\sE|_{X_{S'}}\cong\sO_{X_{S'}}^{\oplus d}$. By full faithfulness, the resulting descent datum on $\sO_{X_{S'}}^{\oplus d}$ with respect to $X_{S'}\ra X_S$ induces a descent datum $\al$ on $\ul{\bE}^{\oplus d}$ with respect to $S'\ra S$, and $(\ul{\bE}^{\oplus d},\al)$ descends to the desired pro-\'etale $\ul{\bE}$-local system on $S$.
\end{proof}

\section{Reductive groups in families}\label{s:reductivegroups}
In this section, we study $p$-adic groups over $E$. We start with some recollections about Chevalley groups over $\bZ$. Then we combine this with results from \S\ref{s:localfields} to spread out quasisplit connected reductive groups over $E_\infty$ to reductive group schemes $G$ over $E$. Using a construction from \S\ref{s:smoothrepresentations}, the latter yields a group topological space $G(\bE)$ over $\bN\cup\{\infty\}$, which we show is lctd over $\bN\cup\{\infty\}$ as in Definition \ref{ss:grouphypotheses}. We also define the analogue $\bK^n$ of congruence subgroups of $G(\bE)$, and we prove that
\begin{align*}
  \ul\End_{G(\bE)}(\cInd^{G(\bE)}_{\bK^n}\ul\La)
\end{align*}
is isomorphic to the constant sheaf $\ul{\cH(G_\infty(E_\infty),\bK^n_\infty)_\La}$ in a way that recovers Ganapathy's isomorphism on stalks.

Next, we turn to the study of $G$-bundles on $X_S$. After proving a Tannakian description of $G$-torsors, we define the moduli stack
\begin{align*}
\Bun_G\ra\ul{\bN\cup\{\infty\}}
\end{align*}
of $G$-bundles on $X_S$. Using results from \S\ref{s:vectorbundles}, we prove that the locus $\Bun^1_G\subseteq\Bun_G$ of trivial $G$-bundles is open, and that $\Bun^1_G$ is isomorphic to the classifying stack of $\ul{G(\bE)}$ over $\ul{\bN\cup\{\infty\}}$. Finally, we prove an invariance property of $D_{\et}(\Bun_G,\La)$ when base changing to $\Bun_G\times_{\ul{\bN\cup\{\infty\}}}\Spd{C}$, by reducing to the situation considered in Fargues--Scholze \cite{FS21}.

\subsection{}\label{ss:groupoverinfty}
Let $(G^{\s},T^{\s},X^*(T^{\s}),\De,\{x_{\wt{a}}\}_{\wt{a}\in\De})$ be a pinned split connected reductive group over $\bZ$ as in \cite[Definition 6.1.1]{Con14}, and recall that the global sections $\sO_{G^{\s}}$ of $G^{\s}$ is a free $\bZ$-module \cite[4.9]{Bor70}. Write $B^{\s}$ for the Borel subgroup of $G^{\s}$ associated with the base $\De$, and note that $\sO_{B^{\s}}$ is also a free $\bZ$-module.

Let $\de:\Ga_\infty\ra\Aut(G^{\s},T^{\s},X^*(T^{\s}),\De,\{x_{\wt{a}}\}_{\wt{a}\in\De})$ be a continuous homomorphism, and write $F_\infty/E_\infty$ for the finite Galois extension such that $\Gal(F_\infty/E_\infty)$ is the image of $\de$. Descending $(G^{\s}_{F_\infty},T^{\s}_{F_\infty},B^{\s}_{F_\infty},\{x_{\wt{a}}\}_{\wt{a}\in\De})$ along the \'etale $\Gal(F_\infty/E_\infty)$-torsor $\Spec F_\infty\ra\Spec E_\infty$ via $\de$ yields a quasisplit connected reductive group $G_\infty$ over $E_\infty$ with a pinning $(B_\infty,T_\infty,\{x_{\wt{a}}\}_{\wt{a}\in\De})$ over $E_\infty$ as in \cite[Definition 2.9.1]{KP23}. Recall that every quasisplit connected reductive group over $E_\infty$ arises this way for a uniquely determined $\de$ (up to conjugation).

\subsection{}
We now spread out $G_\infty$ to a connected reductive group over $E$ as follows. Let $n$ be a positive integer such that the image of $I^n_\infty$ in $\Gal(F_\infty/E_\infty)$ is trivial. \textbf{For the rest of this paper, assume that $e_i\geq n$ for all $i$ in $\bN$.} Recall from \ref{ss:realizingextensions} the function $\psi:[-1,\infty)\ra[-1,\infty)$, the finite Galois extensions $F_i/E_i$, the finite \'etale $E$-algebra $F$, and the ring topological spaces $\bO_{\bF,n}$, $\bO_\bF$, and $\bF$ over $\bN\cup\{\infty\}$.

Corollary \ref{cor:Efet} implies that $\Spec{F}\ra\Spec{E}$ is an \'etale $\Gal(F_\infty/E_\infty)$-torsor. Therefore descending $(G^{\s}_F,T^{\s}_F,B^{\s}_F,\{x_{\wt{a}}\}_{\wt{a}\in\De})$ along $\Spec{F}\ra\Spec{E}$ via $\de$ yields a connected reductive group $G$ over $E$, a maximal subtorus $T$ of $G$, a Borel subgroup $B$ of $G$ containing $T$, and generators $x_{\wt{a}}$ of the rank-$1$ free $F$-module $(\Lie G_F)_{\wt{a}}$ for all $\wt{a}$ in $\De$.

For all $i$ in $\bN\cup\{\infty\}$, write $(G_i,T_i,B_i,\{x_{\wt{a}}\}_{\wt{a}\in\De})$ for the base change to $E_i$ of $(G,T,B,\{x_{\wt{a}}\}_{\wt{a}\in\De})$. Then $G_i$ and $(B_i,T_i,\{x_{\wt{a}}\}_{\wt{a}\in\De})$ are the objects associated with $G_\infty$ and $(B_\infty,T_\infty,\{x_{\wt{a}}\}_{\wt{a}\in\De})$ as in \ref{ss:closefieldscongruence}. In particular, when $i=\infty$ this agrees with the notation of \ref{ss:groupoverinfty}.

\subsection{}\label{ss:G(E)hypothesisB}
By Proposition \ref{ss:relativepoints}, the affine group $G$ over $E=\bE(\bN\cup\{\infty\})$ induces a group topological space $G(\bE)$ over $\bN\cup\{\infty\}$.
\begin{prop*}
The group topological space $G(\bE)$ over $\bN\cup\{\infty\}$ is Hausdorff, locally profinite, and satisfies Definition \ref{ss:grouphypotheses}.b).
\end{prop*}
\begin{proof}
Since $\bE$ is Hausdorff, the same holds for $G(\bE)$. Similarly, because $G$ is of finite type over $E$ and $\bE$ is locally profinite, the same holds for $G(\bE)$.
  
We turn to Definition \ref{ss:grouphypotheses}.b). Let $i$ be in $\bN\cup\{\infty\}$, and note that the fiber $G(\bE)_i$ is $G(E_i)$. Moreover, the evaluation map $\ev_i:\varinjlim_UG(\bE)(U)\ra G(\bE)_i$, where $U$ runs over compact neighborhoods of $i$, equals the composition
  \begin{align*}
    \varinjlim_UG(\bE)(U)=\varinjlim_UG(\bE(U))\ra G(\sO_{\Spa{E},i})\ra G(E_i) = G(\bE)_i,
  \end{align*}
where the first equality follows from Proposition \ref{ss:relativepoints}. The first arrow is an isomorphism because $\sO_{\Spa E,i}=\varinjlim_U\bE(U)$ and $G$ is of finite presentation over $E$. The second arrow is surjective because $E_i$ is the residue field of $\sO_{\Spa E,i}$, the latter is henselian \cite[Lemma 2.4.17 (a)]{KL15}, and $G$ is smooth over $E$. 
\end{proof}

\subsection{}\label{ss:Kn}
To construct our family of congruence subgroups over $\bN\cup\{\infty\}$, we use Theorem \ref{ss:parahoriccongruence}. For all $i$ in $\bN\cup\{\infty\}$, write $\cB(G_i/E_i)$ for the (reduced) building of $G_i$ over $E_i$, and write $o_i$ for the vertex in $\cB(G_i/E_i)$ corresponding to the adjusted Chevalley valuation associated with the pinning $(B_i,T_i,\{x_{\wt{a}}\}_{\wt{a}\in\De})$ over $E_i$. Write $\cK_i$ for the smooth affine model of $G_i$ over $\cO_i$ such that $\cK_i(\cO_i)\subseteq G_i(E_i)$ equals the open compact subgroup $G_i(E_i)_{o_i}^1$ \cite[Proposition 8.3.1]{KP23}, which is maximal.

Let $n$ be a positive integer. For large enough $i$, Theorem \ref{ss:parahoriccongruence} indicates that our isomorphism $\Tr_{e_i}(E_i)\cong\Tr_{e_i}(E_\infty)$ induces an isomorphism $(\cK_i)_{\cO_i/\fp_i^n}\cong(\cK_\infty)_{\cO_\infty/\fp_\infty^n}$ of groups over $\cO_i/\fp_i^n\cong\cO_\infty/\fp_\infty^n$.
\begin{defn*}Recall from \ref{defn:rings} our presentation of $\bN\cup\{\infty\}$.
  \begin{enumerate}[a)]
  \item Let $n$ be a positive integer. Write $\bK_n$ for the group topological space over $\bN\cup\{\infty\}$ given by $\varprojlim_d$ of the discrete group topological spaces
    \begin{align*}
      \Big(\coprod_{i\leq d}\cK_i(\cO_i/\fp_i^n)\Big)\textstyle\coprod\cK_\infty(\cO_\infty/\fp_\infty^n)
    \end{align*}
    over $\{i\in\bN\mid i\leq d\}\cup\{\infty\}$, where, for large enough $d$, the transition map
    \begin{align*}
      \Big(\coprod_{i\leq d+1}\cK_i(\cO_i/\fp_i^n)\Big)\textstyle\coprod\cK_\infty(\cO_\infty/\fp_\infty^n)\ra\Big(\coprod_{i\leq d}\cK_i(\cO_i/\fp_i^n)\Big)\textstyle\coprod\cK_\infty(\cO_\infty/\fp_\infty^n)
    \end{align*}
    sends $\cK_{d+1}(\cO_{d+1}/\fp_{d+1}^n)$ to $\cK_\infty(\cO_\infty/\fp_\infty^n)$ via the isomorphism
    \begin{align*}
      (\cK_{d+1})_{\cO_{d+1}/\fp_{d+1}^n}\cong(\cK_\infty)_{\cO_\infty/\fp_\infty^n}
    \end{align*}
    (since $d$ is large enough) and equals the identity otherwise.
  \item Write $\bK$ for the group topological space over $\bN\cup\{\infty\}$ given by $\varprojlim_n\bK_n$, where $n$ runs over positive integers, and the transition map $\bK_{n+1}\ra\bK_n$ is induced from reduction mod $\pi_i^n$.
  \item Write $\bK^n$ for the kernel of $\bK\ra\bK_n$.
  \end{enumerate}
\end{defn*}
Note that $\bK_n$ is finite locally constant over $\bN\cup\{\infty\}$, so $\bK$ is profinite. Moreover, $\{\bK^n\}_{n\geq1}$ forms a family of compact open group subspaces of $\bK$ over $\bN\cup\{\infty\}$ satisfying Definition \ref{ss:grouphypotheses}.a). We see that $\bK$ also satisfies Definition \ref{ss:grouphypotheses}.b), so it is lctd over $\bN\cup\{\infty\}$.

\subsection{}\label{ss:G(E)hypothesisA}
For all $i$ in $\bN\cup\{\infty\}$, note that the fiber $\bK_i$ is $\cK_i(\cO_i)$. Write $\io:\bK\ra G(\bE)$ for the map over $\bN\cup\{\infty\}$ whose fiber at $i$ equals the inclusion $\cK_i(\cO_i)\subseteq G(E_i)$.
\begin{prop*}
The map $\io$ is an open embedding of group topological spaces over $\bN\cup\{\infty\}$. Consequently, $G(\bE)$ satisfies Definition \ref{ss:grouphypotheses}.a), so it is lctd over $\bN\cup\{\infty\}$.
\end{prop*}
\begin{proof}
  Because $G^{\s}$ is of finite type over $\bZ$ and $\bO_\bF$ is an open ring topological subspace of $\bF$, Proposition \ref{ss:relativepoints} yields an open embedding $G^{\s}(\bO_\bF)\hra G^{\s}(\bF)$ of group topological spaces over $\bN\cup\{\infty\}$. Endow $G^{\s}(\bF)$ with the action of $\Gal(F_\infty/E_\infty)$ arising from $\de$ and Proposition \ref{ss:relativepoints}. Then taking $\Gal(F_\infty/E_\infty)$-invariants yields an open embedding
  \begin{align*}
    G^{\s}(\bO_\bF)^{\Gal(F_\infty/E_\infty)}\hra G^{\s}(\bF)^{\Gal(F_\infty/E_\infty)} = G(\bE)
  \end{align*}
  of group topological spaces over $\bN\cup\{\infty\}$.

  For all $i$ in $\bN\cup\{\infty\}$, we claim that there exists a natural morphism
  \begin{align*}
    \cK_i\ra\R_{\cO_{F_i}/\cO_i}(G^{\s}_{\cO_{F_i}})
  \end{align*}
  of groups over $\cO_i$ such that the resulting diagram
  \begin{align*}
    \xymatrix{\cK_i(\cO_i)\ar[rr]\ar@{^{(}->}[d] & & G^{\s}(\cO_{F_i})\ar@{^{(}->}[d]\\
    G_i(E_i)\ar@{^{(}->}[r] & G_i(F_i)\ar@{=}[r] & G^{\s}(F_i)}
  \end{align*}
  commutes. To see this, consider the image of $o_i$ in $\cB(G_i/F_i)$. Note that $G^{\s}_{\cO_{F_i}}$ is the smooth affine model of $(G_i)_{F_i}$ over $\cO_{F_i}$ such that $G^{\s}(\cO_{F_i})\subseteq G_i(F_i)$ equals $G_i(F_i)^1_{o_i}$. Now $\cK_i(\cO_i)$ equals the intersection of $G_i(E_i)$ with $G_i(F_i)^1_{o_i}$, so $\cK_i$ is the smoothening as in \cite[Definition A.6.3]{KP23} of the Zariski closure of $G_i$ in $\R_{\cO_{F_i}/\cO_i}(G^{\s}_{\cO_{F_i}})$. This yields the claim.

Because Proposition \ref{ss:relativepoints} is compatible with inverse limits in $R$ and $S$, we have
  \begin{align*}
    G^{\s}(\bO_\bF) = \varprojlim_nG^{\s}(\bO_{\bF,n}) = \varprojlim_n\varprojlim_d\Big[\Big(\coprod_{i\leq d}G^{\s}(\cO_{F_i}/\fp_{F_i}^n)\Big)\textstyle\coprod G^{\s}(\cO_{F_\infty}/\fp_{F_\infty}^n)\Big]
  \end{align*}
  as group topological spaces over $\bN\cup\{\infty\}$. For all $i$ in $\bN\cup\{\infty\}$, the claim induces a homomorphism $\cK_i(\cO_i/\fp_i^n)\ra G^{\s}(\cO_{F_i}/\fp_{F_i}^n)$, so taking $\varprojlim_n\varprojlim_d\coprod_{\{i\leq d\}\cup\{\infty\}}$ yields a morphism $\bK\ra G^{\s}(\bO_\bF)$ of group topological spaces over $\bN\cup\{\infty\}$. By checking on fibers, the above work shows that the image of this map is $G^{\s}(\bO_\bF)^{\Gal(F_\infty/E_\infty)}$. Moreover, we see that the composition
  \begin{align*}
    \bK\ra G^{\s}(\bO_\bF)^{\Gal(F_\infty/E_\infty)}\hra G(\bE)
  \end{align*}
equals $\io$, so $\io$ is a morphism of group topological spaces over $\bN\cup\{\infty\}$. Finally, since $\bK$ is compact and $G(\bE)$ is Hausdorff, $\io$ is a homeomorphism onto its image. We already saw that this image is the open subspace $G^{\s}(\bO_\bF)^{\Gal(F_\infty/E_\infty)}$.
\end{proof}

\subsection{}\label{ss:valuationmap}
We have the following version of the valuation map for $T(\breve\bE)$ over $\bN\cup\{\infty\}$. Recall from \ref{ss:realizingEbreve} the $E$-algebra $\breve{E}$, as well as the ring topological spaces $\breve\bO$ and $\breve\bE$ over $\bN\cup\{\infty\}$. Note that the pullback to $\Spec{F}$ of $\ul{\Hom}(T,\bG_m)$ is naturally $\Gal(F_\infty/E_\infty)$-equivariantly isomorphic to the constant $\bZ$-local system with fiber $X^*(T_\infty)$, so elements of $X^*(T_\infty)^{\Ga_\infty}$ induce morphisms $T\ra\bG_m$ of groups over $E$. Similarly, elements of $X^*(T_\infty)^{I_\infty}$ induce morphisms $T_{\breve{E}}\ra\bG_m$ of groups over $\breve{E}$, and elements of $X_*(T_\infty)$ induce morphisms $\bG_m\ra T_{\breve{F}}$ of groups over $\breve{F}$.

For all compact open subsets $U$ of $\bN\cup\{\infty\}$ and $f$ in $T(\breve\bE)(U)$, write
\begin{align*}
v(f):X^*(T_\infty)^{I_\infty}\ra\ul{\bZ}(U)
\end{align*}
for the map that sends $\om$ in $X^*(T_\infty)^{I_\infty}$ to the composition
\begin{align*}
\xymatrix{U\ar[r]^-f & T(\breve\bE)\ar[r]^-{\om(\breve\bE)} & \breve\bE^\times\ar[r]^-v & \bZ\times(\bN\cup\{\infty\}).}
\end{align*}
Note this yields a morphism $v:T(\breve\bE)\ra\ul{\Hom(X^*(T_\infty)^{I_\infty},\bZ)}$ of abelian sheaves on $\bN\cup\{\infty\}$. Identify $\Hom(X^*(T_\infty)^{I_\infty},\bZ)$ with the $\bZ$-torsionfree quotient $X_*(T_\infty)_{I_\infty,\tf}$ of $X_*(T_\infty)_{I_\infty}$.

Write $T(\breve\bE)^1$ for the subspace of $T(\breve\bE)$ whose fiber at $i$ equals $T_i(\breve{E}_i)^1$, and write $T(\bE)^1$ for the subspace of $T(\bE)$ whose fiber at $i$ equals $T_i(E_i)^1$. Proposition \ref{ss:G(E)hypothesisA} shows that $T(\bE)^1$ is an open group topological subspace of $T(\bE)$ over $\bN\cup\{\infty\}$. Similarly, arguing as in the proof of Proposition \ref{ss:G(E)hypothesisA} shows that $T(\breve\bE)^1$ is an open group topological subspace of $T(\breve\bE)$ over $\bN\cup\{\infty\}$.
\begin{lem*}
  The morphism $v$ induces an isomorphism
  \begin{align*}
    v:T(\breve\bE)/T(\breve\bE)^1\ra^\sim\ul{X_*(T_\infty)_{I_\infty,\tf}}
  \end{align*}
  of abelian sheaves on $\bN\cup\{\infty\}$. Moreover, $v$ restricts to an isomorphism
  \begin{align*}
    v:T(\bE)/T(\bE)^1\ra^\sim\ul{(X_*(T_\infty)_{I_\infty,\tf})^{\vp_\infty}}.
  \end{align*}
\end{lem*}
\begin{proof}
By using Proposition \ref{ss:openquotient} to check on stalks, the first statement follows from \cite[Corollary 11.6.2]{KP23}, and the second one follows from \cite[Corollary 11.7.6]{KP23}.
\end{proof}

\subsection{}
In \ref{ss:toruscomparison}, we compare elements of $T(\bE)_i$ for different values of $i$ in $\bN\cup\{\infty\}$. To carry out this comparison compatibly over all $i$ in $\bN\cup\{\infty\}$, we proceed as follows.

Choose a uniformizer $\pi_{F_\infty}$ of $F_\infty$. For all $i$ in $\bN$, choose a uniformizer $\pi_{F_i}$ of $F_i$ whose image in $\cO_{F_i}/\fp_{F_i}^{\psi(e_i)}\cong\cO_{F_\infty}/\fp_{F_\infty}^{\psi(e_i)}$ equals the image of $\pi_{F_\infty}$, and write $\pi_F$ for the continuous section of $\bF\ra\bN\cup\{\infty\}$ whose value on $i$ in $\bN\cup\{\infty\}$ equals $\pi_{F_i}$.

Choose a $\bZ$-basis $\mu_1,\dotsc,\mu_r$ of $X_*(T_\infty)_{I_\infty,\tf}$, and choose representatives $\wt\mu_1,\dotsc\wt\mu_r$ in $X_*(T_\infty)$ of the $\mu_1,\dotsc,\mu_r$. Consider the morphism of abelian group topological spaces over $\bN\cup\{\infty\}$
\begin{align*}
\breve\na:X_*(T_\infty)_{I_\infty,\tf}\times(\bN\cup\{\infty\})\ra T(\breve\bE)
\end{align*}
that sends the constant section $\mu_j$ to $\Nm_{\breve{F}/\breve{E}}(\wt\mu_j\circ\pi_F)$ for all $1\leq j\leq r$. By checking on fibers, we see that $\breve\na$ is a section of the composition
\begin{align*}
\xymatrix{T(\breve\bE)\ar[r] & T(\breve\bE)/T(\breve\bE)^1\ar[r]^-v & X_*(T_\infty)_{I_\infty,\tf}\times(\bN\cup\{\infty\}).}
\end{align*}
Lemma \ref{ss:valuationmap} implies that this composition is a surjective morphism of abelian group topological spaces over $\bN\cup\{\infty\}$ with kernel $T(\breve\bE)^1$, so $\breve\na$ induces an isomorphism of group topological spaces over $\bN\cup\{\infty\}$
\begin{align*}
\big[X_*(T_\infty)_{I_\infty,\tf}\times(\bN\cup\{\infty\})\big]\times_{\bN\cup\{\infty\}} T(\breve\bE)^1\ra^\sim T(\breve\bE).
\end{align*}

\subsection{}\label{ss:Heckealgebrasinfamilies}
Let $\La$ be a ring, and recall the notation of Definition \ref{ss:smoothreps}, Definition \ref{ss:cInd}, and Definition \ref{ss:Heckealgebra}. \textbf{For the rest of this paper, assume that $e_i$ is large enough as in Theorem \ref{thm:Heckealgebracongruence} for all $i$ in $\bN$.} For all $i$ in $\bN\cup\{\infty\}$, Theorem \ref{thm:Heckealgebracongruence} indicates that our isomorphism $\Tr_{e_i}(E_i)\cong\Tr_{e_i}(E_\infty)$ induces an isomorphism of $\La$-algebras
\begin{align*}
\cH(G(E_\infty),\bK^n_\infty)_\La\ra^\sim\cH(G(E_i),\bK^n_i)_\La.
\end{align*}
The following result is crucial for the spreading out argument in our main theorems.
\begin{thm*}
  We have a natural isomorphism
  \begin{align*}
    \vsg:\ul{\cH(G(E_\infty),\bK_\infty^n)_\La}\ra^\sim \ul{\End}_{G(\bE)}(\cInd_{\bK^n}^{G(\bE)}\ul\La)
  \end{align*}
 of sheaves of $\La$-modules on $\bN\cup\{\infty\}$  such that, for all $i$ in $\bN\cup\{\infty\}$, the stalk $\vsg_i$ equals the composition $\cH(G(E_\infty),\bK_\infty^n)_\La\ra^\sim\cH(G(E_i),\bK^n_i)_\La=\End_{G(E_i)}(\cInd^{G(\bE)}_{\bK^n}\ul\La)$.
\end{thm*}
\begin{proof}
  For all $i$ in $\bN\cup\{\infty\}$, Proposition \ref{ss:doublecosetcongruence} yields a natural bijection
  \begin{align*}
    \bK_i^n\bs G(E_i)/\bK_i^n\cong\bK_\infty^n\bs G(E_\infty)/\bK_\infty^n.
  \end{align*}
  We claim that there exists a map $\Ups:\bK_\infty^n\bs G(E_\infty)/\bK_\infty^n\ra G(\bE)(\bN\cup\{\infty\})$ such that, for all $i$ in $\bN\cup\{\infty\}$, the square
  \begin{align*}
    \xymatrix{\bK_i^n\bs G(E_i)/\bK_i^n\ar@{=}[r]^-\sim & \bK_\infty^n\bs G(E_\infty)/\bK_\infty^n\ar[d]^-\Ups\\
   G(E_i)\ar[u] & \ar[l]_-{\ev_i}G(\bE)(\bN\cup\{\infty\})}
  \end{align*}
  commutes. To see this, choose a $\bZ$-basis $\nu_1,\dotsc,\nu_s$ of $(X_*(T_\infty)_{I_\infty,\tf})^{\vp_\infty}$. By Lemma \ref{ss:valuationmap}, there exist $t_1,\dotsc,t_s$ in $T(\bE)(\bN\cup\{\infty\})$ such that, for all $1\leq j\leq s$, the image of $t_j$ under $v$ equals the constant section $\nu_j$. Write
  \begin{align*}
    \na:(X_*(T_\infty)_{I_\infty,\tf})^{\vp_\infty}\ra T(\bE)(\bN\cup\{\infty\})
  \end{align*}
for the homomorphism that sends $\nu_j\mapsto t_j$ for all $1\leq j\leq s$.

  For all $i$ in $\bN\cup\{\infty\}$, write $\cT_i$ for the N\'eron model of $T_i$ over $\cO_i$, and note that $\cT_i$ is naturally a subgroup of $\cK_i$ over $\cO_i$ by Lemma \ref{ss:parahoricstructure}.i). For all $1\leq j\leq s$, write $\ov{t_j}$ for the composition
  \begin{align*}
    \xymatrix{\bN\cup\{\infty\}\ar[r]^-{t_j} & T(\breve\bE) & \ar[l]_-\sim \big[X_*(T_\infty)_{I_\infty,\tf}\times(\bN\cup\{\infty\})\big]\times_{\bN\cup\{\infty\}} T(\breve\bE)^1\ar[r]^-{\pr_2} & T(\breve\bE)^1.}
  \end{align*}
  Because $\ov{t_j}$ is continuous, the image in $\cT_i(\breve\cO_i/\fp_i^n)\cong\cT_\infty(\breve\cO_\infty/\fp_\infty^n)$ of its fiber at $i$ is constant for large enough $i$.

  For all $i$ in $\bN\cup\{\infty\}$, write $T^n_i$ for the kernel of $\cT_i(\cO_i)\ra\cT_i(\cO_i/\fp_i^n)$. Now \ref{ss:toruscomparison} indicates that our choice of $\wt\mu_1,\dotsc,\wt\mu_r$ in $X_*(T_\infty)$, along with our choice of uniformizers $\pi_{F_i}$ and $\pi_{F_\infty}$, induce an isomorphism $T_i(E_i)/T^n_i\cong T_\infty(E_\infty)/T^n_\infty$, and the above implies that the images of $\ev_i(t_j)$ and $\ev_\infty(t_j)$ in $T_i(E_i)/T^n_i\cong T_\infty(E_\infty)/T^n_\infty$ coincide for large enough $i$. Therefore, by modifying $t_j$ at finitely many $i$, we can assume that this holds for all $i$. As $\nu$ runs over $(X_*(T_\infty)_{I_\infty,\tf})^{\vp_\infty,+}$, using the resulting $\ev_i(\na(\nu))$ and $\ev_\infty(\na(\nu))$ in the proof of Proposition \ref{ss:doublecosetcongruence} yields the claim.
  
  Recall the notation of \ref{ss:doublecosetopen}. By checking on fibers, the claim identifies
  \begin{align*}
G(\bE) = \coprod_{g_\infty}\bK^n\Ups(g_\infty)\bK^n
  \end{align*}
as sets, where $g_\infty$ runs over $\bK_\infty^n\bs G(E_\infty)/\bK_\infty^n$. Combining this with Lemma \ref{ss:doublecosetopen} shows that $\bK^n\Ups(g_\infty)\bK^n$ is clopen in $G(\bE)$, so we can form the continuous function $h_{g_\infty}:G(\bE)\ra\La$ whose value on $\bK^n\Ups(g_\infty)\bK^n$ equals $1\circ\pr$ and whose value on $G(\bE)-\bK^n\Ups(g_\infty)\bK^n$ equals $0\circ\pr$.

  By checking on fibers, we see that $h_{g_\infty}$ lies in $\cH(G(\bE),\bK^n)_\La(\bN\cup\{\infty\})$. As $g_\infty$ varies, the $h_{g_\infty}$ induce a morphism of sheaves of $\La$-modules
  \begin{align*}
    \ul{\cH(G(E_\infty),K_\infty^n)_\La}\ra\cH(G(\bE),\bK^n)
  \end{align*}
  on $\bN\cup\{\infty\}$; take our candidate morphism $\vsg$ to be the composition
  \begin{align*}
    \ul{\cH(G(E_\infty),K^n_\infty)_\La}\ra\cH(G(\bE),\bK^n)_\La\ra^\vsg\ul{\End}_{G(\bE)}(\cInd_{\bK^n}^{G(\bE)}\ul\La).
  \end{align*}
Finally, the desired result follows from Proposition \ref{prop:HeckealgebrainjectsintoEnd} and checking on stalks.
\end{proof}

\subsection{}\label{ss:groupalgebracolimit}
Next, we turn to the study of $G$-torsors. More generally, let $H$ be one of $\{G,B\}$, and view $\sO_H$ as a representation of $H$ over $E$ via right translation.
\begin{lem*}The representation $\sO_H$ is isomorphic to $\varinjlim_\al V_\al$ for some directed family $\{V_\al\}_\al$ of objects of $\Rep{H}$.
\end{lem*}
\begin{proof}
  Because $H^{\s}$ is an affine group over $\bZ$ and $\sO_{H^{\s}}$ is a free $\bZ$-module, arguing as in the proof of \cite[II.2.4]{DMOS82} and taking $\bZ$-saturations show that $\sO_{H^{\s}}$ is isomorphic to $\varinjlim_\al V^{\s}_\al$ for some directed family $\{V^{\s}_\al\}_\al$ of objects of $\Rep{H^{\s}}$. Now the finite group $\Gal(F_\infty/E_\infty)$ acts on $\sO_{H^{\s}}$ via $\de$, so after replacing each $V^{\s}_\al$ with the sum of its $\Gal(F_\infty/E_\infty)$-translates, we can assume that each $V^{\s}_\al$ is $\Gal(F_\infty/E_\infty)$-stable.

Note that $\sO_{H^{\s}_F}$ is isomorphic to $\varinjlim_\al(V^{\s}_\al)_F$. The above indicates that each $(V^{\s}_\al)_F$ descends to a constant rank finite projective $E$-module $V_\al$ with a co-action of $\sO_H$, and that $\sO_H$ is isomorphic to $\varinjlim_\al V_\al$. Because $\abs{\Spa{E}}$ is profinite, every constant rank vector bundle on $\Spa{E}$ is free, so \cite[Theorem 2.7.7]{KL15} implies that every constant rank finite projective $E$-module is free. Therefore $V_\al$ is an object of $\Rep{H}$, as desired.
\end{proof}

\subsection{}\label{ss:Gtorsors}
With Lemma \ref{ss:groupalgebracolimit} in hand, we can prove the following Tannakian description of $H$-torsors. Let $Z$ be a sousperfectoid space over $\Spa{E}$, and let $\sH$ be an \'etale $H^{\an}$-torsor on $Z$. For all $V$ in $\Rep{H}$, write $\rho(V)$ for the locally free $\sO_{Z,\et}$-module
\begin{align*}
\rho(V)\coloneqq\sH\times^{H^{\an}}(V^{\an}\otimes_E\sO_{Z,\et}).
\end{align*}
Now \cite[Theorem 8.2.22 (d)]{KL15} implies that $\rho(V)$ is a vector bundle on $Z$, so this yields an exact tensor functor $\rho:\Rep{H}\ra\{\mbox{vector bundles on }Z\}$.
\begin{prop*}
  This association induces an equivalence of categories between
  \begin{enumerate}[a)]
  \item \'etale $H^{\an}$-torsors $\sH$ on $Z$,
  \item exact tensor functors $\rho:\Rep{H}\ra\{\mbox{vector bundles on }Z\}$.
  \end{enumerate}
  Moreover, when $Z=\Spa(R,R^+)$ is affinoid, the above are naturally equivalent to
  \begin{enumerate}[a)]
  \setcounter{enumi}{2}
  \item \'etale $H$-torsors $\sH^{\alg}$ on $\Spec{R}$,
  \item exact tensor functors $\rho:\Rep{H}\ra\{\mbox{finite projective }R\mbox{-modules}\}$.
  \end{enumerate}
\end{prop*}
\begin{proof}
  Both a) and b) are open-local on $Z$, so we can assume that $Z=\Spa(R,R^+)$ is affinoid. Let
  \begin{align*}
    \rho:\Rep{H}\ra\{\mbox{vector bundles on }Z\}\cong\{\mbox{finite projective }R\mbox{-modules}\}
  \end{align*}
  be an exact tensor functor, where the equivalence follows from \cite[Theorem 2.7.7]{KL15}. Lemma \ref{ss:groupalgebracolimit} indicates that $\sO_H\cong\varinjlim_\al V_\al$ for some directed family $\{V_\al\}_\al$ of objects of $\Rep{H}$, so we can define $\rho(\sO_H)$ to be $\varinjlim_\al\rho(V_\al)$. Because the $V_\al$ are finite free $E$-modules, this is independent of the family $\{V_\al\}_\al$.

Note that $\sO_H$ is an $E$-algebra with an $H$-equivariant co-action of $\sO_H$ over $E$ via left translation. Since $\rho$ is a tensor functor, this endows $\rho(\sO_H)$ with the structure of an $R$-algebra with a co-action of $\sO_H$ over $E$. Moreover, $\rho(\sO_H)$ is a direct limit of flat $R$-modules and hence itself flat over $R$. Because $\sO_H$ contains $E$ and $\rho(E)=R$, the exactness of $\rho$ shows that $\rho(\sO_H)$ is supported everywhere on $\Spec{R}$, so $\rho(\sO_H)$ is even faithfully flat over $R$.

Write $\sH^{\alg}$ for $\Spec\rho(\sO_H)$. The above implies that $\sH^{\alg}$ has an action of $H$ over $E$, and we see that $\sH^{\alg}\times_{\Spec R}\sH^{\alg}$ is naturally $H$-equivariantly isomorphic to $H\times_{\Spec{E}}\sH^{\alg}$ over $\sH^{\alg}$. Since $\sH^{\alg}\ra\Spec{R}$ is faithfully flat, this shows that $\sH^{\alg}$ is an fpqc $H$-torsor on $\Spec{R}$. The smoothness of $H$ indicates that $\sH^{\alg}$ is even an \'etale $H$-torsor on $\Spec{R}$. Finally, because $Z$ is sousperfectoid, the analytification of $\sH^{\alg}$ over $Z$ yields an \'etale $H^{\an}$-torsor $\sH$ on $Z$, as desired.
\end{proof}

\subsection{}\label{ss:BunG}
We have the following analogue of the moduli of $G$-bundles on the Fargues--Fontaine curve in our setting.
\begin{defn*}
Write $\Bun_G$ for the presheaf of groupoids on $\Perf_{\bF_q}$ over $\ul{\bN\cup\{\infty\}}$ whose $S$-points equal the category of \'etale $G^{\an}$-torsors on $X_S$.
\end{defn*}
Proposition \ref{ss:Gtorsors} and Proposition \ref{ss:BanachColmezsheaf} imply that $\Bun_G$ is a v-stack, and the proof of \cite[Proposition III.1.3]{FS21} shows that $\Bun_G$ is small.

For all $i$ in $\bN\cup\{\infty\}$, write $\Bun_{G_i}$ for the v-stack as in \cite[Definition III.1.2]{FS21}. Proposition \ref{ss:Wittspecialization} identifies the fiber of $\Bun_G$ at $\ul{\{i\}}$ with $\Bun_{G_i}$.

\subsection{}\label{ss:Newtonuppersemicontinuous}
For all $i$ in $\bN\cup\{\infty\}$, write $\nu_i:\abs{\Bun_{G_i}}\ra(X_*(T_i)_\bQ^+)^{\Ga_i}\cong(X_*(T_\infty)_\bQ^+)^{\Ga_\infty}$ for the Newton map as in \cite[Theorem III.2.3]{FS21}. Because \ref{ss:BunG} identifies the set $\abs{\Bun_G}$ with the disjoint union $\coprod_i\abs{\Bun_{G_i}}$, where $i$ runs over $\bN\cup\{\infty\}$, the maps $\nu_i$ induce a map $\nu:\abs{\Bun_G}\ra(X_*(T_\infty)_\bQ^+)^{\Ga_\infty}$.
\begin{prop*}
The map $\nu:\abs{\Bun_G}\ra(X_*(T_\infty)_\bQ^+)^{\Ga_\infty}$ is upper semicontinuous.
\end{prop*}
\begin{proof}
  Let $\om$ be in $X^*(T_\infty)^{\Ga_\infty,+}$, write $V_\om$ in $\Rep{G}$ for the associated Weyl module, and write $\rho_\om:G\ra\GL_d$ for the corresponding morphism of groups over $E$. For all $i$ in $\bN\cup\{\infty\}$, write $(V_{\om,i},\rho_{\om,i})$ for the base change to $E_i$ of $(V_\om,\rho_\om)$. Identify $\om$ with its image in $X^*(T_i)^{\Ga_i,+}\cong X^*(T_\infty)^{\Ga_\infty,+}$, and note that $V_{\om,i}$ is an irreducible representation of $G_i$ with highest weight $\om$.

Let $\sG$ be a $G^{\an}$-torsor on $X_S$, and let $\nu_0$ be in $(X_*(T_\infty)^+_\bQ)^{\Ga_\infty}$. By definition,
\begin{align*}
  \big\{\ov{s}\in\abs{S}\,\big|\,\nu(\sG|_{X_{\ov{s}}})\geq\nu_0\big\} = \left\{
  \begin{tabular}{c|c}
    \multirow{2}{*}{$\ov{s}\in\abs{S}$} & $\nu_i(\sG|_{X_{\ov{s}}})\geq\nu_0$, where $\{i\}$ is\\
  & the image of $\ov{s}$ in $\bN\cup\{\infty\}$
  \end{tabular}\right\}.
\end{align*}
Now \cite[p.~165]{RR96} shows that this equals
  \begin{align*}
\left\{
  \begin{tabular}{c|c}
    \multirow{2}{*}{$\ov{s}\in\abs{S}$} & $p_{\sG|_{X_{\ov{s}}}(V_{\om,i})}\geq\rho_{\om,i}(\nu_0)$ for all $\om$ in $X^*(T_i)^{\Ga_i,+}$,\\
  &  where $\{i\}$ is the image of $\ov{s}$ in $\bN\cup\{\infty\}$
  \end{tabular}\right\},
  \end{align*}
so the above implies that this equals
  \begin{align*}
    \bigcap_\om\big\{\ov{s}\in\abs{S}\,\big|\,p_{\sG(V_\om)}(\ov{s})\geq\rho_{\om,\infty}(\nu_0)\big\},
  \end{align*}
  where $\om$ runs over $X^*(T_\infty)^{\Ga_\infty,+}$. Proposition \ref{prop:HarderNarasimhan}.ii) indicates that this is a closed subset of $\abs{S}$, and this yields the desired result.
\end{proof}

\subsection{}\label{ss:BunG1}
Write $\Bun^1_G$ for the substack of $\Bun_G$ whose $S$-points consist of \'etale $G^{\an}$-torsors $\sG$ on $X_S$ such that, for all geometric points $\ov{s}$ of $S$, the pullback $\sG|_{X_{\ov{s}}}$ is trivial.

As usual, the trivial locus $\Bun_G^1$ admits the following description in terms of classifying stacks. Recall the notation of \ref{ss:Drinfeldfullyfaithful}. Consider the morphism $\ul{\bN\cup\{\infty\}}\ra\Bun_G$ corresponding to the trivial $G^{\an}$-torsor, which factors through a morphism
\begin{align*}
  \ul{\bN\cup\{\infty\}}\ra\Bun^1_G.
\end{align*}
Proposition \ref{prop:absoluteBanachColmez}.ii) identifies $\ul{G(\bE)} = G(\ul{\bE})$ with $\ul{\bN\cup\{\infty\}}\times_{\Bun_G^1}\ul{\bN\cup\{\infty\}}$ as group v-sheaves over $\ul{\bN\cup\{\infty\}}$, so descent yields a morphism $BG(\bE)\ra\Bun_G^1$.
\begin{prop*}
The substack $\Bun^1_G\subseteq\Bun_G$ is open, and the morphism $BG(\bE)\ra\Bun^1_G$ is an isomorphism.
\end{prop*}
\begin{proof}
  Let $\sG$ be a $G^{\an}$-torsor on $X_S$. We claim that
  \begin{align*}
    \big\{\ov{s}\in\abs{S}\,\big|\,\mbox{the pullback }\sG|_{X_{\ov{s}}}\mbox{ is trivial}\big\}
  \end{align*}
  is an open subset of $\abs{S}$. To see this, \cite[Lemma 2.5]{Sch17} indicates that, after replacing $S$ with a pro-\'etale cover, we can assume that $S$ is strictly totally disconnected. Now
  \begin{align*}
    \big\{\ov{s}\in\abs{S}\,\big|\,\mbox{the pullback }\sG|_{X_{\ov{s}}}\mbox{ is trivial}\big\}\subseteq\big\{\ov{s}\in\abs{S}\,\big|\,\nu(\sG|_{X_{\ov{s}}})=0\big\},
  \end{align*}
  and the right-hand side is an open subset of $\abs{S}$ by Proposition \ref{ss:Newtonuppersemicontinuous}. Therefore, after replacing $S$ with the corresponding open subspace, we can assume that $\nu(\sG|_{X_{\ov{s}}})=0$ for all geometric points $\ov{s}$ of $S$. For all $V$ in $\Rep{G}$, this implies that $p_{\sG(V)}$ has slope-zero, so Corollary \ref{cor:slopezerobundles} shows that $\sG$ corresponds to an exact tensor functor
  \begin{align*}
    \Rep{G}\ra\{\mbox{pro-\'etale }\ul{\bE}\mbox{-local systems on }S\}.
  \end{align*}
Because $S$ is strictly totally disconnected, Lemma \ref{ss:torsorsareproetale} implies that pro-\'etale $\ul{\bE}$-local systems on $S$ are trivial. Hence the natural functor
  \begin{align*}
    \{\mbox{finite free }\ul{\bE}(S)\mbox{-modules}\}\ra\{\mbox{pro-\'etale }\ul{\bE}\mbox{-local systems on }S\}
  \end{align*}
  is an exact tensor equivalence, so $\sG$ corresponds to an exact tensor functor
  \begin{align*}
    \Rep{G}\ra\{\mbox{finite free }\ul{\bE}(S)\mbox{-modules}\}.
  \end{align*}
  
  Proposition \ref{ss:Gtorsors} shows that this corresponds to a $G$-torsor $\sF$ on
  \begin{align*}
    \Spec\ul{\bE}(S) = \Spec\Cont_{\bN\cup\{\infty\}}(\abs{S},\bE) = \Spec\Cont_{\bN\cup\{\infty\}}(\pi_0(\abs{S}),\bE).
  \end{align*}
Since $\pi_0(\abs{S})$ is profinite, we can form the analytic adic space
  \begin{align*}
    Z\coloneqq\ul{\pi_0(\abs{S})}\times_{\ul{\bN\cup\{\infty\}}}\Spa{E}.
  \end{align*}
  Note that $Z$ is affinoid with global sections equal to $\Cont_{\bN\cup\{\infty\}}(\pi_0(\abs{S}),\bE)$. For all $\ov{s}$ in $\abs{S}$, write $z$ for the point of $Z$ corresponding to the image of $\ov{s}$ in $\pi_0(\abs{S})$. Then
  \begin{align*}
   \sO_{Z,z}=\varinjlim_U\Cont_{\bN\cup\{\infty\}}(U,\bE),
  \end{align*}
  where $U$ runs over neighborhoods of $\ov{s}$ in $\pi_0(\abs{S})$, and the residue field of $\sO_{Z,z}$ is $\Cont_{\bN\cup\{\infty\}}(\ov{s},\bE)$. Because $\sO_{Z,z}$ is henselian \cite[Lemma 2.4.17 (a)]{KL15}, the triviality of $\sF|_{\Cont_{\bN\cup\{\infty\}}(\ov{s},\bE)}$ implies the triviality of $\sF|_{\Cont_{\bN\cup\{\infty\}}(U,\bE)}$ for some $U$. Therefore if $\sG|_{X_{\ov{s}}}$ is trivial, then $\sG|_{X_{\wt{U}}}$ is trivial, where $\wt{U}$ denotes the preimage in $\abs{S}$ of $U$. As $\ov{s}$ varies, this yields the claim.

  The claim yields the first statement. For the second statement, note that the above work shows that the morphism $\ul{\bN\cup\{\infty\}}\ra\Bun^1_G$ is a pro-\'etale cover. Hence the second statement follows from descent.
\end{proof}

\subsection{}\label{ss:BunGoverC}
In this subsection, we work over $\Spd\ov\bF_q$ and assume that $\La$ is $\ell$-power torsion. We conclude this section by proving that the second statement in Proposition \ref{ss:Drinfeldfullyfaithful} holds for $Z=\Bun_G$, which will be needed in \S\ref{s:final}. Recall from \ref{ss:Cperfectoid} the $E$-algebra $C$, and recall from \ref{ss:Waction} the morphism $\Spd{C}\ra\ul{\bN\cup\{\infty\}}$.
\begin{prop*}
The pullback functor
  \begin{align*}
    D_{\et}(\Bun_G,\La)\ra D_{\et}(\Bun_G\times_{\ul{\bN\cup\{\infty\}}}\Spd C,\La)
  \end{align*}
is an equivalence of categories.
\end{prop*}
\begin{proof}
Proposition \ref{ss:Cperfectoid} identifies $\Spd{C}$ with $\ul{\bN\cup\{\infty\}}\times\Spa C_\infty$ , so applying \cite[Theorem 19.5 (ii)]{Sch17} to $\ov\bF_q\hra C_\infty$ shows that our functor is fully faithful.
  
We turn to essential surjectivity. Since $\bN$ is discrete, \ref{ss:BunG}. indicates that $\Bun_G\!|_{\ul{\bN}}$ is naturally isomorphic to the disjoint union $\coprod_i\Bun_{G_i}$, where $i$ runs over $\bN$. This identifies the functor
\begin{align*}
  D_{\et}(\Bun_G\!|_{\ul{\bN}},\La)\ra D_{\et}(\Bun_G\!|_{\ul{\bN}}\times_{\ul{\bN\cup\{\infty\}}}\Spd C,\La)
\end{align*}
with $\prod_i D_{\et}(\Bun_{G_i},\La)\ra\textstyle\prod_iD_{\et}(\Bun_{G_i}\times\Spd C_i,\La)$, which is an equivalence by \cite[Corollary V.2.3]{FS21}. Similarly, \ref{ss:BunG} identifies the functor
\begin{align*}
  D_{\et}(\Bun_G\!|_{\ul{\{\infty\}}},\La)\ra D_{\et}(\Bun_G\!|_{\ul{\{\infty\}}}\times_{\ul{\bN\cup\{\infty\}}}\Spd C,\La)
\end{align*}
with $D_{\et}(\Bun_{G_\infty},\La)\ra D_{\et}(\Bun_{G_\infty}\times\Spd C_\infty,\La)$, which also is an equivalence by \cite[Corollary V.2.3]{FS21}. Therefore the essential surjectivity of our functor follows from combining full faithfulness and the excision exact triangle associated with
\begin{gather*}
  \Bun_G\!|_{\ul{\bN}}\times_{\ul{\bN\cup\{\infty\}}}\Spd C\hra\Bun_G\times_{\ul{\bN\cup\{\infty\}}}\Spd C\hla\Bun_G\!|_{\ul{\{\infty\}}}\times_{\ul{\bN\cup\{\infty\}}}\Spd C.\qedhere
\end{gather*}
\end{proof}

\section{Beilinson--Drinfeld affine Grassmannians}\label{s:Grassmannians}
In this section, we introduce the analogue of Beilinson--Drinfeld affine Grassmannians $\Gr^J_G$ over $E$. After proving various basic facts about $\Gr^J_G$ and its affine Schubert subvarieties, we study the analogue of semi-infinite orbits. This lets us apply the machinery of hyperbolic localization as in \cite[Section IV.6]{FS21} to define the constant term functor $\CT_B$ for \'etale sheaves on $\Gr^J_G$. We conclude by giving, in terms of $\CT_B$, criteria for \'etale sheaves on $\Gr^J_G$ to vanish or to be universally locally acyclic.

Our arguments in \S\ref{s:Grassmannians} and \S\ref{s:Satake} closely follow those of Scholze--Weinstein \cite{SW20} and Fargues--Scholze \cite{FS21}. For some statements, we can even reduce to the situation considered in Fargues--Scholze \cite{FS21} by working fiberwise over $\ul{\bN\cup\{\infty\}}$. This lets us avoid (direct) reduction mod $\pi$ arguments over $\ul{\bN\cup\{\infty\}}$, for example.

\subsection{}\label{ss:bundlesondivisors}
Recall from \ref{ss:Div1} the v-sheaf $\Div_X^1$ over $\ul{\bN\cup\{\infty\}}$, and recall the notation of \ref{ss:Drinfeldslemma}. Let $J$ be a finite set. For all morphisms $S\ra(\Div_X^1)^J$ over $\ul{\bN\cup\{\infty\}}$ and $j$ in $J$, write $S_j\ra\Div_X^1$ for the $j$-th projection, write $S^\sharp_j\hra X_S$ for the closed Cartier divisor from \ref{ss:Div1}, and write $\sI_{S^\sharp_j}$ for the corresponding ideal of $\sO_{X_S}$. Write $D_S$ for the closed Cartier divisor $\sum_{j\in J}S_j^\sharp\hra X_S$, and write $\sI_{D_S}$ for the corresponding ideal of $\sO_{X_S}$.
\begin{lem*}
  The presheaf of categories on $\Perf_{\bF_q}$ over $S$ given by
  \begin{align*}
    S'\mapsto\{\mbox{vector bundles on }D_{S'}\}
  \end{align*}
  satisfies v-descent.
\end{lem*}
\begin{proof}
  Let $S'\ra S$ be an affinoid perfectoid v-cover. When $S$ is a geometric point, the image of $\abs{S}$ in $\bN\cup\{\infty\}$ equals $\{i\}$ for some $\bN\cup\{\infty\}$, so the result follows from Proposition \ref{ss:Wittspecialization} and \cite[Proposition VI.1.4]{FS21}. When $S'\ra S$ is finite \'etale, Corollary \ref{ss:Wittfet} implies that $X_{S'}\ra X_S$ and hence $D_{S'}\ra D_S$ is finite \'etale, so the result follows from \cite[Theorem 8.2.22 (d)]{KL15}. Combined with Proposition \ref{ss:Yrationalopensubspace}, this shows that
  \begin{align*}
    \cP\coloneqq\{\mbox{vector bundles satisfy descent with respect to }D_{S'}\ra D_S\}
  \end{align*}
  satisfies the conditions in \cite[Proposition 8.2.20]{KL15}, so \cite[Proposition 8.2.20]{KL15} indicates that the result holds when $S'\ra S$ is \'etale.

  In general, let $\sE'$ be a vector bundle on $D_{S'}$ with a descent datum $\al$ with respect to $D_{S'}\ra D_S$. For all geometric points $\ov{s}$ of $S$, the above enables us to descend $(\sE'|_{D_{S'_{\ov{s}}}},\al|_{D_{S'_{\ov{s}}}})$ to a vector bundle $\sE_{D_{\ov{s}}}$ on $D_{\ov{s}}$. Because $\abs{D_{\ov{s}}}$ is a disjoint union of points, $\sE_{D_{\ov{s}}}$ and hence its pullback $\sE'|_{D_{S'_{\ov{s}}}}$ is free. Therefore \cite[Proposition 5.4.21]{GR03} implies that any basis of $\sE_{D_{\ov{s}}}$ induces a basis of $\sE'|_{D_U}$ for some affinoid perfectoid \'etale neighborhood $U$ of $S'_{\ov{s}}$; in particular, $\sE'|_{D_U}$ is free. In these coordinates, $\al|_{D_U}$ corresponds to a matrix with entries in $\sO_{D_U\times_{D_S}D_U}=\sO_{D_{U\times_SU}}$, and after shrinking $U$, we see that $\al|_{D_U}-\id$ has entries lying in the image of $[\vpi]\sO^+_{\cY_{U\times_SU,[0,q^r]}}$, where $r$ is large enough such that $D_S$ lies in $\cY_{S,[0,q^r]}$. Since $\cO$ is a direct summand of $\cO^{\perf}$ as topological $\cO$-modules, Proposition \ref{ss:Ysousperfectoid} and \cite[Proposition 8.8]{Sch17} show that the v-cohomology group $H^1_v(S,\sO_{\cY_{S,[0,q^r]}^+})$ is annihilated by $[\vpi]^{1/q^m}$ for all non-negative integers $m$. Hence the result follows from arguing as in the proof of \cite[Lemma 17.1.8]{SW20}.
\end{proof}

\subsection{}
Let us introduce the analogue of loop groups, affine Grassmannians, and local Hecke stacks in our setting. Write $B^+_{\dR}(S)$ for the global sections of $\varprojlim_n\sO_{X_S}/\!\sI_{D_S}^n$, and write $B_{\dR}(S)$ for the global sections of $\big(\textstyle\varprojlim_n\sO_{X_S}/\!\sI_{D_S}^n\big)[\frac1{\sI_{D_S}}]$. Lemma \ref{ss:bundlesondivisors} indicates that the global sections of $\sO_{X_S}/\!\sI_{D_S}^n=\sO_{nD_S}$ yields a v-sheaf over $(\Div^1_X)^J$, so the same holds for $B^+_{\dR}$ and $B_{\dR}$.

Let $M$ be an affine scheme over $E$, and let $H$ be one of $\{G,B\}$.
\begin{defn*}\hfill
  \begin{enumerate}[a)]
  \item Write $L^n_JM$, $L^+_JM$, and $L_JM$ for the presheaves on $\Perf_{\bF_q}$ over $(\Div^1_X)^J$ whose $S$-points equal $M(\sO_{nD_S})$, $M(B^+_{\dR}(S))$ and $M(B_{\dR}(S))$, respectively.
  \item Write $\cHck_H^J$ for the presheaf of groupoids on $\Perf_{\bF_q}$ over $(\Div^1_X)^J$ whose $S$-points parametrize data consisting of
    \begin{enumerate}[i)]
    \item two $H$-torsors $\sH$ and $\sH'$ on $\Spec B^+_{\dR}(S)$,
    \item an isomorphism $\rho:\sH|_{B_{\dR}(S)}\ra^\sim\sH'|_{B_{\dR}(S)}$ of $H$-torsors on $\Spec B_{\dR}(S)$.
    \end{enumerate}
  \item Write $\Gr_H^J$ for the presheaf on $\Perf_{\bF_q}$ over $(\Div_X^1)^J$ whose $S$-points parametrize data consisting of
    \begin{enumerate}[i)]
    \item an $H$-torsor $\sH$ on $\Spec B^+_{\dR}(S)$,
    \item an isomorphism $\rho:\sH|_{B_{\dR}(S)}\ra^\sim H$ of $H$-torsors on $\Spec B_{\dR}(S)$.
    \end{enumerate}
  \end{enumerate}
\end{defn*}
Since $M$ is affine, we see that $L^n_JM$, $L_J^+M$, and $L_JM$ are v-sheaves, and the proof of \cite[Proposition III.1.3]{FS21} shows that they are small. We have a natural morphism $\Gr_H^J\ra\cHck_H^J$ given by $(\sH,\rho)\mapsto(\sH,H,\rho)$.

\subsection{}\label{ss:jetgroup}
Let $M$ be a smooth affine scheme over $E$.
\begin{prop*}The morphism $L^n_JM\ra(\Div_X^1)^J$ is representable in locally spatial diamonds, partially proper, and cohomologically smooth of dimension $n\cdot\#J\cdot\dim{M}$.  
\end{prop*}
\begin{proof}
  First, assume that $M=\bG^d_a$, and choose an enumeration $J\cong\{1,\dotsc,\#J\}$. The filtration $0\subseteq\sI_{S_2^\sharp+\dotsb+S_{\#J}^\sharp}^n/\!\sI_{D_S}^n\subseteq\dotsb\subseteq\sI_{S_{\#J}^\sharp}^n/\!\sI_{D_S}^n\subseteq\sO_{nD_S}$ shows that the morphism $L^n_JM\ra(\Div_X^1)^J$ is the iterated extension of $L^n_{\{j\}}M$-torsors, where $j$ runs over $J$, and the filtration $0\subseteq\sI_{S^\sharp_j}^{n-1}/\!\sI_{S_j^\sharp}^n\subseteq\dotsb\subseteq\sI_{S^\sharp_j}/\!\sI^n_{S_j^\sharp}\subseteq\sO_{nS_j^\sharp}$ shows that the pullback to $S$ of $L^n_{\{j\}}M\ra(\Div_X^1)^J$ is the iterated extension of $n$ many $(\bG^{d,\an}_{a,S_j^\sharp})^\Diamond$-torsors. Because $(\bG^{d,\an}_{a,S_j^\sharp})^\Diamond\ra S$ is representable in locally spatial diamonds, partially proper, and cohomologically smooth of dimension $1$, the result for $M$ follows.

  Next, let $\wt{D}\ra nD_S$ be a quasicompact separated \'etale morphism. We claim that the presheaf on $\Perf_{\bF_q}$ over $S$ whose set of $S'$-points equal the set of lifts
  \begin{align*}
    \xymatrix{& \wt{D}\ar[d]\\
    nD_{S'}\ar@{.>}[ur]\ar[r] & nD_S }
  \end{align*}
  is representable by a separated \'etale morphism $\wt{S}\ra S$. To see this, \cite[Proposition 9.7]{Sch17} indicates that, after replacing $S$ with a pro-\'etale cover, we can assume that $S$ is strictly totally disconnected. For all geometric points $\ov{s}$ of $nD_S$, the morphism $\ov{s}\times_{nD_S}\wt{D}\ra\ov{s}$ is quasicompact separated \'etale and hence a finite disjoint union of $\ov{s}$. Since $S$ is strictly totally disconnected, \cite[Lemma 15.6]{Sch17} and \cite[Proposition 11.23 (iii)]{Sch17} imply that there exists an open neighborhood $U$ of $\ov{s}$ such that $U\times_{nD_S}\wt{D}\ra U$ is a finite disjoint union of $U$. Therefore after replacing $\wt{D}$ with an open cover, we may assume that $\wt{D}\ra nD_S$ is an open embedding. Then Lemma \ref{ss:Div1} shows that the image $C$ of $\big|nD_S-\wt{D}\big|$ in $\abs{S}$ is closed. The resulting open subspace $S-C$ of $S$ satisfies the desired property, which yields the claim.

  In general, after replacing $M$ with an open cover, we may assume that $M$ is \'etale over $\bG_a^d$. By the $\Spec$-global sections adjunction, morphisms $S\ra L^n_JM$ over $(\Div_X^1)^J$ correspond to morphisms $nD_S\ra M^{\an}$ over $\Spa{E}$. Applying this to $M$ and $\bG^d_a$ shows that morphisms $S'\ra S\times_{L^n_J\bG_a^d}L^n_JM$ over $S$ correspond to lifts
  \begin{align*}
    \xymatrix{& nD_S\times_{\bG^{d,\an}_a}M^{\an}\ar[d]\\
    nD_{S'}\ar@{.>}[ur]\ar[r] & nD_S. }
  \end{align*}
Now $M^{\an}\ra\bG^{d,\an}_a$ and hence $nD_S\times_{\bG^{d,\an}_a}M^{\an}\ra nD_S$ is quasicompact separated \'etale, so the claim indicates that $L^n_JM\ra L^n_J\bG^d_a$ is separated \'etale. Combined with the case of $\bG^d_a$, this yields the desired result.
\end{proof}

\begin{prop}\label{prop:HeckeandGr}
  The presheaves $\cHck_H^J$ and $\Gr_H^J$ are small v-stacks. Moreover, the natural morphisms $L_JH\ra \Gr^J_H\ra\cHck_H^J$ induce isomorphisms from the \'etale quotients $(L_JH)/(L^+_JH)\ra^\sim\Gr_H^J$ and $(L_J^+H)\bs(L_JH)/(L_J^+H)\ra^\sim\cHck_H^J$.
\end{prop}
\begin{proof}
  We claim that the presheaf of categories on $\Perf_{\bF_q}$ over $(\Div^1_X)^J$ given by
  \begin{align*}
    S\mapsto\{\mbox{finite projective }B^+_{\dR}(S)\mbox{-modules}\}
  \end{align*}
  satisfies v-descent. To see this, note that base change induces an equivalence
  \begin{align*}
    \{\mbox{finite projective }B^+_{\dR}(S)\mbox{-modules}\}\ra^\sim\varprojlim_n\{\mbox{finite projective }\sO_{nD_S}\mbox{-modules}\},
  \end{align*}
  so the claim follows from Lemma \ref{ss:bundlesondivisors} and \cite[Theorem 2.7.7]{KL15}.

  The claim and Proposition \ref{ss:Gtorsors} indicate that $\cHck^J_H$ and $\Gr^J_H$ are v-stacks. For the second statement, let $\sH$ be an $H$-torsor on $\Spec B_{\dR}^+(S)$. For all geometric points $\ov{s}$ of $S$, note that $B_{\dR}(\ov{s})$ is a finite product of strictly henselian local rings, so $\sH|_{B^+_{\dR}(\ov{s})}$ and hence $\sH|_{\sO_{D_{\ov{s}}}}$ is trivial. Because $H$ is smooth affine over $E$, \cite[Proposition 5.4.21]{GR03} implies that there exists an affinoid perfectoid \'etale neighborhood $U$ of $\ov{s}$ such that $\sH|_{\sO_{D_U}}$ is trivial. Since $B^+_{\dR}(U)$ is complete along $\sI_{D_U}$ and $H$ is smooth over $E$, this shows that $\sH|_{B^+_{\dR}(U)}$ is trivial. As $\ov{s}$ varies, we get an affinoid perfectoid \'etale cover $S'\ra S$ such that $\sH|_{B^+_{\dR}(S')}$ is trivial, which yields the second statement.

Finally, the second statement implies that $\cHck^J_H$ and $\Gr^J_H$ are small.
\end{proof}

\subsection{}\label{ss:Grassmannianspecialization}
Our affine Grassmannians and local Hecke stacks specialize to the usual ones from \cite{FS21}, which lets us reduce many facts to their ``classical'' analogues from \cite{FS21}. More precisely, for all $i$ in $\bN\cup\{\infty\}$, write $\Div_{X_i}^1$ for the v-sheaf as in \cite[Definition II.1.19]{FS21}, write $\cHck^J_{G_i}$ for the v-sheaf over $(\Div_{X_i}^1)^J$ as in \cite[Definition VI.1.6]{FS21}, and write $\Gr^J_{G_i}$ for the v-sheaf over $(\Div_{X_i}^1)^J$ as in \cite[Definition VI.1.8]{FS21}. Proposition \ref{ss:Wittspecialization} identifies the fiber at $\ul{\{i\}}$ of $\Div_X^1$, $\cHck^J_G$, and $\Gr^J_G$ with $\Div_{X_i}^1$, $\cHck^J_{G_i}$, and $\Gr^J_{G_i}$, respectively.

\subsection{}
In this subsection, assume that $G$ is split. Let $\mu_\bullet=(\mu_j)_{j\in J}$ be in $(X_*(T_\infty)^+)^J$. We define affine Schubert varieties in this setting.
\begin{defn*}\hfill
  \begin{enumerate}[a)]
  \item Write $\cHck_{G,\leq\mu_\bullet}^J$ for the substack of $\cHck_G^J$ whose $S$-points consist of $(\sH,\sH',\rho)$ such that, for all geometric points $\ov{s}$ of $S$ and $j$ in $J$, the relative position of $\rho|_{B_{\dR}(\ov{s})}$ at the completion of $X_S$ along $\ov{s}^\sharp_j$ is bounded by $\sum_{j'}\mu_{j'}$, where $j'$ runs over elements of $J$ such that $\ov{s}^\sharp_{j'}=\ov{s}^\sharp_j$.
  \item Write $\Gr^J_{G,\leq\mu_\bullet}$ for the pullback to $\Gr_G^J$ of $\cHck^J_{G,\leq\mu_\bullet}$.
  \end{enumerate}
\end{defn*}
Note that the action of $L^+_JG$ on $\Gr^J_G$ preserves $\Gr^J_{G,\leq\mu_\bullet}$.

\subsection{}\label{ss:affineSchubertvarieties}
In this subsection, assume that $G$ is split. Write $2\rho$ in $X^*(T_\infty)$ for the sum of all positive roots, and for all $\mu_\bullet$ be in $(X_*(T_\infty)^+)^J$, and write $d_{\mu_\bullet}$ for $\sum_{j\in J}\ang{2\rho,\mu_j}$. For all positive integers $n$, write $(L^+_JG)^n$ for the kernel of $L^+_JG\ra L^n_JG$. 
\begin{prop*}\hfill
  \begin{enumerate}[i)]
  \item The substack $\cHck^J_{G,\leq\mu_\bullet}\subseteq\cHck_G^J$ is closed, and the morphism $\Gr^J_{G,\leq\mu_\bullet}\ra(\Div_X^1)^J$ is representable in spatial diamonds, proper, and of finite $\dimtrg$.
\item Assume that $n$ is greater than all weights of $\sum_{j\in J}\mu_j$ acting on $\Lie{G}$. Then $(L^+_JG)^n$ acts trivially on $\Gr^J_{G,\leq\mu_\bullet}$.
\item The natural morphism $\varinjlim_{\mu_\bullet}\Gr^J_{G,\leq\mu_\bullet}\ra\Gr_G^J$, where $\mu_\bullet$ runs over $(X_*(T_\infty)^+)^J$, is an isomorphism.
\item For $\mathrm{par}$ in $\bZ/2$, write $(\Gr^J_G)^{\mathrm{par}}$ for $\varinjlim_{\mu_\bullet}\Gr^J_{G,\leq\mu_\bullet}$, where $\mu_\bullet$ runs over elements of $(X_*(T_\infty)^+)^J$ such that the image of $d_{\mu_\bullet}$ in $\bZ/2$ equals $\mathrm{par}$. Then the subspace $(\Gr^J_G)^{\mathrm{par}}\subseteq\Gr^J_G$ is clopen.
\end{enumerate}
\end{prop*}
\begin{proof}
  Part i) follows from arguing as in the proof of \cite[Proposition 20.4.5]{FS21}. For part ii), note that part i) indicates that the equalizer of
  \begin{align*}
    \xymatrix{(L^+_JG)^n\times_{(\Div_X^1)^J}\Gr^J_{G,\leq\mu_\bullet}\ar@<.25pc>[r]^-a \ar@<-.25pc>[r]_-{\pr_2}& \Gr^J_{G,\leq\mu_\bullet}}
  \end{align*}
  is a closed subspace of $(L^+_JG)^n\times_{(\Div_X^1)^J}\Gr^J_{G,\leq\mu_\bullet}$, so \cite[Lemma 12.11]{Sch17} shows that, to see that $a$ and $\pr_2$ agree, it suffices to check on geometric points. This follows from \ref{ss:Grassmannianspecialization} and \cite[Proposition VI.2.8]{FS21}.

  For part iii), the natural morphism $\varinjlim_{\mu_\bullet}\Gr^J_{G,\leq\mu_\bullet}\ra\Gr_G^J$ is injective, so we focus on surjectivity. For any morphism $S\ra\Gr_G^J$ over $(\Div_X^1)^J$, note that its image in $\Gr_G^J$ lies in a finite union of subspaces of the form $\Gr^J_{G,\leq\mu_\bullet}$. Then part i) implies that $\coprod_{\mu_\bullet}\Gr^J_{G,\leq\mu_\bullet}\ra\Gr_G^J$ is a v-cover, which yields surjectivity, as desired. Finally, part iv) follows from $\ang{2\rho,a}$ being even for all roots $a$ in $X_*(T_\infty)$.
\end{proof}

\subsection{}\label{ss:affineSchubertcells}
In this subsection, assume that $G$ is split. For all $j$ in $J$, let $\xi_j$ in $B^+_{\dR}(S)$ be a generator of $\sI_{S^\sharp_j}$, and choose an enumeration $J\cong\{1,\dotsc,\#J\}$. For all $\mu_\bullet$ in $(X_*(T_\infty)^+)^J$, write $[\mu_\bullet]:(\Div^1_X)^J\ra\Gr^J_G$ for the morphism over $(\Div^1_X)^J$ given by sending $S\mapsto\prod_{j=1}^{\#J}(L_J\mu)(\xi_j)$, which is independent of the $\xi_j$. When $G$ is abelian or $\#J=1$, this is independent of the enumeration of $J$.

For the rest of this subsection, assume that $\#J=1$. In this case, we also consider affine Schubert cells. Proposition \ref{ss:affineSchubertvarieties}.i) indicates that 
\begin{align*}
\cHck^{\{*\}}_{G,\mu}\coloneqq\cHck^{\{*\}}_{G,\leq\mu}-\bigcup_{\mu'}\cHck^{\{*\}}_{G,\leq\mu'}
\end{align*}
yields an open substack of $\cHck^{\{*\}}_{G,\leq\mu}$, where $\mu'$ runs over elements of $X_*(T_\infty)^+$ such that $\mu'<\mu$. Write $\Gr^{\{*\}}_{G,\mu}$ for the pullback to $\Gr^{\{*\}}_G$ of $\cHck_{G,\mu}^{\{*\}}$, and note that $[\mu]$ factors through a morphism $[\mu]:\Div^1_X\ra\Gr^{\{*\}}_{G,\mu}$ over $\Div^1_X$. Write $(L^+_{\{*\}}G)_\mu$ for the stabilizer of $[\mu]$ in $L^+_{\{*\}}G$, which is a closed subspace by Proposition \ref{ss:affineSchubertvarieties}.i).
\begin{prop*}
The morphism $[\mu]:\Div^1_X\ra\cHck^{\{*\}}_{G,\mu}$ is a v-cover, so it induces an isomorphism $\Div_X^1\!/(L^+_{\{*\}}G)_\mu\ra^\sim\cHck^{\{*\}}_{G,\mu}$ over $\Div^1_X$.
\end{prop*}
\begin{proof}
This follows from the proof of \cite[Proposition VI.2.4]{FS21}.
\end{proof}

\subsection{}\label{ss:diagonalstratification}
As usual, local Hecke stacks satisfy the following fusion property. For all partitions $\mathtt{P}=\{J_1,\dotsc,J_s\}$ of $J$, write $(\Div^1_X)^{\mathtt{P}}$ for the subsheaf of $(\Div^1_X)^J$ whose $S$-points consist of $S\ra(\Div^1_X)^J$ such that, for all geometric points $\ov{s}$ of $S$ and $j\neq j'$ in $J$ lying in different $J_k$, we have $\ov{s}_j\neq\ov{s}_{j'}$. Lemma \ref{ss:Div1} implies that $(\Div^1_X)^{\mathtt{P}}\subseteq(\Div^1_X)^J$ is an open subspace. Note that we have a natural identification
\begin{align*}
\cHck^J_G|_{(\Div^1_X)^{\mathtt{P}}} = \Big(\prod_{k=1}^s\cHck^{J_k}_G\Big)\Big|_{(\Div^1_X)^{\mathtt{P}}},
\end{align*}
where $\prod$ denotes the product over $\ul{\bN\cup\{\infty\}}$. Under the above identification, when $G$ is split, the closed substack $\cHck^J_{G,\leq\mu_\bullet}|_{(\Div_X^1)^{\mathtt{P}}}\subseteq\cHck^J_G|_{(\Div_X^1)^{\mathtt{P}}}$ corresponds to 
\begin{align*}
\Big(\prod_{k=1}^s\cHck^{J_k}_{G,\leq(\mu_j)_{j\in J_k}}\Big)\Big|_{(\Div^1_X)^{\mathtt{P}}}\subseteq\Big(\prod_{k=1}^s\cHck^{J_k}_G\Big)\Big|_{(\Div^1_X)^{\mathtt{P}}}.
\end{align*}

\subsection{}
In this subsection, assume that $G$ is split. We need the following decomposition of $\Gr_T^J$. For all $i$ in $\bN\cup\{\infty\}$, write $\Sg_i:\abs{\Gr^J_{T_i}}\ra X_*(T_i)\cong X_*(T_\infty)$ for the continuous map as in \cite[(VI.3.1)]{FS21}. Now \ref{ss:Grassmannianspecialization} identifies the set $\abs{\Gr_T^J}$ with the disjoint union $\coprod_i\abs{\Gr_{T_i}^J}$, where $i$ runs over $\bN\cup\{\infty\}$, so the $\Sg_i$ induce a map $\Sg:\abs{\Gr_T^J}\ra X_*(T_\infty)$.
\begin{lem*}
The map $\Sg:\abs{\Gr_{T}^J}\ra X_*(T_\infty)$ is continuous.
\end{lem*}
For all $\nu$ in $X_*(T_\infty)$, write $\Gr^{J,\nu}_{T}$ for the preimage under $\Sg$ of $\nu$.
\begin{proof}
  As $\mu_\bullet$ runs over $X_*(T_\infty)^J$, the $[\mu_\bullet]$ induce a morphism
  \begin{align*}
    \ul{X_*(T_\infty)^J}\times(\Div_X^1)^J\ra\Gr_{T}^J
  \end{align*}
  over $(\Div_X^1)^J$ that is surjective on geometric points. Therefore \cite[Lemma 2.5]{Sch17} indicates that resulting continuous map $X_*(T_\infty)^J\times|(\Div^1_X)^J|\ra\abs{\Gr_{T}^J}$ is a quotient map. Checking on fibers shows that the continuous map
  \begin{align*}
    \xymatrix{X_*(T_\infty)^J\times|(\Div_X^1)^J|\ar[r]& X_*(T_\infty)^J\ar[r]^-\Sg& X_*(T_\infty)}
  \end{align*}
factors through a map $\abs{\Gr_{T}^J}\ra X_*(T_\infty)$ with the desired description on fibers.
\end{proof}

\subsection{}\label{ss:hyperbolicinput}
In this subsection, assume that $G$ is split. We turn to study the analogue of semi-infinite orbits in our setting. For all $\nu$ in $X_*(T_\infty)$, write $\Gr^{J,\nu}_{B}$ for the preimage of $\Gr^{J,\nu}_{T}$ under $\Gr_{B}^J\ra\Gr_{T}^J$. Choose a regular element $\la$ of $X_*(T_\infty)^+$, and consider the action of $\bG^{\an}_m$ on $\Gr_H^J$ given by composing the conjugation action of $L^+_JT$ with
\begin{align*}
\xymatrix{\bG^{\an}_m\ar[r]^-{[-]} & L^+_J\bG_m\ar[r]^-{L^+_J\la} & L^+_JT\ar[r] & L^+_JH.}
\end{align*}
\begin{prop*}\hfill
  \begin{enumerate}[i)]
  \item The morphism $\Gr_{T}^J\ra\Gr_{G}^J$ is a closed embedding, and it identifies $\Gr_{T}^J$ with the fixed point locus of $\bG_m^{\an}$.
  \item The morphism $\Gr_{B}^J\ra\Gr_{G}^J$ induces a bijection on geometric points. Its restriction to $\Gr^{J,\nu}_{B}$ is a locally closed embedding, and the image of $\coprod_{\nu'}\big|\!\Gr^{J,\nu'}_{B}\!\big|$ in $\abs{\Gr_{G}^J}$ is closed, where $\nu'$ runs over elements of $X_*(T_\infty)$ satisfying $\nu'\leq\nu$. Finally, the action of $\bG^{\an}_m$ on $\Gr_{B}^J$ naturally extends to an action of $(\bA^1)^{\an}$.
  \end{enumerate}
\end{prop*}
\begin{proof}
  For part i), the first statement follows from arguing as in the proof of \cite[Lemma 19.1.5]{SW20}. Write $(\Gr_{G}^J)^{\bG^{\an}_m}$ for the fixed point locus of $\bG^{\an}_m$. Now $\Gr_{T}^J\ra\Gr_{G}^J$ factors through a morphism $\Gr_{T}^J\ra(\Gr_{G}^J)^{\bG^{\an}_m}$, which is a closed embedding by the first statement. Note that this closed embedding is a bijection on geometric points, so the second statement follows from \cite[Lemma 12.11]{FS21}.

For part ii), the first statement follows from the Iwasawa decomposition. The action of $\bG_m$ on $B$ given by composing the conjugation action of $T$ with $\la$ naturally extends to an action of $\bA^1$, which implies the last statement. Finally, everything else follows from the proof of \cite[Proposition VI.3.1]{FS21}.
\end{proof}

\subsection{}\label{ss:hyperbolicoutput}
In this subsection, assume that $G$ is split. Proposition \ref{ss:hyperbolicinput} enables us to use hyperbolic localization as in \cite[Section IV.6]{FS21} to define the constant term functor as follows. Let $\La$ be a ring that is $\ell$-power torsion, and for all small v-stacks $Z$ over $(\Div_X^1)^J$, write $D_{\et}(\Gr_{G}^J\!|_Z,\La)^{\bd}$ for the full subcategory of $D_{\et}(\Gr_{G}^J\!|_Z,\La)$ consisting of objects arising via pushforward from a finite union of subspaces of the form $\Gr_{G,\leq\mu_\bullet}^J\!|_Z$.

Note that the action of $\bG_m^{\an}$ on $\Gr_{G}^J$ preserves the subspace $\Gr_{G,\leq\mu_\bullet}^J$. Write $D_{\et}(\Gr_{G}^J\!|_Z,\La)^{\bd,\bG_m^{\an}}$ for the full subcategory of $D_{\et}(\Gr_{G}^J\!|_Z)^{\bd}$ consisting of objects arising from $\bG^{\an}_m$-monodromic objects as in \cite[Definition IV.6.11]{FS21}. Write $\ov{B}$ for the opposite Borel in $G$, and write
\begin{align*}
  \xymatrix{\Gr_{G}^J & \ar[l]_-q\ar[d]^-p\Gr_{B}^J\\
  \Gr_{\ov{B}}^J\ar[u]_-{\ov{p}}\ar[r]^-{\ov{q}} & \Gr_{T}^J }
\end{align*}
for the natural morphisms over $(\Div_X^1)^J$.
\begin{cor*}
  We have a natural transformation $\ov{p}_*\ov{q}^!\ra p_!q^*$ of functors
  \begin{align*}
    D_{\et}(\Gr^J_{G}\!|_Z,\La)^{\bd}\ra D_{\et}(\Gr^J_{T}\!|_Z,\La)^{\bd}.
  \end{align*}
When restricted to $D_{\et}(\Gr^J_{G}\!|_Z,\La)^{\bd,\bG^{\an}_m}$, this natural transformation is an isomorphism, and the resulting functor
  \begin{align*}
    \CT_B:D_{\et}(\Gr^J_{G}\!|_Z,\La)^{\bd,\bG^{\an}_m}\ra D_{\et}(\Gr^J_{T}\!|_Z,\La)^{\bd}
  \end{align*}
is compatible with base change in $Z$ and preserves universal local acyclicity over $Z$ as in \cite[Definition IV.2.1]{FS21}.
\end{cor*}
\begin{proof}
  For all $\mu_\bullet$ in $(X_*(T_\infty)^+)^J$, Proposition \ref{ss:affineSchubertvarieties}.i) shows that $\Gr_{G,\leq\mu_\bullet}^J\ra(\Div_X^1)^J$ is representable in spatial diamonds, proper, and of finite $\dimtrg$. We claim that the action of $\bG_m^{\an}$ on $\Gr^J_{G,\leq\mu_\bullet}$ satisfies \cite[Hypothesis IV.6.1]{FS21}. Proposition \ref{ss:hyperbolicinput}.i) indicates that the fixed point locus of $\bG^{\an}_m$ in $\Gr^J_{G,\leq\mu_\bullet}$ is its intersection with $\Gr_{T}^J$, and this intersection equals $\coprod_\nu\Gr^J_{G,\leq\mu_\bullet}\cap\Gr^{\nu,J}_{T}$, where $\nu$ runs over elements of $X_*(T_\infty)$ such that the Weyl translate of $\nu$ in $X_*(T_\infty)^+$ is bounded by $\sum_{j\in J}\mu_j$. There are only finitely many such $\nu$, so applying Proposition \ref{ss:hyperbolicinput}.ii) to $B$ and $\ov{B}$ provides the desired decompositions
  \begin{align*}
    \Gr^J_{G,\leq\mu_\bullet} = \coprod_\nu\Gr^J_{G,\leq\mu_\bullet}\cap\Gr_{B}^{J,\nu}\mbox{ and }\Gr^J_{G,\leq\mu_\bullet} = \coprod_\nu\Gr^J_{G,\leq\mu_\bullet}\cap\Gr_{\ov{B}}^{J,\nu}.
  \end{align*}
This yields the claim.
  
As $\mu_\bullet$ varies, the claim and \cite[Definition IV.6.4]{FS21} yield the natural transformation $\ov{p}_*\ov{q}^!\ra p_!q^*$, and its restriction to $D_{\et}(\Gr_{G}^J\!|_Z,\La)^{\bd,\bG^{\an}_m}$ is an isomorphism by \cite[Theorem IV.6.5]{FS21}. The resulting functor $\CT_B$ is compatible with base change in $Z$ by \cite[Proposition IV.6.12]{FS21} and preserves universal local acyclicity over $Z$ by \cite[Proposition IV.6.14]{FS21}.
\end{proof}

\subsection{}\label{ss:conservative}
In this subsection, assume that $G$ is split. The constant term functor enjoys the following conservativity property. Write $D_{\et}(\cHck_{G}^J|_Z,\La)^{\bd}$ for the full subcategory of $D_{\et}(\cHck_{G}^J|_Z,\La)$ consisting of objects whose pullback to $\Gr_{G}^J\!|_Z$ lie in $D_{\et}(\Gr_{G}^J\!|_Z,\La)^{\bd}$, and note that the image of the pullback functor
\begin{align*}
 D_{\et}(\cHck_{G}^J|_Z,\La)^{\bd}\ra D_{\et}(\Gr_{G}^J\!|_Z,\La)^{\bd}
\end{align*}
lies in $D_{\et}(\Gr_{G}^J\!|_Z,\La)^{\bd,\bG^{\an}_m}$. Write $\CT_B$ for the composition
\begin{align*}
\xymatrix{D_{\et}(\cHck_{G}^J|_Z,\La)^{\bd}\ar[r]& D_{\et}(\Gr_{G}^J\!|_Z,\La)^{\bd,\bG^{\an}_m}\ar[r]^-{\CT_B} & D_{\et}(\Gr_{T}^J\!|_Z,\La)^{\bd}.}
\end{align*}
\begin{prop*}
Let $A$ be in $D_{\et}(\cHck_{G}^J|_Z,\La)^{\bd}$. If $\CT_B(A)$ is zero, then $A$ is zero.
\end{prop*}
\begin{proof}
To check that $A$ is zero, \cite[Proposition 14.3]{Sch17} implies that it suffices to check on geometric points of $\cHck^J_G|_Z$. Hence Corollary \ref{ss:hyperbolicoutput} indicates that we can assume that $Z$ is a geometric point. Then the image of $\abs{Z}$ in $\bN\cup\{\infty\}$ equals $\{i\}$ for some $i$ in $\bN\cup\{\infty\}$, so the result follows from \ref{ss:Grassmannianspecialization} and \cite[Proposition VI.4.2]{FS21}.
\end{proof}

\subsection{}\label{ss:splittingextension}
We will use the following to reduce many proofs to the case when $G$ is split. Recall that $G_F$ is split. Write $X_S(F)$ for the adic space associated with $F$ as in \ref{ss:Ycurve}, and write $\Div_{X(F)}^1$ for the v-sheaf over $\ul{\bN\cup\{\infty\}}$ associated with $F$ as in \ref{ss:Div1}. Note that we have a cartesian square
\begin{align*}
  \xymatrix{X_S(F)\ar[r]\ar[d] & X_S\ar[d]\\
  \Spa{F}\ar[r] & \Spa{E},}
\end{align*}
so the morphism $X_S(F)\ra X_S$ is finite \'etale. This induces identifications
\begin{align*}
\cHck^J_G|_{(\Div_{X(F)}^1)^J}=\cHck^J_{G_F}\mbox{ and }\Gr^J_G\!|_{(\Div_{X(F)}^1)^J}=\Gr^J_{G_F}
\end{align*}
as v-stacks over $(\Div_{X(F)}^1)^J$. Also, applying \cite[Lemma 15.6]{Sch17} to $\Spa\breve{F}\ra\Spa\breve{E}$ implies that $\Div^1_{X(F)}\ra\Div^1_X$ is finite \'etale.

We even use the above to reduce definitions to the case when $G$ is split:
\begin{defn*}\hfill
  \begin{enumerate}[a)]
  \item When $G$ is split, write $D_{\et}(\Gr_{G}^J\!|_Z,\La)^{\ULA}$ for the full subcategory of
    \begin{align*}
      D_{\et}(\Gr_{G}^J\!|_Z,\La)^{\bd}
    \end{align*}
    consisting of universally locally acyclic objects over $Z$ as in \cite[Definition IV.2.1]{FS21}, and write $D_{\et}(\cHck_{G}^J|_Z,\La)^{\ULA}$ for the full subcategory of
    \begin{align*}
      D_{\et}(\cHck_{G}^J|_Z,\La)^{\bd}
    \end{align*}
    consisting of objects whose pullback to $\Gr_{G}^J\!|_Z$ lie in $D_{\et}(\Gr_{G}^J\!|_Z,\La)^{\ULA}$.
\item Write $D_{\et}(\cHck^J_G,\La)^{\ULA}$ for the full subcategory of $D_{\et}(\cHck_G^J|_Z,\La)$ consisting of objects whose pullback to
  \begin{align*}
    \cHck_G^J|_{(\Div_{X(F)}^1)^J\times_{(\Div_{X}^1)^J}Z}=\cHck^J_{G_F}|_{(\Div_{X(F)}^1)^J\times_{(\Div_{X}^1)^J}Z}
  \end{align*}
lie in $D_{\et}(\cHck_{G_F}^J|_{(\Div_{X(F)}^1)^J\times_{(\Div_X^1)^J}Z},\La)^{\ULA}$ as in Definition \ref{ss:splittingextension}.a).
  \item Write $D_{\et}(\cHck_G^J|_Z,\La)^{\bd}$ and $D_{\et}(\Gr_G^J|_Z,\La)^{\bd}$ for the full subcategories of
    \begin{align*}
      D_{\et}(\cHck_G^J|_Z,\La)\mbox{ and }D_{\et}(\Gr_G^J|_Z,\La),
    \end{align*}
    respectively, consisting of objects whose pullback to 
    \begin{gather*}
      \cHck_G^J|_{(\Div_{X(F)}^1)^J\times_{(\Div_{X}^1)^J}Z}=\cHck^J_{G_F}|_{(\Div_{X(F)}^1)^J\times_{(\Div_{X}^1)^J}Z}\mbox{ and }\\
      \Gr_G^J|_{(\Div_{X(F)}^1)^J\times_{(\Div_{X}^1)^J}Z}=\cHck^J_{G_F}|_{(\Div_{X(F)}^1)^J\times_{(\Div_{X}^1)^J}Z},
    \end{gather*}
    respectively, lie in
    \begin{align*}
      D_{\et}(\cHck^J_{G_F}|_{(\Div_{X(F)}^1)^J\times_{(\Div_{X}^1)^J}Z},\La)^{\bd}\mbox{ and }D_{\et}(\Gr^J_{G_F}|_{(\Div_{X(F)}^1)^J\times_{(\Div_{X}^1)^J}Z},\La)^{\bd}
    \end{align*}
 as in \ref{ss:conservative} and \ref{ss:hyperbolicoutput}, respectively.
  \end{enumerate}
\end{defn*}

\subsection{}
Later, we will need the following fact. Write $\sw:\cHck^J_G\ra\cHck^J_G$ for the automorphism given by $(\sH,\sH',\rho)\mapsto(\sH',\sH,\rho^{-1})$. When $G$ is split, note that $\sw$ restricts to an isomorphism $\cHck^J_{G,\leq\mu_\bullet}\ra^\sim\cHck^J_{G,\leq-w_0(\mu_\bullet)}$ for all $\mu_\bullet$ in $(X_*(T_\infty)^+)^J$, where $w_0$ denotes the longest Weyl element. Consequently, for general $G$ the functor $\sw^*$ preserves $D_{\et}(\cHck^J_G|_Z,\La)^{\bd}$.
\begin{lem*}
Let $A$ be in $D_{\et}(\cHck^J_G|_Z,\La)^{\bd}$. Then $A$ lies in $D_{\et}(\cHck^J_G|_Z,\La)^{\ULA}$ if and only if $\sw^*A$ lies in $D_{\et}(\cHck^J_G|_Z,\La)^{\ULA}$.
\end{lem*}
\begin{proof}
  Definition \ref{ss:splittingextension}.b) indicates that we can assume that $G$ is split. For all $\mu_\bullet$ in $(X_*(T_\infty)^+)^J$, write $(L_JG)_{\leq\mu_\bullet}$ for the preimage in $L_JG$ of $\Gr^J_{G,\leq\mu_\bullet}$. Proposition \ref{ss:jetgroup} and \cite[Proposition IV.2.13 (ii)]{FS21} show that an object of $D_{\et}(\Gr^J_{G,\leq\mu_\bullet}\!|_Z,\La)$ is universally locally acyclic over $Z$ if and only if its pullback to $(L_JG)_{\leq\mu_\bullet}/(L^+_JG)^n$ is universally locally acyclic over $Z$. As $n$ varies, we have a natural isomorphism between the pro-systems $\{(L_JG)_{\leq\mu_\bullet}/(L^+_JG)^n\}_n$ and $\{(L^+_JG)^n\bs(L_JG)_{\leq\mu_\bullet}\}_n$, and universal local acyclicity over $Z$ on the two pro-systems is equivalent by Proposition \ref{ss:jetgroup} and \cite[Lemma VI.6.3]{FS21}. Hence the result follows from the commutative square
  \begin{align*}
    \xymatrix{(L_JG)_{\leq\mu_\bullet}/(L^+_JG)^n\ar[r]^-{(-)^{-1}}\ar[d] &(L^+_JG)^n\bs(L_JG)_{\leq-w_0(\mu_\bullet)}\ar[d] \\
      \cHck^J_{G\leq\mu_\bullet}\ar[r]^-{\sw} & \cHck^J_{G,\leq-w_0(\mu_\bullet)}}
  \end{align*}
and applying the above to $-w_0(\mu_\bullet)$.
\end{proof}

\subsection{}\label{ss:ULAcriterion}
Write $\pi_G:\Gr_{G}^J\ra(\Div_X^1)^J$ for the structure morphism, and for the rest of this subsection, assume that $G$ is split. Proposition \ref{ss:affineSchubertvarieties}.i) indicates that we have a functor $\pi_{G!}:D_{\et}(\Gr_G^J\!|_Z,\La)^{\bd}\ra D_{\et}(Z,\La)$ that agrees with $\pi_{G*}$.

We conclude this section by proving the following criterion for being universally locally acyclic in terms of the constant term functor. 
\begin{prop*}
  Let $A$ be in $D_{\et}(\cHck_{G}^J|_Z,\La)^{\bd}$. Then the following are equivalent:
  \begin{enumerate}[a)]
  \item $A$ lies in $D_{\et}(\cHck_{G}^J|_Z,\La)^{\ULA}$,
  \item $\CT_B(A)$ lies in $D_{\et}(\Gr_{T}^J\!|_Z,\La)^{\ULA}$,
  \item $\pi_{T!}\CT_B(A)$ lies $D_{\lc}(Z,\La)$.
  \end{enumerate}
\end{prop*}
\begin{proof}
Now $\pi_T$ is ind-finite by Proposition \ref{ss:affineSchubertvarieties}.iii), so \cite[Proposition IV.2.28]{FS21} and \cite[Proposition IV.2.9]{FS21} indicate that b)$\iff$c). Corollary \ref{ss:hyperbolicoutput} indicates that a)$\implies$b), so we focus on b)$\implies$a). For all $\mu_\bullet$ in $(X_*(T_\infty)^+)^J$, Proposition \ref{ss:affineSchubertvarieties}.ii) shows that the action of $L^+_JG$ on $\Gr_{G,\leq\mu_\bullet}$ factors through $L^+_JG\ra L^n_JG$, where $n$ is any positive integer greater than all weights of $\sum_{j\in J}\mu_j$ acting on $\Lie{G}$. Then \cite[Proposition VI.4.1]{FS21} and the proof of Proposition \ref{ss:jetgroup} imply that pullback induces an equivalence of categories
  \begin{align*}
    D_{\et}((L^n_JG)\bs\!\Gr_{G,\leq\mu_\bullet}^J\!|_Z,\La)\ra^\sim D_{\et}(\cHck_{G,\leq\mu_\bullet}^J|_Z,\La).
  \end{align*}
   By Proposition \ref{ss:jetgroup}, $(L^n_JG)\bs\!\Gr_{G,\leq\mu_\bullet}^J\!|_Z$ is an Artin v-stack over $Z$. Therefore we can apply \cite[Theorem IV.2.23]{FS21} to see that it suffices to check that the natural morphism
  \begin{align*}\label{eqn:ULA}
    \sHom_\La(A,\pi_G^!\La)\boxtimes_\La A\ra\sHom_\La(\pr_1^*A,\pr_2^!A)\tag{$\ddagger$}
  \end{align*}
  is an isomorphism. Note that $\cHck_{G\times_EG}^J|_Z=\cHck_{G}^J|_Z\times_Z\cHck_{G}^J|_Z$, and (\ref{eqn:ULA}) lies in $D_{\et}(\cHck_{G\times_EG}^J|_Z,\La)^{\bd}$. By applying Proposition \ref{ss:conservative} to the cone of (\ref{eqn:ULA}), it suffices to check that (\ref{eqn:ULA}) becomes an isomorphism after applying $\CT_{\ov{B}\times_EB}$. Because hyperbolic localization is compatible with exterior tensor products, \cite[Proposition IV.6.13]{FS21} shows that the left-hand side of (\ref{eqn:ULA}) becomes
  \begin{align*}
\CT_{\ov{B}}(\sHom_\La(A,\pi^!_G\La))\otimes_\La\CT_B(A)=\sHom_\La(\CT_B(A),\pi^!_T\La)\otimes_\La\CT_B(A),
  \end{align*}
while the right-hand side of (\ref{eqn:ULA}) becomes
  \begin{align*}
    &\quad(p_1\times\ov{p}_2)_*(q_1\times\ov{q}_2)^!\sHom(\pr_1^*A,\pr_2^!A)\\
    &=(p_1\times\ov{p}_2)_*\sHom((q_1\times\ov{q}_2)^*\pr_1^*A,(q_1\times\ov{q}_2)^!\pr_2^!A) & \mbox{by \cite[Proposition 23.3 (ii)]{Sch17}}\\
    &=(p_1\times\ov{p}_2)_*\sHom(\pr_1^*q_1^*A,\pr_2^!\ov{q}_2^!A) \\
    &=(p_1\times\id)_*(\id\times\ov{p}_2)_*\sHom((\id\times\ov{p}_2)^*\pr_1^*q_1^*A,\pr_2^!\ov{q}_2^!A) \\
    &=(p_1\times\id)_*\sHom(\pr_1^*q_1^*A,(\id\times\ov{p}_2)_*\pr_2^!\ov{q}_2^!A) & \mbox{by \cite[Corollary 17.9]{Sch17}}\\
    &=(p_1\times\id)_*\sHom(\pr_1^*q_1^*A,\pr_2^!\ov{p}_{2*}\ov{q}_2^!A) & \mbox{by \cite[Proposition 23.16 (i)]{Sch17}}\\
    &=(p_1\times\id)_*\sHom(\pr_1^*q_1^*A,(p_1\times\id)^!\pr_2^!\ov{p}_{2*}\ov{q}_2^!A)\\
    &=\sHom((p_1\times\id)_!\pr_1^*q_1^*A,\pr_2^!\ov{p}_{2*}\ov{q}_2^!A) & \mbox{by \cite[Proposition 23.3 (i)]{Sch17}}\\
    &=\sHom(\pr_1^*\CT_B(A),\pr_2^!\CT_B(A)).
  \end{align*}
Since $\CT_B(A)$ is universally locally acyclic over $Z$, applying \cite[Theorem IV.2.23]{FS21} again implies that this is indeed an isomorphism, as desired.
\end{proof}

\section{Geometric Satake}\label{s:Satake}
To construct geometric Hecke operators in our context, we need the analogue of geometric Satake over $E$. This is the goal of this section. We begin by defining the perverse $t$-structure on $D_{\et}(\cHck^J_G|_Z,\La)^{\bd}$ via characterizing its $\leq0$ part. To analyze its $\geq0$ part, we prove a characterization of the perverse $t$-structure in terms of (a shift of) the constant term functor $\CT_B$.

After defining the analogue of the Satake category in our context, we prove that it enjoys all of the structures needed to apply Tannakian reconstruction. However, computing the resulting Hopf algebra requires working over rings like
\begin{align*}\label{eqn:hartog}
  \Cont(\bN\cup\{\infty\},\bZ_\ell),\tag{$\Join$}
\end{align*}
which seem inaccessible to Tannakian identification results. Instead, we first prove the result with $\bQ_\ell$-coefficients by explicitly constructing objects and morphisms in the Satake category, using the fact that $\Rep\wh{G}_{\bQ_\ell}$ is semisimple, and then we deduce the result with $\bZ_\ell$-coefficients, using a version of Hartog's lemma for \eqref{eqn:hartog}. In both steps, we reduce to the situation considered in Fargues--Scholze \cite{FS21}.

\subsection{}\label{ss:perversetstructure}
Write $\prescript{p}{}{D}_{\et}^{\leq0}(\cHck^J_G|_Z,\La)^{\bd}$ for the full subcategory of $D_{\et}(\cHck^J_G|_Z,\La)^{\bd}$ consisting of objects $A$ such that, for all geometric points $\ov{s}$ of $Z$ with image $\{i\}$ in $\bN\cup\{\infty\}$ and under the identification $\cHck^J_G|_{\ov{s}} = \cHck^J_{G_i}|_{\ov{s}}$ from \ref{ss:Grassmannianspecialization}, the pullback to $\cHck^J_G|_{\ov{s}}$ of $A$ lies in the full subcategory
    \begin{align*}
      \prescript{p}{}{D}_{\et}^{\leq0}(\cHck_{G_i}^J|_{\ov{s}},\La)^{\bd}\subseteq D_{\et}(\cHck_{G_i}^J|_{\ov{s}},\La)^{\bd}
    \end{align*}
as in \cite[Definition/Proposition VI.7.1]{FS21}.
\begin{lem*}
There exists a unique $t$-structure on $D_{\et}(\cHck_G^J|_Z,\La)^{\bd}$ whose $\leq0$ part equals $\prescript{p}{}{D}_{\et}^{\leq0}(\cHck_G^J|_Z,\La)^{\bd}$.
\end{lem*}
Write $\prescript{p}{}{D}_{\et}^{\geq0}(\cHck_G^J|_Z,\La)^{\bd}\subseteq D_{\et}(\cHck_G^J|_Z,\La)^{\bd}$ for the $\geq0$ part of this $t$-structure, and write $\Perv(\cHck_G^J|_Z,\La)\subseteq D_{\et}(\cHck^J_G|_Z,\La)^{\bd}$ for the heart of this $t$-structure.
\begin{proof}
This follows from \cite[Proposition 17.3]{Sch17} and \cite[Proposition 1.4.4.11]{Lu17}.
\end{proof}

\subsection{}\label{ss:texactness}
When $G$ is split, write $\deg:\abs{\Gr_T^J}\ra\bZ$ for the composition
\begin{align*}
\xymatrix{\abs{\Gr_T^J}\ar[r]^-\Sg & X_*(T_\infty)\ar[r]^-{\ang{2\rho,-}} & \bZ.}
\end{align*}
The perverse $t$-structure and the constant term functor (after shifting by $\deg$) enjoy the following exactness properties.
\begin{prop*}\hfill
  \begin{enumerate}[i)]
  \item Let $A$ be in $D_{\et}(\cHck^J_G|_Z,\La)^{\bd}$, and assume that $G$ is split. Then $A$ lies in $\prescript{p}{}{D}_{\et}^{\leq0}(\cHck^J_G|_Z,\La)^{\bd}$ or $\prescript{p}{}{D}_{\et}^{\geq0}(\cHck^J_G|_Z,\La)^{\bd}$ if and only if $\CT_B(A)[\deg]$ lies in $D_{\et}^{\leq0}(\Gr^J_T\!|_Z,\La)^{\bd}$ or $D_{\et}^{\geq0}(\Gr^J_T\!|_Z,\La)^{\bd}$, respectively.
  \item For all small v-stacks $Z'$ over $Z$, the pullback functor
    \begin{align*}
      D_{\et}(\cHck_G^J|_Z,\La)^{\bd}\ra D_{\et}(\cHck_G^J|_{Z'},\La)^{\bd}
    \end{align*}
    is $t$-exact for the perverse $t$-structure.
  \end{enumerate}
\end{prop*}
\begin{proof}
  For part i), let $A$ be in $\prescript{p}{}{D}^{\leq0}_{\et}(\cHck^J_G|_Z,\La)^{\bd}$ or $\prescript{p}{}{D}^{\geq0}_{\et}(\cHck^J_G|_Z,\La)^{\bd}$. To check that $\CT_B(A)[\deg]$ lies in $D_{\et}^{\leq0}(\Gr^J_T\!|_Z,\La)^{\bd}$ or $D_{\et}^{\geq0}(\Gr^J_T\!|_Z,\La)^{\bd}$, respectively, \cite[Proposition 14.3]{Sch17} implies that it suffices to check on geometric points of $\Gr^J_T\!|_Z$. Hence Corollary \ref{ss:hyperbolicoutput} indicates that we can assume that $Z$ is a geometric point. Then the image of $\abs{Z}$ in $\bN\cup\{\infty\}$ equals $\{i\}$ for some $i$ in $\bN\cup\{\infty\}$, so the result follows from \ref{ss:Grassmannianspecialization} and \cite[Proposition VI.7.4]{FS21}.

  Conversely, assume that $\CT_B(A)[\deg]$ lies in $D_{\et}^{\leq0}(\Gr^J_T\!|_Z,\La)^{\bd}$ or $D_{\et}^{\geq0}(\Gr^J_T\!|_Z,\La)^{\bd}$. The above shows that applying $\CT_B(A)[\deg]$ to the exact triangles
  \begin{gather*}
    \xymatrix{\prescript{p}{}\tau^{\leq0}A\ar[r]& A\ar[r]& \prescript{p}{}\tau^{\geq1}A\ar[r]^-{+1} & }\\
    \xymatrix{\prescript{p}{}\tau^{\leq-1}A\ar[r]& A\ar[r]& \prescript{p}{}\tau^{\geq0}A\ar[r]^-{+1} & }
  \end{gather*}
  yields exact triangles
  \begin{gather*}
    \xymatrix{\tau^{\leq0}\CT_B(A)[\deg]\ar[r]&\CT_B(A)[\deg]\ar[r]& \CT_B(\prescript{p}{}\tau^{\geq1}A)[\deg]\ar[r]^-{+1} & }\\
    \xymatrix{\CT_B(\prescript{p}{}\tau^{\leq-1}A)[\deg]\ar[r]&\CT_B(A)[\deg]\ar[r]& \tau^{\geq0}\CT_B(A)[\deg]\ar[r]^-{+1} &, }
  \end{gather*}
respectively. Therefore $\CT_B(\prescript{p}{}\tau^{\geq1}A)[\deg]$ or $\CT_B(\prescript{p}{}\tau^{\leq-1}A)[\deg]$ is zero, respectively, so Proposition \ref{ss:conservative} indicates that $\prescript{p}{}\tau^{\geq1}A$ or $\prescript{p}{}\tau^{\leq-1}A$ is zero, respectively. This yields the desired result.

For part ii), right $t$-exactness holds by construction, so we focus on left $t$-exactness. Then \cite[1.3.4]{BBDG18} and descent \cite[Proposition 17.3]{Sch17} show that we can replace $Z$ with a v-cover, so \ref{ss:splittingextension} indicates that we can assume that $G$ is split. Finally, the result follows from part i) and Corollary \ref{ss:hyperbolicoutput}.
\end{proof}

\begin{cor}\label{cor:perversefiberwise}\hfill
  \begin{enumerate}[i)]
  \item Let $A$ be in $D_{\et}(\cHck^J_G|_Z,\La)^{\bd}$. Then $A$ lies in $\prescript{p}{}{D}_{\et}^{\geq0}(\cHck^J_G|_Z,\La)^{\bd}$ if and only if, for all geometric points $\ov{s}$ of $Z$, the pullback to $\cHck^J_G|_{\ov{s}}$ of $A$ lies in $\prescript{p}{}{D}_{\et}^{\geq0}(\cHck_G^J|_{\ov{s}},\La)^{\bd}$. Therefore, $A$ lies in $\Perv(\cHck^J_G|_Z,\La)$ if and only if, for all geometric points $\ov{s}$ of $Z$, the pullback to $\cHck^J_G|_{\ov{s}}$ of $A$ lies in $\Perv(\cHck^J_G|_{\ov{s}},\La)$.
  \item For all open substacks $U\subseteq Z$, the $*$-pushforward functor
    \begin{align*}
      D_{\et}(\cHck^J_G|_U,\La)^{\bd}\ra D_{\et}(\cHck^J_G|_Z,\La)^{\bd}
    \end{align*}
    is left $t$-exact for the perverse $t$-structure.
  \end{enumerate}
\end{cor}
\begin{proof}
  For part i), the forward direction follows from Proposition \ref{ss:texactness}.ii). In the other direction, \cite[1.3.4]{BBDG18}, and descent \cite[Proposition 17.3]{Sch17} show that we can replace $Z$ with a v-cover, so \ref{ss:splittingextension} indicates that we can assume that $G$ is split. Then the result follows from Proposition \ref{ss:texactness}.i), Corollary \ref{ss:hyperbolicoutput}, and \cite[Proposition 14.3]{Sch17}.

  For part ii), write $\{i\}$ for the image of $\ov{s}$ in $\bN\cup\{\infty\}$, and identify $\cHck^J_G|_{\ov{s}}$ with $\cHck^J_{G_i}|_{\ov{s}}$ as in \ref{ss:Grassmannianspecialization}. Now $\prescript{p}{}D_{\et}^{\geq0}(\cHck^J_G|_{\ov{s}},\La)^{\bd}$ equals the full subcategory
  \begin{align*}
 \prescript{p}{}D_{\et}^{\geq0}(\cHck^J_{G_i}|_{\ov{s}},\La)^{\bd}\subseteq D_{\et}(\cHck^J_{G_i}|_{\ov{s}},\La)^{\bd}
  \end{align*}
  as in \cite[Definition/Proposition VI.7.1]{FS21}, so the discussion on \cite[p.~221]{FS21} characterizes $\prescript{p}{}D_{\et}^{\geq0}(\cHck^J_G|_{\ov{s}},\La)^{\bd}$ in terms of certain $!$-pullbacks. Therefore the desired result follows from part i) and \cite[Proposition 23.16 (i)]{Sch17}.
\end{proof}

\subsection{}\label{ss:Satakecategory}
We now introduce the Satake category. Write $\Sat(\cHck^J_G|_Z,\La)$ for the full subcategory of $D_{\et}(\cHck_G^J|_Z,\La)^{\ULA}$ consisting of objects $A$ such that, for all $\La$-modules $M$, the tensor product $A\otimes_\La M$ lies in $\Perv(\cHck^J_G|_Z,\La)$. In particular, $\Sat(\cHck^J_G|_Z,\La)$ lies in $\Perv(\cHck^J_G|_Z,\La)$.

The Satake category can be described in terms of the constant term functor:
\begin{prop*}
Let $A$ be in $D_{\et}(\cHck^J_G|_Z,\La)^{\ULA}$, and assume that $G$ is split. Then $A$ lies in $\Sat(\cHck^J_G|_Z,\La)$ if and only if $\pi_{T!}\CT_B(A)[\deg]$ lies in $\LocSys(Z,\La)$.
\end{prop*}
\begin{proof}
  Proposition \ref{ss:ULAcriterion} shows that $\pi_{T!}\CT_B(A)[\deg]$ lies in $D_{\lc}(Z,\La)$, and $\pi_T$ is ind-finite by Proposition \ref{ss:affineSchubertvarieties}.iii). Therefore Proposition \ref{ss:texactness}.i) indicates that $A\otimes_\La M$ lying in $\Perv(\cHck^J_G|_Z,\La)$ for all $\La$-modules $M$ is equivalent to
  \begin{align*}
    \pi_{T!}\CT_B(A\otimes_\La M)[\deg] = (\pi_{T!}\CT_B(A)[\deg])\otimes_\La M
  \end{align*}
being concentrated in degree $0$ for all $\La$-modules $M$. The latter is equivalent to $\pi_{T!}\CT_B(A)[\deg]$ lying in $\LocSys(Z,\La)$ by \cite[Tag 0658]{stacks-project}.
\end{proof}

\subsection{}\label{ss:fiberfunctor}
Write $\pi_{G!}$ for the composition
\begin{align*}
\xymatrix{\Perv(\cHck^J_G|_Z,\La)\ar[r] & D_{\et}(\Gr^J_G\!|_Z,\La)^{\bd}\ar[r]^-{\pi_{G!}} & D_{\et}(Z,\La).}
\end{align*}
The (relative) total cohomology functor on $\Perv(\cHck^J_G|_Z,\La)$, which will be the fiber functor we use to apply Tannakian reconstruction, satisfies the following properties.
\begin{prop*}
  Let $A$ be in $\Perv(\cHck^J_G|_Z,\La)$.
  \begin{enumerate}[i)]
  \item When $G$ is split, we have an isomorphism
    \begin{align*}
      \bigoplus_{d\in\bZ}H^d(\pi_{G!}A) \cong \pi_{T!}\CT_B(A)[\deg].
    \end{align*}
  \item The object $F_Z(A)\coloneqq\bigoplus_{d\in\bZ}H^d(\pi_{G!}A)$ is concentrated in degree $0$, and the resulting functor $F_Z:\Perv(\cHck^J_G|_Z,\La)\ra\Shv_{\et}(Z,\La)$ is conservative, exact, and faithful.
  \item When $F_Z$ is restricted to $\Sat(\cHck^J_G|_Z,\La)$, its image lies in $\LocSys(Z,\La)$. Moreover, $\Sat(\cHck^J_G|_Z,\La)$ admits and $F_Z$ reflects coequalizers of $F_Z$-split pairs.
  \end{enumerate}
\end{prop*}
When $Z=(\Div^1_X)^J$, write $F^J$ for $F_{(\Div^1_X)^J}$.
\begin{proof}
  For part i), note that $\Gr^J_B=\coprod_\nu\Gr^{J,\nu}_B$, where $\nu$ runs over $X_*(T_\infty)$. Therefore $\CT_B(A)=\bigoplus_\nu p_{\nu!}q_\nu^*A$, where $p_\nu:\Gr^{J,\nu}_B\ra\Gr^J_T$ and $q_\nu:\Gr^{J,\nu}_B\ra\Gr^J_G$ denote the restrictions to $\Gr^{J,\nu}_B$ of $p$ and $q$, respectively. Proposition \ref{ss:texactness}.i) indicates that $p_{\nu!}q_\nu^*A$ is concentrated in degree $\ang{2\rho,\nu}$, and because $\pi_T$ is ind-finite by Proposition \ref{ss:affineSchubertvarieties}.iii), this implies that $\pi_{T!}p_{\nu!}q_\nu^*A$ is concentrated in degree $\ang{2\rho,\nu}$.

  Proposition \ref{ss:hyperbolicinput}.ii) yields a stratification of $\Gr^J_G$ by $q_\nu:\Gr^{J,\nu}_B\ra\Gr^J_G$. The excision exact triangles associated with this stratification yield a filtration of $A$ with graded pieces $q_{\nu!}q_\nu^*A$, so applying $\pi_{G!}$ yields a filtration of $\pi_{G!}A$ with graded pieces $\pi_{G!}q_{\nu!}q_\nu^*A=\pi_{T!}p_{\nu!}q_\nu^*A$. The proof of Proposition \ref{ss:hyperbolicoutput} shows that the image of $q_\nu$ lies in the clopen subspace $(\Gr^J_G)^{\ang{2\rho,\nu}}\subseteq\Gr^J_G$ from Proposition \ref{ss:affineSchubertvarieties}.iv), so we can restrict to $\nu$ such that the image of $\ang{2\rho,\nu}$ in $\bZ/2$ is fixed. Then the above implies that the spectral sequence associated with this filtration of $\pi_{G!}A$ degenerates, which yields the desired isomorphism
  \begin{align*}
    \bigoplus_{d\in\bZ}H^d(\pi_{G!}A) \cong \bigoplus_\nu H^{\ang{2\rho,\nu}}(\pi_{T!}p_{\nu!}q^*_\nu A) = \pi_{T!}\CT_B(A)[\deg].
  \end{align*}
  
  We turn to part ii) and part iii). Descent \cite[Proposition 17.3]{Sch17} and \cite[Proposition 22.19]{Sch17} imply that we can replace $Z$ with a v-cover, so \ref{ss:splittingextension} indicates that we can assume that $G$ is split. Then part i) shows that $F_Z(A)\cong\pi_{T!}\CT_B(A)[\deg]$. By Proposition \ref{ss:texactness}.i), the object $F_Z(A)$ lies in $\Shv_{\et}(Z,\La)$, and $F_Z$ is exact. Since $\pi_T$ is ind-finite by Proposition \ref{ss:affineSchubertvarieties}.iii), Proposition \ref{ss:conservative} shows that $F_Z$ is conservative, and by taking differences of morphisms, exactness and conservativity imply that $F_Z$ is faithful. This yields part ii).

  For part iii), Proposition \ref{ss:Satakecategory} indicates that $F_Z(A)$ lies in $\LocSys(Z,\La)$. Finally, by taking differences of morphisms, it suffices to show that, for all morphisms $f:A\ra A'$ in $\Sat(\cHck^J_G|_Z,\La)$ such that $\coker F_Z(f)$ is a direct summand of $F_Z(A)$, the object $\coker{f}$ in $\Perv(\cHck^J_G|_Z,\La)$ lies in $\Sat(\cHck^J_G|_Z,\La)$. This follows from Proposition \ref{ss:ULAcriterion} and Proposition \ref{ss:Satakecategory}.
\end{proof}

\subsection{}\label{ss:Verdierdual}
We can use Proposition \ref{ss:fiberfunctor} to prove that Verdier duality preserves the Satake category. More precisely, write $\bD:D_{\et}(\cHck^J_G|_Z,\La)^{\ULA}\ra D_{\et}(\cHck^J_G|_Z,\La)^{\ULA}$ for the Verdier dual as in \cite[IV.2.3.1]{FS21} on $\cHck^J_{G,\leq\mu_\bullet}|_Z$ relative to $Z/(L^+_JG)$, where $\mu_\bullet$ runs over $(X_*(T_\infty)^+)^J$.
\begin{prop*}
  Let $A$ be in $D_{\et}(\cHck^J_G|_Z,\La)^{\ULA}$.
  \begin{enumerate}[i)]
  \item The object $A$ lies in $\prescript{p}{}{D}^{\leq0}_{\et}(\cHck^J_G|_Z,\La)^{\bd}$ or $\prescript{p}{}{D}^{\geq0}_{\et}(\cHck^J_G|_Z,\La)^{\bd}$ if and only if $\bD(A)$ lies in $\prescript{p}{}{D}^{\geq0}_{\et}(\cHck^J_G|_Z,\La)^{\bd}$ or $\prescript{p}{}{D}^{\leq0}_{\et}(\cHck^J_G|_Z,\La)^{\bd}$, respectively.
  \item If $A$ lies in $\Sat(\cHck^J_G|_Z,\La)$, then $\bD(A)$ lies in $\Sat(\cHck^J_G|_Z,\La)$.
  \end{enumerate}
\end{prop*}
\begin{proof}
  For part i), we start with the first equivalence. Corollary \ref{cor:perversefiberwise}.i) and \cite[Proposition IV.2.15]{FS21} imply that we can assume that $Z$ is a geometric point. Then the result follows from \cite[Proposition 23.3 (ii)]{Sch17} and the proof of Corollary \ref{cor:perversefiberwise}.ii). As for the second equivalence, it follows from \cite[Corollary IV.2.25]{FS21} and the first equivalence.

  For part ii), Corollary \ref{cor:perversefiberwise}.i) implies that we can replace $Z$ with a v-cover, so \ref{ss:splittingextension} indicates that we can assume that $G$ is split. Part i) shows that $\bD(A)$ lies in $\Perv(\cHck^J_G|_Z,\La)$, so Proposition \ref{ss:fiberfunctor}.i) and \cite[Proposition 23.3 (i)]{Sch17} indicate that
  \begin{align*}
    \pi_{T!}\CT_B(\bD(A))[\deg] &\cong \bigoplus_{d\in\bZ}H^d(\pi_{G!}\bD(A)) = \bigoplus_{d\in\bZ}H^d(\pi_{G*}\bD(A)) \\
    &= \bigoplus_{d\in\bZ}H^d(\sHom_\La(\pi_{G!}A,\La)) =\sHom_\La\Big(\bigoplus_{d\in\bZ} H^d(\pi_{G!}A),\La\Big),
  \end{align*}
  since $\bigoplus_{d\in\bZ} H^d(\pi_{G!}A)$ lies in $\LocSys(Z,\La)$, by Proposition \ref{ss:fiberfunctor}.iii). Therefore
  \begin{align*}
    \pi_{T!}\CT_B(\bD(A))[\deg]
  \end{align*}
  also lies in $\LocSys(Z,\La)$, so the result follows from Proposition \ref{ss:Satakecategory}.
\end{proof}

\subsection{}\label{ss:standardobjects}
In this subsection, assume that $G$ is split. The following construction provides an important source of objects in the Satake category. Write $\j_\mu:\cHck^{\{*\}}_{G,\mu}\ra\cHck^{\{*\}}_G$ for the locally closed embedding from \ref{ss:affineSchubertcells}.
\begin{prop*}
  The objects $\prescript{p}{}H^0(\j_{\mu!}\La[d_\mu])$ and $\prescript{p}{}H^0(\j_{\mu*}\La[d_\mu])$ of $\Perv(\cHck^{\{*\}}_G|_Z,\La)$ lie in $D_{\et}(\cHck^{\{*\}}_G|_Z,\La)^{\ULA}$, and their values under $\CT_B[\deg]$ are locally constant with finite free fibers. Consequently, the same holds for their values under $F_Z$, and $\prescript{p}{}H^0(\j_{\mu!}\La[d_\mu])$ and $\prescript{p}{}H^0(\j_{\mu*}\La[d_\mu])$ lie in $\Sat(\cHck^J_G|_Z,\La)$.
\end{prop*}
\begin{proof}
The first sentence follows from the proof of \cite[Proposition VI.7.5]{FS21}. The second sentence follows from Proposition \ref{ss:fiberfunctor}.i) and Proposition \ref{ss:Satakecategory}.
\end{proof}

\subsection{}
We now give our first definition (via convolution) of the monoidal structure on the Satake category. Proposition \ref{prop:HeckeandGr} implies that we have a cartesian square
\begin{align*}
  \xymatrix{\cHck^J_G\ar[r]\ar[d] & (\Div^1_X)^J\!/(L^+_JG)\ar[d] \\
  (\Div^1_X)^J\!/(L^+_JG)\ar[r] & (\Div^1_X)^J\!/(L_JG).}
\end{align*}
Taking into account the discussion on \cite[p.~224]{FS21}, this identifies $D_{\et}(\cHck^J_G|_Z,\La)^{\bd}$ with the category of $1$-endomorphisms of the object $Z/(L^+_JG)$ in the $2$-category $\cC_{Z/(L_JG)}$ as in \cite[IV.2.3.3]{FS21}. Write
\begin{align*}
\star:D_{\et}(\cHck^J_G|_Z,\La)^{\bd}\times D_{\et}(\cHck^J_G|_Z,\La)^{\bd}\ra D_{\et}(\cHck^J_G|_Z,\La)^{\bd}
\end{align*}
for the monoidal structure arising from composition of $1$-morphisms in $\cC_{Z/(L_JG)}$.
\begin{prop*}Let $A$ and $A'$ be in $D_{\et}(\cHck^J_G|_Z,\La)^{\bd}$.
  \begin{enumerate}[i)]
  \item If $A$ and $A'$ lie in $D_{\et}(\cHck^J_G|_Z,\La)^{\ULA}$, then $A\star A'$ lies in $D_{\et}(\cHck^J_G|_Z,\La)^{\ULA}$.
  \item If $A$ and $A'$ both lie in $\prescript{p}{}{D}_{\et}^{\leq0}(\cHck^J_G|_Z,\La)^{\bd}$ or $\prescript{p}{}{D}_{\et}^{\geq0}(\cHck^J_G|_Z,\La)^{\bd}$, then $A\star A'$ lies in $\prescript{p}{}{D}_{\et}^{\leq0}(\cHck^J_G|_Z,\La)^{\bd}$ or $\prescript{p}{}{D}_{\et}^{\geq0}(\cHck^J_G|_Z,\La)^{\bd}$, respectively.
  \item If $A$ and $A'$ lie in $\Sat(\cHck^J_G|_Z,\La)$, then $A\star A'$ lies in $\Sat(\cHck^J_G|_Z,\La)$.
  \end{enumerate}
\end{prop*}
\begin{proof}
Proposition \ref{ss:affineSchubertvarieties}.i) lets us apply \cite[Proposition IV.2.26]{FS21} and \cite[Proposition IV.2.11]{FS21} to show part i). For part ii), Corollary \ref{cor:perversefiberwise}.i) and \cite[Proposition 22.19]{Sch17} imply that we can assume that $Z$ is a geometric point. Then the first statement follows from \ref{ss:Grassmannianspecialization} and \cite[Proposition VI.8.1 (ii)]{FS21}. By Proposition \ref{ss:Verdierdual}.i) and \cite[Theorem IV.2.23]{FS21}, the second statement follows from the first.

For part iii), part i) shows that $A\star A'$ lies in $D_{\et}(\cHck^J_G|_Z,\La)^{\ULA}$. Since
\begin{align*}
  (A\star A')\otimes_\La M = (A\otimes_\La M)\star(A'\otimes_\La M)
\end{align*}
for all $\La$-modules $M$, the result follows from part ii).
\end{proof}

\subsection{}\label{ss:dualizable}
The machinery of \cite[Section IV.2.3.3]{FS21} shows that the Satake category has duals with respect to convolution.
\begin{prop*}
Let $A$ be in $\Sat(\cHck^J_G|_Z,\La)$. Then $A$ is dualizable with respect to $\star$, and its dual is isomorphic to $\sw^*\bD(A)$.
\end{prop*}
\begin{proof}
Taking into account the discussion on \cite[p.~224]{FS21}, we see from \cite[Proposition IV.2.24]{FS21} that the dual of $A$ in $D_{\et}(\cHck^J_G|_Z,\La)^{\bd}$ is isomorphic to $\sw^*\bD(A)$. Proposition \ref{ss:Verdierdual}.ii) shows that $\bD$ preserves $\Sat(\cHck^J_G|_Z,\La)$. Finally, Corollary \ref{cor:perversefiberwise}.i) and the proof of Corollary \ref{cor:perversefiberwise}.ii) indicate that $\sw^*$ preserves $\Sat(\cHck^J_G|_Z,\La)$.
\end{proof}

\subsection{}\label{ss:fusion}
Next, we give our second definition (via fusion) of the monoidal structure on the Satake category. Let $\mathtt{P}=\{J_1,\dotsc,J_s\}$ be an ordered partition of $J$. Under the identification from \ref{ss:diagonalstratification}, write
\begin{align*}
  \ast:\prod_{k=1}^sD_{\et}(\cHck^{J_k}_G,\La)^{\bd}\ra D_{\et}(\cHck^J_G|_{(\Div^1_X)^{\mathtt{P}}},\La)^{\bd}
\end{align*}
for the functor that sends $(A_1,\dotsc,A_s)\mapsto(A_1\boxtimes\dotsb\boxtimes A_s)|_{(\Div_X^1)^{\mathtt{P}}}$. By the K\"unneth formula and \cite[Proposition 22.19]{Sch17}, $\ast$ is monoidal with respect to $\prod_{k=1}^s\star$ on the source and $\star$ on the target.

Because the identification from \ref{ss:diagonalstratification} is independent of the ordering of $\mathtt{P}$, we have a natural isomorphism $A_1\ast\dotsb\ast A_s\cong A_{\sg(1)}\ast\dotsb\ast A_{\sg(s)}$ for all $\sg$ in the symmetric group $\fS_s$. Using Proposition \ref{ss:affineSchubertvarieties}.iv), replace these natural isomorphisms with their modified version as in \cite[p.~228]{FS21}.
\begin{lem}
  If $A_k$ lies in $\Sat(\cHck^{J_k}_G,\La)$ for all $1\leq k\leq s$, then $A_1\ast\dotsb\ast A_s$ lies in $\Sat(\cHck^J_G|_{(\Div^1_X)^{\mathtt{P}}},\La)$.
\end{lem}
\begin{proof}
  The K\"unneth formula indicates that $A_1\ast\dotsb\ast A_s$ lies in
  \begin{align*}
    D_{\et}(\cHck^J_G|_{(\Div^1_X)^{\mathtt{P}}},\La)^{\ULA}.
  \end{align*}
Next, we have $(A_1\ast\dotsb\ast A_s)\otimes_\La M = (A_1\otimes_\La M)\ast\dotsb\ast(A_s\otimes_\La M)$ for all $\La$-modules $M$, so it suffices to show that $A_1\ast\dotsb\ast A_s$ lies in $\Perv(\cHck^J_G|_{(\Div^1_X)^{\mathtt{P}}},\La)$. Corollary \ref{cor:perversefiberwise}.i) implies that we can replace $(\Div_X^1)^{\mathtt{P}}$ with a v-cover, so \ref{ss:splittingextension} indicates that we can assume that $G$ is split. Then Proposition \ref{ss:texactness}.i) shows that the desired condition is equivalent to $\CT_B(A_1\ast\dotsb\ast A_s)[\deg]$ being concentrated in degree $0$. Since hyperbolic localization is compatible with exterior tensor products and $\deg$ is compatible with \ref{ss:diagonalstratification}, Corollary \ref{ss:hyperbolicoutput} indicates that the above is isomorphic to
  \begin{align*}
    \big(\CT_B(A_1)[\deg]\boxtimes\dotsb\boxtimes\CT_B(A_s)[\deg]\big)\big|_{(\Div^1_X)^{\mathtt{P}}},
  \end{align*}
so applying Proposition \ref{ss:texactness}.i) again $s$ times yields the desired result.
\end{proof}

\subsection{}
The Satake category enjoys the following functoriality with respect to surjective maps $\ze:J\ra J'$ of finite sets. Note that we have a natural identification $\cHck^J_G\times_{(\Div_X^1)^J}(\Div_X^1)^{J'} = \cHck^{J'}_G$ and hence a natural morphism $\ze:\cHck^{J'}_G\ra\cHck^J_G$ over the morphism $\ze:(\Div_X^1)^{J'}\ra(\Div_X^1)^J$. The resulting pullback
\begin{align*}
\ze^*:D_{\et}(\cHck^J_G,\La)^{\bd}\ra D_{\et}(\cHck^{J'}_G,\La)^{\bd}
\end{align*}
sends $\Sat(\cHck^J_G,\La)$ to $\Sat(\cHck^{J'}_G,\La)$ by Proposition \ref{ss:texactness}.ii), and the square
\begin{align*}
  \xymatrix{\Sat(\cHck^J_G,\La)\ar[d]^-{F^J}\ar[r]^-{\ze^*} & \Sat(\cHck^{J'}_G,\La)\ar[d]^-{F^{J'}}\\
  \LocSys((\Div^1_X)^J,\La)\ar[r]^-{\ze^*} & \LocSys((\Div^1_X)^{J'},\La)}
\end{align*}
commutes by \cite[Proposition 22.19]{Sch17}.

\subsection{}\label{ss:restrictionfullyfaithful}
We need the following version of ``analytic continuation'' for the Satake category. Write $\j_{\mathtt{P}}:(\Div^1_X)^{\mathtt{P}}\ra(\Div^1_X)^J$ for the open embedding from \ref{ss:diagonalstratification}, and write
\begin{align*}
\i_{\mathtt{P}}:\big[(\Div^1_X)^J-(\Div^1_X)^{\mathtt{P}}\big]\ra(\Div^1_X)^J
\end{align*}
for the complementary closed embedding.
\begin{lem*}
For all $A$ in $\Sat(\cHck^J_G,\La)$ and $B$ in $\LocSys((\Div_X^1)^J,\La)$, the morphisms
    \begin{align*}
      A\ra\prescript{p}{}{H}^0(\j_{\mathtt{P}*}\j^*_{\mathtt{P}}A)\mbox{ and }B\ra H^0(\j_{\mathtt{P}*}\j^*_{\mathtt{P}}B)
    \end{align*}
    are isomorphisms. Consequently, the pullback functors
    \begin{gather*}
      \Sat(\cHck^J_G,\La)\ra\Sat(\cHck^J_G|_{(\Div^1_X)^{\mathtt{P}}},\La)\\
      \LocSys((\Div_X^1)^J,\La)\ra\LocSys((\Div^1_X)^{\mathtt{P}},\La)
    \end{gather*}
    are fully faithful.
  \end{lem*}
  \begin{proof}
If $\#J=1$, then $(\Div^1_X)^{\mathtt{P}}=(\Div^1_X)^J$, and there is nothing to prove. Hence assume that $\#J\geq2$. For the statement about $B$, the excision exact triangle
\begin{align*}
  \xymatrix{\i_{\mathtt{P}*}\i_{\mathtt{P}}^!B\ar[r] & B\ar[r] & \j_{\mathtt{P}*}\j^*_{\mathtt{P}}B\ar[r]^-{+1} &}
\end{align*}
implies that it suffices to prove that $\i_{\mathtt{P}*}\i_{\mathtt{P}}^!B$ lies in $D_{\et}^{\geq2}((\Div_X^1)^J,\La)$. Now Corollary \ref{ss:Div1smooth} shows that $(\Div_X^1)^J$ is cohomologically smooth of dimension $\#J\geq2$ over $\ul{\bN\cup\{\infty\}}$, and it also shows that $(\Div^1_X)^J-(\Div^1_X)^{\mathtt{P}}$ has a stratification by locally closed subsheaves that are cohomologically smooth of dimension $1$ over $\ul{\bN\cup\{\infty\}}$. Because $B$ lies in $\LocSys((\Div^1_X)^J,\La)$, using \cite[Proposition 23.16 (iii)]{Sch17} to work \'etale-locally yields the desired result.

We turn to the statement about $A$. The excision exact triangle
\begin{align*}
  \xymatrix{\i_{\mathtt{P}*}\i_{\mathtt{P}}^!A\ar[r] & A\ar[r] & \j_{\mathtt{P}*}\j^*_{\mathtt{P}}A\ar[r]^-{+1} &}
\end{align*}
implies that it suffices to prove that $\i_{\mathtt{P}*}\i_{\mathtt{P}}^!A$ lies in $\prescript{p}{}D_{\et}^{\geq2}(\cHck^J_G,\La)^{\bd}$. Since
\begin{align*}
\Div^1_{X(F)}\ra\Div^1_X
\end{align*}
is an \'etale cover, \cite[Proposition 23.16 (iii)]{Sch17} and \ref{ss:splittingextension} imply that we can assume that $G$ is split. Then Proposition \ref{ss:texactness}.i) shows that the desired condition is equivalent to $\CT_B(\i_{\mathtt{P}*}\i_{\mathtt{P}}^!A)[\deg]$ lying in $D_{\et}^{\geq2}((\Div_X^1)^J,\La)$, and Corollary \ref{ss:hyperbolicoutput} yields
\begin{align*}
\CT_B(\i_{\mathtt{P}*}\i_{\mathtt{P}}^!A)[\deg] =  \i_{\mathtt{P}*}\i_{\mathtt{P}}^!\CT_B(A)[\deg].
\end{align*}
Proposition \ref{ss:Satakecategory} indicates that $\CT_B(A)[\deg]$ lies in $\LocSys((\Div_X^1)^J,\La)$, so the desired result follows from the above work.
\end{proof}

\subsection{}\label{ss:fusionconvolution}
At this point, we can prove that fusion yields a symmetric monoidal structure on the Satake category, as well as that fusion agrees with convolution.
\begin{prop*}\hfill
  \begin{enumerate}[i)]
  \item The image of
    \begin{align*}
      \ast:\prod_{k=1}^s\Sat(\cHck^{J_k}_G,\La)\ra\Sat(\cHck^J_G|_{(\Div^1_X)^{\mathtt{P}}},\La)
    \end{align*}
    lies in the image of $\Sat(\cHck^J_G,\La)$, and the square
    \begin{align*}
      \xymatrix{\prod_{k=1}^s\Sat(\cHck^{J_k}_G,\La)\ar[r]^-\ast\ar[d]^-{\prod_{k=1}^sF^{J_k}} & \Sat(\cHck^J_G,\La)\ar[d]^-{F^J} \\
     \prod_{k=1}^s\LocSys((\Div^1_X)^{J_k},\La)\ar[r]^-\boxtimes & \LocSys((\Div^1_X)^J,\La)}
    \end{align*}
    commutes. Moreover, this is functorial in refinements and permutations of $\mathtt{P}$.
  \item The functor $\star$ is naturally isomorphic to the composition
    \begin{align*}
      \xymatrix{\Sat(\cHck^J_G,\La)\times\Sat(\cHck^J_G,\La)\ar[r]^-\ast & \Sat(\cHck^{J\coprod J}_G,\La)\ar[r]^-{\ze^*} & \Sat(\cHck^J_G,\La),}
    \end{align*}
    where $\ze:J\coprod J\ra J$ denotes the natural map. Consequently, $\star$ is naturally a symmetric monoidal structure, and $F^J$ is symmetric monoidal with respect to $\star$ on the source and $\otimes_\La$ on the target.
  \end{enumerate}
\end{prop*}
\begin{proof}
We start with part i). The first statement follows from the proof of \cite[Definition/Proposition VI.9.4]{FS21}. As for the second statement, Lemma \ref{ss:restrictionfullyfaithful} shows that it suffices to check that the square
    \begin{align*}
      \xymatrix{\prod_{k=1}^s\Sat(\cHck^{J_k}_G,\La)\ar[r]^-\ast\ar[d]^-{\prod_{k=1}^sF^{J_k}} & \Sat(\cHck^J_G|_{(\Div^1_X)^{\mathtt{P}}},\La)\ar[d]^-{F_{(\Div^1_X)^{\mathtt{P}}}} \\
     \prod_{k=1}^s\LocSys((\Div^1_X)^{J_k},\La)\ar[r]^-\boxtimes & \LocSys((\Div^1_X)^{\mathtt{P}},\La)}
    \end{align*}
commutes, and this follows from the K\"unneth formula. Finally, the last statement follows from \ref{ss:fusion}.

For part ii), part i) and \ref{ss:fusion} imply that the composition yields a symmetric monoidal structure that commutes with $\star$. Therefore the Eckmann--Hilton argument shows that it is naturally isomorphic to $\star$, so the last statement follows from part i).
\end{proof}

\subsection{}\label{ss:Drinfeldapplied}
Let us explicate the target of our fiber functor $F^J$. Recall from \ref{ss:Waction} the group topological space $\bW$ over $\bN\cup\{\infty\}$. Because $\bW$ is lctd over $\bN\cup\{\infty\}$ as in Definition \ref{ss:grouphypotheses}, so is the $J$-fold fiber power $\bW^J$ of $\bW$ over $\bN\cup\{\infty\}$. Recall from Definition \ref{ss:smoothreps} the category $\Rep(\bW^J,\La)$, and write $\Rep(\bW^J,\La)^{\finproj}$ for the full subcategory of $\Rep(\bW^J,\La)$ given by objects whose underlying sheaf of $\La$-modules on $\bN\cup\{\infty\}$ is locally constant with finite projective fibers.

\textbf{For the rest of this paper, we work over $\Spd\ov\bF_q$.}
\begin{lem*}
The category $\Shv_{\et}(\Div_X^1,\La)$ is naturally equivalent to $\Rep(\bW,\La)$, and the category $\LocSys((\Div_X^1)^J,\La)$ is naturally equivalent to $\Rep(\bW^J,\La)^{\finproj}$.
\end{lem*}
\begin{proof}
The second statement follows from Proposition \ref{ss:Drinfeldslemma} and Lemma \ref{ss:equivariantsheaves}. As for the first statement, Proposition \ref{ss:Cperfectoid} identifies $\Spd{C}$ with $\ul{\bN\cup\{\infty\}}\times\Spa C_\infty$, so the pullback functor $D_{\et}(\ul{\bN\cup\{\infty\}},\La)\ra D_{\et}(\Spd C,\La)$ is an equivalence. Hence the desired result follows from Proposition \ref{ss:Drinfeldfullyfaithful} and Lemma \ref{ss:equivariantsheaves}.
\end{proof}

\subsection{}\label{ss:repsaction}
When applying Tannakian reconstruction to the Satake category, we consider the latter as an enriched category in the following way. Using the identification from Lemma \ref{ss:Drinfeldapplied}, we obtain an action
\begin{align*}
\otimes_\La:\Rep(\bW^J,\La)^{\finproj}\times D_{\et}(\cHck^J_G,\La)^{\bd}\ra D_{\et}(\cHck^J_G,\La)^{\bd}
\end{align*}
with respect to $\otimes_\La$ on $\Rep(\bW^J,\La)^{\finproj}$, given by $(B,A)\mapsto \pi_G^*B\otimes_\La A$. By the adjoint functor theorem, this induces a $\Rep(\bW^J,\La)^{\finproj}$-enriched structure on $D_{\et}(\cHck^J_G,\La)^{\bd}$.
\begin{lem*}
The action $\otimes_\La$ preserves $\Perv(\cHck^J_G,\La)$ and $\Sat(\cHck^J_G,\La)$. Moreover, $F^J$ is $\Rep(\bW^J,\La)^{\finproj}$-linear.
\end{lem*}
\begin{proof}
  Let $B$ be in $\Rep(\bW^J,\La)^{\finproj}$, and let $A$ be in $\Perv(\cHck^J_G,\La)$. For all geometric points $\ov{s}$ of $(\Div_X^1)^J$, the pullback $(\pi^*_GB)|_{\ov{s}}$ is a finite free $\La$-module concentrated in degree $0$, so Corollary \ref{cor:perversefiberwise}.i) implies that $\pi^*_GB\otimes_\La A$ lies in $\Perv(\cHck^J_G,\La)$. Next, the projection formula \cite[Proposition 22.23]{Sch17} shows that $F^J$ is $\Rep(\bW^J,\La)^{\finproj}$-linear. Finally, suppose that $A$ lies in $\Sat(\cHck^J_G,\La)$. Since
  \begin{align*}
    (\pi^*_GB\otimes_\La A)\otimes_\La M=\pi^*_GB\otimes_\La(A\otimes_\La M)
  \end{align*}
  for all $\La$-modules $M$, the above indicates that it suffices to check that $\pi^*_GB\otimes_\La A$ lies in $D_{\et}(\cHck^J_G,\La)^{\ULA}$. Using \cite[Proposition IV.2.13 (ii)]{FS21} and \ref{ss:splittingextension}, we can assume that $G$ is split. Then combining Proposition \ref{ss:fiberfunctor}.i) with the above shows that
  \begin{align*}
    \CT_B(\pi^*_GB\otimes_\La A)[\deg] &\cong F^J(\pi^*_GB\otimes_\La A) = B\otimes_\La F^J(A) \cong B\otimes_\La\CT_B(A)[\deg].
  \end{align*}
Proposition \ref{ss:ULAcriterion} indicates that $\CT_B(A)$ lies in $D_{\lc}((\Div^1_X)^J,\La)$. Since $B$ also lies in $D_{\lc}((\Div^1_X)^J,\La)$, applying Proposition \ref{ss:ULAcriterion} again yields the desired result.
\end{proof}

\subsection{}\label{ss:leftadjointSatake}
We now prove the co-representability results needed to apply Tannakian reconstruction. For all $j$ in $J$, let $\Om_j$ be a finite $\Ga_\infty$-stable downwards-closed subset of $X_*(T_\infty)^+$. Proposition \ref{ss:affineSchubertvarieties}.i) shows that the substack $\bigcup_{\mu_\bullet}\cHck^J_{G_F,\leq\mu_\bullet}\subseteq\cHck^J_{G_F}$, where $\mu_\bullet$ runs over $\prod_{j\in J}\Om_j$, is closed, and it descends to a closed substack $\cHck^J_{G,\Om}\subseteq\cHck^J_G$. Write $\Perv(\cHck^J_{G,\Om}|_Z,\La)$ and $\Sat(\cHck^J_{G,\Om}|_Z,\La)$ for the full subcategories of $\Perv(\cHck^J_G|_Z,\La)$ and $\Sat(\cHck^J_G|_Z,\La)$, respectively, consisting of objects that are supported on $\cHck^J_{G,\Om}|_Z$.

For the rest of this subsection, assume that $\#J=1$. Under the identification from Lemma \ref{ss:Drinfeldapplied}, write $L_\Om:\Rep(\bW,\La)\ra\Perv(\cHck^{\{*\}}_{G,\Om},\La)$ for the left adjoint of $F^{\{*\}}:\Perv(\cHck^{\{*\}}_{G,\Om},\La)\ra\Rep(\bW,\La)$, which exists by the adjoint functor theorem.
\begin{prop*}
The object $L_\Om(\ul\La)$ lies in $\Sat(\cHck^{\{*\}}_{G,\Om},\La)$.
\end{prop*}
\begin{proof}
  Corollary \ref{cor:perversefiberwise}.i) and \cite[Proposition IV.2.13 (ii)]{FS21} imply that we can replace $\Div^1_X$ with an \'etale cover, so \ref{ss:splittingextension} indicates that we can assume that $G$ is split. By Proposition \ref{ss:ULAcriterion} and Proposition \ref{ss:Satakecategory}, it suffices to show that $\CT_B(L_\Om(\ul\La))[\deg]$ lies in $\Rep(\bW,\La)^{\finproj}$. Because this is checked on underlying sheaves of $\La$-modules on $\bN\cup\{\infty\}$, Corollary \ref{ss:hyperbolicoutput} indicates that we can work over $Z=\Spd{C}$. Moreover, Proposition \ref{ss:fiberfunctor}.i) lets us replace $\CT_B(L_\Om(\ul\La))[\deg]$ with $F_Z(L_\Om(\ul\La))$, and it suffices to consider $\La=\bZ/\ell^m$.

  For all $\mu$ in $X_*(T_\infty)$, we claim that $\j_\mu^*L_\Om(\ul\La)$ in $D_{\et}(\cHck^{\{*\}}_{G,\mu}|_Z,\La)$ is a constant finite free $\La$-module concentrated in degree $-d_\mu$. Since $\La=\bZ/\ell^m$, Proposition \ref{ss:affineSchubertcells} implies that it suffices to show that $\Hom(\j_\mu^*L_\Om(\ul\La),\La[d_\mu])$ is a constant finite free $\La$-module. By adjunction and Corollary \ref{cor:perversefiberwise}.ii), we have
  \begin{align*}
    \Hom(\j_\mu^*L_\Om(\ul\La),\La[d_\mu]) &= \Hom(L_\Om(\ul\La),\j_{\mu*}\La[d_\mu]) = \Hom(L_\Om(\ul\La),\prescript{p}{}H^0(\j_{\mu*}\La[d_\mu]))\\
    &= \Hom(\ul\La,F_Z(\prescript{p}{}H^0(\j_{\mu*}\La[d_\mu]))) = F_Z(\prescript{p}{}H^0(\j_{\mu*}\La[d_\mu])),
  \end{align*}
so the claim follows from Proposition \ref{ss:standardobjects}.
  
Let $\mu$ be a maximal element of $\Om$. The excision exact triangle associated with
\begin{align*}
\cHck^{\{*\}}_{G,\mu}\hra \cHck^{\{*\}}_{G,\Om} \hla \cHck^{\{*\}}_{G,\Om-\{\mu\}}
\end{align*}
and $L_\Om(\ul\La)$ yields the perverse cohomology long exact sequence
\begin{align*}
\xymatrix{0\ar[r] & K\ar[r] & \prescript{p}{}H^0(\j_{\mu!}\j_\mu^*L_\Om(\ul\La))\ar[r] & L_\Om(\ul\La)\ar[r] & L_{\Om-\{\mu\}}(\ul\La)\ar[r] &0}
\end{align*}
for some $K$ in $\Perv(\cHck^{\{*\}}_{G,\Om}|_Z,\La)$. The claim and Proposition \ref{ss:standardobjects} indicate that $F_Z(\prescript{p}{}H^0(\j_{\mu!}\j_\mu^*L_\Om(\ul\La)))$ is locally constant with finite free fibers, and by induction on $\#\Om$, we can assume that the same holds for $F_Z(L_{\Om-\{\mu\}}(\ul\La))$. Hence Proposition \ref{ss:fiberfunctor}.ii) implies that it suffices to show that $K$ is zero.

Now \cite[Proposition 14.3]{Sch17} shows that it suffices to check on geometric points $\ov{s}$ of $\cHck^{\{*\}}_{G,\Om}|_Z$. The image of $\ov{s}$ in $\bN\cup\{\infty\}$ equals $\{i\}$ for some $i$ in $\bN\cup\{\infty\}$, and under the identification $\cHck^{\{*\}}_G|_{\ul{\{i\}}}=\cHck^{\{*\}}_{G_i}$ from \ref{ss:Grassmannianspecialization}, combining Proposition \ref{ss:fiberfunctor}.i) with Corollary \ref{ss:hyperbolicoutput} implies that $L_\Om$ is compatible with pulling back to $\cHck^{\{*\}}_{G_i}$. Therefore the desired result follows from the proof of \cite[Proposition VI.10.1]{FS21}.
\end{proof}

\subsection{}\label{cor:fiberfunctorrep}
For all $j$ in $J$, Proposition \ref{ss:leftadjointSatake} indicates that $L_{\Om_j}(\ul\La)$ lies in $\Sat(\cHck^{\{j\}}_{G,\Om_j},\La)$. Hence we can form $\bigast_{j\in J}L_{\Om_j}(\ul\La)$ in $\Sat(\cHck^J_{G,\Om},\La)$. 
\begin{cor*}
  Under the identification from Lemma \ref{ss:Drinfeldapplied}, the functor
  \begin{align*}
    F^J:\Sat(\cHck^J_{G,\Om},\La)\ra\Rep(\bW^J,\La)^{\finproj}
  \end{align*}
  is co-represented by $\bigast_{j\in J}L_{\Om_j}(\ul\La)$.
\end{cor*}
\begin{proof}
We have the unit morphism $\ul\La\ra F^{\{j\}}(L_{\Om_j}(\ul\La))$, so taking $\bigboxtimes_{j\in J}$ and applying Proposition \ref{ss:fusionconvolution}.i) yields an element of
  \begin{align*}
    \Hom(\ul\La,F^J(\bigast_{j\in J}L_{\Om_j}(\ul\La))) = \Hom(L_\Om(\ul\La),\bigast_{j\in J}L_{\Om_j}(\ul\La)).
  \end{align*}
  Write $\mathtt{P}$ for the partition of $J$ into singletons. We claim that the resulting morphism
  \begin{align*}
    \j_{\mathtt{P}}^*L_{\Om}(\ul\La)\ra\j^*_{\mathtt{P}}(\bigast_{j\in J}L_{\Om_j}(\ul\La))
  \end{align*}
  is an isomorphism. To see this, \cite[Proposition 14.3]{Sch17} implies that it suffices to check on geometric points $\ov{s}$ of $\cHck^J_G|_{(\Div^1_X)^{\mathtt{P}}}$, so \ref{ss:splittingextension} indicates that we can assume that $G$ is split. Now the image of $\ov{s}$ in $\bN\cup\{\infty\}$ equals $\{i\}$ for some $i$ in $\bN\cup\{\infty\}$, and under the identification $\cHck^J_G|_{\ul{\{i\}}}=\cHck^J_{G_i}$ from \ref{ss:Grassmannianspecialization}, the proof of Proposition \ref{ss:leftadjointSatake} shows that $L_\Om$ is compatible with pulling back to $\cHck^J_{G_i}$. Therefore the claim follows from the proof of \cite[Proposition VI.10.1]{FS21}.

Finally, for all $A$ in $\Sat(\cHck^J_{G,\Om},\La)$ we have
  \begin{align*}
    F^J(A) &= \Hom(\ul\La,F^J(A)) = \Hom(L_\Om(\ul\La),A) \\
           &=\Hom(L_\Om(\ul\La),\prescript{p}{}H^0(\j_{\mathtt{P}*}\j_{\mathtt{P}}^*A)) & \mbox{by Lemma \ref{ss:restrictionfullyfaithful}} \\
           &=\Hom(L_\Om(\ul\La),\j_{\mathtt{P}*}\j_{\mathtt{P}}^*A) & \mbox{by Corollary \ref{cor:perversefiberwise}.ii)}\\
           &=\Hom(\j_{\mathtt{P}}^*L_\Om(\ul\La),\j_{\mathtt{P}}^*A) \\
           &=\Hom(\j_{\mathtt{P}}^*(\bigast_{j\in J}L_{\Om_j}(\ul\La)),\j_{\mathtt{P}}^*A) &\mbox{by the claim} \\
    &=\Hom(\bigast_{j\in J}L_{\Om_j}(\ul\La),A) &\mbox{by Lemma \ref{ss:restrictionfullyfaithful}},
  \end{align*}
as desired.
\end{proof}

\subsection{}\label{ss:modulesandspaces}
To identify the Hopf algebra arising from Tannakian reconstruction, we want to work with $\bQ_\ell$-coefficients. This requires us to first work with $\bZ_\ell$-coefficients, so let us explicate the $\bZ_\ell$-version of the target of our fiber functor $F^J$.

Later, it will be convenient to work in the following generality. Let $L/\bQ_\ell$ be a finite extension, write $\cO_L$ for its ring of integers, and write $\la$ for the maximal ideal of $\cO_L$. Write $\Rep(\bW^J,\ul{\cO_L})^{\finproj}$ for the inverse limit $\textstyle\varprojlim_m\Rep(\bW^J,\cO_L/\la^m)^{\finproj}$.
\begin{lem*}\hfill
  \begin{enumerate}[i)]
  \item The category
    \begin{align*}
      \{\mbox{finite projective }\Cont(\bN\cup\{\infty\},\cO_L)\mbox{-modules}\}
    \end{align*}
    is naturally equivalent to the category of topological $\cO_L$-module topological spaces $V$ over $\bN\cup\{\infty\}$ that, open-locally on $\bN\cup\{\infty\}$, are finite free topological $\cO_L$-module topological spaces.
  \item The category $\Rep(\bW^J,\ul{\cO_L})^{\finproj}$ is naturally equivalent to the category of $V$ as in part i) along with a continuous $\cO_L$-linear action $\bW^J\times_{\bN\cup\{\infty\}}V\ra V$ over $\bN\cup\{\infty\}$.
  \end{enumerate}
\end{lem*}
\begin{proof}
  All finite projective $\cO_L/\la^m$-modules are free, and $\cO_L/\la^m$ is discrete. Therefore \cite[Lemma B.2.5]{HKW17} implies that the category
\begin{align*}
  \{\mbox{finite projective }\Cont(\bN\cup\{\infty\},\cO_L/\la^m)\mbox{-modules}\}
\end{align*}
is naturally equivalent to the category of $\cO_L/\la^m$-module topological spaces $V_m$ over $\bN\cup\{\infty\}$ that, open-locally on $\bN\cup\{\infty\}$, are finite free $\cO_L/\la^m$-module topological spaces. This identifies $\Rep(\bW^J,\cO_L/\la^m)^{\finproj}$ with the category of such $V_m$ along with a continuous $\cO_L/\la^m$-linear action $\bW^J\times_{\bN\cup\{\infty\}}V_m\ra V_m$. Finally, taking $\varprojlim_m$ yields part i) and part ii).
\end{proof}

\subsection{}
We now put everything together to apply Tannakian reconstruction. Write $\Sat(\cHck^J_G,\cO_L)$ for the inverse limit $\textstyle\varprojlim_m\Sat(\cHck^J_G,\cO_L/\la^m)$.
\begin{thm*}
  The $\Rep(\bW^J,\ul{\bZ_\ell})^{\finproj}$-enriched symmetric monoidal category $\Sat(\cHck^J_G,\bZ_\ell)$ is naturally equivalent to the category of representations of the Hopf algebra
  \begin{align*}
    \cH^J_G\coloneqq\varinjlim_{\{\Om_j\}_{j\in J}}F^J(\bigast_{j\in J}L_{\Om_j}(\ul{\bZ_\ell}))\in\Ind(\Rep(\bW^J,\ul{\bZ_\ell})^{\finproj})
  \end{align*}
in $\Rep(\bW^J,\ul{\bZ_\ell})^{\finproj}$, where $\{\Om_j\}_{j\in J}$ runs over collections of finite $\Ga_\infty$-stable downwards-closed subsets of $X_*(T_\infty)^+$ indexed by $J$. Moreover, we have $\cH^J_G=\bigotimes_{j\in J}\cH^{\{j\}}_G$.
\end{thm*}
\begin{proof}
  The second statement follows from Proposition \ref{ss:fusionconvolution}.i), so we focus on the first statement. We claim that the category $\Sat(\cHck^J_G,\bZ_\ell)$ and the functor
  \begin{align*}
    F^J:\Sat(\cHck^J_G,\bZ_\ell)\ra\Rep(\bW^J,\ul{\bZ_\ell})^{\finproj}
  \end{align*}
  satisfy the conditions in \cite[Proposition VI.10.2]{FS21}. Proposition \ref{ss:dualizable} indicates that $\Sat(\cHck^J_G,\bZ_\ell)$ is rigid, and Proposition \ref{ss:fusionconvolution} indicates that $\Sat(\cHck^J_G,\bZ_\ell)$ and $F^J$ are symmetric monoidal. Moreover, $\Sat(\cHck^J_G,\bZ_\ell)$ admits and $F^J$ reflects coequalizers of $F^J$-split pairs by Proposition \ref{ss:fiberfunctor}.iii). Finally, Proposition \ref{ss:affineSchubertvarieties}.iii) shows that $\Sat(\cHck^J_G,\bZ_\ell)$ equals the filtered union $\bigcup_{\{\Om_j\}_{j\in J}}\Sat(\cHck^J_{G,\Om},\bZ_\ell)$, and Corollary \ref{cor:fiberfunctorrep} shows that the restriction of $F^J$ to $\Sat(\cHck^J_{G,\Om},\bZ_\ell)$ is co-represented by $\bigast_{j\in J}L_{\Om_j}(\ul{\bZ_\ell})$. This yields the claim.

The claim and \cite[Proposition VI.10.2]{FS21} yield the first statement.
\end{proof}

\subsection{}\label{ss:Lgroup}
We have the following version of the $L$-group for $G$ over $\bN\cup\{\infty\}$. Recall from \ref{ss:groupoverinfty} the pinned split connected reductive group $(G^{\s},T^{\s},X^*(T^{\s}),\De,\{x_{\wt{a}}\}_{\wt{a}\in\De})$ over $\bZ$ and the homomorphism
\begin{align*}
\de:\Gal(F_\infty/E_\infty)\ra\Aut(G^{\s},T^{\s},X^*(T^{\s}),\De,\{x_{\wt{a}}\}_{\wt{a}\in\De}).
\end{align*}
Write $\wh{G}$ for the associated Langlands dual over $\bZ_\ell$, and topologize $\sO_{\wh{G}}$ as the filtered union of finite free $\bZ_\ell$-submodules with the $\ell$-adic topology. Then $\de$ induces a continuous $\bZ_\ell$-linear action $a:\Gal(F_\infty/E_\infty)\times\sO_{\wh{G}}\ra\sO_{\wh{G}}$. Write $\wh{G}_{\ad}$ for the adjoint group of $\wh{G}$, and note that the composition
\begin{align*}
\xymatrix{\bZ\ar[r]^-{q^{-1}} & \bZ_\ell^\times\ar[r]^-{\rho_{\ad}} & \wh{G}_{\ad}(\bZ_\ell)}
\end{align*}
induces another such action $c$ of $\bZ$ on $\sO_{\wh{G}}$ that commutes with $a$.

Write $\sO_{\wh\bG}$ for the topological $\bZ_\ell$-module topological space over $\bN\cup\{\infty\}$ given by $\sO_{\wh{G}}\times(\bN\cup\{\infty\})$. Because $E\ra F$ induces a morphism $\bW\ra\Gal(F_\infty/E_\infty)\times(\bN\cup\{\infty\})$ of group topological spaces over $\bN\cup\{\infty\}$, we get a continuous $\bZ_\ell$-linear action $a$ of $\bW$ on $\sO_{\wh\bG}$ over $\bN\cup\{\infty\}$. Precomposing $c$ with $\bW\ra\bZ\times(\bN\cup\{\infty\})$ yields another such action, and after multiplying by $a$, endow $\sO_{\wh\bG}$ with the resulting continuous $\bZ_\ell$-linear action over $\bN\cup\{\infty\}$. Under the identification from Lemma \ref{ss:modulesandspaces}, view $\sO_{\wh\bG}$ as an object of $\Ind(\Rep(\bW,\ul{\bZ_\ell})^{\finproj})$.

\subsection{}\label{ss:SatakeoverQl}
First, we identify $\cH_G^{\{*\}}$ after inverting $\ell$. Write $\Rep(\bW,\ul{\bQ_\ell})^{\finproj}$ for the category $\Rep(\bW,\ul{\bZ_\ell})^{\finproj}[\textstyle\frac1\ell]$, and write $\Sat(\cHck^J_G,\bQ_\ell)$ for the category $\Sat(\cHck^J_G,\bZ_\ell)[\textstyle\frac1\ell]$.
\begin{thm*}
  There is a natural isomorphism
  \begin{align*}
    \sO_{\wh\bG}[\textstyle\frac1\ell]\ra^\sim\cH^{\{*\}}_G[\textstyle\frac1\ell]
  \end{align*}
of Hopf algebras in $\Ind(\Rep(\bW,\ul{\bQ_\ell})^{\finproj})$.
\end{thm*}
\begin{proof}
  Write $Z$ for $\Spd{C}$, write $\j_\bN:Z\times_{\ul{\bN\cup\{\infty\}}}\ul\bN\ra Z$ for the open embedding, and recall the notation of \ref{ss:O(1)sections}. For all $\mu$ in $X_*(T_\infty)^+$, write $V_\mu$ in $\Rep(\wh{G}_{\bQ_\ell})$ for the associated Weyl module. For all $i$ in $\bN\cup\{\infty\}$, it follows from \ref{ss:Grassmannianspecialization} and \cite[Theorem VI.11.1]{FS21} that we have a natural $\bQ_\ell$-linear symmetric monoidal equivalence
  \begin{align*}
    \eps_i:\Rep(\wh{G}_{\bQ_\ell})\ra^\sim\Sat(\cHck^{\{*\}}_G|_{Z_i},\bQ_\ell)
  \end{align*}
such that $F_{Z_i}:\Sat(\cHck^{\{*\}}_G|_{Z_i},\bQ_\ell)\ra\Shv_{\et}(Z_i,\bQ_\ell)=\Vect(\bQ_\ell)$ corresponds to the forgetful functor $\Rep(\wh{G}_{\bQ_\ell})\ra\Vect(\bQ_\ell)$, and that $\eps_i(V_\mu)=\prescript{p}{}H^0(\j_{\mu!}\bQ_\ell|_{Z_i}[d_\mu]))$. We use Proposition \ref{ss:texactness}.ii) and \cite[Proposition 22.19]{Sch17} to identify the latter with $\prescript{p}{}H^0(\j_{\mu!}\bQ_\ell[d_\mu]))|_{Z_i}$.

  First, let us construct a $\bQ_\ell$-linear functor $\eps:\Rep(\wh{G}_{\bQ_\ell})\ra\Sat(\cHck^{\{*\}}_G|_Z,\bQ_\ell)$. By \ref{ss:splittingextension}, we can assume that $G$ is split, so Proposition \ref{ss:standardobjects} lets us take
  \begin{align*}
    \eps(V_\mu)\coloneqq\prescript{p}{}H^0(\j_{\mu!}\bQ_\ell[d_\mu]).
  \end{align*}
  For all $\mu'$ in $X_*(T_\infty)^+$ and all morphisms $f:V_\mu\ra V_{\mu'}$, we claim that there exists a unique morphism $\eps(f):\eps(V_\mu)\ra\eps(V_{\mu'})$ such that $\eps(f)|_{Z_i}=\eps_i(f)$ for all $i$ in $\bN\cup\{\infty\}$. To see this, write $\eps_\bN(f):\j^*_\bN\eps(V_\mu)\ra\j^*_\bN\eps(V_{\mu'})$ for the disjoint union $\coprod_i\eps_i(f)$, where $i$ runs over $\bN$. The description of $D_{\et}(\cHck^{\{*\}}_G|_Z,\La)$ in terms of the decomposition
  \begin{align*}
    \coprod_i\cHck^{\{*\}}_G|_{Z_i}\hra\cHck^{\{*\}}_G|_Z\hla\cHck^{\{*\}}_G|_{Z_\infty},
  \end{align*}
which follows from \cite[Proposition 14.3]{Sch17}, indicates that it suffices to check that
  \begin{align*}
    \xymatrix{\eps(V_\mu)|_{Z_\infty}\ar[r]\ar[d]^-{\eps_\infty(f)} & \j_{\bN*}\j_\bN^*\eps(V_\mu)|_{Z_\infty}\ar[d]^-{\j_{\bN*}\eps_\bN(f)_{Z_\infty}}\\
    \eps(V_{\mu'})|_{Z_\infty}\ar[r] & \j_{\bN*}\j_\bN^*\eps(V_{\mu'})|_{Z_\infty} }
  \end{align*}
  is commutative. Corollary \ref{cor:perversefiberwise}.ii) implies that this is equivalent to checking that
  \begin{align*}
    \xymatrix{\eps(V_\mu)|_{Z_\infty}\ar[r]\ar[d]^-{\eps_\infty(f)} & \prescript{p}{}H^0(\j_{\bN*}\j_\bN^*\eps(V_\mu))|_{Z_\infty}\ar[d]^-{\prescript{p}{}H^0(\j_{\bN*}\eps_\bN(f))_{Z_\infty}}\\
    \eps(V_{\mu'})|_{Z_\infty}\ar[r] & \prescript{p}{}H^0(\j_{\bN*}\j_\bN^*\eps(V_{\mu'}))|_{Z_\infty} }
  \end{align*}
is commutative, and Proposition \ref{ss:fiberfunctor}.ii) indicates that it suffices to check after applying $F_{Z_\infty}$. By combining Proposition \ref{ss:fiberfunctor}.i) with Corollary \ref{ss:hyperbolicoutput}, this becomes
  \begin{align*}
    \xymatrix{V_\mu\ar[r]\ar[d]^-f & V_\mu\otimes_{\bQ_\ell}\varinjlim_U\Cont(U-\{\infty\},\bQ_\ell)\ar[d]^-{f\otimes\id}\\
    V_{\mu'}\ar[r] & V_{\mu'}\otimes_{\bQ_\ell}\varinjlim_U\Cont(U-\{\infty\},\bQ_\ell),}
  \end{align*}
where $U$ runs over compact neighborhoods of $\infty$, and the top and bottom maps are induced by the natural map $\bQ_\ell\ra\textstyle\varinjlim_U\Cont(U-\{\infty\},\bQ_\ell)$. This square indeed commutes, so the claim follows.

Since $\Rep(\wh{G}_{\bQ_\ell})$ is semisimple with simple objects given by the $V_\mu$, the claim finishes our construction of $\eps$. Arguing as above also shows that $\eps$ is symmetric monoidal, so it induces a morphism $\sO_{\wh\bG}[\frac1\ell]\ra\cH^{\{*\}}_G[\textstyle\frac1\ell]$ of Hopf algebras in
\begin{align*}
\Ind\{\mbox{finite projective }\Cont(\bN\cup\{\infty\},\bQ_\ell)\mbox{-modules}\}.
\end{align*}
To show that this is an isomorphism that, under the identification from Lemma \ref{ss:modulesandspaces}, preserves the continuous $\bQ_\ell$-linear $\bW$-action, it suffices to check on fibers. Finally, the result follows from \ref{ss:Grassmannianspecialization} and \cite[Theorem VI.11.1]{FS21}.
\end{proof}

\subsection{}\label{ss:SatakeoverZl}
Finally, we use Theorem \ref{ss:SatakeoverQl} to identify $\cH^{\{*\}}_G$ before inverting $\ell$.
\begin{thm*}
  There is a natural isomorphism
  \begin{align*}
    \sO_{\wh\bG}\ra^\sim\cH^{\{*\}}_G
  \end{align*}
  of Hopf algebras in $\Ind(\Rep(\bW,\ul{\bZ_\ell})^{\finproj})$.
\end{thm*}
\begin{proof}
Note that \ref{ss:Grassmannianspecialization} and \cite[Theorem VI.11.1]{FS21} yield a natural isomorphism
  \begin{align*}
    \sO_{\wh\bG}\otimes_{\Cont(\bN\cup\{\infty\},\bZ_\ell)}\Cont(\bN,\bZ_\ell)\ra^\sim\cH^{\{*\}}_G\otimes_{\Cont(\bN\cup\{\infty\},\bZ_\ell)}\Cont(\bN,\bZ_\ell)
  \end{align*}
  of Hopf algebras in
  \begin{align*}
    \Ind\{\mbox{finite projective }\Cont(\bN,\bZ_\ell)\mbox{-modules}\},
  \end{align*}
  and Theorem \ref{ss:SatakeoverZl} yields a natural isomorphism $\sO_{\wh\bG}[\frac1\ell]\ra^\sim\cH^{\{*\}}_G[\textstyle\frac1\ell]$ of Hopf algebras in $\Ind(\Rep(\bW,\ul{\bQ_\ell})^{\finproj})$. Because $\bN$ is dense in $\bN\cup\{\infty\}$, the square
  \begin{align*}
    \xymatrix{\Cont(\bN\cup\{\infty\},\bZ_\ell)\ar[r]\ar[d] & \Cont(\bN\cup\{\infty\},\bQ_\ell)\ar[d]\\
    \Cont(\bN,\bZ_\ell)\ar[r] & \Cont(\bN,\bQ_\ell)}
  \end{align*}
is cartesian. Our natural isomorphisms are compatible with this cartesian square, and since $\sO_{\wh\bG}$ and $\cH^{\{*\}}_G$ are both flat over $\Cont(\bN\cup\{\infty\},\bZ_\ell)$, this induces a natural isomorphism $\sO_{\wh\bG}\ra^\sim\cH^{\{*\}}_G$ of Hopf algebras in
  \begin{align*}
    \Ind\{\mbox{finite projective }\Cont(\bN\cup\{\infty\},\bZ_\ell)\mbox{-modules}\}.
  \end{align*}
To show that, under the identification from Lemma \ref{ss:modulesandspaces}, this isomorphism preserves the continuous $\bZ_\ell$-linear $\bW$-action, it suffices to check on fibers. Finally, the result follows from \ref{ss:Grassmannianspecialization} and \cite[Theorem VI.11.1]{FS21}.
\end{proof}

\section{Local Langlands over close fields}\label{s:final}
In this section, we put everything together to prove Theorem A and Theorem B. We start by proving that \'etale sheaves on certain classifying stacks are equivalent to smooth representations, thereby geometrizing the setting of \S\ref{s:smoothrepresentations}. Next, we use results from \S\ref{s:Satake} to spread out geometric Hecke operators from $E_\infty$ to $E$. After proving basic facts about geometric Hecke operators over $E$, we prove Theorem A using results from \S\ref{s:reductivegroups} and a spreading out argument. Finally, we combine this with results of Bernstein \cite{Ber84} to prove Theorem B.

\subsection{}\label{ss:BHet}
Recall the notation of Proposition \ref{ss:continuitycriterion}. Recall from \ref{ss:Kn} the group topological space $\bK^n$ over $\bN\cup\{\infty\}$, which is naturally an open group subspace of $G(\bE)$ over $\bN\cup\{\infty\}$ by Proposition \ref{ss:G(E)hypothesisA}. Let $\bH$ be one of $\{G(\bE), \bK^n\}$.

Let $\La$ be a ring that is $\ell$-power torsion.
\begin{prop*}
  There is a natural equivalence of categories
  \begin{align*}
    \cF_{(-)}:D(\bH,\La)\ra^\sim D_{\et}(B\bH,\La).
  \end{align*}
\end{prop*}
\begin{proof}
  First, let us construct a natural functor $\cF_{(-)}:\Rep(\bH,\La)\ra\Shv_{\et}(B\bH,\La)$. For all locally profinite topological spaces $X$ over $\bN\cup\{\infty\}$ and small v-stacks $Z$ over $\ul{\bN\cup\{\infty\}}$, one can show that the natural map $\big|\ul{X}\times_{\ul{\bN\cup\{\infty\}}}Z\big|\ra X\times_{\bN\cup\{\infty\}}\abs{Z}$ is a homeomorphism by reducing to the case where $X$ is profinite. Because $G(\bE)$ is locally profinite by Proposition \ref{ss:G(E)hypothesisB}, this implies that, for all $\ul\bH$-torsors $\wt{S}\ra S$ in the category of v-sheaves, $\big|\wt{S}\big|$ naturally has a continuous action of $\bH$ over $\bN\cup\{\infty\}$. For all smooth representations $V$ of $\bH$ over $\La$, write $\cF_V$ for the presheaf of $\La$-modules over $B\bH$ given by sending $\wt{S}\ra S$ to the $\La$-module
  \begin{align*}
    \big\{\mbox{continuous }\bH\mbox{-equivariant maps }\big|\wt{S}\big|\ra V\mbox{ over }\bN\cup\{\infty\}\big\}.
  \end{align*}
By \cite[Proposition 12.9]{Sch17}, $\cF_V$ is a v-sheaf. Since continuous $\bH$-equivariant maps $\bH\times_{\bN\cup\{\infty\}}\abs{S}\ra V$ over $\bN\cup\{\infty\}$ are equivalent to continuous maps $\abs{S}\ra V$ over $\bN\cup\{\infty\}$, the pullback $\cF_V|_{\ul{\bN\cup\{\infty\}}}$ is naturally isomorphic to $\ul{V}$ over $\ul{\bN\cup\{\infty\}}$, so \cite[Remark 14.14]{Sch17} shows that $\cF_V$ lies in $\Shv_{\et}(B\bH,\La)$.

Note that $\cF_{(-)}$ is exact. Hence \cite[Proposition 14.16]{Sch17} indicates that it extends to a functor $\cF_{(-)}:D(\bH,\La)\ra D_{\et}(B\bH,\La)$. Let $\ov{s}$ be a geometric point in $\Perf_{\ov\bF_q}$, and note that the functor $X\mapsto(\ul{X}/\ul\bH)\times\ov{s}$ yields an equivalence of sites
  \begin{align*}
    \{\mbox{\'etale }\bH\mbox{-spaces}\}\ra^\sim(B\bH_{\ov{s}})_{\et}.
  \end{align*}
Therefore we get a commutative diagram
\begin{align*}
  \xymatrix{D(\bH,\La)\ar[r]^-{\cF_{(-)}}\ar[d]^-{\rotatebox{90}{$\sim$}}& D_{\et}(B\bH,\La)\ar[dd] \\
  D(\{\mbox{\'etale }\bH\mbox{-spaces}\},\La)\ar[d]^-{\rotatebox{90}{$\sim$}}& \\
  D((B\bH_{\ov{s}})_{\et},\La)\ar[r] & D_{\et}(B\bH_{\ov{s}},\La),}
\end{align*}
where the top left equivalence follows from Corollary \ref{ss:etaleGspaces}. The right functor is fully faithful by \cite[Proposition 19.5 (ii)]{Sch17}, so it suffices to show that the bottom functor is an equivalence.

For full faithfulness, write $\la:(B\bH_{\ov{s}})_v\ra(B\bH_{\ov{s}})_{\et}$ for the natural morphism of sites. To prove that the natural transformation $\id\ra\la_*\la^*$ on $D((B\bH_{\ov{s}})_{\et},\La)$ is an isomorphism, we can replace $B\bH$ with an \'etale cover. Hence we can assume that $\bH$ is $\bK^n$. Then the fiber $\bH_i$ is a pro-$p$ group for all $i$ in $\bN\cup\{\infty\}$, so checking on stalks shows that the morphism of ringed sites $((B\bH_{\ov{s}})_{\et},\La)\ra((\ul{\bN\cup\{\infty\}}_{\ov{s}})_{\et},\La)$ has cohomological dimension $0$. Because \cite[Lemma 7.2]{Sch17} indicates that
\begin{align*}
((\ul{\bN\cup\{\infty\}}_{\ov{s}})_{\et},\La)
\end{align*}
has cohomological dimension $0$, this shows that $((B\bH_{\ov{s}})_{\et},\La)$ also has cohomological dimension $0$, so \cite[Tag 0D64]{stacks-project} implies that $D((B\bH_{\ov{s}})_{\et},\La)$ is left-complete. Since $D_{\et}(B\bH_{\ov{s}},\La)$ is also left-complete by \cite[Proposition 14.15]{Sch17}, it suffices to restrict to $D^+((B\bH_{\ov{s}})_{\et},\La)$. Then the desired result follows from using \v{C}ech cohomology along the cover $\ul{\bN\cup\{\infty\}}_{\ov{s}}\ra B\bH_{\ov{s}}$ and the fact that $\bH_i$ is a pro-$p$ group for all $i$ in $\bN\cup\{\infty\}$.

We turn to essential surjectivity. Full faithfulness implies that it suffices to consider $\cF$ in $\Shv_{\et}(B\bH_{\ov{s}},\La)$, and pullback yields a v-sheaf $\cF|_{\ul{\bN\cup\{\infty\}}_{\ov{s}}}$ of $\La$-modules over $\ul{\bN\cup\{\infty\}}_{\ov{s}}$ with a descent datum with respect to
\begin{align*}
\ul{\bN\cup\{\infty\}}_{\ov{s}}\ra B\bH_{\ov{s}}.
\end{align*}
Because $\cF_{\ul{\bN\cup\{\infty\}}_{\ov{s}}}$ is \'etale over $\ul{\bN\cup\{\infty\}}_{\ov{s}}$, it corresponds to a $\La$-module \'etale topological space $V$ over $\bN\cup\{\infty\}$. Lemma \ref{ss:equivariantsheaves} shows that the descent datum endows $V$ with a continuous $\La$-linear action $\bH\times_{\bN\cup\{\infty\}}V\ra V$ over $\bN\cup\{\infty\}$, and descent \cite[Proposition 17.3]{Sch17} identifies $\cF_V$ with $\cF$.
\end{proof}

\subsection{}
We have the following analogue of the (global) Hecke stack in our setting. Recall from \ref{ss:bundlesondivisors} the closed Cartier divisor $D_S\hra X_S$.
\begin{defn*}
  Write $\Hck^J_G$ for the presheaf of groupoids on $\Perf_{\ov\bF_q}$ over $(\Div_X^1)^J$ whose $S$-points parametrize data consisting of
  \begin{enumerate}[a)]
  \item two \'etale $G^{\an}$-torsors $\sG$ and $\sG'$ on $X_S$,
  \item an isomorphism
    \begin{align*}
      \rho:\sG|_{X_S-D_S}\ra^\sim\sG'|_{X_S-D_S}
    \end{align*}
    of \'etale $G^{\an}$-torsors on $X_S-D_S$ that is meromorphic along $D_S$.
  \end{enumerate}
\end{defn*}
Proposition \ref{ss:Gtorsors} and Proposition \ref{ss:BanachColmezsheaf} imply that $\cHck^J_G$ is a v-stack, and the proof of \cite[Proposition III.1.3]{FS21} shows that $\Hck^J_G$ is small. Proposition \ref{ss:Gtorsors} also indicates that pullback to the completion of $X_S$ along $D_S$ induces a morphism $\cL:\Hck^J_G\ra\cHck^J_G$. Finally, note that we have natural morphisms
\begin{align*}
\xymatrix{\Bun_G & \ar[l]_-{p_1}\Hck^J_G\ar[r]^-{p_2} & \Bun_G\times_{\ul{\bN\cup\{\infty\}}}(\Div_X^1)^J}
\end{align*}
given by $(\sG,\sG',\rho)\mapsto\sG$ and $(\sG,\sG',\rho)\mapsto\sG'$, respectively. Using Beauville--Laszlo gluing, Proposition \ref{ss:affineSchubertvarieties}.i) and Proposition \ref{ss:affineSchubertvarieties}.iii) imply that $p_2$ is ind-proper.

\subsection{}
Theorem \ref{ss:SatakeoverZl} lets us spread out geometric Hecke operators from $E_\infty$ to $E$ as follows. Recall from \ref{ss:Lgroup} the action $a$ of $\Gal(F_\infty/E_\infty)$ on $\wh{G}$ and the action $c$ of $\bZ$ on $\wh{G}$. Form the semidirect product $\wh{G}\rtimes\Gal(F_\infty/E_\infty)$ via the action $a$. By using $\sqrt{q}$ to untwist the action $c$, Theorem \ref{ss:SatakeoverZl} yields a natural functor
\begin{align*}
  \cS_{(-)}:\Rep\big(\wh{G}\rtimes\Gal(F_\infty/E_\infty)\big)^J_{\bZ_\ell[\sqrt{q}]}\ra\Sat(\cHck^J_G,\ul{\bZ_\ell[\sqrt{q}]}).
\end{align*}
Let $\La$ be a $\bZ_\ell[\sqrt{q}]$-algebra that is $\ell$-power torsion, and write $\cS_{(-)}$ for the composition
\begin{align*}
\xymatrix{\Rep\big(\wh{G}\rtimes\Gal(F_\infty/E_\infty)\big)^J_{\bZ_\ell[\sqrt{q}]}\ar[r]^-{\cS_{(-)}} & \Sat(\cHck^J_G,\ul{\bZ_\ell[\sqrt{q}]})\ar[r]^-{-\otimes\La} & \Sat(\cHck^J_G,\La).}
\end{align*}
For all $V$ in $\Rep\big(\wh{G}\rtimes\Gal(F_\infty/E_\infty)\big)^J_{\bZ_\ell[\sqrt{q}]}$, write
\begin{align*}
  T_V:D_{\et}(\Bun_G,\La)\ra D_{\et}(\Bun_G\times_{\ul{\bN\cup\{\infty\}}}(\Div_X^1)^J,\La)
\end{align*}
for the functor given by $A\mapsto p_{2!}(p_1^*A\otimes_\La\cL^*\cS_V)$.
\begin{prop*}
The pullback functor
\begin{align*}
  D_{\et}(\Bun_G\times_{\ul{\bN\cup\{\infty\}}}B\bW^J,\La)\ra D_{\et}(\Bun_G\times_{\ul{\bN\cup\{\infty\}}}(\Div_X^1)^J,\La)
\end{align*}
is fully faithful, and the image of $T_V$ lies in the image of $D_{\et}(\Bun_G\times_{\ul{\bN\cup\{\infty\}}}B\bW^J,\La)$.
\end{prop*}
\begin{proof}
The first statement follows from applying Proposition \ref{ss:Drinfeldfullyfaithful} $\#J$ times. For the second statement, note that $V\mapsto T_V$ is exact and compatible with exterior tensor products. Since $D_{\et}(\Bun_G\times_{\ul{\bN\cup\{\infty\}}}B\bW^J,\La)$ is left-complete \cite[Proposition 14.15]{Sch17} and $V$ has a resolution whose terms are of the form $\bigboxtimes_{j\in J}V_j$ for $V_j$ in
  \begin{align*}
    \Rep\big(\wh{G}\rtimes\Gal(F_\infty/E_\infty)\big)_{\bZ_\ell[\sqrt{q}]},
  \end{align*}
this indicates that we can assume that $\#J=1$. Then the result follows from Proposition \ref{ss:BunGoverC} and Proposition \ref{ss:Drinfeldfullyfaithful}.
\end{proof}

\subsection{}\label{ss:Heckeoperatorprop}
As usual, geometric Hecke operators enjoy the following properties. For all $V$ in $\Rep\big(\wh{G}\rtimes\Gal(F_\infty/E_\infty)\big)^J_{\bZ_\ell[\sqrt{q}]}$, write $T_V$ for the composition
\begin{align*}
\xymatrix{D_{\et}(\Bun_G,\La)\ar[r]^-{T_V} & D_{\et}(\Bun_G\times_{\ul{\bN\cup\{\infty\}}}B\bW^J,\La)\ar[r] & D_{\et}(\Bun_G,\La).}
\end{align*}
\begin{prop*}
The functor $T_{V^\vee}$ is left and right adjoint to $T_V$. Consequenty, $T_V$ preserves limits, colimits, and compact objects.
\end{prop*}
\begin{proof}
  Note that $V\mapsto T_V$ is monoidal with respect to $\otimes_\La$ on
  \begin{align*}
    \Rep\big(\wh{G}\rtimes\Gal(F_\infty/E_\infty)\big)^J_{\bZ_\ell[\sqrt{q}]}
  \end{align*}
  and composition on the category of functors $D_{\et}(\Bun_G,\La)\ra D_{\et}(\Bun_G,\La)$. Therefore the evaluation and coevaluation morphisms for $V$ induce the counit and unit, respectively, for the desired adjunctions.
\end{proof}

\subsection{}\label{ss:propsubgroup}
For all $i$ in $\bN$, recall from \ref{ss:closefieldssetup} the absolute ramification index $e_i$ of $E_i$, and recall from \ref{ss:Galoisgroups} the isomorphism $\Ga_i/I_i^{e_i}\cong\Ga_\infty/I_\infty^{e_i}$ of topological groups. Let $\bP$ be a compact group subspace of $\bW$ over $\bN\cup\{\infty\}$ satisfying the following properties:
\begin{enumerate}[a)]
\item for all $i$ in $\bN\cup\{\infty\}$, the fiber $\bP_i$ is a normal subgroup of $W_i$ that is an open subgroup of the wild inertia subgroup,
\item there exists a positive integer $d$ such that, for all $i$ in $\bN\cup\{\infty\}$, the fiber $\bP_i$ contains $I_i^d$, and for large enough $i$, the image of $\bP_i$ under the isomorphism $\Ga_i/I_i^d\cong\Ga_\infty/I_\infty^d$ (since $i$ is large enough) equals the image of $\bP_\infty$.
\end{enumerate}
By repeating Definition \ref{ss:Galoisgroups}.a) (except that we replace $\Ga_i/I^d_i$ with $W_i/\bP_i$), we obtain a natural group topological space $\bW/\bP$ over $\bN\cup\{\infty\}$ whose fiber at $i$ is isomorphic to $W_i/\bP_i$. Arguing as in the proof of Lemma \ref{ss:LTextensions} shows that the map
\begin{align*}
  q:\bW\ra\bW/\bP
\end{align*}
whose fiber at $i$ equals $W_i\ra W_i/\bP_i$ is continuous, and by checking on fibers, we see that $q$ is a morphism of group topological spaces over $\bN\cup\{\infty\}$.

\begin{prop}\label{prop:compactnessBunG}\hfill
  \begin{enumerate}[i)]
  \item The pullback functor
  \begin{align*}
    (\id\times q)^*:D_{\et}(\Bun_G\times_{\ul{\bN\cup\{\infty\}}}B(\bW/\bP)^J,\La)\ra D_{\et}(\Bun_G\times_{\ul{\bN\cup\{\infty\}}}B\bW^J,\La)
  \end{align*}
  is fully faithful.
\item Let $A$ be a compact object in $D_{\et}(\Bun_G,\La)$. Then there exists a group subspace $\bP$ of $\bW$ over $\bN\cup\{\infty\}$ as in \ref{ss:propsubgroup} such that, for all finite sets $J$ and $V$ in
  \begin{align*}
    \Rep\big(\wh{G}\rtimes\Gal(F_\infty/E_\infty)\big)^J_{\bZ_\ell[\sqrt{q}]},
  \end{align*}
  the object $T_V(A)$ lies in the image of $D_{\et}(\Bun_G\times_{\ul{\bN\cup\{\infty\}}}B(\bW/\bP)^J,\La)$.
  \end{enumerate}
\end{prop}
\begin{proof}
For part i), it suffices to show that the pullback functor
\begin{align*}
 (\id\times q)^*:D_{\solid}(\Bun_G\times_{\ul{\bN\cup\{\infty\}}}B(\bW/\bP)^J,\La)\ra D_{\solid}(\Bun_G\times_{\ul{\bN\cup\{\infty\}}}B\bW^J,\La)
\end{align*}
is fully faithful. The projection formula \cite[Proposition VII.3.1 (i)]{FS21} indicates that it suffices to show that the counit morphism $(\id\times q)_\natural\La\ra\La$ is an isomorphism, and combining the cartesian square
\begin{align*}
  \xymatrix{\Bun_G\times_{\ul{\bN\cup\{\infty\}}}B\bW^J\ar[r]^-{\id\times q}\ar[d] & \Bun_G\times_{\ul{\bN\cup\{\infty\}}}B(\bW/\bP)^J\ar[d]\\
  B\bW^J\ar[r]^-q & B(\bW/\bP)^J}
\end{align*}
with \cite[Proposition VII.3.1 (iii)]{FS21} implies that it suffices to show that the counit morphism $q_\natural\La\ra\La$ is an isomorphism. Using descent \cite[Proposition 17.3]{Sch17} and applying \cite[Proposition VII.3.1 (iii)]{FS21} again indicate that it suffices to show that the structure morphism $\pr:B\bP^J\ra\ul{\bN\cup\{\infty\}}$ induces an isomorphism $\pr_\natural\La\ra\La$. Finally, this follows from $\bP_i$ being a pro-$p$ group for all $i$ in $\bN\cup\{\infty\}$.

For part ii), note that $V\mapsto T_V$ is exact and compatible with exterior tensor products, so it suffices to consider a tensor generator $V$ in $\Rep\big(\wh{G}\rtimes\Gal(F_\infty/E_\infty)\big)_{\bZ_\ell[\sqrt{q}]}$. Now $D_{\et}(\Bun_G,\La)$ naturally has the structure of a condensed ($\infty$-)category over $\bN\cup\{\infty\}$ via sending any extremally disconnected profinite set $X$ over $\bN\cup\{\infty\}$ to $D_{\et}(\Bun_G\times_{\ul{\bN\cup\{\infty\}}}\ul{X},\La)$, and descent \cite[Proposition 17.3]{Sch17} implies that
\begin{align*}
D_{\et}(\Bun_G\times_{\ul{\bN\cup\{\infty\}}}B\bW,\La)
\end{align*}
is naturally equivalent to the category of objects $A'$ in $D_{\et}(\Bun_G,\La)$ along with a morphism $\bW\ra\Aut(A')$ of condensed animated groups over $\bN\cup\{\infty\}$. Because $T_V(A)$ is a compact object in $D_{\et}(\Bun_G,\La)$ by Proposition \ref{ss:Heckeoperatorprop}, the condensed animated group $\Aut(T_V(A))$ over $\bN\cup\{\infty\}$ is \'etale over $\bN\cup\{\infty\}$. Since $\La$ is $\ell$-power torsion but $\bP_i$ is a pro-$p$ group for all $i$ in $\bN\cup\{\infty\}$, the result follows.
\end{proof}

\subsection{}\label{ss:compactnessBunGapplied}
Before proceeding to our main theorems, we need the following preparations. Write $\j:BG(\bE)\ra\Bun_G$ for the open embedding from Proposition \ref{ss:BunG1}. Now $\j_!$ preserves compact objects, and Proposition \ref{ss:rescIndadjunction} implies that the same holds for $\cInd^{G(\bE)}_{\bK^n}$. Hence $A\coloneqq\j_!\cInd_{\bK^n}^{G(\bE)}\ul\La$ is a compact object in $D_{\et}(\Bun_G,\La)$. Write $\bP$ for the group subspace of $\bW$ over $\bN\cup\{\infty\}$ as in \ref{ss:propsubgroup} provided by applying Proposition \ref{prop:compactnessBunG}.ii) to $A$.

For all $i$ in $\bN\cup\{\infty\}$, choose lifts $\wt\vp_i$ and $\wt\tau_i$ to $W_i/\bP_i$ of absolute $q$-Frobenius and of a (topological) generator of tame inertia, respectively, such that, for large enough $i$ as in \ref{ss:propsubgroup}.b), the images of $\wt\vp_i$ and $\wt\tau_i$ under the isomorphism $W_i/\bP_i\cong W_\infty/\bP_\infty$ equal $\wt\vp_\infty$ and $\wt\tau_\infty$, respectively. Write $\Exc(W_i/\bP_i,\wh{G})$ for the associated excursion algebra over $\bZ_\ell[\sqrt{q}]$ as in \cite[Definition VIII.3.4]{FS21}.

Recall the notation of Definition \ref{ss:Heckealgebra}. Briefly, relax our assumption that $\La$ is $\ell$-power torsion. Write $\cZ_\La(G(E_i),\bK_i^n)$ for the center of $\cH(G(E_i),\bK_i^n)_\La$, and recall from \cite[p.~326]{FS21} the natural map of $\bZ_\ell[\sqrt{q}]$-algebras
\begin{align*}
\Exc(W_i/\bP_i,\wh{G})\ra\cZ_\La(G(E_i),\bK^n_i).
\end{align*}

\subsection{}\label{ss:spreadoutexcursion}
By spreading out excursion operators from $E_\infty$ to $E$, we can prove our first main theorem. For all $i$ in $\bN\cup\{\infty\}$, Theorem \ref{thm:Heckealgebracongruence} indicates that our isomorphism $\Tr_{e_i}(E_i)\cong\Tr_{e_i}(E_\infty)$ induces an isomorphism of $\La$-algebras
\begin{align*}
\cH(G(E_\infty),\bK^n_\infty)_\La\ra^\sim\cH(G(E_i),\bK^n_i)_\La,
\end{align*}
which restricts to an isomorphism of $\La$-algebras
\begin{align*}
\cZ_\La(G(E_\infty),\bK^n_\infty)\ra^\sim\cZ_\La(G(E_i),\bK^n_i).
\end{align*}

Resume our assumption that $\La$ is $\ell$-power torsion.
\begin{thm*}
  For large enough $i$ in $\bN\cup\{\infty\}$, the square
  \begin{align*}
    \xymatrix{\Exc(W_\infty/\bP_\infty,\wh{G})\ar[r]\ar@{=}[d]^-{\rotatebox{90}{$\sim$}} & \cZ_\La(G(E_\infty),\bK_\infty^n)\ar[d]^-{\rotatebox{90}{$\sim$}} \\
    \Exc(W_i/\bP_i,\wh{G})\ar[r] & \cZ_\La(G(E_i),\bK_i^n)}
  \end{align*}
commutes.
\end{thm*}
\begin{proof}
  Recall that $\Exc(W_\infty/\bP_\infty,\wh{G})$ has canonical generators $S_{J,V,x,\xi,\ga_{\infty,\bullet}}$, where $J$ runs over finite sets, $V$ runs over objects in $\Rep\big(\wh{G}\rtimes\Gal(F_\infty/E_\infty)\big)_{\bZ_\ell[\sqrt{q}]}^J$, $x$ and $\xi$ run over morphisms $\mathbf{1}\ra V|_{\De(\wh{G})}$ and $V|_{\De(\wh{G})}\ra\mathbf{1}$ in $\Rep\wh{G}_{\bZ_\ell[\sqrt{q}]}$, respectively, and $\ga_{\bullet,\infty}=(\ga_{j,\infty})_{j\in J}$ runs over $J$-tuples in the subgroup of $W_\infty/\bP_\infty$ generated by $\wt\vp_\infty$, $\wt\tau_\infty$, and the image of the wild inertia subgroup \cite[Corollary VIII.4.3]{FS21}.

  For all $j$ in $J$, let $\ga_j$ be a continuous section of $\bW/\bP\ra\bN\cup\{\infty\}$ such that, for all $i$ large enough as in \ref{ss:propsubgroup}.ii), the image of $\ga_j(i)$ under the isomorphism $W_i/\bP_i\cong W_\infty/\bP_\infty$ equals $\ga_{j,\infty}$. Since $V\mapsto T_V$ is functorial in $J$ and $V$, we can form the composition
\begin{align*}
\xymatrix{A = T_{\mathbf{1}}(A)\ar[r]^-x & T_{V|_{\De(\wh{G})}}(A)=T_V(A)\ar[r]^-{\ga_\bullet} & T_V(A) = T_{V|_{\De(\wh{G})}}\ar[r]^-\xi & T_{\mathbf{1}}(A) = A. }
\end{align*}
Write $S_{J,V,x,\xi,\ga_\bullet}$ for this endomorphism of $A$, which corresponds to an element of $\End_{G(\bE)}(\cInd^{G(\bE)}_{\bK^n}\ul\La)$ because $\j_!$ is fully faithful. For all $i$ in $\bN\cup\{\infty\}$, Proposition \ref{ss:Wittspecialization} implies that $(S_{J,V,x,\xi,\ga_\bullet})(i)$ equals the image of $S_{J,V,x,\xi,\ga_\bullet(i)}$ under the map
\begin{align*}
\Exc(W_i/\bP_i,\wh{G})\ra\cZ_\La(G(E_i),\bK^n_i).
\end{align*}
Therefore, under the identification $\ul{\cH(G(E_\infty),\bK_\infty^n)_\La}\ra^\sim\ul\End_{G(\bE)}(\cInd^{G(\bE)}_{\bK^n}\ul\La)$ from Theorem \ref{ss:Heckealgebrasinfamilies}, there exists a compact neighborhood $U$ of $\infty$ such that
\begin{enumerate}[$\bullet$]
\item all $i$ in $U$ are large enough as in \ref{ss:propsubgroup}.ii),
\item the restriction of $S_{J,V,x,\xi,\ga_\bullet}$ to $U$ corresponds to the constant section valued in the image of $S_{J,V,x,\xi,\ga_{\bullet,\infty}}$ under the map $\Exc(W_\infty/\bP_\infty,\wh{G})\ra\cZ_\La(G(E_\infty),\bK^n_\infty)$.
\end{enumerate}
Theorem \ref{ss:Heckealgebrasinfamilies} and the above show that, for all $i$ in $U$, the square in question commutes for $S_{J,V,x,\xi,\ga_{\bullet,\infty}}$. Finally, since $\Exc(W_\infty/\bP_\infty,\wh{G})$ is of finite type over $\bZ_\ell[\sqrt{q}]$, applying this to finitely many $S_{J,V,x,\xi,\ga_{\bullet,\infty}}$ yields the desired result.
\end{proof}

\subsection{}\label{ss:thmA}
We now prove Theorem A. Write $E'$ for $\bF_q\lp{t}$, and let $G'$ be a quasisplit connected reductive group over $E'$ that splits over an $n$-ramified extension.
\begin{thm*}For all $p$-adic fields $E$ along with an isomorphism $\Tr_e(E)\cong\Tr_e(E')$ for some $e\geq n$, write $G$ for the quasisplit connected reductive group over $E$ associated with $G'$ as in \ref{ss:closefieldscongruence}. There exists an integer $d\geq n$ such that, if $e\geq d$, then
  \begin{enumerate}[i)]
  \item the map $\LLC^{\semis}_G$ restricts to a map
  \begin{align*}
\LLC^{\semis}_G:\left\{{\begin{tabular}{c}
    smooth irreps $\pi$ of $G(E)$\\
    over $\ov\bF_\ell$ with $(\pi)^{K^n}\neq0$
  \end{tabular}}\right\}\lra\left\{{\begin{tabular}{c}
  semisimple $L$-parameters $\rho$\\
  for $G$ over $\ov\bF_\ell$ with $\rho|_{I^d}=1$
  \end{tabular}}\right\},
  \end{align*}
and the same holds for the map $\LLC^{\semis}_{G'}$.

\item the square
  \begin{align*}
    \xymatrix{\left\{{\begin{tabular}{c}
    smooth irreps $\pi'$ of $G'(E')$\\
    over $\ov\bF_\ell$ with $(\pi')^{K'^n}\neq0$
  \end{tabular}}\right\}\ar[r]^-{\LLC_{G'}^{\semis}}\ar[d]^-{\rotatebox{90}{$\sim$}} & \left\{{\begin{tabular}{c}
  semisimple $L$-parameters $\rho'$\\
  for $G'$ over $\ov\bF_\ell$ with $\rho'|_{I'^d}=1$
  \end{tabular}}\right\}\ar[d]^-{\rotatebox{90}{$\sim$}} \\
  \left\{{\begin{tabular}{c}
    smooth irreps $\pi$ of $G(E)$\\
    over $\ov\bF_\ell$ with $(\pi)^{K'^n}\neq0$
  \end{tabular}}\right\}\ar[r]^-{\LLC_G^{\semis}} & \left\{{\begin{tabular}{c}
  semisimple $L$-parameters $\rho$\\
  for $G$ over $\ov\bF_\ell$ with $\rho|_{I^d}=1$
  \end{tabular}}\right\}}
  \end{align*}
commutes.
\end{enumerate}
\end{thm*}
\begin{proof}
  For all positive integers $e$, Krasner's lemma implies that there are finitely many $p$-adic fields $E$ with residue field $\bF_q$ and absolute ramification index $e$ (up to isomorphism). Moreover, there are finitely many isomorphisms $\Tr_e(E)\cong\Tr_e(E')$. Therefore we can take $\{E_i\}_{i\in\bN}$ in \ref{ss:closefieldssetup} to be the family of all $p$-adic fields $E$ with residue field $\bF_q$ (up to isomorphism), where we make $E$ appear with multiplicity equal to the number of isomorphisms $\Tr_e(E)\cong\Tr_e(E')$. Finally, using \cite[Corollary VIII.4.3]{FS21}, part i) follows from \ref{ss:compactnessBunGapplied}, and part ii) follows from Theorem \ref{ss:spreadoutexcursion}.
\end{proof}

\subsection{}
We conclude by proving Theorem B.
\begin{thm*}There exists an integer $d\geq n$ such that, if $e\geq d$, then
  \begin{enumerate}[i)]
  \item the map $\LLC^{\semis}_{G'}$ restricts to a map
  \begin{align*}
\LLC^{\semis}_{G'}:\left\{{\begin{tabular}{c}
    smooth irreps $\pi'$ of $G'(E')$\\
    over $\ov\bQ_\ell$ with $(\pi')^{K'^n}\neq0$
  \end{tabular}}\right\}\lra\left\{{\begin{tabular}{c}
  semisimple $L$-parameters $\rho'$\\
  for $G'$ over $\ov\bQ_\ell$ with $\rho'|_{I'^d}=1$
  \end{tabular}}\right\}.
  \end{align*}

\item the square
  \begin{align*}
    \xymatrix{\left\{{\begin{tabular}{c}
    smooth irreps $\pi'$ of $G'(E')$\\
    over $\ov\bQ_\ell$ with $(\pi')^{K'^n}\neq0$
  \end{tabular}}\right\}\ar[r]^-{\LLC_{G'}^{\semis}}\ar[d]^-{\rotatebox{90}{$\sim$}} & \left\{{\begin{tabular}{c}
  semisimple $L$-parameters $\rho'$\\
  for $G'$ over $\ov\bQ_\ell$ with $\rho'|_{I'^d}=1$
  \end{tabular}}\right\}\ar[d]^-{\rotatebox{90}{$\sim$}} \\
  \left\{{\begin{tabular}{c}
    smooth irreps $\pi$ of $G(E)$\\
    over $\ov\bQ_\ell$ with $(\pi)^{K'^n}\neq0$
  \end{tabular}}\right\}\ar[r]^-{\LLC_G^{\semis}} & \left\{{\begin{tabular}{c}
  semisimple $L$-parameters $\rho$\\
  for $G$ over $\ov\bQ_\ell$ with $\rho|_{I^d}=1$
  \end{tabular}}\right\}}
  \end{align*}
commutes \emph{after restricting to the wild inertia subgroup}. In particular, part i) holds for the map $\LLC^{\semis}_G$.
\end{enumerate}
\end{thm*}
\begin{proof}
  Let $\pi'$ be an irreducible cuspidal representation of $G'(E')$ over $\ov\bQ_\ell$ with $(\pi')^{K'^n}\neq0$. Write $\rho'$ for $\LLC^{\semis}_{G'}(\pi')$, and let $d$ be a positive integer with $\rho'|_{I'^d}=1$. After enlarging $d$, we can assume that it is at least the $d$ from Theorem \ref{ss:thmA}. Now \cite[p.~331]{FS21} indicates that $\LLC^{\semis}_{G'}$ is compatible with twisting by unramified characters, so we can assume that $\pi'$ has finite order central character. Then \cite[(II.4.12)]{Vig96} shows that there exists a finite extension $L/\bQ_\ell(\sqrt{q})$ such that
  \begin{enumerate}[$\bullet$]
  \item $\pi'$ is defined over $L$,
  \item $\pi'$ has an $\cO_L[G'(E')]$-lattice $\pi'_{\cO_L}$ as in \cite[(I.9.1)]{Vig96}.
  \end{enumerate}
Write $\bF_\la$ for the residue field of $L$, and write $\chi':\cZ_{\cO_L}(G'(E'),K'^n)\ra\cO_L$ for the map of $\cO_L$-algebras induced by $\pi'_{\cO_L}$. Since $\rho'$ corresponds to $\chi'\otimes\ov\bQ_\ell$ under \cite[Corollary VIII.4.3]{FS21}, this implies that, after enlarging $L$, the $L$-parameter $\rho'$ has a representative of the form $\rho':W'\ra\wh{G}(\cO_L)$.

Write $\pi$ for the irreducible smooth representation of $G(E)$ over $\ov\bQ_\ell$ associated with $\pi'$ via the isomorphism
\begin{align*}
\cH(G(E),K^n)_{\ov\bQ_\ell}\ra^\sim\cH(G'(E'),K'^n)_{\ov\bQ_\ell}.
\end{align*}
Because the above isomorphism is base changed from an isomorphism
\begin{align*}
  \cH(G(E),K^n)_{\cO_L}\ra^\sim\cH(G'(E'),K'^n)_{\cO_L},
\end{align*}
the representation $\pi$ is also defined over $L$, and $\pi'_{\cO_L}$ induces an $\cO_L[G(E)]$-lattice $\pi_{\cO_L}$ of $\pi$. Write $\chi:\cZ_{\cO_L}(G(E),K^n)\ra\cO_L$ for the map of $\cO_L$-algebras induced by $\pi_{\cO_L}$. Write $\rho$ for $\LLC^{\semis}_G(\pi)$; arguing as above shows that, after enlarging $L$, the $L$-parameter $\rho$ has a representative of the form $\rho:W\ra\wh{G}(\cO_L)$.

The proof of Theorem \ref{ss:thmA} yields certain normal subgroups $P'$ and $P$ of $W'$ and $W$, respectively, that are open subgroups of the wild inertia subgroups, contain $I'^d$ and $I^d$, respectively, and whose images in $W'/I'^d\cong W/I^d$ coincide. Consider the resulting diagram
\begin{align*}
  \xymatrix{\Exc(W'/P',\wh{G})\ar[r]\ar@{=}[d]^-{\rotatebox{90}{$\sim$}} & \cZ_{\ov\bF_\ell}(G'(E'),K'^n)\ar[r]^-{\chi'\otimes\ov\bF_\ell}\ar[d]^-{\rotatebox{90}{$\sim$}} & \ov\bF_\ell\ar@{=}[d]\\
  \Exc(W/P,\wh{G})\ar[r] & \cZ_{\ov\bF_\ell}(G(E),K^n)\ar[r]^-{\chi\otimes\ov\bF_\ell} & \ov\bF_\ell.}
\end{align*}
Theorem \ref{ss:spreadoutexcursion} indicates that the left square commutes, and the right square commutes by construction. The top and bottom rows correspond under \cite[Corollary VIII.4.3]{FS21} to the semisimplification of the images in $\wh{G}(\ov\bF_\ell)$ of $\rho'$ and $\rho$, respectively, so after enlarging $L$, this shows that the semisimplification of the $1$-cocycles
\begin{align*}
W'/P'\lra^{\rho'}\wh{G}(\cO_L)\lra\wh{G}(\bF_\la)\mbox{ and }W/P\lra^\rho\wh{G}(\cO_L)\lra\wh{G}(\bF_\la)
\end{align*}
are cohomologous (under the isomorphism $W'/P'\cong W/P$). Since the kernel $\wh{K}$ of $\wh{G}(\cO_L)\ra\wh{G}(\bF_\la)$ is pro-$\ell$ but the wild inertia subgroup is pro-$p$, the Galois cohomology long exact sequence induced by
\begin{align*}
\xymatrix{1\ar[r] & \wh{K}\ar[r] & \wh{G}(\cO_L)\ar[r] & \wh{G}(\bF_\la)\ar[r] & 1}
\end{align*}
then implies that the $1$-cocycles
\begin{align*}
W'/P'\lra^{\rho'}\wh{G}(\cO_L)\mbox{ and }W/P\lra^\rho\wh{G}(\cO_L)
\end{align*}
are cohomologous \emph{after restricting to the wild inertia subgroup}.

Finally, \cite[2.13]{Ber84} implies that that there are finitely many Bernstein components in $\Rep(G'(E'),\ov\bQ_\ell)$ that contain an irreducible smooth representation $\pi'$ with $(\pi')^{K'^n}\neq0$. Because $\LLC^{\semis}_{G'}$ is compatible with twisting by unramified characters \cite[p.~331]{FS21} and parabolic induction \cite[Corollary IX.7.3]{FS21}, applying the above to these finitely many Bernstein components yields the desired result.
\end{proof}

\appendix

\section{Congruences of disconnected parahoric subgroups}\label{s:congruence}
In the appendix, our goal is to generalize Kazhdan's Hecke algebra isomorphism over close fields \cite[Theorem A]{Kaz86b} to the case of quasisplit connected reductive groups $G$ over $E$. While such a generalization has been proved by Ganapathy \cite[Theorem 4.1]{Gan24}, this relies on a congruence result of Ganapathy \cite[Corollary 6.3]{Gan19} for parahoric subgroups associated with stabilizers in $G(E)^0$ (which have connected special fibers). For applications to \S\ref{s:reductivegroups}, it is convenient instead to use parahoric subgroups associated with stabilizers in $G(E)^1$ (which can have disconnected special fibers), since $G(E)^1$ enjoys Galois descent.

Therefore, we begin by proving a congruence result for disconnected parahoric subgroups. By using birational group laws, we reduce this to the case when $G$ is a torus, where it follows from results of Chai--Yu \cite{CY01}. Afterwards, we follow Kazhdan's arguments to construct an isomorphism of Hecke algebras over close fields. Although the Hecke algebras that we consider are equal to those considered in \cite{Gan24}, we do not know whether the isomorphisms constructed here equal those constructed in \cite{Gan24}.

\subsection{}
We start with some notation. Let $E$ be a nonarchimedean local field, write $\cO$ for its ring of integers, and write $\fp$ for the maximal ideal of $\cO$. Fix a separable closure $\ov{E}$ of $E$, and write $\Ga$ for $\Gal(\ov{E}/E)$. Write $I$ for the inertia subgroup of $\Ga$. For all positive integers $n$, write $I^n$ for the $n$-th ramification subgroup of $I$ in the upper numbering.

Recall the notation of \ref{ss:groupoverinfty}. Let $\de:\Ga\ra\Aut(G^{\s},T^{\s},X^*(T^{\s}),\De,\{x_{\wt{a}}\}_{\wt{a}\in\De})$ be a continuous homomorphism, and write $F/E$ for the finite Galois extension such that $\Gal(F/E)$ is the image of $\de$. Descending $(G^{\s}_F,T^{\s}_F,B^{\s}_F,\{x_{\wt{a}}\}_{\wt{a}\in\De})$ along the \'etale $\Gal(F/E)$-torsor $\Spec{F}\ra\Spec{E}$ via $\de$ yields a quasisplit connected reductive group $G$ over $E$ with a pinning $(B,T,\{x_{\wt{a}}\}_{\wt{a}\in\De})$ over $E_i$ as in \cite[Definition 2.9.1]{KP23}.

\subsection{}\label{ss:parahoricstructure}
Let us gather some facts about disconnected parahoric subgroups. Write $S$ for the maximal split subtorus of $T$, write $\cB(G/E)$ for the (reduced) building of $G$ over $E$, and write $\cA(S)$ for the apartment in $\cB(G/E)$ associated with $S$. Let $x$ be a special point in $\cA(S)$. Write $\cK$ for the smooth affine model of $G$ over $\cO$ such that $\cK(\cO)\subseteq G(E)$ equals the maximal compact subgroup $G(E)^1_x$ \cite[Proposition 8.3.1]{KP23}.

Write $\Phi$ for the relative root system $\Phi(G,S)$, and write $\Phi^\pm$ for the subset of positive or negative roots associated with $B$. By \cite[Proposition B.2.4]{KP23}, the Zariski closure $\cS$ of $S$ in $\cK$ is a split subtorus of $\cK$. Write $\cT$ for the centralizer of $\cS$ in $\cK$, and for all $a$ in $\Phi$, write $\cU_a$ for the $a$-root subgroup of $\cK$.
\begin{lem*}\hfill
  \begin{enumerate}[i)]
  \item The group $\cT$ is isomorphic to the N\'eron model of $T$ over $\cO$.
  \item The multiplication morphism $\prod_a\cU_a\ra\cK$, where $a$ runs over $\Phi^{\pm}_{\red}$ (for some ordering of $\Phi^{\pm}_{\red}$) and $\prod$ denotes the product over $\cO$, is a closed embedding whose image is a subgroup $\cU_{\pm}$ of $\cK$ that is independent of the ordering of $\Phi^\pm_{\red}$.
  \item The multiplication morphism $\cU_-\times_\cO\cT\times_\cO\cU_+\ra\cK$ is an open embedding.
  \end{enumerate}
\end{lem*}
\begin{proof}
Note that that $\cT_E$ equals $T$, and $T(E)^1$ equals the intersection of $T(E)$ with $G(E)^1_x$. Since $\cT$ is smooth over $\cO$ by \cite[Lemma 2.2.4]{Con14}, part i) follows from \cite[Proposition B.7.2]{KP23}. Part ii) and part iii) follow from \cite[Theorem 8.2.5]{KP23}.
\end{proof}

\subsection{}\label{ss:parahoricrootgroups}
Let $a$ be in $\Phi_{\red}$. Let $\wt{a}$ be a lift of $a$ to the absolute root system $\Phi(G,T)$, write $E_{\wt{a}}$ for the field of definition of $\wt{a}$, and write $\cO_{\wt{a}}$ for the ring of integers of $E_{\wt{a}}$. Write $\wt{a}:T\ra\R_{E_{\wt{a}}/E}\bG_m$ for the associated morphism of groups over $E$, as well as the induced morphism $\wt{a}:\cT\ra\R_{\cO_{\wt{a}}/\cO}\bG_m$ of N\'eron models over $\cO$.

Write $U_a$ for the $a$-root subgroup of $G$. There are two possibilities:
\begin{enumerate}[R1:]
\item If $2a$ does not lie in $\Phi$, then the pinning induces an isomorphism $\R_{E_{\wt{a}}/E}\bG_a\ra^\sim U_a$. Moreover, after composing the natural action of $\R_{E_{\wt{a}}/E}\bG_m$ on $\R_{E_{\wt{a}}/E}\bG_a$ with $\wt{a}$, this corresponds to the action of $T$ on $U_a$.
\item If $2a$ lies in $\Phi$, then the pinning induces an isomorphism $\R_{E_{\wt{b}}/E}\wt{U}_a\ra^\sim U_a$, where $\wt{b}$ denotes the lift to $\Phi(G,T)$ of $2a$ induced by $\wt{a}$, and $\wt{U}_a$ denotes the affine group $U_{E_{\wt{a}}/E_{\wt{b}}}$ over $E_{\wt{b}}$ as in \cite[(2.7.2)]{KP23}. Moreover, there is a natural action of $\R_{E_{\wt{a}}/E_{\wt{b}}}\bG_m$ on $\wt{U}_a$, and after applying $\R_{E_{\wt{b}}/E}$ and composing with $\wt{a}$, this corresponds to the action of $T$ on $U_a$.
\end{enumerate}
The (disconnected) N\'eron model $\cT$ naturally acts on root subgroups as follows.
\begin{lem*}The subgroup $\cU_a$ is normalized by $\cT$. There are two possibilities:
  \begin{enumerate}[R1:]
  \item If $2a$ does not lie in $\Phi$, then the pinning induces an isomorphism $\R_{\cO_{\wt{a}}/\cO}\wt\cU_a\ra^\sim\cU_a$, where $\wt\cU_a$ denotes the smooth affine model of $\bG_a$ over $\cO_{\wt{a}}$ such that $\wt\cU_a(\cO_{\wt{a}})\subseteq E_{\wt{a}}$ equals $U_{x,a,0}$ \cite[\S C.2]{KP23}. Moreover, there is a natural action of $\bG_m$ on $\wt\cU_a$, and after applying $\R_{\cO_{\wt{a}}/\cO}$ and composing with $\wt{a}$, this corresponds to the action of $\cT$ on $\cU_a$.
  \item If $2a$ lies in $\Phi$, then the pinning induces an isomorphism $\R_{\cO_{\wt{b}}/\cO}\wt\cU_a\ra^\sim\cU_a$, where $\wt\cU_a$ denotes the smooth affine model of $\wt{U}_a$ over $\cO_{\wt{b}}$ such that $\wt\cU_a(\cO_{\wt{b}})\subseteq\wt{U}_a(E_{\wt{b}})$ equals $U_{x,a,0}$ \cite[\S C.4]{KP23}. Moreover, there is a natural action of $\R_{\cO_{\wt{a}}/\cO_{\wt{b}}}\bG_m$ on $\wt\cU_a$, and after applying $\R_{\cO_{\wt{b}}/\cO}$ and composing with $\wt{a}$, this corresponds to the action of $\cT$ on $\cU_a$.
  \end{enumerate}
\end{lem*}
\begin{proof}
The desired description of $\cU_a$ follows from \cite[Proposition C.5.1]{KP23}. Because $\cT(\cO)\subseteq T(E)$ equals $T(E)^1$, this description indicates that $\cU_a$ is normalized by $\cT$. It also yields the desired description of the action of $\cT$ on $\cU_a$.
\end{proof}

\subsection{}\label{ss:parahoricbirational}
We will use the following to reconstruct $\cK$, which is crucial for our congruence result. Write $\cX$ for the smooth affine scheme $\cU_-\times_\cO\cT\times_\cO\cU_+$ over $\cO$, so that the multiplication morphism $\cX\ra\cK$ is an open embedding by Lemma \ref{ss:parahoricstructure}.iii). Since $x$ is a special point, \cite[Proposition 7.7.11]{KP23} and \cite[Remark 7.7.6]{KP23} imply that $\cX_{\cO/\fp}$ is dense in $\cK_{\cO/\fp}$. Hence the group law on $\cK$ over $\cO$ restricts to an birational group law on $\cX$ over $\cO$ as in \cite[5.1/1]{BLR90}.
\begin{prop*}
For any scheme $\cZ$ over $\cO$, the birational group law on $\cX_\cZ$ over $\cZ$ depends only on the group laws on $(\cU_\pm)_\cZ$ and $\cT_\cZ$ over $\cZ$, along with the action of $\cT_\cZ$ on $(\cU_+)_\cZ$.
\end{prop*}
\begin{proof}
  Applying Lemma \ref{ss:parahoricstructure}.iii) to the opposite Borel shows that the multiplication morphism $\cU_+\times_\cO\cT\times_\cO\cU_-\ra\cK$ is also an open embedding. Write $\cX^{-1}$ for the smooth affine scheme $\cU_+\times_\cO\cT\times_\cO\cU_-$ over $\cO$, and write $\cC$ for the intersection of $\cX$ with $\cX^{-1}$. Note that the birational group law $\cX\times_\cO\cX\dashrightarrow\cX$ equals
  \begin{align*}
    \cX\times_\cO\cX &= \cU_-\times_\cO(\cT\times_\cO\cU_+\times_\cO\cU_-)\times_\cO(\cT\times_\cO\cU_+) \\
                     &\cong \cU_-\times_\cO(\cU_+\times_\cO\cT\times_\cO\cU_-)\times_\cO(\cT\times_\cO\cU_+) \\
                     &\supseteq\cU_-\times_\cO\cC\times_\cO(\cT\times_\cO\cU_+) \\
                     &\subseteq\cU_-\times_\cO(\cU_-\times_\cO\cT\times_\cO\cU_+)\times_\cO(\cT\times_\cO\cU_+) \\
                     &\cong(\cU_-\times_\cO\cU_-)\times_\cO(\cT\times_\cO\cT)\times_\cO(\cU_+\times_\cO\cU_+)\\
    &\ra\cU_-\times_\cO\cT\times_\cO\cU_+=\cX,
  \end{align*}
where the isomorphisms are induced by the action of $\cT$ on $\cU_+$, and the last morphism is induced by the group laws on $\cU_\pm$ and $\cT$ over $\cO$. After base changing to $\cZ$, this yields the desired result.
\end{proof}

\subsection{}\label{ss:closefieldscongruence}
We now introduce our setup of close fields. Let $n$ be a positive integer such that the image of $I^n$ in $\Gal(F/E)$ is trivial. Let $E'$ be a nonarchimedean local field along with an isomorphism $\Tr_l(E)\cong\Tr_l(E')$ for some $l\geq n$, which induces an isomorphism $\Tr_n(E)\cong\Tr_n(E')$. Using our isomorphism $\Tr_n(E)\cong\Tr_n(E')$, we obtain a canonical isomorphism $\Ga/I^n\cong\Ga'/I'^n$ of topological groups (up to conjugation) \cite[(3.5.1)]{Del84}.

Write $F'/E'$ for the finite Galois extension corresponding to $F/E$ via the isomorphism $\Ga/I^n\cong\Ga'/I'^n$. We have an isomorphism $\Gal(F/E)\cong\Gal(F'/E')$, so descending $(G^{\s}_{F'},T^{\s}_{F'},B^{\s}_{F'},\{x_{\wt{a}}\}_{\wt{a}\in\De})$ along $\Spec{F'}\ra\Spec{E'}$ via $\de$ yields a quasisplit connected reductive group $G'$ over $E'$ with a pinning $(B',T',\{x'_{\wt{a}}\}_{\wt{a}\in\De})$ over $E'$.

Write $S'$ for the maximal split subtorus of $T'$. Note that we have natural isomorphisms $X^*(T)\cong X^*(T')$ and $X^*(S)\cong X^*(S')$ that preserve the absolute and relative root systems, respectively, as well as the positive roots. By using our pinnings $(B,T,\{x_{\wt{a}}\}_{\wt{a}\in\De})$ and $(B',T',\{x'_{\wt{a}}\}_{\wt{a}\in\De})$ as basepoints, this induces a canonical isomorphism $\cA(S)\cong\cA(S')$. Write $x'$ for the image of $x$ under the isomorphism $\cA(S)\cong\cA(S')$, and write $\cK'$ for the unique smooth affine model of $G'$ over $\cO'$ such that $\cK'(\cO')\subseteq G'(E')$ equals $G'(E')^1_{x'}$.

Write $\psi:[-1,\infty)\ra[-1,\infty)$ for the Hasse--Herbrand function associated with $F/E$ as in \cite[Chap. IV, \S3]{Ser79}, and write $\ups$ for the inertia degree of $F/E$. Recall that our isomorphism $\Tr_l(E)\cong\Tr_l(E')$ induces an isomorphism $\Tr_{\psi(l)}(F)\cong\Tr_{\psi(l)}(F')$ \cite[(3.4.1)]{Del84}. Because $\psi(x)\to\infty$ as $x\to\infty$, for large enough $l$ this yields an isomorphism $\cO_L/\fp_L^{n\ups_L}\cong\cO_{L'}/\fp_{L'}^{n\ups_L}$ for all subextensions $L/E$ in $F$ with corresponding subextension $L'/E'$ in $F'$, where $\ups_L$ denotes the inertia degree of $L/E$.

\subsection{}\label{ss:parahoriccongruence}
For the rest of this section, assume that $l$ is large enough as in \ref{ss:closefieldscongruence}. The following is our main congruence result for disconnected parahoric subgroups.
\begin{thm*}
For large enough $l$, the isomorphism $\Tr_l(E)\cong\Tr_l(E')$ induces an isomorphism $\cK_{\cO/\fp^n}\cong\cK'_{\cO'/\fp'^n}$ of groups over $\cO/\fp^n\cong\cO'/\fp'^n$.
\end{thm*}
\begin{proof}
Assume that $l$ is large enough as in the proof of \cite[(9.4)]{CY01}. Then \cite[(9.2)]{CY01} indicates that our isomorphism $\Tr_l(E)\cong\Tr_l(E')$ induces an isomorphism $\cT_{\cO/\fp^n}\cong\cT'_{\cO'/\fp'^n}$ of groups over $\cO/\fp^n\cong\cO'/\fp'^n$. Under this identification, \cite[(9.4)]{CY01} implies that, for all $\wt{a}$ in $X^*(T)\cong X^*(T')$, the square of groups over $\cO/\fp^n\cong\cO'/\fp'^n$
\begin{align*}
  \xymatrix{\cT_{\cO/\fp^n}\ar@{=}[r]^-\sim\ar[d]^-{\wt{a}_{\cO/\fp^n}} & \cT'_{\cO'/\fp'^n}\ar[d]^-{\wt{a}_{\cO'/\fp'^n}}\\
  \R_{(\cO_{\wt{a}}/\fp_{\wt{a}}^{n\ups_{\wt{a}}})/(\cO/\fp^n)}\bG_m\ar@{=}[r]^-\sim & \R_{(\cO_{\wt{a}}'/\fp_{\wt{a}}'^{n\ups_{\wt{a}}})/(\cO'/\fp'^n)}\bG_m}
  \end{align*}
commutes, where $\ups_{\wt{a}}$ denotes the inertia degree of $E_{\wt{a}}/E$.
  
  Let $a$ be in $\Phi_{\red}$. In the two cases of \ref{ss:parahoricrootgroups}, the proof of \cite[Lemma 4.7]{Gan19} shows that the isomorphism $\Tr_l(E)\cong\Tr_l(E')$ induces:
  \begin{enumerate}[R1:]
  \item If $2a$ does not lie in $\Phi$, an isomorphism $(\wt\cU_a)_{\cO_{\wt{a}}/\fp_{\wt{a}}^{n\ups_{\wt{a}}}}\cong(\wt\cU'_a)_{\cO'_{\wt{a}}/\fp_{\wt{a}}'^{n\ups_{\wt{a}}}}$ of groups over $\cO_{\wt{a}}/\fp_{\wt{a}}^{n\ups_{\wt{a}}}\cong\cO'_{\wt{a}}/\fp_{\wt{a}}'^{n\ups_{\wt{a}}}$ such that the actions of $(\bG_m)_{\cO_{\wt{a}}/\fp_{\wt{a}}^{n\ups_{\wt{a}}}}\cong(\bG_m)_{\cO'_{\wt{a}}/\fp_{\wt{a}}'^{n\ups_{\wt{a}}}}$ agree,
  \item If $2a$ lies in $\Phi$, an isomorphism $(\wt\cU_a)_{\cO_{\wt{b}}/\fp_{\wt{b}}^{n\ups_{\wt{b}}}}\cong(\wt\cU_a')_{\cO_{\wt{b}}'/\fp_{\wt{b}}'^{n\ups_{\wt{b}}}}$ of groups over $\cO_{\wt{b}}/\fp_{\wt{b}}^{n\ups_{\wt{b}}}\cong\cO_{\wt{b}}'/\fp_{\wt{b}}'^{n\ups_{\wt{b}}}$ such that the actions of
    \begin{align*}
      \R_{(\cO_{\wt{a}}/\fp_{\wt{a}}^{n\ups_{\wt{a}}})/(\cO_{\wt{b}}/\fp_{\wt{b}}^{n\ups_{\wt{b}}})}\bG_m\cong\R_{(\cO_{\wt{a}}'/\fp_{\wt{a}}'^{n\ups_{\wt{a}}})/(\cO_{\wt{b}}'/\fp_{\wt{b}}'^{n\ups_{\wt{b}}})}\bG_m
    \end{align*}
    agree, where $\ups_{\wt{b}}$ denotes the inertia degree of $E_{\wt{b}}/E$.
  \end{enumerate}
Therefore Lemma \ref{ss:parahoricrootgroups} yields an isomorphism $(\cU_a)_{\cO/\fp^n}\cong(\cU'_a)_{\cO'/\fp'^n}$ of groups over $\cO/\fp^n\cong\cO'/\fp'^n$ such that the actions of $\cT_{\cO/\fp^n}\cong\cT'_{\cO'/\fp'^n}$ agree. By Proposition \ref{ss:parahoricbirational}, this induces an isomorphism $\cX_{\cO/\fp^n}\cong\cX'_{\cO'/\fp'^n}$ of birational group laws over $\cO/\fp^n\cong\cO'/\fp'^n$. Now $\cK_{\cO/\fp^n}$ and $\cK'_{\cO'/\fp'^n}$ are solutions as in \cite[5.1/2]{BLR90} of the birational group laws $\cX_{\cO/\fp^n}$ and $\cX'_{\cO'/\fp'^n}$, respectively, so the uniqueness of solutions \cite[5.1/3]{BLR90} induces an isomorphism $\cK_{\cO/\fp^n}\cong\cK'_{\cO'/\fp'^n}$ of groups over $\cO/\fp^n\cong\cO'/\fp'^n$. This yields the desired result.
\end{proof}

\subsection{}\label{ss:valuationsection}
At this point, we begin preparations for our Hecke algebra isomorphism over close fields. Write $\breve{E}$ for the completion of the maximal unramified extension of $E$, and write $\breve\cO$ for its ring of integers. Write $\bF_q$ for the residue field of $E$, and write $\vp:\breve{E}\ra\breve{E}$ for the lift of absolute $q$-Frobenius.

Choose a uniformizer $\pi_F$ of $F$. Choose a $\bZ$-basis $\mu_1,\dotsc,\mu_r$ of the $\bZ$-torsionfree quotient $X_*(T)_{I,\tf}$ of $X_*(T)_I$, and choose representatives $\wt\mu_1,\dotsc,\wt\mu_r$ in $X_*(T)$ of the $\mu_1,\dotsc,\mu_r$. Sending $\mu_j\mapsto\Nm_{\breve{F}/\breve{E}}\wt\mu_j(\pi_F)$ for all $1\leq j\leq r$ yields a homomorphism
\begin{align*}
\breve\na:X_*(T)_{I,\tf}\ra T(\breve{E}).
\end{align*}
Note that $\breve\na$ is a section of the valuation map $v:T(\breve{E})\ra X_*(T)_{I,\tf}$, and because $v$ is surjective with kernel $T(\breve{E})^1$, this induces an isomorphism
\begin{align*}
  X_*(T)_{I,\tf}\times T(\breve{E})^1\ra^\sim T(\breve{E}).
\end{align*}

Lemma \ref{ss:parahoricstructure}.i) indicates that $T(\breve{E})^1$ equals $\cT(\breve\cO)$. Write $\breve{T}^n$ for the kernel of $\cT(\breve\cO)\ra\cT(\breve\cO/\fp^n)$, write $T^n$ for the kernel of $\cT(\cO)\ra\cT(\cO/\fp^n)$, and write $K^n$ for the kernel of $\cK(\cO)\ra\cK(\cO/\fp^n)$. Note that $\breve{T}^n$, $T^n$, and $K^n$ remain unchanged if we replace the groups $\cG$ over $\cO$ with their relative neutral components $\cG^\circ$, since $\cG^\circ$ is open in $\cG$, and the only open subscheme of $\Spec\cO$ containing $\Spec\cO/\fp^n$ is $\Spec\cO$.

\subsection{}\label{ss:toruscomparison}
We want to compare elements of $T(E)$ with elements of $T'(E')$ in a way that is compatible with Theorem \ref{ss:parahoriccongruence}, so we proceed as follows.

Choose a uniformizer $\pi_{F'}$ of $F'$ whose image in $\cO_{F'}/\fp_{F'}^{n\ups}\cong\cO_F/\fp_F^{n\ups}$ equals the image of $\pi_F$. Using the image of the $\wt\mu_1,\dotsc,\wt\mu_r$ in $X_*(T')\cong X_*(T)$, we obtain a homomorphism $\breve\na':X_*(T')_{I',\tf}\ra T'(\breve{E}')$ as in \ref{ss:valuationsection} and hence an isomorphism $X_*(T')_{I',\tf}\times \cT'(\breve\cO')\ra^\sim T'(\breve{E}')$.

For the rest of this section, assume that $l$ is large enough as in Theorem \ref{ss:parahoriccongruence}. Then \cite[(9.2)]{CY01} indicates that our isomorphism $\Tr_l(E)\cong\Tr_l(E')$ induces an isomorphism $\cT_{\cO/\fp^n}\cong\cT'_{\cO'/\fp'^n}$ of groups over $\cO/\fp^n\cong\cO'/\fp'^n$; write $\tau$ for the resulting composition
\begin{align*}
T(\breve{E})/\breve{T}^n\lea^\sim X_*(T)_{I,\tf}\times\cT(\breve\cO/\fp^n)\cong X_*(T')_{I',\tf}\times\cT'(\breve\cO'/\fp'^n)\ra^\sim T'(\breve{E}')/\breve{T}'^n.
\end{align*}
The proof of \cite[Lemma 2.5]{Gan24} shows that $\tau$ intertwines the actions of $\vp$ and $\vp'$, so \cite[Corollary B.10.14]{KP23} implies that taking $\vp$-invariants and $\vp'$-invariants yields an isomorphism $\tau:T(E)/T^n\ra^\sim T'(E')/T'^n$.

\subsection{}\label{ss:stabilizercongruence}
Note that the action of $\cK(\cO)\times\cK(\cO)$ on $K^n\bs G(E)/K^n$ by left and right translation factors through $\cK(\cO)\times\cK(\cO)\ra\cK(\cO/\fp^n)\times\cK(\cO/\fp^n)$.

Our comparison from \ref{ss:toruscomparison} satisfies the following compatibility with Theorem \ref{ss:parahoriccongruence}. Let $t$ be in $T(E)$, and choose $t'$ in $T'(E')$ whose image in $T'(E')/T'^n\cong T(E)/T^n$ equals the image of $t$.
\begin{lem*}
Under the identification $\cK(\cO/\fp^n)\cong\cK'(\cO'/\fp'^n)$ from Theorem \ref{ss:parahoriccongruence}, the stabilizer of $K^ntK^n$ in $\cK(\cO/\fp^n)\times\cK(\cO/\fp^n)$ equals the stabilizer of $K'^nt'K'^n$ in $\cK'(\cO'/\fp'^n)\times\cK'(\cO'/\fp'^n)$.
\end{lem*}
\begin{proof}
  Let $k$ and $k'$ be elements of $\cK(\cO)$ and $\cK'(\cO')$, respectively, whose images in $\cK'(\cO'/\fp'^n)\cong\cK(\cO/\fp^n)$ coincide. We claim that if $t^{-1}kt$ lies in $\cK(\cO)$, then $t'^{-1}k't'$ lies in $\cK'(\cO')$, and their images in $\cK'(\cO'/\fp'^n)\cong\cK(\cO/\fp^n)$ also coincide. By the Bruhat decomposition, it suffices to check this in the following two cases:
  \begin{enumerate}[$\bullet$]
  \item If $k$ and $k'$ lie in $\cT(\cO)$ and $\cT'(\cO')$, respectively, then this follows from $T(E)$ being commutative.
  \item If $k$ and $k'$ lie in $\cU_a(\cO)$ and $\cU'_a(\cO')$, respectively, for some $a$ in $\Phi_{\red}$, then this follows from Lemma \ref{ss:parahoricrootgroups}.
  \end{enumerate}
  
  Next, note that the map $k\mapsto(k,t^{-1}kt)$ induces an isomorphism
  \begin{align*}
    \io:\big(\cK(\cO)\cap t\cK(\cO)t^{-1}\big)\big/\big(K^n\cap tK^nt^{-1}\big)\ra^\sim\stab_{\cK(\cO/\fp^n)\times\cK(\cO/\fp^n)}(K^ntK^n).
  \end{align*}
We have an analogous isomorphism
  \begin{align*}
    \io':\big(\cK'(\cO')\cap t'\cK'(\cO')t'^{-1}\big)\big/\big(K'^n\cap t'K'^nt'^{-1}\big)\ra^\sim\stab_{\cK'(\cO'/\fp'^n)\times\cK'(\cO'/\fp'^n)}(K'^nt'K'^n),
  \end{align*}
  and the claim implies that the images of $\io$ and $\io'$ in
  \begin{align*}
    \cK(\cO/\fp^n)\times\cK(\cO/\fp^n)\cong\cK'(\cO'/\fp'^n)\times\cK'(\cO'/\fp'^n)
  \end{align*}
  coincide. This yields the desired result.
\end{proof}

\subsection{}\label{ss:liftsofcombinatorics}
Let us recall the Cartan decomposition for disconnected parahoric subgroups, as well as some of its consequences for Hecke algebras. Choose a $\bZ$-basis $\nu_1,\dotsc,\nu_s$ of $(X_*(T)_{I,\tf})^\vp$, and choose lifts $t_1,\dotsc,t_s$ to $T(E)$ of the $\nu_1,\dotsc,\nu_s$. Sending $\nu_j\mapsto t_j$ for all $1\leq j\leq s$ yields a homomorphism $\na:(X_*(T)_{I,\tf})^\vp\ra T(E)$. By construction, $\na$ is a section of the valuation map $v:T(E)\ra(X_*(T)_{I,\tf})^\vp$.

Let $\La$ be a ring. Consider the (noncommutative) $\La$-algebra $\cH(G(E),K^n)_\La$ as in Definition \ref{ss:Heckealgebra}, and for all $g$ in $G(E)$, write $h_g$ in $\cH(G(E),K^n)_\La$ for the indicator function on $K^ngK^n$. As $g$ runs over $K^n\bs G(E)/K^n$, the $h_g$ form a $\La$-basis of $\cH(G(E),K^n)_\La$.
\begin{lem*}\hfill
  \begin{enumerate}[i)]
  \item The group $G(E)$ equals the disjoint union
    \begin{align*}
      \coprod_\nu\cK(\cO)\na(\nu)\cK(\cO),
    \end{align*}
    where $\nu$ runs over $(X_*(T)_{I,\tf})^{\vp,+}$.
  \item For all $\nu$ and $\mu$ in $(X_*(T)_{I,\tf})^{\vp,+}$, we have $h_{\na(\nu)}*h_{\na(\mu)}=h_{\na(\nu+\mu)}$.
  \item For all $\nu$ in $(X_*(T)_{I,\tf})^{\vp,+}$ and $k$ and $j$ in $\cK(\cO)$, we have
    \begin{align*}
      h_k*h_{\na(\nu)}*h_j=h_{k\cdot\na(\nu)\cdot j}.
    \end{align*}
  \end{enumerate}
\end{lem*}
\begin{proof}
  Since $x$ is a special point, part i) follows from the proof of \cite[Theorem 5.2.1]{KP23}. Next, note that $\cH(G(E),K^n)_\La$ is base changed from the $\La=\bZ$ case, so it suffices to consider $\La=\bZ$. Then part i) and Lemma \ref{ss:parahoricstructure}.iii) imply that $\De\coloneqq\cK(\cO)$, $\De_0\coloneqq K^n$, and $\mathbf{D}\coloneqq\na((X_*(T)_{I,\tf})^{\vp,+})$ satisfy \cite[Condition H-1]{All72} and \cite[Condition H-2]{All72}. Therefore part ii) and part iii) follow from \cite[Theorem 2]{All72}.
\end{proof}

\subsection{}\label{ss:doublecosetcongruence}
At this point, we can compare double coset spaces over close fields as follows. Using $t'_1,\dotsc,t'_s$ in $T'(E')$ whose images in $T'(E')/T'^n\cong T(E)/T^n$ equal the images of the $t_1,\dotsc,t_s$, we obtain a homomorphism $\na':(X_*(T')_{I',\tf})^{\vp'}\ra T'(E')$ as in \ref{ss:liftsofcombinatorics}.
\begin{prop*}
We have a natural bijection $K^n\bs G(E)/K^n\cong K'^n\bs G'(E')/K'^n$.
\end{prop*}
\begin{proof}
  Let $\nu$ be in $(X_*(T)_{I,\tf})^{\vp,+}\cong(X_*(T')_{I',\tf})^{\vp',+}$. Note that the diagram
  \begin{align*}
    \xymatrix{(X_*(T)_{I,\tf})^\vp\ar[r]^-{\na}\ar@{=}[d]^-{\rotatebox{90}{$\sim$}} & T(E)\ar[r] & T(E)/T^n\ar@{=}[d]^-{\rotatebox{90}{$\sim$}}\\
    (X_*(T')_{I',\tf})^{\vp'}\ar[r]^-{\na'} & T'(E')\ar[r] & T'(E')/T'^n}
  \end{align*}
  commutes by construction, so Lemma \ref{ss:stabilizercongruence} indicates that the stabilizer of $K^n\na(\nu)K^n$ and the stabilizer of $K'^n\na'(\nu)K'^n$ in
  \begin{align*}
    \cK(\cO/\fp^n)\times\cK(\cO/\fp^n)\cong\cK'(\cO'/\fp'^n)\times\cK'(\cO'/\fp'^n)
  \end{align*}
  coincide. This induces a natural bijection between the set of $K^n$-double cosets lying in $\cK(\cO)\na(\nu)\cK(\cO)$ and the set of $K'^n$-double cosets lying in $\cK'(\cO')\na'(\nu)\cK'(\cO')$. As $\nu$ varies, Lemma \ref{ss:liftsofcombinatorics}.i) yields the desired result.
\end{proof}

\subsection{}\label{ss:Heckealgebrahomomorphism}
Proposition \ref{ss:doublecosetcongruence} and \ref{ss:liftsofcombinatorics} induce a natural isomorphism of $\La$-modules
\begin{align*}
\eta_n:\cH(G(E),K^n)_\La\ra\cH(G'(E'),K'^n)_\La.
\end{align*}

To prove that $\eta_n$ is an isomorphism of $\La$-algebras, we need the following result. Let $C\subseteq(X_*(T)_{I,\tf})^{\vp,+}$ be a finite subset, and write $G(E)_C$ for $\coprod_{\nu\in C}\cK(\cO)\na(\nu)\cK(\cO)$. View $C$ as a subset of $(X_*(T')_{I',\tf})^{\vp',+}\cong(X_*(T)_{I,\tf})^{\vp,+}$.
\begin{lem*}
For large enough $l$, we have $\eta_n(h_1*h_2)=\eta_n(h_1)*\eta_n(h_2)$ for all $h_1$ and $h_2$ in $\cH(G(E),K^n)_\bC$ supported on $G(E)_C$.
\end{lem*}
\begin{proof}
  Because $K^n$ remains unchanged if we replace $\cK$ with its relative neutral component $\cK^\circ$, this follows from \cite[Lemma 4.6]{Gan24}.
\end{proof}

\subsection{}\label{thm:Heckealgebracongruence}
Finally, we prove our Hecke algebra isomorphism over close fields.
\begin{thm*}
For large enough $l$, the map $\eta_n:\cH(G(E),K^n)_\La\ra\cH(G'(E'),K'^n)_\La$ is an isomorphism of $\La$-algebras.
\end{thm*}
\begin{proof}
Note that $\eta_n$ is base changed from the $\La=\bZ$ case, so it suffices to consider $\La=\bZ$. Because $\bZ\ra\bC$ is injective, the $\La=\bZ$ case follows from that of $\La=\bC$.

  Therefore assume that $\La=\bC$. Let $k_1,\dotsc,k_a$ be a set of representatives in $\cK(\cO)$ of $\cK(\cO/\fp^n)$, and let $\nu^+_1,\dotsc,\nu^+_b$ be generators of the monoid $(X_*(T)_{I,\tf})^{\vp,+}$. Then Lemma \ref{ss:liftsofcombinatorics}.ii) and Lemma \ref{ss:liftsofcombinatorics}.iii) indicate that the $h_{k_1},\dotsc,h_{k_a},h_{\na(\nu^+_1)},\dotsc,h_{\na(\nu^+_b)}$ generate $\cH(G(E),K^n)_\bC$ over $\bC$. Now \cite[2.13]{Ber84} and \cite[3.4]{Ber84} imply that $\cH(G(E),K^n)_\bC$ is finitely presented over $\bC$, so there exist $f_1,\dotsc,f_r$ in the noncommutative polynomial ring $\bC\ang{X_1,\dotsc,X_{a+b}}$ such that the surjective $\bC$-algebra homomorphism
  \begin{align*}
    \bC\ang{X_1,\dotsc,X_{a+b}}\ra\cH(G(E),K^n)_\bC
  \end{align*}
induced by the $h_{k_1},\dotsc,h_{k_a},h_{\na(\nu^+_1)},\dotsc,h_{\na(\nu^+_b)}$ has kernel $(f_1,\dotsc,f_r)$.

Write $D$ for $\max_{1\leq j\leq r}\deg f_j$. By compactness, there exists a finite subset $C\subseteq(X_*(T)_{I,\tf})^{\vp,+}$ such that $G(E)_C$ contains the $D$-fold product of $G(E)_{\{0,\nu_1^+,\dotsc,\nu_b^+\}}$ in $G(E)$. Assume that $l$ is large enough for $C$ as in Lemma \ref{ss:Heckealgebrahomomorphism}.ii). Since every at most $D$-fold product of the $h_{k_1},\dotsc,h_{k_a},h_{\na(\nu_1^+)},\dotsc,h_{\na(\nu_b^+)}$ is supported on $G(E)_C$, Lemma \ref{ss:Heckealgebrahomomorphism}.ii) shows that
  \begin{align*}
    &\quad f_j(\eta_n(h_{k_1}),\dotsc,\eta_n(h_{k_a}),\eta_n(h_{\na(\nu^+_1)}),\dotsc,\eta_n(h_{\na(\nu^+_b)}))\\
    &= \eta_n(f(h_{k_1},\dotsc,h_{k_a},h_{\na(\nu^+_1)},\dotsc,h_{\na(\nu^+_b)})) = 0
  \end{align*}
  for all $1\leq j\leq r$. Hence the $\bC$-algebra homomorphism
  \begin{align*}
    \bC\ang{X_1,\dotsc,X_{a+b}}\ra\cH(G'(E'),K'^n)
  \end{align*}
  induced by the $\eta_n(h_{k_1}),\dotsc,\eta_n(h_{k_a}),\eta_n(h_{\na(\nu^+_1)}),\dotsc,\eta_n(h_{\na(\nu^+_b)})$ factors through a $\bC$-algebra homomorphism $\vt:\cH(G(E),K^n)_\bC\ra\cH(G'(E'),K'^n)_\bC$.

  For all $1\leq i\leq a$, let $k_i'$ be an element of $\cK'(\cO')$ whose image in $\cK'(\cO'/\fp'^n)\cong\cK(\cO/\fp^n)$ equals the image of $k_i$. Note that $\eta_n(h_{k_i})$ equals $h'_{k'_i}$. Similarly, for all $1\leq m\leq b$, we have $\eta_n(h_{\na(\nu_m^+)})=h'_{\na'(\nu_m^+)}$. Therefore Lemma \ref{ss:liftsofcombinatorics}.ii) and Lemma \ref{ss:liftsofcombinatorics}.iii) indicate that $\vt$ equals $\eta_n$, so $\eta_n$ is a homomorphism of $\bC$-algebras. Finally, we already saw in \ref{ss:Heckealgebrahomomorphism} that $\eta_n$ is a bijection.
\end{proof}

\bibliographystyle{../habbrv}
\bibliography{biblio}

\def\cprime{$'$}
\begin{thebibliography}{10}

\bibitem{All72}
N.~D. Allan.
\newblock Hecke rings of congruence subgroups.
\newblock {\em Bull. Amer. Math. Soc.}, 78:541--545, 1972.

\bibitem{AV24}
A.-M. {Aubert} and S.~{Varma}.
\newblock {On congruent isomorphisms for tori}.
\newblock {\em arXiv e-prints}, page arXiv:2401.08306, Jan. 2024, 2401.08306.

\bibitem{Ber84}
J.~N. Bernstein.
\newblock Le ``centre'' de {B}ernstein.
\newblock In {\em Representations of reductive groups over a local field},
  Travaux en Cours, pages 1--32. Hermann, Paris, 1984.
\newblock Edited by P. Deligne.

\bibitem{BBDG18}
A.~A. Be\u{\i}linson, J.~Bernstein, and P.~Deligne.
\newblock Faisceaux pervers.
\newblock In {\em Analysis and topology on singular spaces, {I} ({L}uminy,
  1981)}, volume 100 of {\em Ast\'{e}risque}, pages 5--171. Soc. Math. France,
  Paris, 2nd edition edition, 2018.

\bibitem{BMS18}
B.~Bhatt, M.~Morrow, and P.~Scholze.
\newblock Integral {$p$}-adic {H}odge theory.
\newblock {\em Publ. Math. Inst. Hautes \'{E}tudes Sci.}, 128:219--397, 2018.

\bibitem{Bor70}
A.~Borel.
\newblock Properties and linear representations of {C}hevalley groups.
\newblock In {\em Seminar on {A}lgebraic {G}roups and {R}elated {F}inite
  {G}roups ({T}he {I}nstitute for {A}dvanced {S}tudy, {P}rinceton, {N}.{J}.,
  1968/69)}, volume Vol. 131 of {\em Lecture Notes in Math.}, pages 1--55.
  Springer, Berlin-New York, 1970.

\bibitem{BLR90}
S.~Bosch, W.~L\"{u}tkebohmert, and M.~Raynaud.
\newblock {\em N\'{e}ron models}, volume~21 of {\em Ergebnisse der Mathematik
  und ihrer Grenzgebiete (3) [Results in Mathematics and Related Areas (3)]}.
\newblock Springer-Verlag, Berlin, 1990.

\bibitem{CY01}
C.-L. Chai and J.-K. Yu.
\newblock Congruences of {N}\'{e}ron models for tori and the {A}rtin conductor.
\newblock {\em Ann. of Math. (2)}, 154(2):347--382, 2001.
\newblock With an appendix by Ehud de Shalit.

\bibitem{Con14}
B.~Conrad.
\newblock Reductive group schemes.
\newblock In {\em Autour des sch\'{e}mas en groupes. {V}ol. {I}}, volume 42/43
  of {\em Panor. Synth\`eses}, pages 93--444. Soc. Math. France, Paris, 2014.

\bibitem{Del84}
P.~Deligne.
\newblock Les corps locaux de caract\'{e}ristique {$p$}, limites de corps
  locaux de caract\'{e}ristique {$0$}.
\newblock In {\em Representations of reductive groups over a local field},
  Travaux en Cours, pages 119--157. Hermann, Paris, 1984.

\bibitem{DMOS82}
P.~Deligne, J.~S. Milne, A.~Ogus, and K.-y. Shih.
\newblock {\em Hodge cycles, motives, and {S}himura varieties}, volume 900 of
  {\em Lecture Notes in Mathematics}.
\newblock Springer-Verlag, Berlin-New York, 1982.

\bibitem{Dri76}
V.~G. Drinfel{\cprime}d.
\newblock Coverings of {$p$}-adic symmetric domains.
\newblock {\em Funkcional. Anal. i Prilo\v{z}en.}, 10(2):29--40, 1976.

\bibitem{FF18}
L.~Fargues and J.-M. Fontaine.
\newblock Courbes et fibr\'{e}s vectoriels en th\'{e}orie de {H}odge
  {$p$}-adique.
\newblock {\em Ast\'{e}risque}, (406):xiii+382, 2018.
\newblock With a preface by Pierre Colmez.

\bibitem{FS21}
L.~{Fargues} and P.~{Scholze}.
\newblock {Geometrization of the local Langlands correspondence}.
\newblock {\em arXiv e-prints}, page arXiv:2102.13459, Jan. 2024, 2102.13459.

\bibitem{Fon77}
J.-M. Fontaine.
\newblock {\em Groupes {$p$}-divisibles sur les corps locaux}, volume No. 47-48
  of {\em Ast\'{e}risque}.
\newblock Soci\'{e}t\'{e} Math\'{e}matique de France, Paris, 1977.

\bibitem{GR03}
O.~Gabber and L.~Ramero.
\newblock {\em Almost ring theory}, volume 1800 of {\em Lecture Notes in
  Mathematics}.
\newblock Springer-Verlag, Berlin, 2003.

\bibitem{GHSBP21}
W.~T. {Gan}, M.~{Harris}, W.~{Sawin}, and R.~{Beuzart-Plessis}.
\newblock {Local parameters of supercuspidal representations}.
\newblock {\em arXiv e-prints}, page arXiv:2109.07737, June 2022, 2109.07737.

\bibitem{Gan19}
R.~Ganapathy.
\newblock Congruences of parahoric group schemes.
\newblock {\em Algebra Number Theory}, 13(6):1475--1499, 2019.

\bibitem{Gan24}
R.~{Ganapathy}.
\newblock {A Hecke algebra isomorphism over close local fields}.
\newblock Mar. 2024.
\newblock Revised version communicated to the author by Radhika Ganapathy.

\bibitem{GL17}
A.~{Genestier} and V.~{Lafforgue}.
\newblock {Chtoucas restreints pour les groupes r{\'e}ductifs et
  param{\'e}trisation de Langlands locale}.
\newblock {\em arXiv e-prints}, page arXiv:1709.00978, Sept. 2017, 1709.00978.

\bibitem{HKW17}
D.~Hansen, T.~Kaletha, and J.~Weinstein.
\newblock On the {K}ottwitz conjecture for local shtuka spaces.
\newblock {\em Forum Math. Pi}, 10:Paper No. e13, 79, 2022.

\bibitem{Hub94}
R.~Huber.
\newblock A generalization of formal schemes and rigid analytic varieties.
\newblock {\em Math. Z.}, 217(4):513--551, 1994.

\bibitem{Hub96}
R.~Huber.
\newblock {\em \'{E}tale cohomology of rigid analytic varieties and adic
  spaces}.
\newblock Aspects of Mathematics, E30. Friedr. Vieweg \& Sohn, Braunschweig,
  1996.

\bibitem{KP23}
T.~Kaletha and G.~Prasad.
\newblock {\em Bruhat-{T}its theory---a new approach}, volume~44 of {\em New
  Mathematical Monographs}.
\newblock Cambridge University Press, Cambridge, 2023.

\bibitem{Kaz86b}
D.~Kazhdan.
\newblock Representations of groups over close local fields.
\newblock {\em J. Analyse Math.}, 47:175--179, 1986.

\bibitem{KL15}
K.~S. Kedlaya and R.~Liu.
\newblock Relative {$p$}-adic {H}odge theory: foundations.
\newblock {\em Ast\'{e}risque}, (371):239, 2015.

\bibitem{Kot05}
R.~E. Kottwitz.
\newblock Harmonic analysis on reductive {$p$}-adic groups and {L}ie algebras.
\newblock In {\em Harmonic analysis, the trace formula, and {S}himura
  varieties}, volume~4 of {\em Clay Math. Proc.}, pages 393--522. Amer. Math.
  Soc., Providence, RI, 2005.

\bibitem{Kra50}
M.~Krasner.
\newblock Quelques m\'{e}thodes nouvelles dans la th\'{e}orie des corps
  valu\'{e}s complets.
\newblock In {\em Alg\`ebre et {T}h\'{e}orie des {N}ombres}, volume no. 24 of
  {\em Colloq. Internat. CNRS}, pages 29--39. CNRS, Paris, 1950.

\bibitem{Lan02}
S.~Lang.
\newblock {\em Algebra}, volume 211 of {\em Graduate Texts in Mathematics}.
\newblock Springer-Verlag, New York, third edition, 2002.

\bibitem{LH23}
S.~D. {Li-Huerta}.
\newblock {Local-global compatibility over function fields}.
\newblock {\em arXiv e-prints}, page arXiv:2301.09711, Aug. 2023, 2301.09711.

\bibitem{Lu17}
J.~Lurie.
\newblock {Higher Algebra}.
\newblock \url{https://www.math.ias.edu/~lurie/papers/HA.pdf}, 2017.

\bibitem{MM94}
S.~Mac~Lane and I.~Moerdijk.
\newblock {\em Sheaves in geometry and logic}.
\newblock Universitext. Springer-Verlag, New York, 1994.
\newblock A first introduction to topos theory, Corrected reprint of the 1992
  edition.

\bibitem{MP19}
M.~Mishra and B.~Pattanayak.
\newblock A note on depth preservation.
\newblock {\em J. Ramanujan Math. Soc.}, 34(4):393--400, 2019.

\bibitem{RR96}
M.~Rapoport and M.~Richartz.
\newblock On the classification and specialization of {$F$}-isocrystals with
  additional structure.
\newblock {\em Compositio Math.}, 103(2):153--181, 1996.

\bibitem{Sch14}
P.~Scholze.
\newblock Perfectoid spaces and their applications.
\newblock In {\em Proceedings of the {I}nternational {C}ongress of
  {M}athematicians---{S}eoul 2014. {V}ol. {II}}, pages 461--486. Kyung Moon Sa,
  Seoul, 2014.

\bibitem{Sch17}
P.~{Scholze}.
\newblock {Etale cohomology of diamonds}.
\newblock {\em arXiv e-prints}, page arXiv:1709.07343, Jan. 2022, 1709.07343.

\bibitem{SW13}
P.~Scholze and J.~Weinstein.
\newblock Moduli of {$p$}-divisible groups.
\newblock {\em Camb. J. Math.}, 1(2):145--237, 2013.

\bibitem{SW20}
P.~Scholze and J.~Weinstein.
\newblock {\em Berkeley {L}ectures on $p$-adic {G}eometry}, volume 207 of {\em
  Annals of Mathematics Studies}.
\newblock Princeton University Press, Princeton, NJ, 2020.

\bibitem{Ser79}
J.-P. Serre.
\newblock {\em Local fields}, volume~67 of {\em Graduate Texts in Mathematics}.
\newblock Springer-Verlag, New York-Berlin, 1979.
\newblock Translated from the French by Marvin Jay Greenberg.

\bibitem{stacks-project}
{T}he {S}tacks~{P}roject {A}uthors.
\newblock {S}tacks {P}roject.
\newblock \url{http://stacks.math.columbia.edu}, 2024.

\bibitem{Vig96}
M.-F. Vign\'{e}ras.
\newblock {\em Repr\'{e}sentations {$l$}-modulaires d'un groupe r\'{e}ductif
  {$p$}-adique avec {$l\ne p$}}, volume 137 of {\em Progress in Mathematics}.
\newblock Birkh\"{a}user Boston, Inc., Boston, MA, 1996.

\end{thebibliography}
\end{document}